\documentclass[12pt,leqno]{report}

\usepackage{graphicx, amsfonts, amsthm, amsxtra, amssymb, verbatim, makeidx}
\usepackage{subeqnarray, relsize}
\usepackage[mathscr]{euscript}
\usepackage{hyperref, tikz-cd}
\usepackage{aliascnt}
\usepackage[nameinlink,capitalize,noabbrev]{cleveref}
\hypersetup{
    colorlinks=true,       
    linkcolor=blue,          
    citecolor=blue,        
    filecolor=blue,      
    urlcolor=blue           
}

\newtheorem{theo}{Theorem}[chapter]
\newtheorem{coro}{Corollary}[chapter]

\newtheorem{prop}{Proposition}[chapter]
\newtheorem{conj}{Conjecture}[chapter]

\theoremstyle{remark}
\newtheorem{rem}{Remark}[chapter]

\theoremstyle{definition}
\newtheorem{defi}{Definition}[chapter]
\newtheorem{exam}{Example}[chapter]

\begin{document}

\def\U{{\mathcal U}}   
\def\res{{\rm res}}
\def\cf{r}   

\def\Nn{{\sf n}}
\def\Am{{\sf A}}

\def\Z{\mathbb{Z}}                   
\def\Q{\mathbb{Q}}                   
\def\C{\mathbb{C}}                   
\def\N{\mathbb{N}}                   
\def\Ff{\mathbb{F}}                  
\def\uhp{{\mathbb H}}                
\def\A{\mathbb{A}}                   
\def\dR{{\rm dR}}                    
\def\F{{\mathcal F}}                     
\def\Sp{{\rm Sp}}                    
\def\Gm{\mathbb{G}_m}                 
\def\Ga{\mathbb{G}_a}                 
\def\Tr{{\rm Tr}}                      
\def\tr{{{\mathsf t}{\mathsf r}}}                 
\def\spec{{\rm Spec}}            
\def\proj{{\rm Proj}}
\def\ker{{\rm ker}}              
\def\GL{{\rm GL}}                

\def\k{{\sf k}}                     
\def\ring{{\sf R}}                   
\def\sk{{\mathfrak k }}             
\def\sring{{\mathfrak R }}          

\def\X{{\sf X}}                      
\def\T{{\sf T}}                      
\def\V{{    V}}                   

\def\Ts{{\sf S}}
\def\cmv{{\sf M}}                    
\def\BG{{\sf G}}                       
\def\podu{{\sf pd}}                   
\def\ped{{\sf U}}                    
\def\per{{\sf  P}}                   
\def\gm{{\sf  A}}                    
\def\gma{{\sf  B}}                   
\def\ben{{\sf b}}                    

\def\Rav{{\mathfrak M }}                     
\def\Ram{{\mathcal C}}                         
\def\Rap{{i(\Lie(\BG))}}                    

\def\nov{{  n}}                    
\def\mov{{  m}}                    
\def\Yuk{{\sf Y}}                     
\def\Ra{{\sf R}}                      

\def\Da{{\sf D}}                      

\def\hn{{\sf h}}                      
\def\cpe{{\sf C}}                     
\def\g{{\sf g}}                       
\def\t{{   t}}                       
\def\v{{   v}}                       

\def\pedo{{\sf  \Pi}}                  

\def\Der{{\rm Der}}                   
\def\MMF{{\sf MF}}                    
\def\codim{{\rm codim}}                
\def\dim{{\rm    dim}}                
\def\Lie{{\rm Lie}}                   
\def\gg{{\mathfrak  g}}                

\def\u{{\sf u}}                       

\def\imh{{  \Psi}}                 
\def\imc{{  \Phi }}                  
\def\stab{{\rm Stab }}               
\def\Vec{{\Theta}}                 
\def\prim{{0}}                  
\def\Zero{{\rm Zero}}                  

\def\Fg{{\sf F}}     
\def\hol{{\rm hol}}  
\def\non{{\rm non}}  
\def\alg{{\rm alg}}  
\def\an{{\rm an}}   
\def\for{{\rm for}}  

\def\bcov{{\rm \O_\T}}       

\def\leaves{{\mathcal L}}        

\def\Hse{{\rm HS}}        
\def\Hpo{{\rm HP}}        
\def\Hfu{{\rm HF}}        
\def\Hsc{{\rm Hilb}}     

\def\TS{\mathlarger{{\bf T}}}                
\def\IS{\mathlarger{{\mathcal I}}}                

\def\vf{{\sf v}}                      
\def\wf{{\sf w}}                      

\def\red{{\rm red}}                           

\def\Ua{{   L}}                      
\def\plc{{ Z_\infty}}    

\def\gru{\mu} 
\def\pg{{ \sf S}}               
\def\group{{ G}}            

\def\GM{{\rm GM}}

\def\perr{{\sf q}}        
\def\perdo{{\mathcal K}}   
\def\sfl{{\mathrm F}} 
\def\sp{{\mathbb S}}  

\newcommand\diff[1]{\frac{d #1}{dz}} 
\def\End{{\rm End}}              

\def\sing{{\rm Sing}}            
\def\cha{{\rm char}}             
\def\Gal{{\rm Gal}}              
\def\jacob{{\rm jacob}}          
\def\tjurina{{\rm tjurina}}      
\newcommand\Pn[1]{\mathbb{P}^{#1}}   
\def\P{\mathbb{P}}                   
\def\Ff{\mathbb{F}}                  

\def\O{{\mathcal O}}                     
\def\as{\mathbb{U}}                  
\def\ring{{\mathsf R}}                         
\def\R{\mathbb{R}}                   

\newcommand\ep[1]{e^{\frac{2\pi i}{#1}}}
\newcommand\HH[2]{H^{#2}(#1)}        
\def\Mat{{\rm Mat}}              
\newcommand{\mat}[4]{
     \begin{bmatrix}
            #1 & #2 \\
            #3 & #4
       \end{bmatrix}
    }                                
\newcommand{\matt}[2]{
     \begin{bmatrix}                 
            #1   \\
            #2
       \end{bmatrix}
    }
\def\cl{{\rm cl}}                

\def\hc{{\mathsf H}}                 
\def\Hb{{\mathcal H}}                    
\def\pese{{\sf P}}                  

\def\PP{\tilde{\mathcal P}}              
\def\K{{\mathbb K}}                  

\def\M{{\mathcal M}}
\def\RR{{\mathcal R}}
\newcommand\Hi[1]{\mathbb{P}^{#1}_\infty}
\def\pt{\mathbb{C}[t]}               
\def\W{{\mathcal W}}                     
\def\gr{{\rm Gr}}                
\def\Im{{\rm Im}}                
\def\Re{{\rm Re}}                
\def\depth{{\rm depth}}
\newcommand\SL[2]{{\rm SL}(#1, #2)}    
\def\sl{{\rm SL}}                    
\newcommand\PSL[2]{{\rm PSL}(#1, #2)}  
\def\Resi{{\rm Resi}}              

\def\L{{\mathcal L}}                     
\def\Aut{{\rm Aut}}              
\def\any{R}                          
\newcommand\ovl[1]{\overline{#1}}    

\newcommand\mf[2]{{M}^{#1}_{#2}}     
\newcommand\mfn[2]{{\tilde M}^{#1}_{#2}}     

\newcommand\bn[2]{\binom{#1}{#2}}    
\def\ja{{\rm j}}                 
\def\Sc{\mathsf{S}}                  
\newcommand\es[1]{g_{#1}}            
\newcommand\WW{{\mathsf W}}          
\newcommand\Ss{{\mathcal O}}             
\def\rank{{\rm rank}}                
\def\Dif{{\mathcal D}}                   
\def\gcd{{\rm gcd}}                  
\def\zedi{{\rm ZD}}                  
\def\BM{{\mathsf H}}                 
\def\plf{{\sf pl}}                             
\def\sgn{{\rm sgn}}                      
\def\diag{{\rm diag}}                   
\def\hodge{{\rm Hodge}}
\def\HF{{\sf F}}                                
\def\WF{{\sf W}}                               
\def\HV{{\sf HV}}                                
\def\pol{{\rm pole}}                               
\def\bafi{{\sf r}}
\def\Id{{\rm Id}}                               
\def\gms{{\sf M}}                           
\def\Iso{{\rm Iso}}                           

\def\hl{{\rm L}}    
\def\imF{{\rm F}}
\def\imG{{\rm G}}

\def\HL{{\rm Ho}}     
\def\NLL{{\rm NL}}   

\def\RG{{\bf G}}          
\def\rg{{\bf g}}     
\def\rbullet{{\cdot}}
\def\Ld{{\mathcal L}}      
\def\Ro{{\rm R}}     
\def\ZS{{\rm ZeSc}}     
\def\ZI{{\rm ZeId}}     
 \def\integ{{\rm Int}}  

\def\tmap{{\sf t}}

\def\ivhs{{\rm IVHS}}    
\def\ivhsmaps{{{\Delta}_{}}}   
\def\sch{{\rm Sch}}   
\def\mk{{\mathfrak  m}}   
\def\pk{{\mathfrak  p}}   
\def\qk{{\mathfrak  q}}   

\newcommand\licy[1]{{\mathbb P}^{#1}} 

\begin{center}
{\LARGE\bf Leaf schemes and Hodge loci\footnote{\today}
}
\\
\vspace{.1in} {\large {\sc Hossein Movasati}}
\footnote{
Instituto de Matem\'atica Pura e Aplicada, IMPA, Estrada Dona Castorina, 110, 22460-320, Rio de Janeiro, RJ, Brazil,
{\tt \href{http://w3.impa.br/~hossein/}{www.impa.br/$\sim$hossein}, hossein@impa.br.}}
\end{center}
\newpage

\begin{figure}[t]
\begin{center}
\includegraphics[width=0.6\textwidth]{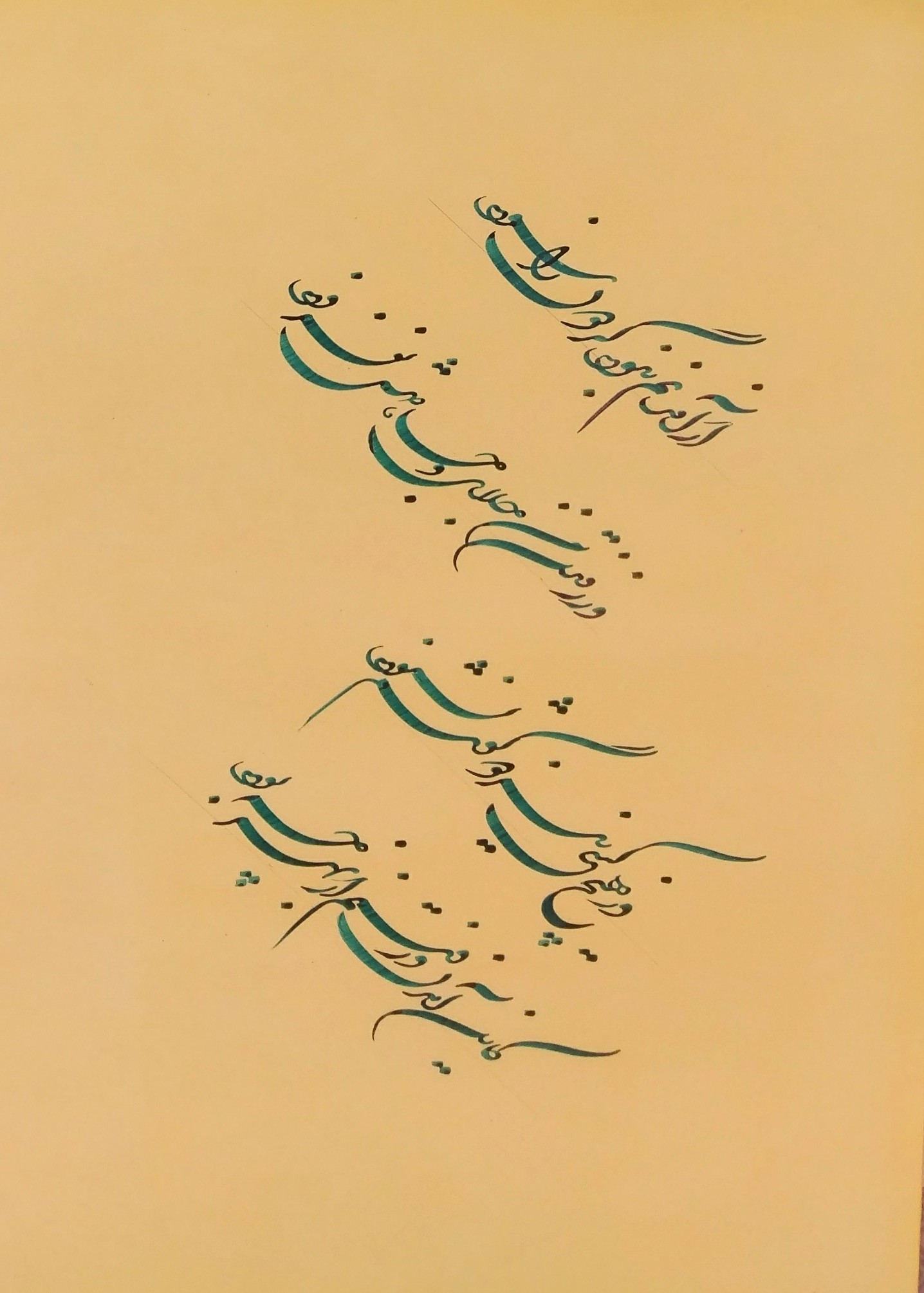}
\end{center}
\end{figure}

{
Omar Khayyam,  Robaiyat 60.
Calligraphy:  Tahereh Aladpoosh \\
There was no use of my coming to the world,\\
neither it was glorified from my going,\\
I never heared sombody saying, \\
for what reason I came and went.  \\
(Author's translation)
}

\newpage

{\it Preface:}
This is a collection of articles, written as chapters, on arithmetic properties of differential equations, holomorphic foliations, Gauss-Manin connections, and Hodge loci.
Each chapter is independent from the others and it  has its own abstract and introduction, and the reader might get an insight to the text by reading the introduction of each chapter.
The main connection between them is through comments in the footnotes. 
Our major aim is to develop a theory of leaf schemes over finitely generated subrings of complex numbers, such that the leaves are also equipped with a scheme structure. 
We also aim to formulate a local-global conjecture for leaf schemes.
\cref{18112023restraintinghappiness} and \cref{10082025sharr} are of an expository nature, although the author was unable to find many materials in these chapters elsewhere. \cref{29082024omidint} is purely computational and experimental and it can be understood with a basic knowledge of polynomial manipulations modulo primes.  \cref{Ramanujan27062024} is mixed expository and research level. The reader with interests on elliptic curves can go directly to this chapter.   \cref{29082024tanha} explains the current status of few conjectures developed in the author's book \cite{ho13}. \cref{21112023bihude} is a research level text which aims to put the content of \cref{29082024tanha} in the framework of local-global principles in arithmetic algebraic geometry.

\chapter*{Introduction}
The analogies  and relations between complex and arithmetic algebraic geometry have occupied the attention of mathematicians for more than a century. The most concrete example of this is Deligne's theory of mixed Hodge structures in \cite{de71-0} mainly inspired by results in $p$-adic cohomologies. In the same time, linear differential equations under the modern name of connections on vector bundles, have been investigated  as a branch of algebraic geometry, and such analogies and relations have been enlarged to this context. The most famous example is the Grothendieck-Katz $p$-curvature conjecture  for algebraicity of solutions of linear differential equations and local systems. It was suggested by  Grothendieck and popularized by Katz in \cite{Katz1970, Katz1972}. Whereas Deligne's theory is based upon analogies between cohomology theories, and consequently there is no concrete bridge   between  complex and $p$-adic Hodge theories,  Grothendieck-Katz conjecture is a concrete statement about the modulo primes properties of linear differential equations and its consequence on their solutions over complex numbers. One of the main goals of the present text is to put the Grothendieck-Katz conjecture into a larger framework of leaf schemes  and Hodge loci and investigate a local-global principle for such objects. Our study aims to provide  a new approach to the   Cattani-Deligne-Kaplan theorem in \cite{cadeka} on the algebraicity of Hodge loci by a local analysis of Taylor series of periods, whereas the original proof uses the global analysis of the monodromy group.
We aim to introduce modulo primes tools for proving  some conjectures in \cite[Chapter 19]{ho13} and \cite{ho2019}. The first one predicts the existance of certain Hodge loci for  cubic hypersurfaces of dimension $6$ and $8$ and the second one provides a conjectural counter example to a conjecture of  J. Harris for special components of Noether-Lefschetz loci in the case of octic surfaces.


Beyond linear differential equations, we have  the theory of holomorphic foliations in complex manifolds which has its origin  in the second part of Hilbert's 16th problem on the uniform boundedness of number of limit cycles. In the last decades, it has also been investigated as a branch of algebraic geometry. Despite this,  considering foliations on non-reduced schemes or with non-reduced leaves might seem abstract non-sense.
In the present  text we develop a theory in which a holomorphic foliation  is identified with its module of differential $1$-forms (which may not be saturated) and so it is not just the underlying geometric object. Moreover, its leaves might be non-reduced with different codimenions for which we use the name leaf scheme.
The main examples of leaf schemes are Hodge loci.
The origin of  the theory of leaf schemes goes back to \cite{GMCD-NL} and \cite[Chapter 5,6]{ho2020} and we follow and further develop these texts.
For experts in holomorphic foliations the theory of holomorphic foliations in which leaves are replaced with leaf schemes might seem an abstract non-sense. This is mainly due to the lack of motivation and simple examples. We are mainly motivated by applications to Hodge loci, even though the corresponding foliations cannot be written explicitly.

\section*{Summary of the text}
Each chapter is independent from the others and it  has its own abstract and introduction. The best way to
get an insight to the text is by reading the introduction of each chapter. In order to achieve this we have allowed ourselves to repeat definitions and notations if this makes the reading smoother.
The connection between chapters  is through comments in footnotes. In the following I will give a summary of the present text based on the author's way of thinking and logical dependence. There is no mathematical reason for the order of chapters.  The current order reflects the timeline of each chapter and  the fact that once the author decided to name a chapter by a number, no changes were done afterwards. 

The birth of the present text started from \cref{29082024tanha} and the desire to introduce local-global principles for Hodge loci. In particular, the starting goal was to analyse the following conjectures thorough modulo prime reductions:
\begin{conj}
\label{03112025bimsa-1}
Let $T$ be the parameter space of smooth hypersurfaces of degree $3$ in ${\mathbb P}^{7}_{\mathbb C}$. Let also $L$ be the locus of hypersurfaces containing two linear spaces $Z_1\cong {\mathbb P}^3$ and $Z_2\cong {\mathbb P}^3$ such that $Z_1\cap Z_2$ is a point. We can verify easily that the codimension of $L$ is $8$. Fix $r\in\Q$ with $r\not=0$. The codimension of the Hodge locus containing $L$ and parametrizing deformations of hypersurfaces $X$ in $L$ such that the cohomology class of $Z_1+rZ_2$ remains a Hodge class (that is, it remains of type $(3,3)$) is $7$.
\end{conj}
\begin{conj}
\label{03112025bimsa-2}
Let $T$ be the parameter space of smooth hypersurfaces of degree $3$ in ${\mathbb P}^{9}_{\mathbb C}$. Let also $L$ be the locus of hypersurfaces containing two linear spaces $Z_1\cong {\mathbb P}^4$ and $Z_2\cong {\mathbb P}^4$ such that $Z_1\cap Z_2$ is a line. We can verify easily that the codimension of $L$ is $20$. Fix $r\in\Q$ with $r\not=0$. The codimension of the Hodge locus containing $L$ and parametrizing deformations of hypersurfaces $X$ in $L$ such that the cohomology class of $Z_1+rZ_2$ remains a Hodge class (that is, it remains of type $(4,4)$) is $19$.
\end{conj}

\begin{conj}
\label{03112025bimsa-3}
Let $T$ be the parameter space of smooth surfaces of degree $8$ in ${\mathbb P}^{3}_{\mathbb C}$. Let also $L$ be the locus of surfaces containing two curves $Z_1$ and $Z_2$ such that $Z_1$ is a complete intersection of type $(3,3)$ in ${\mathbb P}^{3}_{\mathbb C}$ and $Z_2$ is a line and $Z_1\cap Z_2=\emptyset$.
The Noether-Lefschetz loci attached to the cohomology classes of $Z_1+r Z_2$ for $r\in\Q,\ r\not=0$ are distinct $31$ codimensional subvarieties intersecting 
each other  in a $32$ codimensional subvariety of $T$.  
\end{conj}
The main goal of \cref{29082024tanha} is to gather all the evidence that \cref{03112025bimsa-1} and \cref{03112025bimsa-2} must be true. \cref{03112025bimsa-3} is mainly discussed in \cite{ho2019}. The main common feature of these conjectures is that they are supported not only by some theoretical arguments but also by heavy computer calculations  for many days and even a month. 
The main ingredient has been the computation of the Taylor series  of the elements of the ideal of a Hodge locus with coefficients in $\Q(\zeta_{2d})$, where $d$ is the degree of the underlying hypersurface. For the expression of such Taylor series see \cref{InLabelNadasht?}. Using computer implementation of these Taylor series  we can prove that \cref{03112025bimsa-1}, \cref{03112025bimsa-2} and  \cref{03112025bimsa-3} are true up to $14$-order, $6$-th order and $5$-order approximation of the underlying Hodge locus, respectively.  When this chapter was written more evidences came out of Kloosterman's work \cite{Kloosterman2023}. 

In general we know that the Taylor series of periods have coefficients in a finitely generated subring $\sring\subset \C$ tensored with rational numbers, see \cref{05082024retaliation}.  The appearance of primes in the denominator of coefficients is very slow, however, a priori all primes might appear in the denominators of such coefficients soon or later, see  \cref{05aug2024delaram}. The situation for the periods defining a Hodge locus is different. By Cattani-Deligne-Kaplan theorem, we know that a Hodge locus is algebraic. After considering a Hodge locus in an enlarged parameter space, see \cref{26062024omidbita}, we observe that for the ideal defining a Hodge locus we only need to invert  finitely many primes, see \cref{21june2024baran}. This takes us to consider reduction modulo $p$ of Hodge loci. 

In \cref{08082024ls} we define the notion of a leaf scheme and in \cref{04dec2023khar} we realize a Hodge locus as a leaf scheme of a foliation. In this way, the study of Hodge loci and its algebraicity can be studied in the general framework of foliations. A leaf scheme is a concept more general than the classical notion of a leaf in the literature. For instance, its underlying analytic variety can be a subset of the singular set of the foliation.  The content of \cref{21112023bihude} is the highlight of the present text. In this chapter, we investigate a local-global principle for leaf schemes, see \cref{02112023mainconj}. The main motivation behind this, is actually the integrality of coefficients of a leaf scheme discussed in \cref{19062024kontsevich}. We can rewrite this conjecture in the framework of classical leaves of foliations.  Even in the case of  foliations in $\P^2_\C$, this kind of statements is widely open, and we only mention the following particular case.  
\begin{conj}
\label{04112025dodokhtarbimsa}
Let us consider the differential equation 
 $$
 \frac{\partial y}{\partial x}=\frac{P(x,y)}{Q(x,y)},\ \ P,Q\in\Z[x,y],\ 
Q(0,0)\not=0,\ 
 $$
 If the Taylor series of its solution $y(x)=\sum _{n=0}y_nx^n$ through $(0,0)$ is defined over $\Z[\frac{1}{N}]$ for some $N\in\N$, that is all coefficients $y_n$ are in $\Z[\frac{1}{N}]$,  then $y(x)$ is algebraic, that is, there is a polynomial $P\in\C[X,Y]$ such that $P(x,y(x))=0$. 
\end{conj}
We may even improve the hypothesis of this conjecture saying that  that there is $N\in\N$ such that $N^ny_n\in\Z$ for all $n\in\N_0$. When the the present text was  being written, the author knew of \cite{LamLitt2025}, in which the authors announce  a similar conjecture and prove it for isomonodromic foliations of Gauss-Manin connections. The algebraicity criterion \cref{19062024kontsevich} seems to be the most general criterion available as it also includes the algebraicity of Hodge loci, and in order to make advertisement for it we have reproduced it in many different parts of the present text: \cref{04112025dodokhtarbimsa} for foliations in $\P^2_\C$, for linear differential equations over $\P^1_\C$  which is the converse of \cref{19102023chinagain-2} and it is stated  in comments after this theorem, for vector fields in \cref{11july2024autoescola} and \cref{11072024graphofscience}. From all these, one can derive local-global conjectures, among which the Grothendieck-Katz conjecture is the most famous one.    
In a similar way we have written the most general local-global principle studied in this article. This is namely \cref{02112023mainconj}  which is stated for leaf schemes. In  the case of linear differential equations or connections on vector bundles this reduces to   the Grothendieck-Katz $p$-curvature conjecture.

As the integrality of solutions of differential equations is much stronger than the vanishing of $p$-curvature, in \cref{18112023restraintinghappiness} we raise some doubts about the $p$-curvature conjecture. \cref{19102023chinagain}, \cref{19102023chinagain-2} say that if a solution of a linear differential equation is algebraic then  certain quantities derived from this solution divided by factorials still live in some finitely generated  rings, and hence all but a finite number of primes cancel from the denominator. We formulate this property in terms of the vanishing of the so called $m_{p,k}$-curvature. For $m_{p,1}$ we recover the classical vanishing of $p$-curvature. Converses of \cref{19102023chinagain}, \cref{19102023chinagain-2} are conjectures, see for instance \cref{05112023tina} and the notes after \cref{19102023chinagain-2}, and we explain them for the following simple class of linear differential equations: 
\begin{equation}
\label{07112025jinmohamedlucas}
\frac{\partial Y}{\partial z}=\left( \frac{A_0}{z}+\frac{A_1}{z-1}+\frac{A_t}{z-t}\right)Y, 
\end{equation}
where $t\in\Z,\ t\not=0,1$ and $A_0,A_1,A_t$'s are $n\times n$ matrices with entries in $\Z$. Let $\Am(z)$ be the matrix in parentheses. We define $\Am_n$ recursively by $\Am_1=\Am,\ \ \Am_{n+1}=\frac{\partial \Am_n}{\partial z}+\Am_n\Am$.  The matrix $z^n(z-1)^n(z-t)^n\Am_n$ has entries in $\Z[z]$, and \cref{19102023chinagain} says that if all the solutions of \eqref{07112025jinmohamedlucas} are algebraic then there is some $N\in\N$ such that for all $n\in\N$ the matrix $z^n(z-1)^n(z-t)^nN^n\frac{\Am_n}{n!}$ has entries in $\Z[z]$, that is, all except a finite number of primes are canceled from the denominator of $\frac{\Am_n}{n!}$.  \cref{05112023tina} is just the converse of this statement. The Eisenstein criterion for algebraicity of holomorphic functions in many variables is stated in \cref{29062024ajib} and it seems to appear for the first time in print.

In order to find an algebraicity criterion beyond solutions of linear differential equations, we introduce the notion of an integral function and in \cref{09122023sonymuseum} we prove that algebraic functions are integral. The converse statement seems to be related to Grothendieck-Katz conjecture, or its generalizations, but we were not able to make a direct link.

The content of \cref{18112023restraintinghappiness} leads us to believe that the Grothendieck-Katz conjecture might be false. It does not matter whether this belief is based upon tiny evidences. What matters is that this belief forces us to start computing $p$-curvature of many interesting differential equations, where we might be able to find a counterexample. After analyzing carefully the last status of well-known results on $p$-curvature conjecture, in \cref{29082024omidint}  we study the $p$-curvature for  Lamé equations. In \cref{26062024kamar-2},  we find Lamé equations for which for a single prime the $p$-curvature is zero but the $m_{p,k}$-curvature is not zero.  The main outcome of this chapter is a bunch of Lamé equations for which the number of primes $p$  for which $p$-curvature is zero seems to be more than non-zero $p$-curvatures. This leads us to introduce the $p$-curvature density for differential equations, for which there is no literature.     

Linear differential equations in the name of connections on vector bundles are well-studied in the framework of algebraic geometry, however, vector fields $\vf$ in algebraic varieties $\T$ are not studied as ordinary  differential equations as they are mainly defined as sections of tangent bundles. For this reason, in  \cref{10082025sharr} we have collected basic properties of vector fields as differential equations.  In this chapter we have collected many similar topics as in \cref{18112023restraintinghappiness}. We first prove the fundamental theorem of ordinary differential equations, namely \cref{5jan2015}. The algebraicity of a solution of a vector a field $\vf$ manifests itself as cancelation of all but a finite number of primes in $\frac{\vf^m}{m!}$. This is stated in \cref{12112023rik-2} and its converse \cref{11072024graphofscience} is a generalization of \cref{04112025dodokhtarbimsa}.  The local-global principle  for vector fields in \cref{02112023mainconj-vv} is the natural output of all these discussions.  
After discussing the relation of few classical  concepts such as isogeny of elliptic curves, Hasse-Witt invariant and p-curvature with vector fields, we introduce  Poincar\'e linearization theorem for vector fields in \cref{30102025shanghai}.

One of the chapters which the author enjoyed so much writing it, as he learned many classical mathematics, rediscovered  some of them and even discovered some new statements, is \cref{Ramanujan27062024}. It is about a single vector field called Ramanujan vector field. Its rich properties is due to the fact that it lives on a moduli space of elliptic curves $\T$ which is called  "ibiporanga", see \cref{21may2024raisimord}. The reader who is  annoyed with the mathematical nonsense of \cref{21112023bihude}, should  first read this chapter as its content contains many classical and new things related to elliptic curves. The most striking phenomena in this chapter is the loci of parameters for which $\vf^p=\vf$ for a fixed prime. Although this loci only exists in characteristic $p$, it has very useful consequences in the integrality of the coefficients of modular forms. The main goal of this chapter to prepare the ground for such integrality properties for Calabi-Yau modular forms introduced in \cite{GMCD-MQCY3, HosseinMurad}. Another important  novelty in this chapter is the characterization of CM elliptic curves using the Ramanujan vector field in \cref{22102025mauricio}.

\section*{Acknowledgement:}
Since the first messy draft of the present text was written in 2023, it was in the author's webpage for sharing ideas. Throughout the text, there appears the name of many mathematicians who have made comments regarding parts of the text.  Email, online and personal communications with them have improved the text a lot.
The reader after a PDF search of their names can find the place of their contributions.
This includes:  João Pedro dos Santos, Jorge V. Pereira, Yves Andr\'e, H\'el\`ene Esnault, Wadim Zudilin, Stefan Reiter, Frits Beukers, 
Daniel Litt, Joshua Lam, Nick Katz, David Urbanik, Felipe Voloch, Frederico Bianchini,  Pierre Deligne, Remke Kloosterman, Roberto Villaflor, Florian Fürnsinn.  
My sincere thanks go to all of them.



\tableofcontents

\def\sinone{{ W}}
\def\Smat{{\sf S}}
\def\hn{{\sf h}} 
\def\IS{\mathlarger{{\mathcal I}}}      


\chapter{Grothendieck-Katz conjecture}
\label{18112023restraintinghappiness}

{\it
Man sollte weniger danach streben, die Grenzen der mathematischen Wissenschaften zu erweitern, als vielmehr danach,
den bereits vorhandenen Stoff aus umfassenderen Gesichtspunkten zu betrachten, \href{https://openscience.ub.uni-mainz.de/bitstream/20.500.12030/3484/1/930.pdf}{E. Study}, (see the preface of \cite{Edwards1990}).
}
\\
\\

{\it Abstract:} In this article we prove that linear differential equations with only algebraic solutions
have zero $m$-curvature modulo $p^k$ for all except a finite number of primes $p$ and all $k,m\in\N$ with
${\rm ord}_pm!\geq  k$. This provides us with a reformulation of Grothendieck-Katz conjecture  with stronger hypothesis.

\section{Introduction}
\label{27012024tanhaperu}
Let us consider a linear differential equation:\index{$\Am$, matrix of linear differential equation}
\begin{equation}
\label{06122023dingxin}
\frac{\partial y}{\partial z}=\Am(z) y,
\end{equation}
where $\Am=\Am(z)$ is an $\Nn\times \Nn$ matrix with entries which are rational functions in $z$ and coefficients in $\C$. We are looking for holomorphic solutions of this differential equation. These are $\Nn\times 1$ matrices $y$ whose entries are holomorphic functions in an open set $U$ in $\C$. Let $\Delta\in\C[z]$ be the common denominator of the entries of $\Am$, and so the entries of $\Delta\Am$ are in $\C[z]$. We take  $\sring\subset\C$ any finitely generated  $\Z$-algebra generated by the coefficients of $\Delta$ and $\Delta\Am$.\index{$\sring$, a finitely generated subring of $\C$}
It is easy to see that $y^{(n)}=\Am_ny$, where $\Am_n$ are recursively computed by
$$
\Am_1=\Am,\ \ \Am_{n+1}=\frac{\partial \Am_n}{\partial z}+\Am_n\Am.
$$
We can easily check that $\Delta^n\Am_n$ has entries in $\sring[z]$.\index{$\Am_n$, a matrix obtained from $\Am$}
\begin{theo}
\label{19102023chinagain}
If all the entries of solutions of $\frac{\partial y}{\partial z}=\Am y$ are algebraic functions over $\C(z)$, that is, they are in the algebraic closure $\overline{\C(z)}$ of $\C(z)$ then there is $N\in\sring$ such that
\begin{equation}
\label{06032024bitavisa}
\hbox{ entries of } \frac{(N\Delta)^m\cdot \Am_m}{m!}\  \in\  \sring[z],\ \ \forall m\in \N.
\end{equation}
In particular, for all primes $p$ and $k,m\in\N$ with ${\rm ord}_pm!\geq  k$ 
we have  $\Am_m\equiv_{p^k} 0$, that is, $\Am_m$ is zero in the ring $\sring[z, \frac{1}{N\Delta}]/p^k\sring[z,\frac{1}{N\Delta}]$.
\end{theo}
If $\sring$ is a ring of integers of a number field, in the above theorem we can take $N\in\N$, see \cref{27042024frankfurt}. \index{$N$, a natural number containg bad primes}
This is natural as in the left hand side of \eqref{06032024bitavisa} only integers are inverted and not 
an element of $\sring-\Z$.  The only reason that we assume that $\sring$ is the ring of integers of a number field is \cref{30012023espreafico}. 
The conclusion in the second part of \cref{19102023chinagain} is equivalent to $\Delta^m\Am_m\equiv_{p^k}0$  which means $\Delta^m\Am_m=0$ in the ring
$\sring[z,\frac{1}{N}]/p^k\sring[z,\frac{1}{N} ]$.
We may expect that the converse of \cref{19102023chinagain} holds.
\begin{conj}
\label{05112023tina}
For a linear differential equation \eqref{06122023dingxin} if there is $N\in\sring$ such that for all $m\in\N$ the entries of
$\frac{(N\Delta)^m\Am_m}{m!}$ are in $\sring[z]$
then  its solutions are algebraic functions over $\C(z)$.
\end{conj}
\href{https://math.stackexchange.com/questions/1735391/majoration-of-the-p-adic-valuation-of-a-factorial}{It is easy to verify that ${\rm ord}_pm!\leq \frac{m}{p-1}$ and so $m$ in the above conjecture satisfies $m\geq (p-1)k$.}
For a given prime and $k\in\N$ let $m_{p,k}$ be the smallest $m\in\N$ such that ${\rm ord}_pm!\geq  k$.  We have $m_{p,1}=p,\ m_{p,2}=2p$ and  $(p-1)k\leq m_{p,k}\leq pk$ and it is not hard to see that the conclusion of \cref{19102023chinagain}, and hence the hypothesis of \cref{05112023tina}, is equivalent to the same statement with $m=m_{p,k}$.

For a prime number $p$, the matrix $\Am_p$
considered as a matrix with entries in $\sring[z,\frac{1}{\Delta}]/p\sring[z,\frac{1}{\Delta}]$ is usually called the $p$-curvature of $\Am$ (there is no $N$ in the denominator).\index{$p$-curvature}
The statement in the second part of \cref{19102023chinagain} for $k=1$ and $m=p$ and with $N\in\N$ (see \cref{27042024frankfurt}) says that the $p$-curvature vanishes for all except a finite number of primes. Its converse, which is a stronger version of \cref{05112023tina}, is known as Grothendieck-Katz (p-curvature) conjecture.
 It was suggested by A. Grothendieck and popularized by N. Katz in \cite{Katz1970, Katz1972}.
Note that there is a finite number of primes which are invertible in $\sring[\frac{1}{N}]$ and the second statement
in the main theorem becomes empty for these primes. Therefore, the Grothendieck-Katz conjecture is: If for all but a finite number of primes the $p$-curvature of $y'=\Am y$ vanishes then all its solutions are algebraic. 
\begin{prop}
\label{10022024carnaval}
If Grothendieck-Katz conjecture is true for $y'=\Am y$ then the vanishing of $p$-curvature
 for all except a finite number of  primes $p$ implies the vanishing of $\Am_{m_{p,k}}$  modulo $p^k$ for all except a finite number of primes $p$ and all $k\in\N$.
\end{prop}
A random search, without no strategy, for finding a counterexample to the conclusion of \cref{10022024carnaval} for single prime is also not easy. That is, for for many differential equations the conclusion of \eqref{10022024carnaval} is true. However, we were able to find: 
\begin{prop}
\label{26062024kamar}
For the  Lamé differential equation
$$
(4z^3-1) \frac{d^2y}{dz^2}+6z^2\frac{dy}{dz}-\frac{7}{36}zy=0
$$
we have  $\Am_5\equiv_50$, however, $\Am_{25}\not \equiv_{5^6}0$.
\footnote{For the computer code which proves this proposition see \cref{26062024kamar-2}.}
\end{prop}
Therefore, any proof of the conclusion of \cref{10022024carnaval} without assuming the Grothendieck-Katz conjecture,  must use a mixed prime argument.
João Pedro dos Santos took my attention to the articles \cite{MatzatPut2003} in which the authors right at the end of the article makes the following conjecture: Let $F$ be a number field and $M$ be a differential module over $F(z)$.
If for all but a finite number of primes $p$, the iterative differential module with respect to $p$ exists and has a finite differential Galois group $G$, then the differential Galois group of $M$ is isomorphic to $G$. The property \eqref{06032024bitavisa} implies the existence of such an iterative differential module and so \cref{05112023tina} implies this conjecture.  A similar conjecture has been stated in \cite{KisinEsnault2018} using the language of $D$-modules: Let $X/\C$ be a smooth projective variety over $\C$ and $M=(E,\nabla)$ be  a vector bundle with an integrable connection on $X$.  If there is a dense open subscheme
$U\subset {\rm Spec}(\sring)$ such that for all closed points $s\in U$, $M_s$ underlies a $D_{X_s}$-module then $M$ has finite monodromy. Another similar conjecture has been advertised by M. Kontsevich in \cite{Kontsevich2023}.  Hopefully the reader will be convinced that the hypothesis of conjectures in  \cite{MatzatPut2003,  KisinEsnault2018, Kontsevich2023} are hard to check whereas \eqref{06032024bitavisa} can be even implemented in a computer and it is characteristic zero statement.   Moreover, we would like to emphasize that more divisibility properties are missing in the hypothesis of Grothendieck-Katz conjecture. Y. André in a personal communication took my attention to the works \cite{BombieriSperber1982, dw94, Andre2004}. The following is written based on his comments.
The $p$-adic radius of convergence of a full set of solutions of a linear differential equation is $\geq p^{-1/(p-1)}$, see \cite{BombieriSperber1982},  and if this is strict, for instance if the $p$-curvature is nilpotent,  then $\frac{\Am_m(z)}{m!}$  becomes  divisible by very high powers of $p$ for large $m$.
The vanishing of $p$-curvatures for a fixed $p$ implies that the generic radius $R_p$ of convergence satisfies $R_p> p^{-1/(p-1)}$ and the $\tau$-invariant is $0$ in this case, see \cite[Section 5]{Andre2004}.
H. Esnault informed me about \cite[Proposition 8.1, 8.2]{Esnault2020} in which the authors verify the Grothendieck-Katz conjecture for two partial cases. W. Zudilin has called phenomenon like \eqref{06032024bitavisa} the cancellation of factorials, see \cite{Zudilin2001}. He took my attention to \cite{Delaygue-Rivoal} in which the interest in algebraic solutions of rank one systems is attributed to Abel.

\section{Algebraic functions}
The following proposition for $\sring=\Z$  is due to G. Eisenstein, see \cite{Landau1904} and the references therein.
In the following $y(z)$ is a holomorphic function $y:U\to\C$ defined in some connected open subset $U$ of $\C$. \index{Eisenstein theorem}
\begin{prop}[G. Eisenstein]
\label{26012024margetifi?-2}
 Let $y(z)$ be an algebraic function and $z_0$ be in its definition domain. We have 
 \begin{enumerate}
 \item
The Taylor series of $y$ at $z_0\in U $ has  coefficients in a finitely generated  $\Z$-subalgebra of $\C$.
\item
More precisely, if the polynomial equation $P(z,y(z))=0$ is defined over a ring $\sring$ and $z_0,y(z_0)\in\sring$ then there is $N\in\sring$ such that the Taylor series of $y(Nz+z_0)$ at $z=0$  has coefficients in $\sring$, and hence $y(z)\in \sring[\frac{1}{N}][[z-z_0]]$.
\item
For $\sring, z_0,y(z_0)$ as in the previous item there is $N\in\sring$ such that
$$
\frac{N^my^{(m)}(z_0)}{m!}\in\sring,\ \ \ \forall m\in\N.
$$
\end{enumerate}
\end{prop}
\begin{proof}
It is clear that 2 and 3 are equivalent and that 1 is a particular case of 2. Therefore, we only prove 2.
 We write the Taylor series $y(z)=\sum_{i=0}^\infty y_i\cdot (z-z_0)^i$  of $y(z)$ at $z=z_0$, where $y_i$'s are unknown coefficients and substitute in $P(z,y(z))=0$. Let $\sring$ be the $\Z$-algebra generated by coefficients of $P$, $z_0$ and $y_0:=y(z_0)$. Let also $\Delta:=\frac{\partial P}{\partial y}(z_0,y_0)$.
 Computing the coefficient of $(z-z_0)^n$ we get a recursion of type
 $$
 \Delta y_n=\hbox{ a polynomial of degree $\leq n$ in $y_i,\ 1\leq i<n$,  with coefficients in $\sring$},
 $$
 for instance $\Delta y_1=-\frac{\partial P}{\partial z}(z_0,y_0)\in\sring$.
 This implies that $y_n$ lies in the finitely generated $\Z$-algebra $\sring[\frac{1}{\Delta}]$. More precisely, by induction we can show that $y_n$ has a pole order at most $2n-1$ at $\Delta$. For $n=1$ this is easy.  If this is true for all $m<n$ then $\Delta y_n$ is a sum of monomials $y_{i_1}y_{i_2}\cdots y_{i_k}$ with $i_1+i_2+\cdots+i_k\leq n$ and $1\leq i_1,i_2,\ldots,i_k<n$ and coefficients in $\sring$. If $k=1$ then $2i_1-1\leq 2n-2$ and we are done. If $k\geq 2$ then  $2i_1-1+2i_2-1\cdots+2i_k-1\leq 2n-k\leq 2n-2$ and we are done again.
For the second part of the proposition we put $N:=\Delta^2$.
\end{proof}
\begin{rem}\rm
\label{27jan2024toop}
If for a ring $\sring\subset\C$, the Taylor series $y(z)=\sum_{i=0}^\infty y_i\cdot (z-z_0)^i\in\sring[[z-z_0]]$ is algebraic over $\C(z)$ then
there is a polynomial $P\in\sring[z,y]$ such that $P(z,y(z))=0$. In order to see this we first take  $P\in\C[z,y]$ and regard the coefficients $P_i$ of $P$ as unknowns. The equalities derived from  $P(z, \sum_{i=0}^\infty y_i\cdot (z-z_0)^i)=0$ are linear equations in $P_i$ and with coefficients in $\sring$. They have solutions in $\C$ and so they have solutions in $\sring$.
\end{rem}
Let us now consider the differential equation $y'=\Am y$.
\begin{prop}
\label{17dec2023handcrafts}
 Let $\check\sring$ be a finitely generated $\Z$-subalgebra of $\C$ containing  the coefficients
 $\Am$, $z_0\in\C$, $\frac{1}{\Delta(z_0)}$ and the entries of $y_0=[y_{0,1},y_{0,2},\cdots,y_{0,n}]^\tr\in\C^\Nn$.
 A solution of $y'=\Am y$ with the initial condition $y(z_0)=y_0$ is a formal power series with
 coefficients in $\check\sring_\Q:=\check\sring\otimes_\Z\Q$:
 $$
 y(z)\in\check\sring_\Q[[(z-z_0)]].
 $$
\end{prop}
\begin{proof}
 We
can find recursively a unique formal solution $y=\sum_{i=0}^\infty y_n(z-z_0)^n\in\check\sring_\Q[[(z-z_0)]]^ \Nn$. The recursion is given by
$$
ny_{n-1}=\sum_{i=0}^{n-1}\Am_{n-1-i} y_i,
$$
where we have written the Taylor series $\Am=\sum_{i=0}^\infty \Am_i(z-z_0)^i$. Since the entries of $\Am$ are rational functions with coefficients in $\check\sring$, the entries of $\Am_i$ are in $\check\sring$. This follows from the fact that for a polynomial $P\in\check\sring[z]$ we have:
$$
\frac{P(z)}{\Delta(z)}=\frac{P(z)}{\Delta(z_0)}\frac{1}{1-(1-\frac{\Delta(z)}{\Delta(z_0)})}=
\frac{P(z)}{\Delta(z_0)}\sum_{i=0}^\infty \left(1-\frac{\Delta(z)}{\Delta(z_0)}\right )^i.
$$
\end{proof}

\begin{proof}[Proof of \cref{19102023chinagain}]
Let $\sring$ be as \cref{27012024tanhaperu}, that is, $\sring$ is a finitely generated $\Z$-subalgebra of $\C$ containing  the coefficients of $\Delta$ and the coefficients of  the polynomials in $\Delta(z)\Am$. By \cref{17dec2023handcrafts} we have a $\Nn\times\Nn$ matrix $Y(z)$ whose entries
 are formal power series in $z-z_0$ and with coefficients in the ring $\sring_\Q[z_0, \frac{1}{\Delta(z_0)}]$, $Y(z_0)$ is the identity matrix and
 $dY=\Am Y$. This is called the fundamental system at $z=z_0$.
 From another side we know that the entries of $Y$ are algebraic, and so by \cref{27jan2024toop} we can assume that the corresponding polynomials are defined over $\sring_\Q[z_0,\frac{1}{\Delta(z_0)}]$. A multiplication by a power of $\Delta(z_0)$ and possibly an integer,
 we can assume that such polynomials are defined over $\sring[z_0]$.  Now by \cref{26012024margetifi?-2} we know that the  coefficients of the Taylor series of $Y$ are in $\check\sring:=\sring[z_0,\frac{1}{N}]$ for some $N\in\sring[z_0]$. This $N$ is the product of all $N$'s attached to the polynomial equation of each entry of $Y$.
Since for any  $f=(z-z_0)^n$  
 we have
 $$
 f^{(m)}=\binom{n}{m} m!(z-z_0)^{n-m},\
 $$
 we conclude that $Y^{(m)}=0$ in  $(\check\sring/m!\check\sring)[[z-z_0]]$, and hence
 \begin{equation}
 \label{09022024airbnb}
\Am_mY=0 \hbox{  in } (\check\sring/m!\check\sring)[[z-z_0]].
\end{equation}
From now on take $z_0\in\sring$ and hence $\check\sring=\sring[\frac{1}{N}]$.
The number  $N$ in the announcement of theorem is exactly the number $N$ that we have obtained from Eisenstein theorem. We have $\check\sring/m!\check\sring=0$ if and only if
 \begin{equation}
 \label{06032024ss}
  m!| N^d,\hbox{ in $\sring$  for some  $d\in\N$}.
 \end{equation}
In this case \eqref{06032024bitavisa} is automatic.   
Therefore, we assume that $\check\sring/m!\check\sring\not=0$, and in particular $1\not=0$ in this ring.  In this ring $Y=I_{\Nn\times \Nn}+\sum_{i=1}^\infty y_i(z-z_0)^ i$ is invertible, that is the entries of $Y^{-1}$ are formal power series
in $(z-z_0)$ and coefficients in $\check\sring/m!\check\sring$. This together with \eqref{09022024airbnb} imply that  the Taylor series of $\Am_m(z)$, and hence $\Delta(z)^m\Am_m(z)$, at $z=z_0$,
 is zero in $\check \sring/m!\check\sring[[z-z_0]]$.  But $R(z):=\Delta(z)^m\Am_m(z)$ is a matrix of polynomials, and hence, its entries are in
 $m!\check\sring[z]$. This finishes the proof of the fact that the entries of $\frac{\Delta^m\cdot \Am_m}{m!}$ are $\sring[z,\frac{1}{N}]$, for all $m\in \N$ which is slightly weaker than \eqref{06032024bitavisa}. The proof of this stronger statement is similar.
 \end{proof}

 \section{Bad primes}
 The number $N\in\sring$ in \cref{19102023chinagain} is the main responsible for the finite number of exceptional primes in the Grothendieck-Katz conjecture. In this section we describe these primes.
 \footnote{For examples of bad primes of Lam\'e equations see \cref{29082024omidint}.}

\begin{prop}
\label{30012023espreafico}
 Let $\sring$ be a ring of integers of a number field and $N\in\sring$.
 There is a finite number of primes $p_1,p_2,\ldots,p_s$ (depending only on $\sring$ and $N$) such that the following property holds: for all $m\in\N$ coprime with all $p_i$'s  and
 all $M\in\sring$ and $d\in\N$ if $m|N^dM$ in $\sring$  then $m|M$.
 In particular, such a property hold for all except a finite number of prime numbers $m$.
\end{prop}
\begin{proof}
 If $\sring=\Z$ then this proposition follows from the unique factorization. The bad primes are those dividing $N$. Let us now  consider a number field $\sk$ and its ring of integers $\sring=\O_\sk$. By unique factorization theorem for ideals in $\O_\sk$, we know that there   are only a finite number of prime ideals in $\O_\sk$ containing $N\O_\sk$. Let $p_1,p_2,\ldots,p_s\in\Z$ be the list of characteristics of the corresponding residue fields. Take an arbitrary $m\in\Z$ coprime with $p_i$'s and such that $m|N^dM$ in $\sring$.
 By construction $m$ and $N^d$ are coprime, that is, there is no prime ideal containing both $m\O_\sk$ and $N^d\O_\sk$. Therefore, $m\O_\sk+N^d\O_\sk=\O_\sk$. In particular, there is $a,b\in\O_\sk$ such that $ma+N^db=1$. Multiplying this equality with $M$ we conclude that $M\in m\O_\sk$.
\end{proof}

\begin{theo}
 \label{27042024frankfurt}
Let $\sring$ be a ring of integers of a number field. In \cref{19102023chinagain} we can take $N\in\Z$. 
\end{theo}
\begin{proof}
We would like to replace $N$ with another one $\tilde N$ in $\N$. 
Let $p_1,p_2,\ldots, p_s$ be the list of primes in \cref{30012023espreafico} attached to $\sring$ and $N$ and define $\tilde N:=p_1p_2\cdots p_s$.  In general, the invertiblity of $N$ in $\sring[\frac{1}{N}]$ does not imply the invertiblity of $p_i$'s. We claim that for all $m\in\N$ the entries of $\frac{\Am_m}{m!}$ are in
$\sring[z, \frac{1}{\tilde N\Delta}]$. 
 A coefficient $M\in\sring$ of $\Delta^m\Am_m$ satisfies $M=m!\frac{S}{N^d}$ for some $S\in\sring$ and
 $d\in\N$, and so $m!\mid MN^n$ in $\sring$. We write $m!=N_1N_2$, where $N_1$ contains only the primes $p_i$, and $N_2$ is free of $p_i$'s.   By 
 \cref{30012023espreafico} we have $N_2| M$ in $\sring$, and hence, $\frac{M}{m!}\in\sring[\frac{1}{N_1}]\subset \sring[\frac{1}{\tilde N}]$.  

\end{proof}

\section{Algebraicity of a single solution}
In the theory of holomorphic foliations one is also interested in the algebraicity of the leaf passing through a given point. In the framework of linear differential equations this is:

\begin{theo}
\label{19102023chinagain-2}
Assume that $\sring$ is a finitely generated $\Z$-subalgebra of $\C$, $z_0\in\sring,\ \Delta(z_0)\not=0$ and $y_0\in\sring^n$.
If a solution of $\frac{\partial y}{\partial z}=\Am y$ with the initial value $y(z_0)=y_0$ is algebraic then there is $N\in\sring$ such that
\begin{equation}
\label{06032024bitavisa-2}
\hbox{ entries of } \frac{N^m\Am_m(z_0)y_0}{m!}\  \in\  \sring,\ \ \forall m\in \N.
\end{equation}
In other words, for all primes $p$ and $k,m\in\N$ with ${\rm ord}_pm!\geq  k$ 
we have  $\Am_m(z_0)y_0\equiv_{p^k} 0$, that is, $\Am_m(z_0)y_0$ is zero in the ring $\sring[\frac{1}{N}]/p^k\sring[\frac{1}{N}]$.
\end{theo}
The proof is the same as the proof of \cref{19102023chinagain}. In the last step \eqref{09022024airbnb}, instead of $Y$ we use $y$ itself and evaluate it in $z=z_0$. We have also $\Delta(z_0)\in\sring$ which is absorbed by $N$, that is, instead of $\sring[\frac{1}{N\Delta(z_0)}]$ we have written  $\sring[\frac{1}{N}]$.
It is already evident what is the analog of \cref{05112023tina}  in this framework: If for a solution $y(z)$  of
$\frac{\partial y}{\partial z}=\Am y$  with $y(z_0)=y_0$ we have \eqref{06032024bitavisa-2} then $y(z)$ must be algebraic. 
\footnote{Daniel Litt and Joshua Lam kindly reminded me that this conjecture is related to a conjecture of 
Y. Andr\'e and G. Christol in \cite[Remarque 5.3.2]{Andre2004}. In \cite[Conjecture 1.12]{BostanCarusoRoques2024} this is actually attributed to them.}
The analog of
Grothendieck-Katz conjecture is: if for all but a finite number of  primes $p$,  $\Am_p(z_0)y_0$ vanishes modulo $p$ then $y(z)$ is algebraic.
\footnote{This statement is false. See \cref{07092024omidviolento}.}

\section{Integral functions}
\label{15122023chikarmikoni}
Our main goal in the present section is to define integral functions, and prove that algebraic functions are integral (\cref{09122023sonymuseum}) and conjecture that this is an if and only if statement. In the following, for $a\in\C$ and $n\in\N_0$ we will use the notation
$$
(a)_n:=a(a-1)\cdots (a-n+1),\ \  [a]_n:=\frac{(a)_n}{n!}.
$$
Moreover, for a holomorphic function $y:U\to\C$ defined in an open set $U$ of $\C$ with the coordinate $z$, by $y^{(n)}$ we mean its $n$-th derivative with respect to $z$ and
$$
y^{[n]}:=\frac{y^{(n)}(z)}{n!}.
$$
With this notation we have
\begin{eqnarray}
\label{26nov2023zendegi-1}
(y_1y_1)^{[n]} &=& \sum_{i=0}^ny_1^{[i]}y_2^{[n-i]},\\  \label{26nov2023zendegi-1.5}
   (y^{-1})^{[n]} &=& -y^{-1}\sum_{i=0}^{n-1} y^{[n-i]}(y^{-1})^{[i]}, \\ \label{26nov2023zendegi-2}
(y^m)^{[n]} &=& \sum_{i_1+i_2+\cdots+i_m=n} y^{[i_1]}y^{[i_2]}\cdots y^{[i_m]},\\ \label{26nov2023zendegi-3}
(y^{[n]})^{[m]} &=& \binom{n+m}{n}y^{[n+m]}, \\
(z^{a})^{[n]} &=& [a]_nz^{a-n},\ \ \hbox{In particular } (z^{-1})^{[n]} = (-1)^nz^{-n-1}.
\end{eqnarray}
\begin{defi}\rm
 A finitely generated $\Z$-algebra $\Z[x]=\Z[x_1,x_2,\ldots,x_N]$ generated by holomorphic functions $x_i: U\to \C$ is called integral if it is closed under
 \begin{equation}
 \label{06022024guapa}
\Z[x]\to\Z[x],\ \ P\mapsto P^{[n]},
\end{equation}
for all $n\in\N_0$.  A holomorphic function $y:U\to \C$ is called integral (of length $\leq N$) if it belongs to some integral $\Z$-algebra $\Z[x]$ generated by at most $N$ elements. We call $\Z[x]$  the associated $\Z$-algebra.
\end{defi}
\begin{rem}\rm
\label{15122023hossein}
 The operator \eqref{06022024guapa} is $\Z$-linear, and so in order to show that $\Z[x]$ is closed under  \eqref{06022024guapa}, we need to show it  for a  monomial in $x$.
From \eqref{26nov2023zendegi-1} it follows that in order to verify that $\Z[x]$ is an integral algebra it is enough to prove that $x_i^{[n]}\in\Z[x]$ for all $n\in\N_0$ and for a set of generators $x=(x_1,x_2,\ldots,x_N)$ of $\Z[x]$.
\end{rem}
\begin{prop}
\label{09122023asab}
 The set of integral functions is a field.
\end{prop}
\begin{proof}
If $\Z[x]$ and $\Z[\tilde x]$ are integral $\Z$-algebra, then by \cref{15122023hossein}, $\Z[x,\tilde x]$ is also integral. This implies that if  $y$ and $\tilde y$ are two integral functions with the associated $\Z$-algebras $\Z[x]$ and $\Z[\tilde x]$ then
 $y+\tilde y, \ y\tilde y$ are also integral with the associated $\Z$-algebra $\Z[x,\tilde x]$.
 For an integral function $y\in\Z[x]$, $\Z[y^{-1},x]$ is also an integral algebra. This follows from \eqref{26nov2023zendegi-1.5} and \cref{15122023hossein}.
\end{proof}

\begin{prop}
\label{09122023sonymuseum}
 Algebraic functions are integrals.
\end{prop}
\begin{proof}
Let $y(z)$ be an algebraic function, that is,  there is a polynomial $P(z,y)=\sum_{i=0}^m p_i(z)y^i\in\C[z,y]$ such $P(z,y(z))=0$. Let $d:={\rm max}\{\deg(p_i),\ i=0,1,\ldots,m\}$. We take the $k$-th derivative of $P(z,y(z))=0$ with $k\geq d+1$ and we have
\begin{eqnarray*}
0 &=& \sum_{i=0}^m\sum_{j=0}^k p_i^{[j]}(y^i)^{[k-j]} \\
&=&   \sum_{i=1}^m\sum_{j=0}^d p_i^{[j]}(y^i)^{[k-j]}\\
&=&  \sum_{i=1}^m\sum_{j=0}^d\sum_{ i_1+i_2+\cdots+i_i=k-j} p_i^{[j]}y^{[i_1]}y^{[i_2]}\cdots y^{[i_i]}. 
\end{eqnarray*}
The largest derivative in this equality is in the term $\frac{\partial P}{\partial y}y^{[k]}$. This means that if we invert $\frac{\partial P}{\partial y}$ then we can write $y^{[k]}$ as linear combination of the previous derivatives. The conclusion is
\begin{equation}
\label{08122023}
y^{[k]}\in\Z[x]:=\Z\left [ \left(\frac{\partial P}{\partial y}\right)^{-1},\ z, {\rm coef}(P),\  y^{[i]},\ \ \ i=0,1,\ldots,d \right ],\ \ k\geq d+1,
\end{equation}
where ${\rm coef}(P)$ is the list of coefficients of $P$ (they might be transcendendtal numbers).
By \cref{15122023hossein} we need to show $P^{[n]}\in \Z[x]$ for generators of $\Z[x]$.
For $P=\left(\frac{\partial P}{\partial y}\right)^{-1}$ this follow from \eqref{26nov2023zendegi-1.5}. For $P=y^{[n]}$ this  follows from  \eqref{26nov2023zendegi-3} and \eqref{08122023}.
\end{proof}

\begin{rem}\rm
 It is clear from \eqref{08122023} the the length of an algebraic function $y$ is less that or equal
 $$
 N:=d+3+\hbox{number of monomials in $P$ with non-zero coefficients}-1.
 $$
 If $P$ is defined over $\Z$ then $N=d+3$. Precise formulas can be obtained by using the ring of definition of $P$.
\end{rem}
\begin{rem}
\label{26012024margetifi?}
Let $y(z)$ be an integral function and $z_0$ be in its domain of definition.
The Taylor series of $y$ at $z_0 $ has  coefficients in a finitely generated ring $\sring$, that is,
$y(z)\in \sring [[(z-z_0)]]$.  A weaker version of \cref{26012024margetifi?-2} is this statement
for algebraic functions.
\end{rem}

We believe that the converse of \cref{09122023sonymuseum} is also true.
\begin{conj}
\label{15july2024impatech}
 Integral functions are algebraic.
\end{conj}
For a linear differential equations with only algebraic solutions, \cref{19102023chinagain} gives us an integral algebra for its solutions of smaller length than the one given by \cref{09122023sonymuseum}. For an algebraic solution $y=[y_1,y_2,\ldots, y_n]^\tr$ of $y'=\Am(z)y$, we have
$\frac{y^{(m)}}{m!}=\frac{\Am_m}{m!}y$, and so
\begin{equation}
\label{06032024bitaufrj}
\hbox{ entries of } \frac{y^{(m)}}{m!}\  \in\  \sring[z, \frac{1}{N\Delta},y_1,y_2,\ldots, y_n],\ \ \forall m\in \N.
\end{equation}
 If the  Grothendieck-Katz conjecture is true then similar to \cref{10022024carnaval} we must have:
\begin{conj}
\label{22062024privatizar}
If the $p$-curvature of $y'=\Am(z) y$ is zero for almost all prime then  the entries of $y$ are integral.
\end{conj}
\begin{rem}\rm
Let us consider a  differential equations $y'=\Am y$ satisfying the hypothesis of \cref{05112023tina}.  It is a straightforward computation to show that 
$$
U=\sum_{k=0}^{\infty}\frac{(-1)^k\Am_k(z)}{k!}(z-z_0)^k, 
$$
with $U(0)=$ the identity matrix, satisfies the differential equation $U'=-U\Am$, and hence, $Y:=U^{-1}$ is a fundamental matrix of $Y'=\Am Y$, see  \cite[Remark 3.20]{BostanCarusoRoques2024} .  We can consider  $U$ and $U^{-1}$ as  formal power series in $(z-z_0)$ and the coefficients are in the ring $\sring[z_0, \frac{1}{N\Delta(z_0)}]$ . Note that we have also 
$$
U=\sum_{k=0}^{\infty}\frac{\Am_k(z_0)}{k!}(z-z_0)^k, 
$$
which follows again by the hypothesis of \cref{05112023tina} and the formula for Taylor series. 
\end{rem}

\section{An attempt}
The author's impression is that \cref{22062024privatizar} which is a consequence of Grothendieck-Katz conjecture might be false. In order to transfer this feeling to the reader, in this section,  we attempt  to translate the vanishing of $p$-curvatures into the $p$-integrality of its solutions.  T
Let $\sring$ be the $\Z$-algebra generated by  the coefficients of $\Am$ and $\check\sring=\sring[z_0,\frac{1}{\Delta(z_0)},y_0]$ for  $y_0\in\C^\Nn$, $z_0\in\C$ which is not a pole of $\Am$.
\begin{prop}
\label{24092024velinho}
If the $p$-curvature of $y'=\Am y$ is zero then  we have a formal power series $\check y\in\check\sring[[z-z_0]]$ with $\check y=y_0$ such that $\check y'=\Am \check y$ in  $\check\sring/p\check\sring$.
\end{prop}
\begin{proof}
We give two proofs of this statement. The first proof is well-known, see \cite[Remark 3.20]{BostanCarusoRoques2024}, and it is less intuitive. It is a mere computation to see that if $\Am_p=0$ in $\check\sring/p\check\sring$ then the sum  $U=\sum_{k=0}^{p-1}\frac{(-1)^k}{k!}(z-z_0)^k\Am_k$ with $U(0)=$ identity matrix, satisfies the differential equation $U'=-U\Am$, and hence, $U^\tr$ is the solution of the dual differential equation $y'=-\Am^{\tr}y$. It turns out that $U^{-1}$ is a fundamental matrix of $y'=\Am y$. 

The second proof is more natural and it is a simplification and  adaptation of \cite[Theorem 2]{Seshadri1960} into our context.
Let $\sring_p:=\sring/p\sring$.
We consider the larger ring $\sring_p[z,y]$ and the derivation
$$
\vf: \sring_p[z,y]\to \sring_p[z,y],\ \ \vf(z)=1,\ \vf(y)=\Am(z)y,
$$
and let
 $$
 \ker_p(\vf):=\{ f\in \sring_p[z,y]\ \Big |  \vf(f)= 0 \}.
 $$
 We have clearly $f^p\subset \ker_p(\vf)$ for all $f\in\sring_p[z,y]$.
 The hypothesis on $p$-curvature is equivalent to the fact that $\vf^p(f)= 0$ for all $f$.
 Let $\ker_p(\vf)[z]$ be the polynomial ring in $z$ and coefficients in $\ker_p(\vf)$.
Since $z^p\in\ker_p(\vf)$ such polynomials can be taken of degree $\leq p-1$ in $z$.
We claim that
$$
\sring_p[z,y]= \ker_p(\vf)[z].
$$
If not, there is a $f\in\sring_p[z,y]$ such that $f\not\in \ker_p(\vf)[z]$.
Since $\vf^pf= 0$, for some $1\leq e\leq p-1$ we have
$g=\vf^ef\not\in \ker_p(\vf)[z]$ but $\vf g\in\ker_p(\vf)[z] $.
Let us write
$$
\vf g=a_0+a_1z+a_2z^2+\cdots+a_{p-1}z^{p-1}, \ \ a_i \in \ker_p(\vf).
$$
By our hypothesis $\vf(a_i)= 0$, and so,   $0=\vf^{p-1}\vf g= (p-1)! a_{p-1}$.  Now we use the fact that $p$ is a prime, and hence $a_{p-1}=0$ in $\sring_p$.
Therefore, we can find $\tilde g=a_0z+\frac{a_1}{2}z^2+\cdots+\frac{a_{p-2}}{p-1} \in \ker_p(\vf)[z]$ such that $\vf g=\vf \tilde g$. Note that $\frac{1}{i}\in\Z$ is any representative for the inverse of $i$ in $\Ff_p$, and hence $\frac{1}{i}i-1\in p\Z$, and  it is zero in $\sring_p$. Once again the construction of $g$ works only for $p$ prime.
We conclude that that $g-\tilde g\in \ker_p(\vf)[z]$, and so, $g\in \ker_p(\vf)[z]$ which is a contradiction.

Let us take $y_0\in\C^\Nn$ and $z_0\in\C$ which is not a pole of $\Am$. From now on we use the larger ring $\check\sring=\sring[y_0,z_0, \frac{1}{\Delta(z_0)}]$ and the equalities are in $\check\sring_p:=\check\sring/p\check\sring$.
We write
$$
y-y_0= a_{0}+a_{1}(z-z_0)+a_{2}(z-z_0)^2+\cdots+a_{p-1}(z-z_0)^{p-1}, \ \ a_i \in \ker_p(\vf)^n.
$$
This equality implies that the algebraic variety $a_0(z,y)=0$ in $\A^{n+1}_{\check\sring_p}$ contains the point $(y,z)=(y_0,z_0)$. By definition of $\ker_p(\vf)$, it is tangent to the vector field $\vf$.
We get the map
$$
Y: \A^{\Nn+1}_{\check\sring_p }\to  \A^{\Nn+1}_{\check \sring_p } ,\ \ \ Y(z,y)=(z,w):=(z,a_0(z,y)), \ Y(z_0,y_0)=(z_0,0), 
$$
whose derivative at $(z_0, y_0)$ has the determinant $1$.
By inverse function theorem over a ring, see \cref{02jan2023impa}, 
we have power series
$\check y\in\check\sring_p[[z-z_0,w]]$ such that
$a_0(z,\check y)=w$. We redefine $\check y$ to be $\check y$ restricted to $w=0$ and get $\check y\in\check \sring_p[[z-z_0]]$. Therefore,  $a_0(z,\check y)=0$. This together with $\vf(a_0)= 0$, imply that
$$
\frac{\partial \check y}{\partial z}= \Am(z)\check y, \ \ \check y(z_0)=y_0,
$$
this is because these equalities imply $\frac{\partial a_0}{\partial z}+\frac{\partial a_0}{\partial y}\frac{\partial \check y}{\partial z}=0$ and $\frac{\partial a_0}{\partial z}+\frac{\partial a_0}{\partial y}\Am(z)\check y=0$.
\end{proof}

\begin{rem}\rm
\label{02jan2023impa}
 The inverse function theorem and implicit function theorem  are true over an arbitrary ring $\sring$ as follows. Let  $A=\A_{\sring}^n=\spec(\sring[z_1,z_2,\ldots,z_n])$ and $\t=(t_1,t_2,\ldots,t_n),\ \t_i \in\sring$ be an $\sring$-valued point of $A$. Let also $\O_{A^\for,\t}$  (resp. $\O_{A^\hol,\t}$) be the ring of formal (resp. convergent for some embedding $\sring\subset \C$) power series $\sum_{i\in\N_0^n} a_i(z-t)^i$.  
 \begin{enumerate}
  \item
Let $F=(F_1,F_2,\ldots, F_n)$ with $F_i\in \O_{A^\hol,\t}$  (resp. $\O_{A^\for,\t}$).  If $\det[\frac{\partial F_i}{\partial z_j}(0)]$ is invertible in $\sring$ then there is $G=(G_1,G_2,\ldots, G_n)$, with $G_i\in \O_{A^\hol,\t}  $ (resp. $\O_{A^\for,\t} $) such that $F\circ G=G\circ F=(z_1,z_2,\ldots,z_n)$.
 \item
Let $F=(F_1,F_2,\ldots, F_m)$ with $F_i\in \O_{A^\hol,\t}$  (resp. $\O_{A^\for,\t}$). If the rows of $[\frac{\partial F_i}{\partial z_j}(0)]$ can be completed to a basis of $\sring^n$  then there is $G=(G_1,G_2,\ldots, G_{n-m})$, with $G_i\in \O_{B^\hol,0}$ (resp. $\O_{B^\for,\t} $) and
$B=\A_{\sring}^{n-m}=\spec(\sring[z_1,z_2,\ldots,z_{n-m}])$
such that $[\frac{\partial G_i}{\partial z_j}(0)]$ is invertible in $\sring$ and
$$
F(\check z,G(\check z))=0,
$$
where $\check z=(z_1,z_2,\ldots,z_{n-m})$.
 \end{enumerate}
\end{rem}
\begin{defi}\rm Let $\sring$ be a  finitely generated subring of $\C$ such that a primes number $p$ is not invertible in $\sring$.    We say that a formal power series is $p$-integral (relative to $\sring$) if all its coefficients are in $\sring\otimes_\Z\tilde\Z_p$, where $\tilde\Z_p$ is the ring  of rational numbers whose denominator is not divisable by $p$. Note that $\tilde\Z_p[\frac{1}{p}]=\Q$.  
\end{defi}
The following is an immediate consequence of the definition.
\begin{prop}
Consider a differential equation $y'=\Am(z)y$  defined over $\sring$. A solution of this differential equation  with $y(z_0)=y_0, \ z_0\in\sring,\  y_0\in \sring^n$, is $p$-integral if 
$$
\frac{\Am_m(z_0)y_0}{m!}\in\sring\otimes_\Z\tilde\Z_p,\ \ \forall m\in\N.
$$ 
\end{prop}
In \cref{19102023chinagain-2} we have proved that for an algebraic solution of a differential equation we have always $N\in\sring$ such that  it is $p$-integral for all primes except a finite number and  relative to the ring $\sring[\frac{1}{N}]$.
\begin{defi}\rm
We say that a differential equation $y'=\Am(z)y$ defined over $\sring$ is $p$-integral (relative to $\sring$) if 
$$
\frac{\Delta(z)^m\Am_m(z)}{m!}\in\sring\otimes_\Z\tilde\Z_p[z],\ \ \forall m\in\N.
$$ 
where $\Delta\in\sring[z]$ is a polynomial such the entries of $\Delta\Am$ are also in $\sring[z]$. 
\end{defi}
If a differential equation is $p$-integral  relative to $\sring$ then all its solutions at $z=z_0\in\sring$ with the initial value $y_0\in\sring^n$, are $p$-integral, relative to $\sring[\frac{1}{\Delta(z_0)}]$, provided that $p$ is not invertible in this ring.  However, the converse does not seem to be true, but the author is not aware of a counterexample. Note that the integer-valued polynomials like  $\frac{z(z+1)}{2}$ have  values in $\Z$ for all $z\in\Z$, however, they are  not defined over $\Z$. 

The example in  \cref{26062024kamar} shows that the $p$-curvature of a differential equation might be zero but its solutions might not be $p$-integral. We can also observe this in a theoretical way as follows. 
 We take a representative for the coefficients of $y(p):=\check y$ in $\check\sring$ and use $y(p)$ to emphasize that it depends on $p$.
Therefore,  $y(p)\in\check\sring[[z-z_0]]^\Nn$ and $\frac{\partial y(p)}{\partial z}\equiv_p\Am(z)y(p)$. The main difficulty in proving that the entries of $y$ are integral functions is to glue the data of all $y(p)$ for different $p$ together.  By \cref{17dec2023handcrafts} we know that over $\check\sring\otimes_\Z\Q$ we have a unique solution $y$ with with $y(z_0)=y_0$ (here we must assume that $\Delta(z_0)$ is invertible in $\check\sring$). We might claim that $y$ is $p$-integral and  
$y\equiv_p y(p)$. The first strategy is  prove this by induction on $n$ for the coefficient of $(z-z_0)^n$.
This is trivially true for the coefficients of $(z-z_0)^i,\ \ i=0,1$. Since in the larger ring $\check\sring\otimes_\Z\tilde\Z_p$, we have inverted all $N$'s with $p\nmid N$, this is also true for $n<p$.
The coefficient of $(z-z_0)^n$ in $y$ is computed by
\begin{equation}
\label{06122023maosheng}
ny_{n}=\sum_{i+j=n-1} \Am_iy_j.
\end{equation}
We have also $ny(p)_{n}\equiv_p\sum_{i+j=n-1} \Am_i y(p)_j$.
For $n\leq p^2-1$ with $p\mid n$, the right hand side of \eqref{06122023maosheng} is divisible by $p$ and so $y_n$ can be computed in $\check\sring\otimes_\Z\tilde\Z_p$.  However, for $n=p^2$, we observe that $y_n$ might have $p$ in its denominator and the argument breaks.  One might try to give a recursive formula for the order of $p$ in the denominator of $y_n$. In case both $z_0$ and the differential equation are defined over $\Q$, the classical lower bound ${\rm ord}_p(y_n)\geq \frac{-n}{p(p-1)}$ (see for instance \cite[Proposition 3.38]{BostanCarusoRoques2024}) does not see this phenomenon, for instance,  it says that for $n=p^2-1$, the prime $p$ might divide the denominator of $y_n$, whereas by our recursion this does not happen.   


\section{Final conclusions}
Apart from \cref{19102023chinagain} with $m=p$ a prime, there are many other evidences for Grothendieck-Katz conjecture. It is verified for ``suitable direct factors'' of Gauss-Manin connections, see \cite{Katz1972}, and connections on rank one vector bundles, see \cite{ChudChud1985-1}. Moreover,  if the $p$-curvature of  a given differential equation is zero for all but a finite number of primes then its singularities are Fuchsian and all the residue matrices are diagonolizable with eigenvalues  in $\Q$, see \cite[Chapter 9]{Katz1996}, see also \cite[Theorem 8.1]{Katz1982}. These are true if the differential equation has algebraic solutions. In other words, the local data of a differential equation with zero $p$-curvatures, looks like a local data of a linear differential equation with only algebraic solutions. The differential equations over $\P^1$ which are uniquely determined by their local data (rigid differential equations) satisfy the Grothendieck-Katz conjecture, see \cite[Chapter 9]{Katz1996}. Therefore,
if there is a counterexample to this conjecture over $\P^1$ we must look for it among non-rigid differential equations, that is, they depend one some auxiliary parameters.
Example of such differential equations come from Painlevé VI, Heun and Lamé equations, see \cite{ho16,ho17}. Among these equations, the Lamé equation is the most simple one:
$$
p(x) \frac{d^2y}{dz^2}+\frac{1}{2} p'(z) \frac{dy}{dz}-(n(n+1)z+B)y=0,
$$
where $p(z)=4z^3-g_2z-g_3$. It depends on the parameters $g_2,g_3\in\C, n\in\Q,B\in\C$. Our search, with no specific strategy, gave us only Lamé equations with algebraic solutions. These are classified in \cite[Table 4]{Beukers2004}. We also found the Lamé equation with $B=0,\ n=\frac{7}{4},g_2=0,g_3=1$ which is not listed in the mentioned reference. \footnote{For further details see \cref{29082024omidint}.}
\cref{19102023chinagain} can be reformulated for Gauss-Manin connections, and its converse provides us with a conjectural description of Gauss-Manin connections. This conjecture with  $k=1$ appears in  \cite[Appendix of Chapter V]{Andre1989} and the main evidence for this comes from \cite{Katz1972}.

\section{Appendix: A finitely generated ring}
We would like to highlight a very particular case of \cref{09122023sonymuseum} which results in the following elementary statement.
\begin{prop}
\label{11122023babak}
 Let $a$ be a positive rational number. The $\Z$-algebra generated by $[a]_k:=\frac{a(a-1)\cdots(a-k+1)}{k!},\ k\in\N_0$ is of the form $\Z[\frac{1}{N}]$, for some $N\in\N$ which is a product of distinct primes.   In particular, the number of primes appearing in the denominator of $[a]_k,\ k\in\N$ is finite.
\end{prop}
\begin{proof}
Let $a=\frac{d}{n}, \ \gcd(n,d)=1$. We first prove that such a $\Z$-algebra $A$ is finitely generated by $[a]_k,\ k=0,1,2,\ldots,d$ and $\frac{1}{n}$.
 For $k>d$ we have
 $$
 \sum_{i_1+i_2+\cdots+i_n=k}[a]_{i_1}[a]_{i_2}\cdots [a]_{i_n}=0,
 $$
 which follows from \eqref{26nov2023zendegi-2} applied to the equality $y^n-z^d=0$ with $y:=z^a$.
 This implies that
 $n[a]_k$ is a $\Z$-linear combination of products of $[a]_s,\ s<k$. Therefore, $A$ is finitely generated.
 For rational  numbers $a_i=\frac{d_i}{n_i}, \ \gcd(n_i,d_i)=1$ we have $A=\Z[a_1,a_2,\ldots]=\Z[\frac{1}{n_1},\frac{1}{n_2},\cdots]$ which  follows from $s_in_i+r_id_i=1$ or equivalently $s_i+r_ia_i=\frac{1}{n_i}$ for some $r_i,s_i\in\Z$. In a similar way, $\Z[\frac{1}{n_1},\frac{1}{n_2}]=\Z[\frac{1}{[n_1,n_2]}]$ and then by induction $A=\Z[\frac{1}{N}]$ for some $N\in\N$. If $N=N_1^2N_2$ is not square free then $\Z[\frac{1}{N_1^2N_2}]=\Z[\frac{1}{N_1N_2}]$ and the result follows.
\end{proof}
\cref{11122023babak} implies that for any set of positive rational numbers $a_1,a_2,\ldots,a_k\in\Q$ there exists $N\in\N$ such that for all $n_1,n_2,\ldots,n_k\in\N$ there exists $s$ with
$$
N^s\frac{(a_1)_{n_1}(a_2)_{n_2}\cdots(a_k)_{n_k}}{n_1!n_2!\cdots n_k!}\in\Z.
$$
The main interest in the literature is those cases in which $s=n_1+n_2+\cdots+n_k$, see for instance \cite{zu02}. 

The converse of \cref{11122023babak} must be also true, that is, if  the  $\Z$-algebra generated by $[a]_k,\ k\in\N_0$ is finitely generated then $a$ must be a rational number. For the proof we might use:
\begin{theo}[Kronecker]
\label{10022024nemikham-2}
Let $\sk$ be a number field and $a$ be an element in the ring of integers $\O_\sk$  of $\sk$.
The number $a$ is in $\Z$ if and only if for all except a finite number maximal prime ideals $\pk\subset \O_\sk$, its class in $\O_\sk/\pk\cong \Ff_{p^n}$ is in $\Ff_p$.
\end{theo}
Another useful version of this theorem is:
\begin{theo}
\label{10022024nemikham}
Let $P$ be an irreducible non-constant polynomial in $\Z[x]$ of degree $d$. If for all except a finite number primes $p$, $P$ has $d$ roots  in $\Ff_p$ then $P$ is of degree $1$, and hence, it has a unique rational root.
\end{theo}
\cref{10022024nemikham-2}  is known under the name Kronecker's density theorem, see \cite{Esnault2023} in which the author refers to  M. Bauer.\footnote{M. Bauer, Über einen Satz von Kronecker. Arch. d. Math. u. Phys. 6, 212–222 (1904).} For the proof one mainly proves its generalizations such as Frobenius and Chebotarev density theorems. The author was not able to find an elementry proof.

\section{Appendix: Eisenstein theorem in many variables}
In the following $y(z)$ is a holomorphic function $y:U\to\C$ defined in some connected open subset $U$ of $\C^a$ with $z=(z_1,z_2,\ldots,z_a)\in U$. The following is the generalization of Eisenstein's theorem in \cref{26012024margetifi?-2}, to a multivariable case. In the following for $n\in\N_0^a$ we write $|n|:=\sum_{j=1}^a n_j$ and for two such vectors $n,m$ when we write $n<m$ we mean that $n_j\leq m_j$ for all $j=1,2,\ldots,a$ and at least one of them is a strict inequality.
\begin{prop}
\label{29062024ajib}
 Let $y(z)$ be an algebraic function and $z_0\in U$. If the polynomial equation $P(z,y(z))=0$ is defined over a ring $\sring$, $z_0\in \sring^a$  and $y(z_0)\in\sring$ then there is $N\in\sring$ such that
$$
\frac{N^{|m|}y^{(m)}(z_0)}{m!}\in\sring,\ \ \ \forall m\in\N_0^a.
$$
\end{prop}
\begin{proof}
 We write the Taylor series $y(z)=\sum_{i\in\N_0^a}y_i\cdot (z-z_0)^i$  of $y(z)$ at $z=z_0$, where $y_i$'s are unknwon coefficients and substitute in $P(z,y(z))=0$. Let $\sring$ be the $\Z$-algebra generated by coefficients of $P$, $z_0$ and $y_0:=y(z_0)$. Let also $\Delta:=\frac{\partial P}{\partial y}(z_0,y_0)$.
 Computing the coefficient of $(z-z_0)^n, n\in\N_0^a$ we get a recursion of type
 $$
 \Delta \cdot y_n=\hbox{ a polynomial of degree $\leq |n|$ in $y_i,\ i<n$,  with coefficients in $\sring$},
 $$
for instance
 $$
 \Delta y_{(0,\cdots,\underbrace{1}_{j-th},0,\cdots,0)}=-\frac{\partial P}{\partial z_j}(z_0,y_0)\in\sring.
 $$
By induction we can show that $y_n$ has a pole order at most $2|n|-1$ at $\Delta$. For $|n|=1$ this follows from the above equality.  If this is true for all $m<n$ then $\Delta y_n$ is a sum of monomials $y_{i_1}y_{i_2}\cdots y_{i_k}$ with $i_1+i_2+\cdots+i_k\leq n$ and $i_1,i_2,\ldots,i_k<n$ and coefficients in $\sring$. If $k=1$ then $2|i_1|-1\leq 2|n|-2$ and we are done. If $k\geq 2$ then  $2|i_1|-1+2|i_2|-1\cdots+2|i_k|-1\leq 2|n|-k\leq 2|n|-2$ and we are done again. It follows that  $N:=\Delta^2$ satisfy the desired property.
\end{proof}
Next, we give an example of algebraic function in many variables.
We consider the polynomial $f$ of degree $d$:
\begin{equation}
\label{30/03/2024}
f:=x^d+1+t_{1}x^{d-1}+\cdots+t_{d-1}x+t_d,
\end{equation}
and regard $t=(t_1,\ldots,t_d)\in\T:\C^d\backslash\{\Delta=0\}$ as parameters. For $t=0$ it is $x^d+1$, and its zeros are $d$-th roots of minus unity. Let $x_i(t),\ i=1,2$ be two roots of $f$ with $x_i(0)=\zeta_i,\ \ \zeta_i^d=-1$. These are holomorphic functions in $t$ and we may ask a student to write its Taylor series at $t=0$. 
In this section we do this.  For a non-integer  positive number $r$ let 
$$
\langle r\rangle :=(r-1)(r-2)\cdots (\{r\}+1)\{r\}=\frac{\Gamma(r)}{\Gamma(\{r\})}.
$$
For $0<r<1$ by definition $\langle r\rangle=1$. For an integer $a$ we denote by $\bar a$ it unique representative modulo $d$ in $\{0,1,\ldots,d-1\}$. The following is a special case of \cite[Theorem 13.4]{ho13}
with $n=0,\ k=1$
 \begin{theo}
 \label{07092024balegh}
For 
$\beta\in\N_0$ with $0\leq \beta\leq d-2$ we have 
 \begin{equation}
\label{300620224pilar}
\frac{x_2^{\beta}(t)}{f'(x_2(t))}-\frac{x_1^{\beta}(t)}{f'(x_1(t))}=
\mathlarger{\mathlarger{\mathlarger{\sum}}}_{a\in\N_0^d,\  d\nmid  (\beta+1-\sum_{i=1}^d ia_i)}
\left(\frac{1}{ a_1!a_1!\cdots a_d! }D_{a}{\sf p}_{a }    \right) t^a,
\end{equation}
where 
 \begin{eqnarray*}
  D_{a} &:=&
  \left\langle \frac{d-1-\beta+\sum_{i=1}^d ia_i}{d}\right\rangle 
  \left\langle \frac{\beta+1+\sum_{j=0}^{d-1} ja_{d-j}}{d}\right\rangle,
\\
{\sf p}_{a} &:=&  \frac{-1}{d}\left(\zeta_2^{\overline{\beta+1+\sum_{j=0}^{d-1} ja_{d-j}}}-\zeta_1^{\overline{\beta+1+\sum_{j=0}^{d-1} ja_{d-j}}}\right).
 \end{eqnarray*}
\end{theo}

\chapter{Lamé equation}
\label{29082024omidint}
{\it 
 The trademark of Lamé's career was moving from one topic to another in a quite logical way but he often ended up studying problems very far removed from the original [...] for the French seemed to feel that he was too practical for a mathematician and yet too theoretical for an engineer, see \cite{MTH}.
 }
\\
\\

{\it Abstract:}
We analyse Lamé equations with zero $p$-curvature. We investigate the cases in which the $m$-curvature modulo 
$p^k$ is not zero for some  $k,m\in\N$ with
${\rm ord}_pm!\geq  k$.

\section{Introduction}
There are many evidences for Grothendieck-Katz conjecture in the literature. The most important one is due to N. Katz in \cite{Katz1972}: It is true for Gauss-Manin connections and its ``suitable direct factors''. For connections on rank one vector bundles on curves  it is proved by D. V. Chudnovsky and G. V. Chudnovsky in  \cite[Theorem 8.1]{ChudChud1985-1}, however, the author's impression is that the proof must be revised, see \cref{03092024omidsemrespeito} and \cite[footnote page 177]{bost01}. This has been generalized 
when the differential Galois group of the connection  has a solvable neutral component, see \cite[Chapter VIII, Exercises 5 and 6, Section 3]{Andre1989} and \cite{bost01}. 
An important class of linear differential equations $y'=\Am y$ for which the Grothendieck-Katz conjecture is not known is of the format
\begin{equation}
\label{28jan2024nemikhaham}
\Am:=\sum_{i=1}^r\frac{\Am_i}{z-z_i},
\end{equation}
where $z_1,\ldots,z_r\in\C$ are distinct complex numbers and $\Am_i$'s are $\Nn\times\Nn$ matrices with entries in $\C$. The rank one case $\Nn=1$ is easy to handle and it follows from Kronecker's criteria.
The rank two case $\Nn=2$ with two singularities $r=2$ is  not so difficult as it is reduced to verify the conjecture for Gauss hypergeometric equation. This and the case of one (finite) singularity, that is,  $r=1$ follows from N. Katz's work below. 
If the $p$-curvature of  a given differential equation is zero for all but a finite number of primes then its singularities are Fuchsian and all the residue matrices are diagonolizable with eigenvalues  in $\Q$, see \cite[Chapter 9]{Katz1996}, see also \cite[Theorem 8.1]{Katz1982}. This implies that in \eqref{28jan2024nemikhaham} we have to consider diagonolizable matrices $\Am_i$'s with rational eigenvalues. 
These are true if the differential equation has algebraic solutions. In other words, the local data of a differential equation with zero $p$-curvatures, looks like a local data of a linear differential equation with only algebraic solutions. 
The differential equations over $\P^1$ which are uniquely determined by their local data (rigid differential equations) satisfy the Grothendieck-Katz conjecture, see \cite[Chapter 9]{Katz1996}.
 Therefore, if there is a counterexample to this conjecture we must look for it among non-rigid differential equations, that is, they depend on some accessory parameters and they are not pull-back of rigid ones. 
An important  class of differential equations \eqref{28jan2024nemikhaham} for which the Grothendieck-Katz conjecture is still open and depend on some accessory parameters  is the case $r=3$ and $\Nn\geq 2$ for which by a linear transformation we can assume that $z_1=0,\ z_2=1$ and $z_3=t$: 
\begin{equation}
\label{14022024cebolinhasmorreram}
\Am=\frac{\Am_0}{z}+\frac{\Am_1}{z-1}+\frac{\Am_t}{z-t}.
\end{equation}
There are three main examples of linear differential equations of this format: Lamé, Heun and Painlevé VI. It turns out that  Lamé equation is a particular case of Heun equation, and this in turn is a particular case of  the linear differential equation underlying Painlevé VI. All these three differential equations are of the format \eqref{14022024cebolinhasmorreram},  see \cite{ho16,ho17}. In the present text  we consider the Lamé equation 
which is traditionally written as a linear differential equation:
\begin{equation}
\label{07082024un}
P(z) \frac{d^2y}{dz^2}+\frac{1}{2} P'(z) \frac{dy}{dz}-(n(n+1)z+B)y=0,
\end{equation}
where $P(z)=4z^3-g_2z-g_3$ with $27g_3^2-g_2^3\not=0$. It depends on the parameters $g_2,g_3\in\C, n\in\Q,B\in\C$ ($B$ is called the accessory parameter).  
Our main reference for Lamé equation is \cite{Beukers2004} and the references therein. 
In a personal communication with S. Reiter, he sent me the article \cite{Hofmann2012} in which the J. Hofmann computes Lamé equations with the monodromy group of type  $(1; e)$.

\section{Lamé equation as a connection on elliptic curves}
We denote  by
$$E_{g_2,g_3}: y^2=4x^3-g_2x-g_3,$$ the corresponding Weierstrass family of elliptic curves. For this we assume that $\Delta:=27g_3^2-g_2^3\not=0$, otherwise, the curve $E_{g_2,g_3}$ is singular and the Heun equation is a pull-back of the  Gauss hypergeometric equation, see \cref{23june2024beukers}.
The algebraic group $\C^*$ acts on $\spec(\C[g_2,g_3, B])$ by
$$
k\bullet (g_2,g_3,B):=(g_2 k^2,g_3k^3, Bk),\ k\in\C^*
$$
and it is easy to show that if $f(z)$ satisfies the Lamé equation with parameters $g_2,g_3,B,n$ then $f(k^{-1}z),\ \ k\in\C^*$ satisfies
 the Lamé equation with parameters $g_2k^2,g_3k^3, Bk, n$ ($n$ is unchanged). 
 It is proved that any two `equivalent' Lamé equations are necessarily obtained by the above $\C^*$-action, see \cite{Beukers2004}, and so, the moduli space of Lamé equations for fixed $n$ is the weighted projective space $\P^{1,2,3}\backslash\{\Delta=0\}$, which is the projectivization of the homogeneous ring $\C[g_2,g_3, B]$ with $\deg(g_2)=2,\ \deg(g_3)=3,\ \deg(B)=1$. The Lamé equation does not change under $n\mapsto -1-n$ and so we can assume that $n\geq -\frac{1}{2}$.

Let $\Lambda$ be a lattice in $\C$ and $\wp(z,\Lambda)$ be the Weierstrass $\wp$ function. If $f(z)$  is a solution of the the Lamé equation then $f(z):=f(\wp(z))$ is a solution of
 $$
 \frac{d^2f}{dz^2}-(n(n+1)\wp(z)+B)f=0.
 $$
 which is called the elliptic function form of the Lamé equation (we also call \eqref{07082024un} the algebraic Lamé equation). 
 Let $z$ be the coordinate system on $\C$ and hence on the torus $\C/\Lambda$. This gives us the vector field $\frac{\partial }{\partial z}$ on $\C/\Lambda$ as it is invariant under the translation by elements of $\Lambda$.
We remark that under Weierstrass uniformization the vector field  $\frac{\partial }{\partial z}$
is mapped to the vector field
\begin{equation}
\label{23062024consumismo}
\vf:=y\frac{\partial }{\partial x}+ \frac{1}{2}P'(x)\frac{\partial }{\partial y}.
\end{equation}
This follows from $\frac{\partial \wp}{\partial z}=\wp'$ and
$\frac{\partial \wp'}{\partial z}=\frac{1}{2}P'(\wp(z))$. Note that $\vf$ is tangent to $E_{g_2,g_3}$:
$$
\vf(y^2-P(x))=0
$$
and so it is the correct vector field to consider (in the literature sometimes one uses $y\frac{\partial }{\partial x}$).
Therefore, the elliptic function form of the Lamé equation is also an algebraic differential equation, not over $\P^1$, but over the elliptic curve $E_{g_2,g_3}$:
\begin{equation}
\left(y\frac{\partial }{\partial x}+ (6x^2-\frac{1}{2}g_2)\frac{\partial }{\partial y}\right)^2f=
(n(n+1)x+B)f.
\end{equation}
We can also write this as a system.
On the elliptic curve $E_{g_2,g_3}$ we consider the following linear differential equation
\begin{equation}
\label{22062024ulis}
 dY=\Am Y,\  \ \ \Am:=\mat{0}{\frac{dx}{y}}{n(n+1)\frac{xdx}{y}+B\frac{dx}{y}}{0}.
\end{equation}
If we set $Y=[f,\vf f]^\tr$  then $\vf Y=\Am(\vf)Y$. The local exponents of \eqref{22062024ulis} at the point at infinity $\infty=[0:1:0]$ of $E_{g_2,g_3}$ are
$-n$ and $n+1$ and so if it has only algebraic solutions then $n\in\Q$.
Note that the exponents of Lamé equation itself at infinity is
$-\frac{n}{2}$ and $\frac{n+1}{2}$, and the difference is due to the fact that the projection $E_{g_2,g_3}\to \P^1, \ (x,y)\to x$ is two to one map ramified at infinity.

\section{Some special cases}
\label{03092024omidsemrespeito}
\begin{theo}
\label{26june2024ansiedade}
 The Grothendieck-Katz conjecture holds for Lamé equations with  $n+\frac{1}{2}\in\N_0$. 
\end{theo}
\begin{proof}
By Frobenius basis of linear differential equations, we know that if $n+\frac{1}{2}\in \Z$, there might be a logarithmic solution of the Lamé equation at infinity. By a theorem of Brioschi and Halphen we know that  for 
$n+\frac{1}{2}\in \N_0$, there exists a weighted homogeneous polynomial  $p_n(B,g_2,g_3)\in\Z[B,\frac{g_2}{4}, \frac{g_3}{4}]$ of degree $n+\frac{1}{2}$ and monic in $B$ such that the Lamé equation has no logarithmic
solutions at $\infty$ if and only if $p_n=0$. It turns out that this is also equivalent to the Lamé equation having finite monodromy, for details and references see \cite[Theorem  2.2 and 2.3]{Beukers2004}. If the $p$-curvature of the Lamé equation is zero for all but a finite number of primes then its local solution at infinity  cannot be logarithmic.
\end{proof}
In \cite[Theorem 7.2 page 91]{ChudChud1985-1}, brothers Chudnovsky claim the following:  
For $n\in\N_0$ the Lamé equation  with $27g_3^2-g_2^3\not=0$ never satisfies the assumptions of the Grothendieck-Katz  conjecture, that is, its  $p$-curvature is non-zero for infinitely many $p$. 
Unfortunately, this statement is false.  We have many Lamé equations with $n\in\N_0$ and 
finite monodromy, and hence, they have  zero $p$-curvature for all except a finite number of primes.
It has been proved in the literature that if the Lamé equation with $n\in\N_0$ has finite monodromy, then its monodromy group is a dihedral
group $D_N$ of order $2N$ for some $N\in\N$, see \cite[Corollary 3.3]{Beukers2004}.
In this reference, and also Waal's Ph.D. thesis \cite[Remark 6.7.10]{WaallThesis},  we can find many examples, such as 
\begin{equation}
\label{26062024ihes}
n=2, \ B=21, \ g_2=327,\  g_3=1727, 
\end{equation}
of such Lamé equations with finite monodromy, see also \cref{25062024david}. It is strange that in an earlier work \cite[page 2773]{Chiarellotto1995} the author, which mentions \cite{ChudChud1985-1} in the introduction, has even computed the underlying (smooth) elliptic curve $E_{g_2,g_3}$ (unfortunately without computing $B$) but has not complained about the results in \cite[Theorem 7.2 page 91]{ChudChud1985-1}.
After a consult with F. Beukers he sent me the article \cite{Dahmen2007}  in which the author even counts the number of such Lamé equations.
In  \cite[Theorem 7.2 page 91]{ChudChud1985-1}, it has been also claimed the following. For $n\in\N_0$, 
there exists a weighted homogeneous polynomial  $l_n(B,g_2,g_3)\in\Q[B,g_2, g_3]$ of degree $2n+1$ and monic in $B$  with the following property. The $p$-curvature of Lamé equation  is nilpotent for all but a finite number of primes  if and only if $l_n=0$. 
The polynomial $l_n$ seems to be the polynomial described in \cite[Theorem 3.2]{Beukers2004}. In any case, the examples of Lamé equations mentioned in the above references, and also \cref{25062024david},  provide counterexamples to this statement too.
{\tiny
}
A special class of Lamé equations, which is usually excluded from discussions is: 
\begin{prop}
\label{23june2024beukers}
The Lamé equation with $27g_3^2-g_2^3=0$ is a pull-back of Gauss hypergeometric equation, and hence, the Grothendieck-Katz conjecture is true for these differential equations.  
\end{prop}
\begin{proof}
 The curve $E_{g_2,g_3}$ is singular and hence after a blow-up it is a rational curve. More explicitly, let us take $g_2=3a^2,\ g_3=a^3$ and so $p(x)=4(x+\frac{a}{2})^2(x-a)$. Under the desingularization map
 $$
 \P^1\to E_{g_2,g_3},\ \ \ z\mapsto (z^2+a, 2z(z+\frac{3}{2}a))
 $$
 the vector field $(z^2+\frac{3}{2}a)\frac{\partial}{\partial z}$ is mapped to $\vf$ in
 \eqref{23062024consumismo} and we get the differential equation:
 $$
 \left((z^2+\frac{3}{2}a)\frac{\partial}{\partial z}\right)^2=(n(n+1)(z^2+a)+B).
 $$
 which can be clearly transformed into a Gauss hypergemetric equation. 
\end{proof}

\section{$p$-curvature of Lamé  equation}
\label{25062024david}
The section is based on the author's search for a possible counterexample to Grothendieck-Katz conjecture 
among Lamé equations. Of course, the author has not been successful. Instead, we have produced many examples such
that the number of primes for which the $p$-curvature vanishes seems to be much more than the number of non-vanishing cases. The author's hope is that the experiments presented in this section motivate the reader to systematically study the density of primes associated to differential equations. Another starting point for this is \cite[Proposition 6.2, Corollary 6.3]{ChudChud1985-1}. As one can feel it by reading the mentioned reference, Cheobotarev density theorem seems to be the  tip of an iceberg in the framework of differential equations. 
 
We first write the Lamé equation \eqref{07082024un} as a system in  a canonical way.
The matrix $Y:=[f,\frac{\partial f}{\partial z}]^\tr$ satisfies
$$
Y'=\Am Y,\ \ \Am:=\mat{0}{1}{-\frac{1}{2}\frac{P'(z)}{P(z)}}{ \frac{n(n+1)z+B}{P(z)}}.
$$
We consider only the cases $g_2,g_3,B,n\in\Q$, and hence, the entries of $P(z)\Am$ are polynomials in $z$ with 
rational coefficients. Let $N$ be the common denominator of all these parameters.
It is easy to see that $y^{(n)}=\Am_ny$, where $\Am_n$ are computed recursively by
$\Am_1=\Am,\ \ \Am_{n+1}=\frac{\partial \Am_n}{\partial z}+\Am_n\Am$, and all $\Am_n$'s have entries in
$\Z[z,\frac{1}{NP(z)}]$.  
For a prime number $p$, the matrix $\Am_p$
considered as a matrix with entries in $\Z[z,\frac{1}{NP(z)}]/p\sring[z,\frac{1}{NP(z)}]$ is usually called the $p$-curvature of the Lam\'e equation. We have proved:
\footnote{This is \cref{19102023chinagain} for Lam\'e equations.}
\begin{theo}
\label{11092024omid}
If all the entries of solutions of the Lam\'e equation are algebraic functions then 
for all but  a finite number of primes $p$ and $k,m\in\N$ with ${\rm ord}_pm!\geq  k$ 
we have  $\Am_m\equiv_{p^k} 0$, that is, $\Am_m$ is zero in the ring $\Z[z, \frac{1}{NP(z)}]/p^k\Z[z,\frac{1}{NP(z)}]$.
\end{theo}
For our purpose, we have written the procedure {\tt BadPrD} in {\tt foliation.lib}.
It computes the bad ($p$-curvature non-zero) and good ($p$-curvature zero) primes of a linear differential defined over $\Q$  for primes less than or equal to a given number. If the differential equation  is not written as a system then the procedure transforms it into a system in a canonical way. The output consists of three list of primes 1. primes in the denominator of the system. For the Lamé equation these are primes 2, and those in the denominator of $n$ and $B$.  We completely ignore these primes.  2. Bad primes: primes such that the
$p$-curvature is not zero 3. Good primes:  primes that the $p$-curvature is zero. 
 The fourth entry is optional. If it is given and it is a positive integer $k$ then the procedure verifies whether the $m_{p,k}$-curvature  is zero or not, where for a given prime and $k\in\N$,  $m_{p,k}$ is the smallest $m\in\N$ such that ${\rm ord}_pm!\geq  k$. 
 Only the result of the last two lists might change.

We first check that algebraic Lamé equations (those with only algebraic solutions=finite monodromy) in \cite[Table 4]{Beukers2004} have  zero $p$-curvature.  We only consider examples in this table which are defined over $\Q$.
{\tiny 
\begin{verbatim}
LIB "foliation.lib";
   ring r=(0,z),x,dp;
   matrix lde[1][3]; 
   list L=list(1/4,0,0,1), list(3/4,3/8,-168,622), list(1/6,0,1,0), list(5/6,0,1,0),list(1/6,1/6,60,90),
   list(1/10,0,0,1), list(3/10, 3/100,3,5/4), list(7/10,0,0,1), list(7/4,0,0,1);
   number B; number n; number g2; number  g3; int ub=100;
   number P; 
   for (int i=1;i<=size(L);i=i+1)
      { 
        n=L[i][1]; B=L[i][2]; g2=L[i][3]; g3=L[i][4];
        P=4*z^3-g2*z-g3; lde=-n*(n+1)*z-B, 1/2*diffpar(P,z), P; 
        BadPrD(lde, z, ub)[2]; " "; 
      }
\end{verbatim}  
}
Apart from the the primes that appear in the expression of the $8$ Lamé equations above, we observe that we have the following list of bad primes for each of them:
$$
\{3\}, \{5\}, \{\}, \{5\}, \{5\},\{3\},\{2\},\{3,7\},\{3,7\}
$$
respectively.  This data might be useful for the study of bad primes of differential equations. The last entry $(n,B,g_2,g_3)=(7/4,0,0,1)$ is not listed in \cite[Table 4]{Beukers2004} and we found it by our random search of differential equations with zero $p$-curvature. We can verify experimentally that it has finite monodromy. 
\footnote{Therefore, it is not a counterexample to the conclusion of \cref{10022024carnaval}, and hence the Grothendieck-Katz conjecture. }
Below,  we check that for this example $m_{p,k}$-curvature vanishes for all good primes $p\leq 23,\ p\not=2,3,7$ and all $k\leq 6$:
{\tiny
\begin{verbatim}
 LIB "foliation.lib";
 ring r=(0,z),x,dp;
 matrix lde[1][3];  number n=7/4; number B=0; number g2=0; number g3=1; int ub=23;
   number P=4*z^3-g2*z-g3;
   lde=-n*(n+1)*z-B, 1/2*diffpar(P,z), P;
   for (int i=1;i<=6;i=i+1){BadPrD(lde, z, ub,i);}
\end{verbatim}
}Therefore, this examples must have only algebraic solutions. According to \cite{Beukers2004} its monodromy group must be $G_{12}$.  The way that we have found this example is explained in the next paragraph.

The main difficulty in finding Lamé equations with vanishing $p$-curvature for all except a finite number of primes is that there are Lamé equations which the first non-vanishing $p$-curvature is obtained for large $p$'s. 
For example,  for 
$(n, B,g_2,g_3)=(\frac{12}{89},0,0,1)$, the first bad primes are $83, 107, 113,$ 
$127, 149,...$. For $(n, B,g_2,g_3)=(\frac{5}{87},0,0,1)$ the bad primes below $150$ are
$17, 97, 107, 109, 113, 127, 131, 137$.  
For $(n, B,g_2,g_3)=(\frac{4}{65},0,0,1)$, the bad primes below $150$ are
$71,73,89,97,101,103,107,109,113,127,137,139$. For all these we have used: 
{\tiny 
\begin{verbatim}
LIB "foliation.lib";
ring r=(0,z),x,dp;
matrix lde[1][3]; 
list L=list(12/89,0,0,1), list(5/87,0,0,1), list(4/65,0,0,1); 
   number B; number n; number g2; number  g3; int ub=150; number P; 
   for (int i=1;i<=size(L);i=i+1)
      { 
        n=L[i][1]; B=L[i][2]; g2=L[i][3]; g3=L[i][4];
        P=4*z^3-g2*z-g3; lde=-n*(n+1)*z-B, 1/2*diffpar(P,z), P; 
        BadPrD(lde, z, ub)[2]; " "; 
      }
\end{verbatim}            
}
In order to search for Lamé equations with only few bad primes, let us say $\leq m$,  in a large range of primes we have written the procedure
{\tt BadPrDLess} which computes $p$-curvature until the number of bad primes exceed $m$. Using this we can for instance observe that
\begin{prop}
Among the Lamé equations with $n=\frac{j}{i}, \ 1\leq j,i\leq 100$ and 
$$
(B,g_2, g_3)=(\frac{3}{8},-168,622),(\frac{1}{60},60,90),(\frac{3}{100}, 3,\frac{5}{4})
$$
only the three cases $n=\frac{3}{4}, \frac{1}{6}, \frac{3}{10}$, respectively attached to each case above, have at most three bad primes in the range $2\leq p\leq 100$.
\end{prop}
Note that these three Lamé equations have finite monodromy \cite[Table 4]{Beukers2004}. The situation for Lamé equations with the underlying elliptic curves $y^2=4x^3-1$
is different. 
We take the Lamé equation with $B=0, n=\frac{j}{i}, \ 1\leq j,i\leq 100,\ \  g_2=0,\ g_3=1$ and compute all $n$ with at most three bad primes $\leq 100$. Here, we have listed $n$ together with its bad primes. 
{\tiny 
\begin{verbatim}
1; 3                1/2 ;             1/4 ; 3          1/10 ; 3    
3/2 ;               4 ; 3             5/87 ; 17 97     7 ; 3 5 7
7/2 ; 7             7/4 ; 3 7         7/10 ; 3 7       7/95 ; 7 43 97
8/81 ; 11 53 97     9/2 ; 5           9/83 ; 29 31 61  9/95 ; 13 79
10 ; 3 7            10/67 ; 3 7 29    11/37 ; 29 59 67 11/57 ; 79 89
11/79 ; 23          12/89 ; 83        13/2 ; 7 13      13/4 ; 3 5 13
13/10 ; 3 13        13/83 ; 13 17 97  14/99 ; 7 23 83  15/2 ; 11
16 ; 3 11 13        16/81 ; 31 97     16/99 ; 29       17/93 ; 23 43 89
18/83 ; 7 17 89     19/10 ; 3 19      21/2 ; 7 11 17   21/97 ; 7 29 41  
23/67 ; 17 73 97    24/91 ; 11 37 97  25/2 ; 13 19     25/4 ; 3 13 17
27/53 ; 17 23 97    27/91 ; 5 29 47   28/95 ; 7 23 43  29/77 ; 3 5 89 97
29/81 ; 17 47 79    31/2 ; 13 19 31   31/10 ; 3 11 31  33/20 ; 17 43 71
34/53 ; 11 43 59    36/61 ; 7 19 97   39/85 ; 13 31 97 40/81 ; 7 13 23
41/69 ; 43 67 89    41/91 ; 23 47 89  42/83 ; 7 19 73  44/79 ; 43 59 71
45/73 ; 11 23 31    45/83 ; 11 13 61  45/97 ; 11 17 79 48/91 ; 11 17 73 
50/87 ; 11 17 67    51/76 ; 47 89     51/77 ; 47 71 97 51/100 ; 47 61 71
53/91 ; 83 89 97    54/59 ; 11 17 79  54/77 ; 5 37 83  54/91 ; 19 47 73
55/17 ; 53 97       55/21 ; 19 59 67  55/32 ; 31 61 71 60/91 ; 43 67 97
61/96 ; 61 73 97    62/87 ; 17 31 61  62/91 ; 5 31 43  64/59 ; 11 17 79
64/71 ; 43 83       64/81 ; 11 19 97  65/99 ; 13 73 97 67/89 ; 13 67 73
67/93 ; 17 67 79    68/3 ; 11 53 67   68/13 ; 61 71 79  68/99 ; 5 47        
69/13 ; 73 79 83    69/91 ; 31 41 61  70/89 ; 7 43 61   71/69 ; 7 29
78/71 ; 13 43 83    80/19 ; 13 37     80/43 ; 7 11 29   81 ; 13 29 97
81/77 ; 31 37 79    81/95 ; 43 59     82/69 ; 17 47 59  83/31 ; 11 19
83/87 ; 11 23 31    83/96 ; 53 89     85/21 ; 13 61 83  87/7 ; 53 71 97
87/19 ; 59 79       87/80 ; 37 89 97  87/83 ; 41 59 67  87/89 ; 23 59 83
88/63 ; 47 67 83    88/65 ; 83 89     90/61 ; 31 71 83  90/97 ; 61 67
92/7 ; 61 97        92/27 ; 7 41 71   92/73 ; 41 53 67  92/85 ; 19 41 97
94/45 ; 37 73 97    96/91 ; 41 73 97  99/4 ; 41 53 89    100/77 ; 37 83
\end{verbatim}
} 
The cases $n=\frac{1}{4},\frac{1}{10},\frac{7}{10}$ correspond to the algebraic Lamé equations found in  \cite[Table 4]{Beukers2004}. It turns out that $n=\frac{7}{4}$ is also among this class. For
$$
n=1,4,7, 10,16,
$$
the number of bad primes does not increase in the range $p\leq 150$. They seem to be algebraic Lamé equations, and the proof must not be so hard. 
Using techniques used in \cite{Beukers2004} it is quite accessible to prove that these provides counterexamples to \cite[Theorem 7.2 page 91]{ChudChud1985-1}. The example \eqref{26062024ihes} has only the bad prime $5$ in the range $p\leq 150$.
The case $n=81$ is exceptional as its bad primes become  $13, 29,97, 103,109,127,139$ in this range. 
The cases with $n+\frac{1}{2}\in\Z$ must fit in \cref{26june2024ansiedade}.  In the other cases the number of bad primes increases as one increases the range $p\leq 100$ of tested primes (we have just tested this for random choice of $n$). For instance for $n=16/99$ we have only the bad prime $29$ in the range $p\leq 100$ but it increases to 
$29, 103,109,127,131,137,149$ in the range $p\leq 150$. 
A similar search with the same data for $g_2=1$, and $g_3=0$ has produced only the list
{\tiny
\begin{verbatim}
1/2 ; 
1/6 ;
5/2 ; 3 5
5/6 ; 5
9/2 ; 5 7
13/6 ; 7 13
17/6 ; 5 11 17
47/87 ; 23 41 79
\end{verbatim}
}

For $n=\frac{1}{6},\frac{5}{6}$ we get the algebraic Lamé equations which are listed  in  \cite[Table 4]{Beukers2004}.
For our computations  we have used the code: 
{\tiny
\begin{verbatim}
LIB "foliation.lib";
ring r=(0,z),x,dp;
matrix lde[1][3]; number B=0; number n; number g2=1; number g3=0; int ub=100; int nBdPr=3;
number P=4*z^3-g2*z-g3; int i; int j; list final; intvec prli=primes(1,ub);
 for (j=1;j<=100;j=j+1)
    {
       for (i=1;i<=100;i=i+1)
       {
        if (lcm(intvec(i,j))==i*j)
           {
           n=number(j)/number(i);
           lde=-n*(n+1)*z-B, 1/2*diffpar(P,z), P;
           final=BadPrDLess(lde, z, ub,nBdPr);
           if (size(final[2])+size(final[1])+size(final[3])==size(prli))
           {final=final[2]; n, ";", final[1..size(final)];}
           }
       }
    }
\end{verbatim}
}
Finally, we would like to highlight the following consequence of Grothendieck-Katz conjecture, classification of finite groups for Lamé equations and \cite[Theorem 2.3, Corollary 3.4, Theorem 4.4]{Beukers2004}.
\begin{conj}
 If the $p$-curvature of a Lamé equation is zero for all but a finite number of primes then
 $$
 n\in \{0,\pm \frac{1}{2},\pm \frac{1}{6},\pm \frac{1}{4},\pm\frac{3}{10},\pm\frac{1}{10}\}+\Z.
 $$
\end{conj}
Our initial goal was to find a possible counterexample to  the Grothendieck-Katz conjecture.  
But, we were only able to find examples like:
\footnote{We wanted to find a counterexample to the conclusion of 
\cref{10022024carnaval}. This proposition  is  the same as \cref{26062024kamar}.}
\begin{prop}
\label{26062024kamar-2}
For the  Lamé differential equation
$$
(4z^3-1) \frac{d^2y}{dz^2}+6z^2\frac{dy}{dz}-\frac{7}{36}zy=0
$$
we have  $\Am_5\equiv_50$, however, $\Am_{25}\not \equiv_{5^6}0$. 
\end{prop}
\begin{proof}
The computer code of this can be found here:
{\tiny
\begin{verbatim}
 LIB "foliation.lib";
ring r=(0,z),x,dp;
matrix lde[1][3]; number B=0; number n=1/6; int g2=0; int g3=1; int ub=20;
number P=4*z^3-g2*z-g3;
lde=-n*(n+1)*z-B, 1/2*diffpar(P,z), P;
BadPrD(lde, z, ub);
BadPrD(lde, z, ub,6);
\end{verbatim}
}
\end{proof}

\section{Pull-back Lamé equations}
This section is based on many e-mail exchanges with S. Reiter who did most of the computation.  I told him about a list of
Lamé equations 
\href{https://w3.impa.br/~hossein/WikiHossein/files/Singular%20Codes/2024_07_14_g2=0g3=1B=0LameWithOneBadPrime.txt}
{$g_2=0,g_3=1,B=0$ with at most one bad prime among all primes $\leq 200$}
and a similar list for
\href{https://w3.impa.br/~hossein/WikiHossein/files/Singular%20Codes/2024_07_14_g2=1g3=0B=0LameWithTwoBadPrimes.txt}
{$g_2=1,g_3=0,B=0$ with at most two bad primes among all primes $\leq 100$}.
For example, the Lamé equation with $g_2=0,g_3=1,B=0,n=347/480$ , has only the bad prime $197$ among all primes $<=200$. The primes $p=2,3,5$ are not considered as they are  factors of $480$. 
\begin{prop}
\label{12082025samantha}
 The Lamé equation with $g_2=0, g_3=1$, that is over the elliptic curve $y^2=4x^3-1$ with $B=0$,  is the pull-back by the map 
 $z\mapsto z^3$ of the linear differential equation 
 \begin{equation}
 \label{09082025-AI1}
(36z^2-9z)\frac{\partial^2}{\partial z^2}+(42z-6)\frac{\partial}{\partial z}-n(n+1),
\end{equation}
with the Riemann scheme $0, \ [0, \frac{1}{3}]$, $\frac{1}{4}, \  [0, \frac{1}{2}]$ and $\infty, \ [- \frac{n}{6}, \frac{n}{6} + \frac{1}{6}]$. In a similar way, the Lamé equation with $g_2=1, g_3=0$, that is over the elliptic curve $y^2=4x^3-x$ with $B=0$,  is the pull-back by the map 
 $z\mapsto z^2$ of the linear differential equation 
 \begin{equation}
 \label{09082025-AI2}
L:=(16z^2-4z)\frac{\partial^2}{\partial z^2}+(20z-3)\frac{\partial}{\partial z}-n(n+1)
\end{equation}
with the Riemann scheme  $0, \ [0, \frac{1}{4}]$, $\frac{1}{4}, \ [0, \frac{1}{2}]$, $\infty,\  [- \frac{n}{4}, \frac{n}{4} + \frac{1}{4}]$. 
\end{prop}
By a linear transormation in $z$, one can see easily that the differential equation \eqref{09082025-AI1} and \eqref{09082025-AI2} can be transformed into the Gauss hypergeometric equation with $(a,b,c)=(\frac{-n}{6},\frac{n}{6}+\frac{1}{6},\frac{2}{3})$ and $(\frac{-n}{4},\frac{n}{4}+\frac{1}{4},\frac{3}{4})$, respectively. 
S. Reiter also informed me of the classification of Pull-back Lamé equations with non Dihedral monodromy group
\href{https://www.math.fsu.edu/~hoeij/Heun/2F1_Heun_Sorted_by_Degree}{by M. van Hoeij and R. Vidunas}.
Finally, the following question is rised by S. Reiter which is natural in view of computations done in \cite{Beukers2004}. For any Lamé equation with parameters $(n,B,g_2,g_3)$ and $k\in\N$ there is a $B_k$ such that Lamé equations with parameters $(n,B,g_2,g_3)$ and $(n+k,B_k,g_2,g_3)$ are equivalent. For a suggestion on what an equivalence relation should mean see \cite[page 3]{Beukers2004}.

We finish this section with some experiments with the Gauss hypergeometric equation:
\begin{equation}
z(1-z)y''+(c - (a+b+1)z)y'-aby=0,\ \ a,b,c\in\C. 
\end{equation}
For a random choice of parameters $a,b,c$ in the Gauss hypergeometric equation it seems that the $p$-curvature is always non-zero except for a finite number of primes $p$.  For example,  for $(a,b,c)=(\frac{1}{2},\frac{1}{2},1)$ and $p\not=2$, the $p$-curvature seems to never  vanish. We have checked this for primes $\leq 150$. 
{\tiny
\begin{verbatim}
LIB "foliation.lib";
ring r=(0,z),x,dp; matrix lde[1][3]; int i; int j; 
number a=1/2; number b=1/2; number c=1;
int ub=150; number parm=z; 
          lde=-a*b, c-(a+b+1)*z, z*(1-z);
          BadPrD(lde, parm, ub);
\end{verbatim}
}
From another side in \cref{12082025samantha} we have got the Gauss hypergeometric equations with $(a,b,c)=(\frac{-n}{6},\frac{n}{6}+\frac{1}{6},\frac{2}{3})$ and $(\frac{-n}{4},\frac{n}{4}+\frac{1}{4},\frac{3}{4})$, for which we can experimentally observe that the number of both vanishing and non-vanishing curvature is infinite. 
{\tiny
\begin{verbatim}
LIB "foliation.lib";
ring r=(0,z),x,dp;matrix lde[1][3]; int i; int j; 
number a; number b; number c=2/3;
int ub=50; number parm=z; number N=100;
   for (i=1; i<=N; i=i+1)
      {
       for (j=1; j<=N; j=j+1)
          {
          a=-i/number(6*j);  b=-i/number(6*j)+1/6;
         lde=-a*b, c-(a+b+1)*z, z*(1-z);
          BadPrD(lde, parm, ub);
          }
      }
\end{verbatim}
}

\section{Differential equations and density of primes}
The discussion in this section motivates us to define the density of primes for which the $p$-curvature vanishes. Recall that we have two notions of density. For a subset $S$ of primes, we say that $S$ has the natural, respectively analytic or Dirichlet,  density $\delta$ if the following limit exists and it equals to $\delta$:
$$
\lim_{x\to\infty} \frac{ \#\{p\in S\ | \ p\leq x\}}{\#\{p \hbox{ prime} \ | \ p\leq x\}} \ \ \hbox{ respectively } 
\ \ \lim_{s\to 1} \frac{ \sum_{p\in S} \frac{1}{p^s}}{ \sum_{p \hbox{ prime}} \frac{1}{p^s}}.
$$
If a set of primes has a natural density, then it has also an analytic density and the two densities are equal, but the converse is false, see \cite{LenstraStevenhagen}.    
For a linear differential equation $\frac{\partial y}{\partial z}=\Am(z) y$, we define its curvature denisty (natural or analytic), the density of primes $p$ such that its $p$-curvature  vanishes. The density of a finite number of primes (resp. all primes except for a finite number) is zero (resp. one). The converse is not necessarily true, see \href{https://mathoverflow.net/questions/16846/infinite-sets-of-primes-of-density-0#:~:text=A%20neat%20way%20to%20show%20that%20a,P%20has%20zero%20density%20in%20the%20primes.}{the mathoverflow link}. One might try to formulate a stronger version of the  Grothendieck-Katz conjecture speculating  that if the curvature density of a differential equation  is one then all its solutions are algebraic. The only well-known result about curvature density is the following:
\begin{prop}
\label{27072025morrodaurca}
Let $a$ be an algebraic number. The natural and analytic curvature density of  $y'=\frac{a}{z}y,\ a\in\C^*$ is  
$$
\frac{1}{\#\Gal(\k/\Q)},  
$$
where $\k$ is the splitting field of the minimal polynomial of $a$. 
\end{prop}
\begin{proof}
Let $f(x)$ be the minimal polynomial of $a$. If the $p$-curvature of $y'=\frac{a}{z}y$ is zero then $a^p-a=0$ in $\Z[a]/p\Z[a]$. This implies that $f$ divides $x^p-x$ in $\Ff_p[x]$. In particular, $f$ modulo $p$ is a product of linear equations. 
The proposition is now a direct consequence of the Frobenius density theorem:  
Let $f\in\Q[x]$ be a polynomial. The density of primes $p$ such that 
$$
f \ {\rm  mod} \ p =f_1f_2\cdots f_t,\ f_i\in \Ff_p[x] \hbox{ irreducible },\ \deg(f_i)=n_i,\
$$
is equal to 
\begin{equation}
\label{18072025crepiza}
\frac{ 
\#\{\sigma\in \Gal(\k/\Q) \ | \sigma \hbox{ with cycle pattern } n_1,n_2,\ldots,n_t \ \}
}{\#\Gal(\k/\Q) }.
\end{equation}
see \cite{LenstraStevenhagen}. 
For the application of this theorem we consider the case $n_1=n_2=\cdots=n_t=1$
\end{proof}
Frobenius' density theorem is a particular case of Chebotarev's density theorem. The author tried to put these theorems in terms of curvature density of differential equations but there was no succces beyond \cref{27072025morrodaurca}.  Even for the Gauss hypergeometric function obtained in \cref{12082025samantha}, the author is not aware of any result  regarding the density of vanishing $p$-curvatures.

Each computation of the present chapter is done in less than an hour, and it turns out that if we let our codes run for a day or more we can get more accurate density statements, for instance, for the   Lamé equation with $(n, B,g_2,g_3)=(\frac{5}{87},0,0,1)$ after a day of computation, the author was able to compute all good ($p$-curvature zero) and bad primes $\leq 797$. After the first $50$ primes the first digit of the density of good primes seems to stabilize in $0.7$, however, the second digit does not seem to stabilize. The density of good primes $p\leq 797$ is $0.71428$. Similar statements seem to be true for the Lamé equation with $(n, B,g_2,g_3)=(\frac{12}{89},0,0,1)$. For the computer codes and data \href{https://w3.impa.br/~hossein/singular-code/2025_12_14_p_curvature_density.txt}{see  the author's webpage}.

\chapter{Local-global principle for leaf schemes}
\label{21112023bihude}

{\it
One may ask whether imposing a certain Hodge class upon a generic member of an algebraic family 
of polarized algebraic varieties amounts to an algebraic condition upon the parameters, 
A. Weil in \cite{Weil1977}. 
}
\\
\\

{\it Abstract: 
We study Hodge loci as leaf schemes of foliations.
The main ingredient is the Gauss-Manin connection matrix of families of projective varieties.
We also aim to investigate a conjecture on the ring of definition of leaf schemes and its consequences such as the algebraicity of leaf schemes (Cattani-Deligne-Kaplan theorem in the case of Hodge loci). This conjecture is  a consequence of a local-global principle for leaf schemes. }

\section{Introduction}
\label{26sept2024baddahan}
The present text arose from an attempt to combine Grothendieck-Katz $p$-curvature conjecture on linear differential equations with only algebraic solutions, and the Cattani-Deligne-Kaplan theorem on the algebraicity of Hodge loci. 
For our purpose we generalize Hodge loci into leaf schemes of foliations. This has first appeared in  \cite{GMCD-NL}.  Foliations in the present text are given over finitely generated subrings of the field of complex numbers, and it makes sense to manipulate them modulo primes. We introduce conjectural criterions, \cref{19062024kontsevich} and \cref{02112023mainconj}, involving modulo $p$ manipulations of foliations which guarantee the algebraicity of leaf schemes.

For a holomorphic function $f$ in several complex variables, the  Grothendieck-Katz conjecture provides a modulo primes criterion for algebraiciy of $f$ provided that $f$ satisfies a linear differential equation. Our main local-global conjecture, \cref{02112023mainconj},  provides a modulo primes criterion for the algebraicity of the zero locus  $f=0$ of $f$, and in general complex analytic ideals, provided that $f=0$ is part of the ideal of  a leaf scheme. A simple, but yet non-trivial example of this situation, can be constructed from the Gauss hypergeometric function $F(z):=F(\frac{1}{2},\frac{1}{2},1|t)$. 
For any $N\in\N$, the holomorphic function  $F(1-t_1)F(t_2)-NF(1-t_2)F(t_1)$ is not algebraic, however, we know that its zero locus in $(t_1,t_2)\in\C^2$ is an algebraic curve which is a singular model of a covering of the modular curve $X_0(N)$.
After adding more three variables, and constructing an ideal with four generators in $5$ dimension, it can be seen as a leaf scheme, for further details see \cref{04dec2023khar}. Another example is the following series in $\binom{12}{4}=495$ variables $t_\alpha$:
\begin{equation}
 \label{09082024david}
\mathlarger{\mathlarger{\mathlarger{\sum}}}_{
\stackrel{
a: I{}\to \N_0,\  \beta:= \frac{1}{4}\left(\sum_{\alpha}a_\alpha\cdot \alpha+(1,1,1,1)\right),
}
{ \beta_i\not\in\Z,\ \beta_0+\beta_1,\beta_2+\beta_3\in \Z}
}
\frac{
(-1)^{[\beta_0]+[\beta_2]}    
\langle\beta_0\rangle \langle\beta_1\rangle\langle\beta_2\rangle\langle\beta_3\rangle
}{ 
\prod_{\alpha\in I{}}a_\alpha!
}\cdot  \prod_{\alpha\in I{}}t_\alpha^{a_\alpha},
\end{equation}
where $I:=\{\alpha\in\N_0^4\ | \sum_{i=0}^3\alpha_i=4\}$, 
for a rational number $r$, $[r]$ is the integer part of $r$, that is $[r]\leq r<[r]+1$, $\{r\}=r-[r]$ and  
$\langle r\rangle=(r-1)(r-2)\cdots (\{r\})$. This function is a period of a holomorphic $2$-form on smooth surfaces of degree $4$ in $\P^3$ which is not algebraic. However, its zero locus is algebraic and parameterizes such surfaces with a line, for more details see \cref{05aug2024delaram}.

\paragraph{Summary of the text:}
In \cref{08082024hl} we explain what we mean by our local-global principle for Hodge loci. 
Hopefully, this will motivate the reader for the definition of leaf scheme in \cref{08082024ls}.
The main results in \cref{08082024hl}, namely   \cref{17072024ziquan} and \cref{09082024armand}, are proved much later in \cref{13082024shekam}  after elaborating the concept of leaf scheme.
\cref{19062024kontsevich} in \cref{08082024ls} is not suitable for experimental purposes and so in \cref{08082024lsmp} we introduce a stronger conjecture whose hypothesis is implementable in a computer. This involves constructing many vector field in the ambient space and we discuss this in \cref{04aug2024fouad}.
In \cref{04112023weixinpay} we explain natural leaf schemes which arise in the framework of linear differential equations (local systems), and in particular  Gauss-Manin connections. We construct the Hodge loci as leaf schemes in \cref{04dec2023khar} and discuss its ring of definition in \cref{08082024dr}.  In \cref{05aug2024delaram} we experimentally observe that natural generators of the ideal of a Hodge locus  have all primes inverted in their expressions. Finally, in \cref{10082024kk} we prove \cref{06072024claire} using two results of N. Katz. This is a consequence of the Hodge conjecture for Hodge-Tate varieties. 

\paragraph{Notations:} Throughout the text,   $\sring$ is a finitely generated subring of $\C$ and $\sring_\Q:=\sring\otimes _\Z\Q=\sring[\frac{1}{2},\frac{1}{3},\frac{1}{5},\cdots]$ which is an infinitely generated subring of $\C$. We also take a finitely generated ring extension $\sring\subset \ring$ and set $\T:=\spec(\ring)$. By an $\sring$-scheme we simply mean  the affine $\sring$-scheme $\T$. We consider an $\sring$-valued point of $\T$ which is nothing but an $\sring$-linear morphism $\t: \ring\to\sring$ and set $ \mathfrak{m}_{\T,\t}:={\rm ker}(\t)$.
By definition, $\O_\T$ is just the ring $\ring$, $\Omega^1_\T$ is the $\O_\T$-module of K\"ahler differentials (differential $1$-forms) in $\T$. For $f\in\O_\T$, $f(\t):=\t(f)$ is just the evaluation at $\t$. 
The $\O_\T$-module $\Theta_\T$ of vector fields/derivations in $\T$ is the dual of the $\O_\T$-module $\Omega^1_\T$. 
We also look at a vector field $\vf\in\Theta_\T$
as a derivation $\vf: \O_\T\to\O_\T,\ \ f\mapsto \vf(f):=\vf(df)$. For an $\sring$-valued point $\t$, we have also the well-defined $\sring$-linear  map $\vf(t):=\t\circ \vf: \mathfrak{m}_{\T,\t}/\mathfrak{m}_{\T,\t}^2\to\sring$ which is called  the evaluation of $\vf$ at $\t$. The dual $\TS_\t\T:=(\mathfrak{m}_{\T,\t}/\mathfrak{m}_{\T,\t}^2)^\vee$ is the tangent space of $\T$ at $\t$, and so, we have the evaluation  map $\Theta_\T\to\TS_\t\T$ which is $\sring$-linear. 
As we do not need the language of sheaves, a sheaf on $\T$ is identified with the set of its global sections.  We take a submodule $\Omega$ of $\Omega_\T$ and denote its dual by $\Theta:=\{\vf\in\Theta_\T\ |\ \vf(\Omega)=0 \}$ which is a submodule of $\Theta_\T$.


We will assume that $\Omega$ is integrable in the strongest format, that is, $d\Omega\subset \Omega\wedge \Omega_\T$. 
Even though, we rarely use the integrability of $\Omega$,
we will frequently use the notation $\F(\Omega)$ to denote the foliation induced by $\Omega$. From an algebraic point of view $\F(\Omega)$ is just $\Omega$ and nothing more.
For any other subring $\check\sring$  of $\C$, we denote by 
$\T_{\check\sring}:=\T\times_\sring{\rm Spec}(\check\sring)$. The ring $\check\sring$ might be infinitely generated and the main example of this in this text is $\sring_\Q$. We denote by $(\T^\hol_{\check\sring},\t)$ (resp. $(\T^\for_{\check\sring},\t)$) the analytic (resp. formal) scheme underlying $\T$ and $\O_{\T^\hol_{\check\sring}, \t}$ (resp. $\O_{\T^\for_{\check\sring}, \t}$) is the ring of holomorphic functions in a neighborhood of $t$ (resp. formal power series) and with coefficients in $\check\sring$. We say that an $\sring$-valued point $\t$ of $\T$ is smooth if there are $z_1,z_2,\ldots,z_n\in\O_{\T^\hol_{\sring_\Q},\t}$ which generate it freely as $\sring_\Q$-algebra. We call $z=(z_1,z_2,\ldots,z_n)$ a  holomorphic coordinate system.

We use the letter $L$ to denote a subscheme of $(\T^\hol,\t)$, that is, we have an ideal $\IS\subset \O_{\T^\hol, \t}$ and $\O_{L}=\O_{\T^\hol, \t}/\IS$ (this might have zero divisors or nilpotent elements). 
Let $\Theta$  be a submodule of the $\O_\T$-module $\Theta_\T$.  Its rank is the number $a\in\N$ such that $\wedge^{a+1} \Theta$ is a torsion sheaf and $\wedge^{a} \Theta$ is not. For $\vf\in\Theta_\T$ we define the scheme $\sch(\vf\in\Theta)$ given by the ideal generated by 
 $$
\left|\begin{matrix}
\vf(P_1) & \wf_1(P_1) & \wf_2(P_1) &\cdots & \wf_a(P_1) \\
\vf(P_2) & \wf_1(P_2) & \wf_2(P_2) &\cdots & \wf_a(P_2) \\
\vdots &   \vdots    & \vdots & \ddots & \vdots \\
\vf(P_a) & \wf_1(P_a) & \wf_2(P_a) &\cdots & \wf_a(P_a)
\end{matrix}\right|,\ \ 
$$
$$
P_1,P_2,\cdots,P_a\in\O_\T,\ \ \wf_1,\wf_2,\ldots,\wf_a\in\Theta.
$$
In geometric terms if $\sring$ is an algebraically closed field, $\sch(\vf\in\Theta)$ is the loci of points $\t$ in $\T$ such that the vector
$\vf(\t)$ is in the $\sring$-vector space generated by $\wf(\t),\ \wf\in\Theta$. For two vector fields $\vf$ and $\wf$ the collinearness scheme $\vf||\wf$ is just $\sch(\vf\in\O_\T\wf)$. 
For a primes $p$ which in not invertible in $\sring$, the ring $\sring_p:=\sring/p\sring$ is non-zero, and we define $\T_p:=\T\times_\sring \spec(\sring_p)$ which is modulo $p$ reduction of $\T$. For a vector field $\vf$ in $\T_p$, it is known that $\vf^p$ is also a vector field, and this phenomena does not exist in $\T$.

\paragraph{Acknowledgement:}
The main ideas of the present text took place during short visits to BIMSA and YMSC at Beijing in 2023 and a first draft of it for putting in arxiv is written at IHES in  2024. My since thanks go to all these institutes. I would also like to thank D. Urbanik for some useful discussions regarding the content of \cref{08082024hl}. Thanks go to Daniel Litt and Joshua Lam for comments on \cref{07092024omidviolento}. 

\section{Hodge loci}
\label{08082024hl}
In order to motivate experts in Hodge theory, we explain a foliation free statement of our main local-global principle for Hodge loci. This has been the main motivation for writing the present text. 
\begin{defi}\rm
\label{23.11.2017}
Let $Y$ be a smooth projective variety.  A Hodge cycle is any element in the intersection of the integral 
 cohomology $H^{\mov}(Y,\Z)\subset H^\mov_\dR(Y)$ ($H^{\mov}(Y,\Z)$ is considered modulo torsion) and $F^{\frac{\mov}{2}}\subset H^\mov_\dR(Y)$, where 
 $F^{\frac{\mov}{2}}= F^{\frac{\mov}{2}}H^\mov_\dR(Y)$ is the $\frac{\mov}{2}$-th piece of the Hodge filtration of $H^\mov_\dR(Y)$. \index{Hodge cycle}
\end{defi}
Therefore, the $\Z$-module of Hodge cycles is simply the intersection 
$H^{\mov}(Y,\Z)\cap H^{\frac{\mov}{2},\frac{\mov}{2}}= H^{\mov}(Y,\Z)\cap F^{\frac{\mov}{2}}$.
We have the intersection pairing/polarization $\langle\cdot,\cdot\rangle: 
H^{\mov}(Y,\Z)\times H^{\mov}(Y,\Z)\to \Z$ and so it makes sense to talk about the self-intersection or norm 
$\langle \delta,\delta\rangle$ of a Hodge cycle $\delta$. 
Now, let $Y\to V$ be a family of smooth complex projective varieties ($Y\subset \Pn N\times V$ and $Y\to V$
is obtained by projection on the second coordinate). 
\begin{defi}\rm
Let $k\in\N$ and $\T:=F^{\frac{\mov}{2}}H^\mov_\dR(Y/V)$ be the total space of the vector bundle of $F^{\frac{\mov}{2}}$ pieces of the Hodge 
 filtration of $H^{\mov}_\dR(Y_t),\ \ \t\in V$.
 The locus $L_k$ of Hodge cycles (or simply Hodge locus)  of self-intersection equal to $k$ is the subset of 
 $F^{\frac{\mov}{2}}H^\mov_\dR(Y/V)$ containing such Hodge cycles.
 \index{Locus of Hodge cycles}
\end{defi}
Note that $\T:=F^{\frac{\mov}{2}}H^\mov_\dR(Y/V)$ is an algebraic bundle, however, the Hodge locus  is a union
of local analytic varieties $L$.  Moreover, such an $L$ has a natural analytic scheme structure, that is, the natural ideal defining $L$ might not be reduced, and hence, $\O_{L}$ might have zero divisors or nilpotent elements, see \cref{04dec2023khar}.
\begin{theo}
\label{CADEKA}\index{Cattani-Deligne-Kaplan theorem}
 (Cattani-Deligne-Kaplan, \cite[Theorem 1.1]{cadeka})
The locus $L_k$ of Hodge cycles  of self-intersection equal to $k$ is an algebraic subset of 
$F^{\frac{\mov}{2}}H^\mov_\dR(Y/V)$.  
\end{theo}
Let $\sring$ be a finitely generated subring of $\C$ such that $Y\to V$, $\T$ and $Y_{t}$ have smooth models over $\spec(\sring)$ (we avoid introducing new notation for these models).
For an explicit  construction of $\T$ as an $\sring$-scheme see \cref{04dec2023khar}.
Let $L$ be a local analytic Hodge locus as above. It is given by an ideal $\IS\subset \O_{\T^{\hol}_{\sring_\Q},\t}$.  We define 
\begin{equation}
\label{03082024torio}
\Theta_{\T,L}:=\{\vf\in\Theta_\T\ | \ \vf(\IS)\subset\IS\},
\end{equation}
and we call it  the $\O_\T$-module of vector fields in $\T$ tangent to $L$. 
One of our main motivations in the present text is the author's not yet successful attempt to prove the following corollary of \cref{CADEKA} without using it.
\begin{coro}
\label{17072024ziquan}
 Let $\Theta_{\T,L}$ be the $\O_{\T}$-module  of vector fields in $\T$ tangent to $L$. For all except a finite number of primes and all $\vf\in\Theta_{\T,L}$, the vector field $\vf^p$ is also in $\Theta_{\T, L}\otimes_\sring \sring_p$ (that is modulo $p$).
\end{coro}
The above statement is the main inspiration for one of the main conjectures of the present text, see \cref{02112023mainconj}. This conjecture says that \cref{17072024ziquan} implies \cref{CADEKA}, and so they are equivalent. 
\begin{rem}\rm 
Note that even though $L$ modulo $p$ might not make sense, as it is defined over complex numbers (actually $\sring_\Q$), see \cref{05082024retaliation} and \cref{05aug2024delaram},  the module of vector fields $\Theta_{\T,L}\subset \Theta_{\T}$ is algebraic, so it makes sense to talk about modulo $p$ of this module and $\vf^p\in \Theta_{\T, L}\otimes_\sring \sring_p$.   
It says that there is a lift $\wf\in\Theta_\T$ of 
$\vf^p\in\Theta_{\T_p}$ such that it lies in 
$\Theta_{\T,L}$.
Note that a direct definition of  $\Theta_{\T_p,L}$ as in \eqref{03082024torio} does not make sense.
\end{rem}
We can write down a weaker version of \cref{17072024ziquan} as follows. 
Let 
\begin{equation}
\label{08082024decathlon}
\alpha: \TS_t\T\times H^{\frac{m}{2}}(Y_t,\Omega_{Y_t}^{\frac{m}{2}})\to H^{\frac{m}{2}+1}(Y_t,\Omega_{Y_t}^{\frac{m}{2}-1}),\ t\in\T
\end{equation}
be the IVHS map (infinitesimal variation of Hodge structures in Griffiths and his coauthor's terminology, see \cite{CGGH1983}). Here, $\TS_\t\T$ is the dual of the $\sring$-module $\mathfrak{m}_{\T,\t}/\mathfrak{m}_{\T,\t}^2$ and it is the tangent space of $\T$ at $\t$. If $\t$ is an $\sring$-valued point of $\T$ lying in a Hodge locus, by definition of $\T$, it comes together with a a Hodge cycle $\delta_\t\in H^m(Y_t,\Z)$ and hence an element  $\bar \t\in H^{\frac{m}{2}}(Y_t,\Omega_{Y_t}^{\frac{m}{2}})$. We denote by $L$ the Hodge locus corresponding to variations of $\delta_t$. The Hodge cycle $\delta_\t$ is called general if $\alpha(\cdot, \bar\t)$ has maximal rank among all Hodge cycles.
This is equivalent to say that the Hodge locus $L$ is typical in the sense of \cite{BKU}.
\begin{coro}
\label{09082024armand}
 Assume that $\delta_t$ is a general Hodge cycle  and take any $v\in \TS_\t\T$  with $\alpha(v,\bar\t)=0$. There is a vector field $\vf$ in $\T$ such that $\vf(\t)=v$ and for all except a finite number of primes $p$  we have $\alpha(\vf^p(\t),\bar\t)=0$ modulo $p$, where in the last equality we have considered smooth models of $Y,V,\T, Y_{t}$  over a finitely generated subring $\sring$ of $\C$ and then reduction modulo $p$ is performed. 
\end{coro}
We have intentionally not used the language of reduction modulo a closed point of $\spec(\sring)$ with residue field of characteristic $p$, in order to highlight the classical modulo $p$ manipulations. Note that since $Y_\t$ is a projective smooth scheme over $\spec(\sring)$, the cohomology groups $H^{i}(Y_t,\Omega_{Y_t}^{j})$ are free $\sring$-mdoules. 
 In \cref{17072024ziquan} we might ask how big is $\Theta_{\T,L}$. We have discussed this in \cref{04aug2024fouad} in the framework of leaf schemes. 
 Since $L$ has a natural structure of an analytic scheme, our definition of $\vf\in\Theta_{\T,L}$ is stronger than the geometric definition: $\vf(\t)\in\TS_\t L$ for any complex point $\t$ of $\T$.  The following example might clarify the situation better. Let $\T:=\A^2_\sk$ with the coordinate system $(x,y)$ and $L$ be the subscheme of $\T$ given by 
$y^2=0$. In this case the Zariski tangent space of $L$ at each closed point is the whole tangent space of $\T$ at that point, and so, any vector field in $\T$ is tangent to $L$ in the geometric framework, but not necessarily in a scheme theoretic framework, that is, $y^2\mid \vf(y^2)$ is not valid in general.
 If $\vf$ is a vector field in $\T$ and tangent to $L$ in the geometric sense then we have 
 $\alpha(\vf(\t), \bar\t)=0$ for $t\in L$ but there is no reason to believe that $\alpha(\vf^p(\t), \bar\t)=0$ after taking reduction modulo $p$. This is the main reason why in \cref{09082024armand} we assume that $\delta_t$ is a general Hodge cycle and hence $L$ is smooth.  In this case both geometric and scheme theoretical definitions of tangency to $L$ are equivalent.



\section{Leaf scheme}
\label{08082024ls}
 As in this text we care about the field or ring of definition of leaf schemes, we rewrite the definition of leaf scheme in \cite[Section 5.4]{ho2020} in the algebraic framework. 
 It might be better for the reader to read first the definition over complex numbers before reading its algebraic/arithmetic version which involves some heavy notations due to the fact that we have to distinguish algebraic and holomorphic objects, and we have to insert the ring of definition into our notations. However, if there is no confusion, we will use simplified notations as in \cite[Chapter 5]{ho2020}.

\begin{defi}\rm \index{Leaf scheme}
\label{18apr2018-2}
Let $\T$ be an $\sring$-scheme, $\F(\Omega)$ be a foliation in $\T$ and $\t$ be an  $\sring$-valued point of  $\T$. 
A subscheme  $L$ of $(\T^{\hol}_{\check\sring},\t)$  with $\O_L:=\O_{\T^\hol_{\check\sring}}/\IS$ is called a leaf scheme of $\F(\Omega)$ defined over $\check\sring$ if 
$\Omega\otimes_{\O_{\T}}\O_{\T^{\hol}_{\check\sring},\t }$ and  $\O_{\T^{\hol}_{\check\sring},\t }\cdot d\IS$
projected to  $(\Omega^1_{\T}\otimes_{\O_{\T}}\O_{\T^{\hol}_{\check\sring},\t })/\IS\Omega^1_{\T}$  and regarded as
$\O_{\T^\hol_{\check\sring},\t}/\IS$-modules are equal. In other words, $\Omega\otimes_{\O_{\T}}\O_{\T^{\hol}_{\check\sring},\t }$ and $\O_{\T^\hol_{\check\sring},\t} d\IS$ are equal modulo $\IS\Omega^1_{\T}$.
\end{defi}
\begin{defi}\rm
If in the above definition $\IS\subset\O_{\T^\hol_{\check\sring},\t}$  (resp. $\IS\subset\O_{\T^\for_{\check\sring},\t}$ or $\IS\subset\O_{\T_{\check\sring}}$) then we say that $L$ is a holomorphic leaf (resp. formal leaf or algebraic leaf) of $\F(\Omega)$ defined over $\check\sring$ and write it $L^\hol_{\check\sring}$ (resp. $L^\for_{\check\sring}$ or $L^\alg_{\check\sring}$) if it is necessary to emphasize its property of being holomorphic, formal or algebraic, and its ring of definition. 
\index{$L^{\hol}, L^{\for}, L^{\alg}$, holomorphic, formal and algebraic leaf}
\end{defi}

Let $\t$ be a smooth point of $\T$, that is,  there are $z_1,z_2,\ldots,z_n\in\O_{\T^\hol_{\sring_\Q},\t}$ which generate it freely as $\sring_\Q$-algebra. We call $z=(z_1,z_2,\ldots,z_n)$ a  holomorphic coordinate system. 
By Frobenius theorem, see \cite[Theorem 5.8]{ho2020}, we can choose such a coordinate system such that $\Omega\otimes_{\O_\T}\O_{\T^\hol_{\sring_\Q},\t}$ is generated freely by $dz_i,\ i=1,2,\ldots,k$. It turns out that the analytic scheme given by the ideal $\IS=\langle z_1,z_2,\ldots,z_k\rangle$ is a leaf scheme of $\F(\Omega)$ and this is the classical notion of a leaf in the literature. We call it a general leaf. \index{General leaf} A general leaf is smooth (scheme theoretically) in the sense of \cite[Definition 5.8]{ho2020}.

As the reader might have noticed from the definition of a general leaf, for a general definition of leaf scheme we assume that integers are invertible in the underlying ring and that is why we must use $\check\sring:=\sring_\Q$.  
\begin{defi}\rm 
 We say that a leaf scheme $L$ is algebraic if $\IS\otimes_{\sring_\Q}\C$ is generated by elements in $\O_{\T_\C}$.
 In other words, there is a subscheme of $V$ of  $\T_\C$ containing the point $\t$ such that $L_\C$ is just $(V^\hol,\t)$.
\end{defi}
\begin{prop}
\label{10082024ulis}
 If a leaf scheme $L$ defined over a ring $\sring_\Q$ is algebraic then there is $N\in\sring$   such that $L$ is defined over $\sring[\frac{1}{N}]$.
\end{prop}
\begin{proof}
We  have  an ideal $J\subset \O_{\T_\C}$ such that it is  basically $\IS$:
 \begin{equation}
 \label{26062024greg}
 \IS\otimes_{\sring_\Q}\C=J\otimes_{\O_{\T_\C}}\O_{\T_\C^\hol,\t}. 
 \end{equation}
 Let $J=\langle P_1,P_2,\ldots,P_s\rangle$. In particular, we have regular algebraic functions $P_i$ defined  in $\T_\C$ such that they  define the germ of the analytic scheme $L$. After adding finitely many coefficients used in the expression of $P_i$'s to the quotient field $\sk$ of $\sring$, it follows that we have  a finitely generated field extension $\sk\subset \tilde\sk$ such that $J$ is defined over $\tilde\sk$, and hence by abuse of notation, we write $J\subset \O_{\T_{\tilde\sk}}$.
 In \eqref{26062024greg} we can replace all $\C$'s with $\tilde\sk$ and this implies that the algebraic ideal $J$ is invariant under the Galois group $\Gal(\tilde\sk/\sk)$. From this we can deduce that it is generated by elements defined over $\sk$, see \cite[page 19 Lemma 2]{Weil1946} and so we can assume that there are new generators  $P_i$'s  of $J$ defined over $\sk$ and in \eqref{26062024greg} we can replace all $\C$'s with $\sk$. After multiplication  of $P_i$'s by some elements of $\sring$  we get $P_i\in\O_\T$.
 The number $N\in \sring$ is the product of $N_i$'s attached to a set of generators $f_i$ of $\IS$ such that $N_if_i$ is in the ideal $\langle P_1,P_2,\ldots,P_s\rangle\subset \O_{\T}$.
\end{proof}

We believe that converse of \cref{10082024ulis} is true.
The statement that a leaf scheme $L$ is still defined over a finitely generated subring of $\C$ has strong consequences.
\begin{conj}
\label{19062024kontsevich}
 Let $\sring$ be a finitely generated subring of $\C$ and $\T$ be an $\sring$-scheme. 
 Let also $\F(\Omega)$ be a foliation on $\T$. If a leaf scheme is defined over $\sring$ then it is algebraic. 
\end{conj}

\section{Leaf scheme modulo prime}
\label{08082024lsmp}
In this section we explain a generalization of Grothendieck-Katz conjecture for leaf schemes  in the sense of \cref{18apr2018-2}. We call it  a  local-global principle for leaf schemes.
Recall that the leaves of foliations in our context have scheme structure and their structural sheaf might have  nilpotent  elements. They are also defined over rings, for instance by Frobenius theorem 
if we start with a foliation defined over $\sring$, the general leaves are defined over $\sring_\Q$.
They might also  have different codimensions for a given foliation. 
\begin{defi}\rm
We define\index{$\Theta_{\T, L}$, vector fields in $\T$ tangent to $L$}
\begin{equation}
\Theta_{\T, L}:=\left\{\vf\in\Theta_{\T}\   |\  \vf(\IS)\subset \IS \right \}, 
\end{equation}
and call it the module of vector fields in $\T$ tangent to $L$.
\end{defi}
\begin{conj}[Main local-global conjecture]
\label{02112023mainconj}
Let $\T$ be an $\sring$-scheme, $\t$ be an $\sring$-valued point of $\T$, $\F(\Omega)$ be a foliation on $\T$ and $L$ be a leaf scheme of $\F(\Omega)$ through $\t$.
If for all vector fields in $\T$ tangent to $L$, that is $\vf\in\Theta_{\T,L}$,  and all but a finite number of primes $p$, the vector field $\vf^p$ in $\T_p$ is also tangent to $L$ modulo $p$,  that is,  $\vf^p\in \Theta_{\T, L}\otimes_\sring \sring_p$, 
then the leaf $L$ is algebraic.
\end{conj} 
One may formulate a stronger conjecture with weaker hypothesis:
\begin{conj}\rm
\label{07092024omidviciado}
\cref{02112023mainconj}  is true if we replace in its hypothesis $\vf^p\in  \Theta_{\T, L}\otimes_\sring \sring_p$  with the weaker 
hypothesis  $\sch(\vf^p\in \Theta_{\T, L}\otimes_\sring \sring_p)$ contains the point $t$, that is, the vector field $\vf^p$ is tangent to $L$ at $t$. 
\end{conj} 
 It is not so hard to see that  $\Theta\subset \Theta_{\T, L}$ and even if we start with $\vf\in\Theta$ then $\vf^p$ might be in $\Theta_{\T,L}$ and not $\Theta$.
\begin{conj}
(Ekedahl, Shepherd-Barron, Taylor and Luntz see \cite[Page 165]{bost01} and \cite{Kontsevich2023})
\label{14112023tifi}
Let $\T$ be an $\sring$-scheme and $\F$ be a non-singular foliation on $\T$.
If for all vector fields $\vf$ in $\T$ tangent to the leaves of $\F$ and all but a finite number of primes $p$, $\vf^p$ is also tangent to the leaves of $\F$ then all the leaves of $\F$ are algebraic.
\end{conj}
Note that this is a particular case of \cref{02112023mainconj}. 
The main evidence to this conjecture (and the next one) is due to Bost in \cite{bost01} in which he proves this with an extra hypothesis on the leaves (Liouville property).

\begin{prop}
\label{11112023omid}
\cref{02112023mainconj} implies \cref{19062024kontsevich}.
\end{prop}
\begin{proof}
Let us take the leaf scheme $L$ defined over $\sring$. If $\vf\in\Theta_{\T,L}$ then 
we have $\vf(\IS)\subset \IS$ and hence $\vf^p(\IS)\subset \IS$. Since the leaf $L$ is also defined over $\sring$, it makes sense to talk about its reduction $L_p$ modulo $p$. In particular, a direct definition $\Theta_{\T_p, L_p}$ as in \eqref{03082024torio} is possible and 
$\Theta_{\T, L}\otimes_\sring\sring_p=\Theta_{\T_p, L_p}$.
This implies that $\vf^p$ is tangent to the leaf $L_p$ in $\T_p$. In particular, 
$\sch(\vf^p\in\Theta_{\T, L}\otimes_\sring \sring_p)\subset \T_p$ contains the point $\t$. This is exactly the hypothesis of \cref{02112023mainconj}.
\end{proof}

\begin{rem}\rm
\label{07092024omidviolento}
\cref{02112023mainconj} can be rewritten for foliations  given by a single vector field as follows:
Let $\T$ be an $\sring$-scheme, $\t$ be an $\sring$-valued smooth point of $\T$, and $\vf$ be a vector field  in $\T$ with $\vf(\t)\not=0$.
If for all but a finite number of primes $p$, $\vf^p$ is tangent to $L$ (the collinearness scheme $\vf||\vf^p$ contains $L$),  then $L$ is algebraic.
In a similar  way \cref{07092024omidviciado} in this case is the following: If we have $\vf$ as before
such that for all but a finite number of primes $p$, $\vf$ is collinear with $\vf^p$ at the point $\t$ (the collinearness scheme $\vf||\vf^p$ contains the point $\t$) then the solution $L$  of $\vf$ through $\t$ is algebraic.
In a private communication with Daniel Litt and Joshua Lam, they explained the author  that the second  conjecture is wrong. Their counterexample is $\vf:=zy\frac{\partial}{\partial y}+\frac{\partial}{\partial z}$ in $\A^2_{\Z}$. One can easily check that 
$\vf^ny=P_n(z)y$, where $P_n$ can be computed recursively $P_{n+1}=P_n'+zP_n$. For $n$ even and odd, $P_n$ is respectively even and odd polynomial. This implies that for $n$ odd  we have $P_n(0)=0$ which implies that $\vf^p(0,1)=0$ for all primes $p\not=2$. The solution through $(0,1)$ is given by $(z,e^{\frac{1}{2}z^2})$ which is not algebraic.
\footnote{Apart from \cref{02112023mainconj} we have also another correction to this conjecture which is the converse of \cref{19102023chinagain-2}.  
}
\end{rem}

\section{Vector fields tangent to leaf schemes}
\label{04aug2024fouad}
The formulation of \cref{02112023mainconj} with the module $\Theta_{\T,L}$ arises the question how big this module is? For instance, if $\Theta_{\T,L}=0$ then the hypothesis of \cref{02112023mainconj} is automatically satisfied and one might doubt this conjecture. From another side, if $\t$ is a smooth point of $\F(\Omega)$ then  $L$ is a general leaf and $\Theta\subset \Theta_{\T,L}$.
In this section we gather some statements about $\Theta_{\T,L}$.

Let $\F(\Omega)$ be a foliation in an $\sring$-scheme  $\T$, $L$ be a leaf scheme of $\F(\Omega)$ defined over a larger ring $\check \sring$) and through an $\sring$-valued point $\t$ of $\T$. 
\begin{defi}\rm
We define  $\bar L$ to be the Zarsiki closure of $L$ in $\T$, that is,\index{Zariski closure}
\begin{equation}
 \bar L:=\Zero \left(\bar\IS\right),\ \ \bar\IS:=\O_{\T}\cap \IS.
\end{equation}\index{$\bar L$, Zariski closure of $L$}
We say that $L$ is Zariski dense in $\T$ if $\bar L=\T$. In other words, 
$$
\O_{\T}\cap \IS=\{0\},
$$ 
which means that $\IS$ has not algebraic  elements. 
\end{defi} 
Note that the concept of Zariski closure depends on the underlying ring $\check\sring$. For instance $L$ might be Zariski dense for $\sring$ but not for an extension $\check\sring$ of $\sring$.

\begin{prop}
Let $\T$ be an $\sring$-scheme and $\t$ be an $\sring$-valued point of $\T$, not necessarily smooth, and $L$ be a leaf scheme of $\F(\Omega)$ through $\t$. We have
 \begin{equation}
 \label{11082024hugo}
 \Theta_{\T,L}=\{ \vf\in \Theta_{\T} \ |\ \vf(\Omega)\subset \bar\IS\}.
 \end{equation}
\end{prop}
We may call the right hand side of \eqref{11082024hugo} the dual of $\Omega$ along the Zariski closure of $L$.
It is a kind of surprising that we do not need to insert $\vf(\bar\IS)\subset \bar\IS$ in \eqref{11082024hugo} as it will be clear in the proof.
 \index{Dual along $\bar L$}\index{$\Theta^{\bar L}$, Dual along $\bar L$}
\begin{proof}
Proof of $\subset$:  If $\vf\in \Theta_{\T,L}$ then $\vf(\IS)\subset \IS$, and since $\vf$ is algebraic, that is 
$\vf(\O_\T)\subset \O_\T$),  we get 
$\vf(\bar\IS)\subset \bar\IS$. Moreover, by definition of a leaf scheme  we have $\Omega=\O_\T d\IS$ modulo $\IS\Omega_\T$ (and after tensoring with holomorphic functions). Taking $\vf$ from both sides we have $\vf(\Omega)\subset \IS$, and since $\vf(\Omega)\in\O_\T$, we get
$\vf(\Omega)\subset \bar \IS$. 

Proof of $\supset$: Let $\vf\in \Theta_{\T}$  with $\vf(\Omega)\subset \bar\IS$.
By definition of a leaf scheme $d\IS$ is in $\Omega\otimes_{\O_\T}\O_{\T^\hol,\t}+\IS\Omega_\T$.
Therefore, $\vf(\IS)$ is in $\vf(\Omega)\O_{\T^\hol,\t}+\IS$. Since we have $\vf(\Omega)\subset \bar\IS\subset\IS$ we get the result. 

\end{proof}

\begin{prop}
\label{27062024pezeshkian}
If there is an element of $\Theta_\T(\Omega)$ which is not a zero divisor (in particular if $\O_\T$ has no zero divisors and  $\Theta\subsetneq\Theta_\T$)
then the following are equivalent:
\begin{enumerate}
\item
$L$ is Zariski dense in $\T$.
\item 
$
\Theta=\Theta_{\T, L}$.
\end{enumerate}
\end{prop}
\begin{proof}
1. implies 2.: If $L$ is Zariski dense then $\bar\IS=0$ and this follows from \eqref{11082024hugo}.
2. implies 1.: 
Assume that $L$ is not Zariski dense. The algebraic ideal 
$\bar\IS$ has a non-zero element $f$, and $f\Theta_{\T}\subset \Theta_{\T,L}$. If  $\Theta_{\T,L}=\Theta$ then 
$f\Theta_{\T}(\Omega)=0$. But by our assumption, $\Theta_{\T}(\Omega)$ has an element which is not a zero divisor and we get a contradiction. 
\end{proof}

\begin{prop}\rm
 We have a natural $\sring$-bilinear pairing $\Omega_\T\times\TS_\t\T\to\sring,\ (\omega,v)\mapsto \omega(v)$ such that 
 $$
 \TS_\t L=\{v\in \TS_\t\T\ \mid \ \Omega(v)=0\}.
 $$
\end{prop}
\begin{proof}
 First, let us define the pairing. For $f,g\in\O_\T$ and $\omega=fdg$ we define $\omega(v):=f_0v(g-g_0)$, where $g_0=g(\t), \ f_0=f(\t)$. Since $\Omega_\T$ as $\sring$-module is generated by $fdg$'s, the definition extend to $\Omega^1_\T$. It is well-defined because
 {\tiny 
 \begin{eqnarray*}
 (d(fg)-fdg-gdf)(v) &=& v(d((f-f_0+f_0)(g-g_0+g_0)-f_0g_0))-f_0v(g-g_0)-g_0v(f-f_0)\\
 &=& v((f-f_0)(g-g_0))=0.
 \end{eqnarray*}
 }
 If $v\in \TS_\t L$ then we know that $\TS_\t L=\left(\frac{\mk_{\T^\hol,t}}{\IS+\mk_{\T^\hol,\t}^2}\right)^\vee$, and so we get a map $v:\O_{\T^\hol,t}\to\sring$ with $v(\IS)=0$. We use the definition of a leaf scheme and we have $\Omega\subset \O_{\T^\hol,t}d\IS+\IS\Omega_\T$. Applying $v$ to both sides we get $\Omega(v)=0$. Conversely if we have $v\in \TS_\t\T$ with $\Omega(v)=0$ then we use again the definition of leaf scheme and we have $d\IS\subset \Omega\otimes \O_{\T^\hol,t}+\IS\Omega_\T$. Applying $v$ to both sides we conclude that $v(\IS)=0$ 
 \end{proof}
 
\begin{prop}
\label{11082024ku}
 The evaluation at $\t$ induces a map $\Theta_{\T,L}\to \TS_\t L$ and if $L$ is a general leaf then it is surjective. 
\end{prop}
\begin{proof}
 Let $\vf\in \Theta_{\T,L}$ and so $\vf(\IS)\subset\IS$.  This implies that $\vf(\IS)(\t)=0$ and so
 $\vf(\t)\in \TS_\t L$. Therefore, the evaluation at $\t$ induces a map $\Theta_{\T,L}\to \TS_\t L$. We have 
 $\Theta\subset \Theta_{\T,L}$, and if $L$ is a general leaf we claim that $\Theta\to \TS_\t L$ is surjective. In order to see this, we first define $\Theta^\hol:=\{\vf\in \Theta_{\T^{\hol},\t} \ \mid\  \vf(\Omega)=0\}$. Since $\t$ is a smooth point of $\T$, we have a holomorphic coordinate system given by the Frobenius theorem and in this coordinate system
 the surjectivity of $\Theta^{\hol}\to \TS_\t L$ can be checked easily. The proof finishes with
 $\Theta\otimes_{\O_\T }{\O_{\T^\hol,\t}}=\Theta^\hol$ which is just a linear algebra. 

\end{proof}
\begin{rem}\rm
Note that the kernel of $\Theta_{\T,L}\to \TS_\t L$ contains vector fields $\vf$ such that
$\vf(\t)=0$. In this case $\t$ is called the singularity of $\vf$.
Moreover, for an arbitrary leaf scheme $L$, the map $\Theta_{\T,L}\to \TS_\t L$ may not be  surjective. For instance,  let 
$\T:=\A^2_\sring$ with the coordinates $(x,y)$, $L: xy=0$ and $\Omega=\O_\T (xdy+ydx)$. It can be easily checked that $L$ is a leaf scheme of $\F(\Omega)$, $\Theta_{\T,L}=\{xf\frac{\partial}{\partial x}+yg\frac{\partial}{\partial y},\ \ f,g\in\O_\T \}$, $\TS_0L=\TS_0\T$ and $\Theta_{\T,L}\to \TS_\t L$ is the zero map. 
\end{rem}

\section{Foliations of linear differential equations}
\label{04112023weixinpay}
Let $\V$ be an $\sring$-scheme  and $\gma $ be a $\hn\times\hn$ matrix with entries which are global sections of
$\Omega^1_\V$. Let also $0$ be an $\sring$-valued smooth point of $\V$.
Recall that $\sring_\Q:=\sring\otimes _\Z\Q$ and if $\sring\subset\C$ then this is the smallest subring of $\C$ containing both $\sring$ and $\Q$. The first fundamental theorem of linear differential equations is:
\begin{theo}\rm
\label{13112023israel}
For any $y_0\in\sring^\hn$, the linear differential equation $dy=\gma y$ has a unique solution 
$y\in\O_{\V_\C^\hol,0}^\hn$ with $y(0)=y_0$ if and only if  $\gma $ satisfies the integrability condition $d\gma =\gma \wedge \gma $.
Moreover, if $\gma $ is defined over a ring $\sring\subset\C$ then $y_i$'s are convergent.
\end{theo}
\begin{proof} 
The direction $\Rightarrow$ is easy and we only prove this. 
We take a basis $e_i,\ i=1,2,\ldots,\hn$ of $\sring^\hn$ and find $\hn$ linearly independent solutions $y_i,\ y_i(0)=e_i \ i=1,2,3,\ldots,\hn$.
We put all $\hn$ solutions $y_1,y_2,\cdots, y_{\hn}$  inside a $\hn\times \hn$ matrix $Y=[y_1,y_2,\ldots,y_\hn]$, and by our hypothesis on the initial values of $y_i$'s, we have $\det(Y(0))\not=0$ and so $\gma =dY\cdot Y^{-1}$. Therefore,
$$
d\gma =-dY\cdot d(Y^{-1} )=dY\cdot Y^{-1}\cdot dY\cdot Y^{-1}=\gma \wedge \gma .
$$

\end{proof}
Let us consider new variables $x_1,x_2,\cdots, x_{\hn}$, and define
$$
O=\spec\left(\sring[x_1,x_2,\ldots,x_{\hn}] \right).
$$
The entries of $dx-\gma x$ are differential forms in $\T:=\V\times_\sring O$.
If $\gma $ is integrable in the sense of \cref{13112023israel} then the  $\O_\V$-module generated by the entries of $dx-\gma x$ is integrable. This follows from 
$$
d(dx-\gma x)=-d(\gma x)=\gma\wedge dx-(d\gma) x=\gma \wedge (dx-\gma x)+(\gma\wedge \gma-d\gma)x.
$$
We denote the corresponding foliation in $\T$ by $\F(dx-\gma x)$.
\begin{prop}
All the leaves of $\F:=\F(dx-\gma x)$ are general and hence smooth.
\end{prop}
\begin{proof}
First note that for a point $(0,y_0)\in \T:=\V\times O$, there is a unique solution $y(\t)$ of $dy=\gma y, \ y(0)=y_0$. This gives us the leaf $L$ whose ideal is generated by the entries of $x-y(t)$.
The linear part of the entries of $x-y(t)$ is $x$ minus the linear part of $y(t)$ at $0$.
These are trivially $\hn$ linearly independent functions due to the presence of $x$.
\end{proof}

Let $W$ be a  subvariety of $V$. It turns out that the foliation $$\F(dx-\gma x)\cap W:=\F((dx-\gma x)|_W)
$$ might have non-smooth leaves and one of the main goals of the present text is to study such foliations. The main example for $W$ is $x_1=x_2=\cdots=x_k=0$ for some $k<n$.

\paragraph{Gauss-Manin connection}

 Let $Y\to \V$ be a family of smooth complex projective varieties over the field $\sk$ and let $\V$ be irreducible, smooth and affine. Around any point of $\V$ we can find  global
sections $\omega$ of the $\mov$-th relative de Rham cohomology sheaf of $Y/\V$ such that $\omega_i, i=1,2,\ldots,\hn$ at each fiber
$H_\dR^\mov(Y_t), \ t\in \V$ form a basis compatible with the Hodge filtration. If it is necessary we may
replace $\V$ with a Zariski open subset of $\V$. We write the Gauss-Manin connection of $Y/\V$ in the basis $\omega=[\omega_1,\omega_2,\ldots,\omega_{\hn}]^\tr$:  $\nabla\omega=\gma \otimes\omega$, where $\gma $ is ${\hn}\times {\hn}$ matrix with entries which are differential forms in $\V$. As all our objects $Y\to \V,\ \t,\ \omega_i$ etc. use a finite number of coefficients in $\sk$,  we can take a model of all these over a finitely generated ring $\sring$ so that $\sk$ is the quotient field of $\sring$. For simplicity, we use the same notations for these objects defined over $\sring$.
\begin{theo}
\label{27nov2023shhshahani}
 Let $\sring\subset \C$.
The linear differential equation $dy=\gma y$ for an unknown ${\hn}\times 1$ matrix  $y$ with entries which are holomorphic functions in $(\V,\t)$, has a basis of solutions given by $\int_{\delta_t}\omega$, where $\delta_t$ ranges in a basis of $H_\mov(Y_t,\Z)$.
\end{theo}
This is classical statement in Hodge theory, see for instance or \cite[Theorem 9.3]{ho13-Roberto}.

 \section{Foliations attached to Hodge loci}
 \label{04dec2023khar}
 In this section, we write  the expression of a foliation in $\T:=F^{\frac{\mov}{2}}H^\mov_\dR(Y/V)$ with Hodge loci as leaf schemes.
 A version of these foliations developed in \cite[Chapter 5,6]{ho2020} is not suitable for our purpose and we mainly use the version in \cite{GMCD-NL}.
 
 Let $\hn:=\hn^{\mov,0}+\cdots+\hn^{\frac{\mov}{2},\frac{\mov}{2}}+\cdots+\hn^{0,\mov}$ be the decomposition of $\hn$ into Hodge numbers of $H^\mov_\dR(Y_t)$ and $\hn^{i}:=\hn^{\mov,0}+\cdots+\hn^{i,\mov-i}$. 
 We take variables $x_1,x_2,\ldots, x_{\hn^{\frac{\mov}{2}}}$ and put them in a $\hn\times 1$ matrix $x$ as below.
The first $\frac{\mov}{2}$ Hodge blocks of $x$  are zero and $x_i$'s are listed in the next blocks:
\begin{equation}
\label{desesperado}
x=\begin{pmatrix}
      0\\
      \vdots\\
      0\\
      x^{\frac{\mov}{2}}\\
      \vdots\\
      x^{\mov}
     \end{pmatrix}. 
\end{equation}
Here, $x^i$ is a $\hn^{\mov-i,i}\times 1$  matrix. 
We take $\cpe$ the constant matrix which is obtained by replacing $x_i$ with $0$ in $x$ except for $x_1$, which is replaced with $1$ (this is the first coordinate of $x^{\frac{\mov}{2}}$).
Let $\Smat$ be a  Hodge block lower triangular  $\hn\times \hn$
matrix which is  obtained from the identity matrix by replacing the $\hn^{\frac{\mov}{2}+1}+1$ column with $x$. It is defined in this way to have the equality  \begin{equation}
\label{08august2014}
\Smat\cdot \cpe=x.
\end{equation}
In this way $\Smat^{-1}$ is obtained from $\Smat$ by
replacing $x_1$ with $x_1^{-1}$ and $x_i,\ i\geq 2$ with
$-x_ix_1^{-1}$. Note that $\det(\Smat)=x_1$. Define
$$
O:=\spec\left(
\sring\left [x_1,x_2,\ldots, x_{\hn^{\frac{\mov}{2}}}, \frac{1}{x_1}\right ]
\right).
$$
We consider the family $\X\to\T$, where $\X:=Y\times O,\
\T:= \V \times O$. It is obtained from
$Y\to \V$ and the identity map $O\to O$. We also define $\alpha$ by
\begin{equation}
\label{august2014}
\alpha:=\Smat^{-1}\cdot\omega.
\end{equation}
Let $\nabla: H^\mov_\dR(Y/\V)\to \Omega_\V\otimes_{\O_\V}
H^\mov_\dR(Y/\V)$ be the algebraic Gauss-Manin connection. We can write $\nabla$
in the basis $\omega$ and define the $\hn\times\hn$ matrix $\gma$ by the equality:
$$
\nabla\omega=\gma\otimes \omega.
$$
The entries of $\gma$ are differential 1-forms in $\V$. In a similar way we can compute the Gauss-Manin connection  of $\X/\T$ in the basis
$\alpha$:
\begin{equation}
 \label{omidtaiwan}
\nabla\alpha=\gm\otimes \alpha,\ \ \gm=-\Smat^{-1} d\Smat+\Smat^{-1}\cdot \gma\cdot \Smat.
 \end{equation}
This follows from the construction of the global sections $\alpha$ in \eqref{august2014} and the Leibniz rule.
We call $\gma$ (resp. $\gm$) the Gauss-Manin connection matrix of the pair $(Y/\V, \omega)$ (resp. $(\X/\T,\alpha)$).
 From the integrability of the Gauss-Manin connection it follows that
\begin{equation}
\label{nancy2014}
d\gm=\gm\wedge \gm.
\end{equation}
\begin{defi}\rm
The entries of $\gm\cpe$ induce  a holomorphic foliation $\F(\cpe)$ in $\T$. The integrability 
follows from
(\ref{nancy2014}):
$$
 d(\gm\cdot \cpe)=d\gm \cdot \cpe=\gm\wedge (\gm\cdot \cpe).
$$
\end{defi}
\begin{prop}
The foliation $\F(\cpe)$ in  $\T$ is given by
\begin{eqnarray}
\label{7aug2014}
0&=&\gma^{\frac{\mov}{2}-1,\frac{\mov}{2}} x^{\frac{\mov}{2}} \\
\label{7aug2014-2}
dx^\frac{\mov}{2} &=&  \gma^{\frac{\mov}{2},\frac{\mov}{2}} x^\frac{\mov}{2}+\gma^{\frac{\mov}{2},\frac{\mov}{2}+1} x^{\frac{\mov}{2}+1} ,
\\
\label{7aug2014-1}
dx^i &=& \sum_{j=\frac{\mov}{2}}^\mov \gma^{i,j} x^j,\ \ \ \ \ i=\frac{\mov}{2}+1,\ldots,\mov.
\end{eqnarray}
\end{prop}
\begin{proof}
 For this we use \eqref{omidtaiwan} and we conclude that $\F(\cpe)$ is given by
 $(-\Smat^{-1} d\Smat+\Smat^{-1}\cdot \gma\cdot \Smat)\cpe$. Since $\cpe$
is a constant vector and $\Smat$ is an invertible matrix and we have (\ref{08august2014}),
we conclude that $\F(\cpe)$ is given by the entries of
$$
dx-\gma x=0.
$$
Opening this equality and
using the zero blocks of $x$ in (\ref{08august2014}) we
get \eqref{7aug2014}, \eqref{7aug2014-2} and \eqref{7aug2014-1}.
Note that by Griffiths transversality $\gma^{i,j}=0$ for $j-i\geq 2$.
\end{proof}

Let $\delta_t\in H_\mov(\X_t,\Q)\otimes_\Q \C,\ t\in(\T,0)$
be a continuous family of cycles, that is, the Poincar\'e dual of $\delta_t$ is a flat section
of the Gauss-Manin connection: $\nabla \delta_t=0$. Here, $(\T,0)$ is a small neighborhood of $0$ in $\T$ in
the usual topology.
\begin{prop}
 \label{21062024fred}
The following
\begin{equation}
\label{16aug2014}
L_{\delta_t}:= \left \{ \ \ \ t\in (\T,0) \ \ \ |  \ \ \
\int_{\delta_t}\alpha=\cpe. \right \},
\end{equation}
is a leaf scheme of $\F(\cpe)$. In other words, the ideal 
$\IS_{\delta_t}$ generated by the entries of $\int_{\delta_t}\alpha-\cpe$ gives a leaf scheme of $\F(\cpe)$. Moreover, $\delta_t$ is a general Hodge cycle if and only if $0$ is a smooth point of $\F(\Omega)$, and hence, $L_{\delta_t}$ is general. 
\end{prop}
\begin{proof}
We have the holomorphic function
$$
f: (\T,0)\to \C^{\hn}, \ \
f(t):=\int_{\delta_t}\alpha  -\cpe
$$
which satisfies
\begin{equation}
 \label{28sep2014}
 df=\int_{\delta_t}\nabla\alpha =\gm\cdot \int_{\delta_t}\alpha =\gm\cdot \cpe+\gm\cdot f.
 \end{equation}
 This implies that
 $L_{\delta_t}$ is a
 leaf scheme of $\F(\cpe)$. For the second part, note that 
 $\gma^{\frac{\mov}{2}-1,\frac{\mov}{2}} x^{\frac{\mov}{2}}$ is just the matrix format of $\alpha$ in \eqref{08082024decathlon}, see 
 \cite[Proposition 5]{GMCD-NL}. The singular locus of $\F(\cpe)$ corresponds to those $t\in\T$ such that the kernel of the entries of $\gma^{\frac{\mov}{2}-1,\frac{\mov}{2}} x^{\frac{\mov}{2}}$ is not minimal, and the statement follows. 
 \end{proof}
One might conjecture that all the leaf schemes  of $\F(\cpe)$ are of the form $L_{\delta_t}$.
As we do not need this kind of statements, we leave it to the reader. 

\begin{defi}\rm
\label{26062024omidbita}
The Hodge locus  with constant periods $\cpe$ is defined to be $L_{\delta_t}$ in \eqref{16aug2014}
with $\delta_t\in H^\mov(\X_\t,\Q)$. Its ideal is given by
$$
\IS_{\delta_t}:=\left\langle \int_{\delta_t}\alpha-\cpe\right\rangle=\left\langle \int_{\delta_t}\omega-x\right\rangle\subset \O_{\T^\hol,0}.
$$
\end{defi}
From the zero blocks of $\cpe$, it follows that the Poincar\'e dual $\delta_\t^\podu$  of $\delta_t$ is in $H^{\frac{\mov}{2},\frac{\mov}{2}}\cap H^{\mov}(\X_\t,\Z)$ and so $\delta_t$ is a Hodge cycle in the classical sense.

\begin{rem}\rm
 For the example mentioned in the Introduction, we take $E_{z}: y^2=x(x-1)(x-t)$ the Legendre family of elliptic curves, and $Y_{t_1,t_2}:=E_{t_1}\times E_{t_2}$. The de Rham cohomology $H^2_\dR(Y_{t_1,t_2})$ modulo the cohomology class of fibers of projections in each factor, is four dimensional with Hodge numbers $1,2,1$. The loci of isogenies of degree $N$ between $E_{t_1}$ and $E_{t_2}$ is an algebraic curve and we can construct the corresponding leaf scheme and foliation in $(t_1,t_2,x_1,x_2,x_3)\in\C^5$. For a more conceptual treatment of this example see \cite[Chapter 10]{ho2020}.
\end{rem}

\section{The definition ring of Hodge loci}
\label{08082024dr}
Let $Y\to\V$ be a family of smooth  projective varieties over $\C$ and take a smooth  model of this over $\sring\subset\C$.
Let also $0$ be an $\sring$-valued point of $\V$. We take  a topological cycle
$\delta_0\in H_m(Y_0,\Z)$ and  enlarge  $\sring$ to
$$
\sring(\delta_0):=\sring\left [\frac{1}{(2\pi i)^{\frac{m}{2}}}\int_{\delta_0}\omega, \omega\in H^m_\dR(Y_0/\sring)\right ],
$$
that is, $\sring(\delta_0)$  is the ring of polynomials in the periods $(2\pi i)^{-\frac{m}{2}}\int_{\delta_0}\omega, \omega\in H^m_\dR(Y_0/\sring)$. Note that these numbers are conjecturally in $\bar\sk$, where $\sk$ is the quotient field of $\sring$. This follows from the Hodge conjecture or if we assume that $\delta_0$ is an absolute Hodge cycle. 
We will not need this for our investigation, and hence, $\sring(\delta_0)$   
might contain new transcendental numbers.
Consider the monodromy $\delta_t\in H_m(Y_t,\Z)$ of $\delta_0$ to nearby fibers $t\in(\V,0)$. 
\begin{theo}
\label{05082024retaliation}
Let $0$ be an $\sring$-valued point  of $\V$ and $\omega\in H^{m}_\dR(Y/\V)$. Then the Taylor series of $\int_{\delta_t}\omega$ at $0$ has coefficients in $\sring(\delta_0)_{\Q}:=\sring(\delta_0)\otimes_\Z\Q$, that is,
$$
\int_{\delta_t}\omega\in \O_{\V^\hol_{\sring(\delta_0)_\Q},0}.
$$
\end{theo}
\begin{proof}
 This follows from \cref{27nov2023shhshahani} and \cref{13112023israel}. The main ingredient is that
the Gauss-Manin connection matrix of of $Y/\V$ is defined over $\sring$.
\end{proof}
It might be too naive to believe  that if $\delta_0$ is a Hodge cycle then it has coefficients in $\sring(\delta_0)$ itself, that is, $\int_{\delta_t}\omega\in \O_{\V^\hol_{\sring(\delta_0)},0}$. In \cref{05aug2024delaram} we experimentally observe that this is false. 
\begin{theo}
\label{21june2024baran}
The Hodge locus $L_{\delta_0}$ defined in \cref{26062024omidbita} is defined over the ring $\sring(\delta_0)[\frac{1}{N}]$ for some $N\in \sring(\delta_0)$.
\end{theo}
Using \cref{21062024fred} this theorem  is a particular  case of  \cref{19062024kontsevich}.
Actually \cref{19062024kontsevich}  is inspired by \cref{21june2024baran}.
\begin{proof}
The Cattani-Deligne-Kaplan theorem, see \cite[Theorem 1.1]{cadeka},  implies that there is an algebraic subvariety $L$ of $\T_\C$ such that $L_{\delta_t}$ is just an analytic germ/small open subset of $L$. By monodromy of $\delta_t$ along any path in $L$, we can see that $L$ is covered by such open sets. Therefore, $L$ has the structure of an analytic scheme given by analytic ideals $\IS_{\delta_t}$.  We have to consider a compactification $\bar\T_\C$ and see that the analytic scheme structure of $L$ extends to $\bar L$. This is not explicitly mentioned in \cite{cadeka}, but their proof implies this. Now by GAGA for analytic subschemes of projective varieties we conclude that  $\IS_{\delta_t}$ is algebraic. \cref{10082024ulis} finishes the proof. 
\end{proof}

\begin{rem}\rm
Let $\sk$ be the quotient field of $\sring$. If $\delta_0$ is an  absolute Hodge cycle then all its periods $\frac{1}{(2\pi i)^{\frac{m}{2}}}\int_{\delta_0}\omega, \omega\in H^m_\dR(Y_0)$ are in the algebraic closure $\bar\sk$ of $\sk$. Since $\F(\cpe)$ is defined over $\sk$, by taking the action of the Galois group $\Gal(\bar\sk/\sk)$ on the coefficients of $L$ and the base point $(0,x_0)$, we get finitely many leaf schemes of $\F(\Omega)$ which come from Hodge cycles too. If $\delta_0$ is not absolute then some of the period of $\delta_0$ are transcendental numbers, that it they do not belong to $\bar\sk$. In this case by some standard arguments, see for instance \cite[Theorem 5.18]{ho2020}, we can transform such transcendental numbers into variables and get continuous families of algebraic leaf schemes for $\F(\cpe)$. In general it is open whether such leaf schemes come from Hodge cycles or not. For some related results see \cite[Section 7.5]{ho2020}. 
\end{rem}

\section{Proofs}
\label{13082024shekam}

\begin{proof}{(of \cref{17072024ziquan})}
 The proof starts with \cref{CADEKA} which together with \cref{10082024ulis} implies \cref{21june2024baran}. We conclude that a Hodge locus is defined over a finitely generated subring of $\C$ and in the same way as in the proof of \cref{11112023omid}, we conclude that for all but a finite number of primes 
 reduction modulo $p$ of $L$ makes sense and  the result follows. 
\end{proof}
\begin{proof}{(of \cref{09082024armand})}
 It is well-known that
$$
\TS_\t L=\{v\in \TS_t\T | \alpha(v, \bar\t)=0\},
$$
see for instance, \cite{CGGH1983}. In other words, $v\in \TS_\t L$ is characterised by the fact that 
the infinitesimal monodromy of $\bar\t$ along the vector $v$ is still in $\T$. If the Hodge cycle is general then by \cref{21062024fred} its Hodge locus is a general leaf scheme and by \cref{11082024ku}, the map $\Theta_{\T,L}\to \TS_\t L$ is surjective. This together with \cref{17072024ziquan} finishes the proof. 
\end{proof}

\section{Taylor series of periods over Hodge cycles}
\label{05aug2024delaram}
In order to prove \cref{17072024ziquan} and \cref{09082024armand} and verify the hypothesis of \cref{19062024kontsevich}  by a local analysis, we have to investigate the defining ideal of Hodge loci. These are given explicitly in terms of periods, and the most general fact about their ring of definition is \cref{05082024retaliation} which is not enough.
In this section we give closed formulas for such coefficients for families of hypersurfaces near the Fermat variety and experimentally observe that all primes might be inverted in the Taylor series of periods, and hence, an strategy is needed to modify them and obtain new generators of the defining ideal of Hodge loci.

Let us consider the hypersurface $X_t$ in the projective space 
$\Pn {n+1}$ given by the homogeneous polynomial:
\begin{equation}
\label{15dec2016-2}
f_t:=-x_0^{d}+x_1^{d}-x_2^d+x_3^d+\cdots-x_n^d+x_{n+1}^d-\sum_{\alpha}t_\alpha x^\alpha=0,\ \ 
\end{equation}
$$
t=(t_\alpha)_{\alpha\in I}\in(\T,0), 
$$
where $\alpha$ runs through a finite subset $I$ of 
$\N_0^{n+2}$ with $\sum_{i=0}^{n+1} \alpha_i=d$. 
For a rational number $r$ 
let $[r]$ be the integer part of $r$, that is $[r]\leq r<[r]+1$, and  $\{r\}:=r-[r]$. Let
also $(x)_y:=x(x+1)(x+2)\cdots(x+y-1),\ (x)_0:=1$ be the Pochhammer symbol. 
We compute the Taylor series of 
the integration of differential forms 
over monodromies of the algebraic cycle 
$$
\licy{\frac{n}{2}}: x_0-x_1=x_2-x_3=\cdots=x_n-x_{n+1=0},
$$ inside the Fermat variety $X_0$. 
The following has been proved in \cite[Theorem 18.9]{ho13}
\begin{theo}
\label{InLabelNadasht?}
Let $\delta_{t}\in H_n(X_t,\Z),\ t\in(\T,0)$ be the monodromy (parallel transport) of the cycle 
$\delta_0:=[\licy{\frac{n}{2}}]\in H_n(X_0,\Z)$ along a path which 
connects $0$ to $t$. 
For a monomial $x^\beta=x_0^{\beta_0} x_1^{\beta_1}x_2^{\beta_2}\cdots x_{n+1}^{\beta_{n+1}}$ with 
$k:=\sum_{i=0}^{n+1}\frac{\beta_i+1}{d}\in\N$  
we have 
 \begin{equation}
 \label{15.12.16}
 \frac{(-1 )^{\frac{n}{2}} \cdot   d^{\frac{n}{2}+1}  \cdot (k-1)!}{  (2\pi \sqrt{-1})^{\frac{n}{2}}}
 \mathlarger{\mathlarger{\int}}_{\delta_t}\Resi\left(\frac{x^\beta\Omega}{f^{k}_t}\right)=
\mathlarger{\mathlarger{\mathlarger{\sum}}}_{a: I{}\to \N_0}
\frac{(-1)^{E_{\beta+a^*}} D_{\beta+ a^*} }{ a! }\cdot  t^a,
\end{equation}
where  the sum runs through all $\#I{}$-tuples $a=(a_\alpha,\ \ \alpha\in I{})$
of non-negative integers such that for $\check\beta:=\beta+a^*$ we have 
\begin{equation}
\label{TheLastMistake2017}
\left\{\frac{ \check\beta_{2e}+1}{d} \right\}+   \left\{\frac{ \check\beta_{2e+1}+1}{d} \right\}=1,\ \ \ 
\forall  e=0,1,\cdots, \frac{n}{2}, 
\end{equation}
and 
\begin{eqnarray*}
t^a:&=&\prod_{\alpha\in I{}}t_\alpha^{a_\alpha}, \ \ \ \
a!:=\prod_{\alpha\in I{}}a_\alpha!,  \ \   \ \   a^* := \sum_{\alpha}a_\alpha\cdot \alpha, \\ 
D_{\check\beta} &:=& 
\mathlarger{\prod}_{i=0}^{n+1}\left( \left\{\frac{\check\beta_i+1}{d}\right\}\right)_{ \left[\frac{\check\beta_i+1}{d}\right ]},\ \ \ \ 
 E_{\check\beta} := 
		   \sum_{e=0}^\frac{n}{2} \left[\frac{\check\beta_{2e}+1}{d}\right]. 
\end{eqnarray*}
\end{theo}
Note that for two types of $a$ the coefficient of $t^a$ in \eqref{15.12.16} is zero.
First, when $\beta+a^*$ does not satisfy \eqref{TheLastMistake2017}. 
Second,   when an entry of $\beta+a^*$ plus one is divisible by $d$ 
(this is hidden in the definition of $D_{\beta+a^*}$). The coefficients of
the Taylor series are in $\Q$. The Taylor series \eqref{09082024david} is just obtained from
\eqref{15.12.16} by setting $n=2,\ d=4,\ \beta=0, \ k=1$. Since in this case $H^{2,0}$ is one dimensional and it is generated by $\omega:=\Resi\left(\frac{\Omega}{f_t}\right)$, the Hodge locus in $V$ corresponding to the homology class of the line $\P^1: x_0-x_1=x_2-x_3=0$  is given by the zero locus of $\int_{\delta_t}\omega$ whose Taylor series at $\t=0$ is given in  \eqref{09082024david}.

In the following, we use the computer implementation of \eqref{TheLastMistake2017} and its general format in \cite[Theorem 18.9]{ho13}. The full family  of hypersurfaces has too much parameters, and so, one has to consider lower truncation of the Taylor series.
Our main goal is to show that the natural generators of the ideal of a Hodge locus, are not necessarily defined over $\sring(\delta_0)$ and one has to invert  infinitely many primes.  In particular, in \cref{19062024kontsevich} the natural generators of the ideal might not prove this conjecture, and that is why, in this conjecture we have claimed that the ideal is defined over the ring $\sring$ and not its natural generators. For example, we consider the family
$$
X:\ \ x_1^4+x_2^4+x_3^4+x_0^4-t_0x_0x_1^3-t_1x_1x_2^3-t_2x_2x_3^3-t_3x_3x_0^3=0. 
$$
In this case all the Griffiths basis of differential forms for $H^2_\dR(X_t)$ 
has Taylor series with primes appearing in their denominators, but very slowly. For truncation  with degree $\leq 30$ we get the following denominator:
{\tiny
$$
1, 2^{84}\cdot 5\cdot 7\cdot 11\cdot 13 : 
x_2^2x_3^2, 2^{86}\cdot 7\cdot 11\cdot 13:
x_1x_2x_3^2, 2^{84}\cdot 11:
x_0x_2x_3^2 , 2^85\cdot 7\cdot 11\cdot 13:
x_1^2x_3^2 , 2^{86}\cdot 11\cdot 13:
$$
$$
x_0x_1x_3^2 ,2^{86}\cdot 7\cdot 11\cdot 13:
x_0^2x_3^2, 2^{86}\cdot 7\cdot 11\cdot 13:
x_1x_2^2x_3 , 2^{85}\cdot 7\cdot 11\cdot 13:
x_0x_2^2x_3, 2^{85}\cdot 7\cdot 11\cdot 13:
x_1^2x_2x_3, 2^{86}\cdot 7\cdot 11\cdot 13:
$$
$$
x_0x_1x_2x_3, 2^{85}\cdot 11:
x_0^2x_2x_3, 2^{86}\cdot 11:
x_0x_1^2x_3, 2^{84}\cdot 11:
x_0^2x_1x_3, 2^{85}\cdot 7\cdot 11\cdot 13:
x_1^2x_2^2, 2^{86}\cdot 7\cdot 11\cdot 13:
$$
$$
x_0x_1x_2^2, 2^{86}\cdot 11:
x_0^2x_2^2, 2^{86}\cdot 11\cdot 13:
x_0x_1^2x_2, 2^{85}\cdot 7\cdot 11\cdot 13:
x_0^2x_1x_2, 2^{85}\cdot 7\cdot 11\cdot 13:
x_0^2x_1^2, 2^{86}\cdot 7\cdot 11\cdot 13:
$$
$$
x_0^2x_1^2x_2^2x_3^2, 2^{88}\cdot 7\cdot 11\cdot 13:
$$
}
where we have written the monomial $x^\beta$ in 
$\frac{x^\beta\Omega}{f^{k}_t}$ and then the denominator of its period, separated by two points. For the computer code used for this computation, see 
\href{https://w3.impa.br/~hossein/WikiHossein/files/Singular%20Codes/2024_08_09_Primes_Taylor_Series_Periods.txt}
{the author's webpage here} or the latex text of the present paper in arxiv.
{\tiny

}

\section{Hodge-Tate varieties}
\label{10082024kk}
One of the main goals of the present text has been to elaborate a generalization of Grothendieck-Katz conjecture which implies the algebraicity of the Hodge loci. This is also a consequence of the Hodge conjecture. It turns out that we can formulate  another statement on the algebraicity of periods which is a common consequence of both Hodge and the original Grothendieck-Katz conjecture.

\begin{defi}\rm
 The $m$-th cohomology of a smooth projective variety $X$ is called of Hodge-Tate  type if
 $H^n_\dR(X)=H^{\frac{m}{2},\frac{m}{2}}$. In other words,
 $$
 H^p(X,\Omega_X^q)=0,\ \forall p+q=m,\ \ p\not=q.
 $$
\end{defi}
As the Hodge conjecture can be reduced to middle cohomology, we say that a variety is of Hodge-Tate type if its middle cohomology is of Hodge-Tate type. The following very particular case of the Hodge conjecture is still open. 
\begin{conj}
\label{05082024ki}
 Hodge conjecture is true for Hodge-Tate varieties $X$, that is, there are algebraic cycles $Z_i,\ i=1,2,\ldots,a$ in $X$  of dimension $\frac{m}{2}$ such that $[Z_i]$'s generate the homology group $H_m(X,\Q)$. 
\end{conj}
Note that for a Hodge-Tate variety all the cycles $\delta\in H_m(X_z,\Z)$ are Hodge. 
\begin{conj}
For a Hodge-Tate variety $X$ defined over $\bar\Q$ and $\delta\in H_m(X,\Z)$ we have
 $$
 \frac{1}{(2\pi i)^{\frac{m}{2}}} \int_{\delta_z}\omega\in\bar\Q,\ 
 \forall \omega\in H^m_\dR(X/\bar\Q).
 $$
\end{conj}
This is a consequence of the Hodge conjecture, see \cite[Proposition 1.5]{dmos}.
A complex version of the above statement is as follows.
\begin{theo}
\label{06072024claire}
Let $X_z,\ z\in V$ be a family of varieties of Hodge-Tate type and let
$\omega\in H^m_\dR(X/V)$. The holomorphic multi-valued function $\int_{\delta_z}\omega$ is algebraic (as  a function in $z$).
 \end{theo}
\begin{proof}
 Both Hodge and Grothendieck-Katz conjectures imply \cref{06072024claire}.
 Hodge conjecture is not known for varieties of Hodge-Tate type and Grothendieck-Katz conjecture for Gauss-Manin connections and its factors is known in \cite{Katz1972}. 
 
 If the Hodge conjecture for Hodge-Tate varieties is true then $\delta_z=[Z_z]$ and the result follows from a version of \cite[Proposition 1.5]{dmos} for families of algebraic cycles.

Let $a:=\#\{(p,q) |  p+q=m,\ h^{p,q}(X_t)\not=0\}$. N. Katz in  \cite[Corollary 7.5, page 383]{Katz1970} proves that the $p$-curvature of the Gauss-Manin connection of $X/V$ is nilpotent of order at most $a$. He attributes this theorem to  P. Deligne.  If $X/V$ is of Hodge-Tate type  then $a=1$ and so 
we know that  the $p$-curvature of the Gauss-Manin connection of $X/V$ is zero. 
The Grothendieck-Katz conjecture for Gauss-Manin connections and its factors is known in \cite{Katz1972},
and so the solutions to the Gauss-Manin connection as a differential equation are algebraic. 
These are exactly the periods $\int_{\delta_z}\omega$, see \cref{27nov2023shhshahani}. 
\end{proof}

 As far as the author is aware of it, there is no classification of Hodge-Tate varieties, and  \cref{05082024ki} is as open as the Hodge conjecture itself. Toric varieties are Hodge-Tate varieties, and we can look for complete intersections in Toric varieties which are Hodge-Tate. For instance, let $X$ be a smooth hypersurface of degree $d$ in the weighted projective space
 $\P^{v_0,v_1,\cdots,v_{n+1}}$. By a result of J. Steenbrink in \cite{st77} we know that  for $v_0=1$ if $\frac{n}{2}\leq  \frac{\sum_{i=1}^{n+1}v_i}{d}$ then $X$ is of Hodge-Tate type. For an example of this, and the resulting algebraic periods see \cite[Section 16.9]{ho13}.

\chapter{Ramanujan vector field}
\label{Ramanujan27062024}
{\it
On se propose de donner un dictionnaire heuristique entre énoncés en cohomologie
$l$-adique et énoncés en théorie de Hodge. Ce dictionnaire a notamment pour sources [...]
et la théorie conjecturale des motifs de Grothendieck [...]. Jusqu'ici, il a surtout servi
a formuler, en théorie de Hodge, des conjectures, et il en a parfois suggéré une démonstrdtion,
(P. Deligne in  \cite{de71-0}).
}
\\
\\

{\it Abstract:} In this article we prove that for all primes $p\not=2,3$, the Ramanujan vector field has an invariant algebraic curve and then
we give a moduli space interpretation of this curve in terms of Cartier operator acting on the de Rham cohomology of elliptic curves. The main ingredients of our study are due to Serre, Swinnerton-Dyer and Katz in 1973.
We aim to generalize this for  the theory of Calabi-Yau modular forms, which includes the generating function of genus $g$ Gromov-Witten invariants.
The integrality of $q$-expansions of such modular forms is still a main conjecture which has been only established for special Calabi-Yau varieties, for instance those whose periods are hypergeometric functions. For this the main tools are Dwork's theorem.
We present an alternative project which aims to prove such integralities using modular vector fields and Gauss-Manin connection in positive characteristic.

\section{Introduction}
We are going to study algebraic leaves of the Ramanujan vector field
\begin{equation}
\label{16112023yau}
\vf:=
(t_1^2-\frac{1}{12}t_2)\frac{\partial}{\partial t_1}+
(4t_1t_2-6t_3)\frac{\partial}{\partial t_2}+
(6t_1t_3-\frac{1}{3}t_2^2)\frac{\partial}{\partial t_3}
\end{equation}
in characteristic $p\not=2,3$. All the algebraic leaves of this vector field over $\C$ are inside the hypersurface $27t_3^2-t_2^3=0$ and it has a transcendental solution given by the Eisenstein series $a_1E_2,a_2E_4,a_3E_6$, where $a=(-\frac{1}{12}, \frac{1}{12}, -\frac{1}{216})$, and that is why it carries this name.
We consider it as a vector field in
$\A^3_\sring=\spec(\sring[t_1,t_2,t_3])$, where $\sring$ is a ring of characteristic zero with $2$ and $3$ invertible. We usually take $\sring=\Z[\frac{1}{6}]$ and so $\sring/p\sring=\Ff_p$.
Let $\sk$ be the quotient field of $\sring$.  In \cite{ho14} the author gave a moduli space interpretation of the Ramanujan vector field:
\begin{theo}
\label{21may2024raisimord}
Let $\T$ be the moduli space of triples  $(E,\alpha, \omega)$, where $E$ is an elliptic curve defined over $\sk$ and $\alpha,\omega$ form a basis of $H^1_\dR(E/\sk)$ with
$\alpha\in F^1H^1_\dR(E/\sk)$ and $\langle\alpha,\omega\rangle=1$.
This moduli space is the affine variety
$$
\T=\spec \Q[t_1,t_2,t_3, \frac{1}{\Delta}], \ \Delta:=27t_3^2-t_2^3,
$$
and we have a universal family over $\T$ given by ${\sf E}\to \T$, where
\begin{eqnarray*}
\label{khodaya}
{\sf E} & : & zy^2-4(x-t_1z)^3+t_2z^2(x-t_1z)+t_3z^3=0,\\
& & [x;y;z]\times (t_1,t_2,t_3)\in\P ^2\times \T,\ \\
(\alpha,\omega) &:=& \left(\left[\frac{dx}{y}\right], \left[\frac{xdx}{y}\right]\right)\ \ \hbox{given in the affine coordinate $z=1$.}
\end{eqnarray*}
The natural action of the algebraic group
$$
\BG:=\left\{\mat{k}{k'}{0}{k^{-1}}\Bigg|  k'\in \sk, k\in \sk-\{0\}\right\} 
$$
by change of basis  on $\T$ is given by
$$
t\bullet g:=(
 t_1k^{-2}+k'k^{-1},
t_2k^{-4}, t_3k^{-6}), t=(t_1,t_2,t_3)\in\T,\ \
g=\mat {k}{k'}{0}{k^{-1}}\in \BG.
$$
Moreover, if $\nabla: H^1_\dR({\sf E}/\T)\to \Omega^1_{\T}\otimes_{\O^1_{\T}} H^1_\dR({\sf E}/\T)$ is the Gauss-Manin connection of ${\sf E}/\T$ then the Ramanujan vector field is the unique vector field in $\T$ with the property
$$
\nabla_\vf\alpha=-\omega,\ \ \nabla_\vf\omega=0.
$$
 \end{theo}
The moduli space $\T$ in \cite{ho2020} is called  "ibiporanga", a word in Tupi language meaning pretty land.  Once we master all the preliminaries of \cref{21may2024raisimord}, the proof becomes an easy exercise, see also \cite[Appendix 1, page 158]{ka73}, and we realize that being an elliptic curve does not play a significant role in this theorem.
 The author took the job of generalizing \cref{21may2024raisimord} and wrote many papers with the title ``Gauss-Manin connection in disguise". The project has been summarized in the book \cite{ho2020}, and it has been so far written over fields of characteristic zero. For arithmetic applications, it seems to be necessary to consider fields of positive characteristic.   We redefine $\T$ to be $\spec \Z[t_1,t_2,t_3, \frac{1}{\Delta}]$ which has a similar moduli space interpretation. 
 In this article we prove the following:
 \begin{theo}
\label{21may2024raisimord-2}
For primes $p\not=2,3$, there is a curve  in $\T/\Ff_p$ which is invariant by $\vf$, and its $\sk$-rational points correspond to triples $(E,\alpha,\omega)$
such that $C(\alpha)=\alpha$ and $C(\omega)=0$, where $C$ is the Cartier operator.
 \end{theo}
It is shown by J.V. Pereira in \cite{Pereira2002}, that foliations modulo primes might have algebraic leaves, even though in
characteristic zero they do not have such leaves. \cref{21may2024raisimord-2} is a manifestation of this phenomena.
Most of the ingredients of the proof of \cref{21may2024raisimord-2} comes from the articles of Swinnerton-Dyer, Serre and Katz in LNM 350 and they are the building blocks of the theory of $p$-adic modular forms. The integrality of the coefficients of $E_2,E_4, E_6$ play a main role in the proof of \cref{21may2024raisimord-2}. For Calabi-Yau varieties such integralities, and in particular the integrality of the mirror map,  have been experimentally observed by physicists, and proved for special class of Calabi-Yau varieties, see \cite[Appendix C]{GMCD-MQCY3} and the references therein, but in general it is an open problem.
In \cite{GMCD-MQCY3, HosseinMurad} we have introduced the theory of Calabi-Yau modular forms and proved a similar statement as in \cref{21may2024raisimord} for Calabi-Yau threefolds. We strongly believe that  \cref{21may2024raisimord-2} for Calabi-Yau varieties is equivalent to the $p$-integrality of Calabi-Yau modular forms. In this direction we observe the following.
For a moment forget that the Ramanujan vector field has a solution given by  the Eisenstein series. Instead, consider $\vf$ as an ordinary differential equation: 
\begin{equation}
\label{raman}
{\rm R}:
 \left \{ \begin{array}{l}
\dot t_1=t_1^2-\frac{1}{12}t_2 \\
\dot t_2=4t_1t_2-6t_3 \\
\dot t_3=6t_1t_3-\frac{1}{3}t_2^2
\end{array} \right. .
\end{equation}
We write each $t_i$ as a formal power series in $q$, $t_i=\sum_{n=0}^\infty t_{i,n}q^n,\ i=1,2,3$ and substitute in the above differential
equation with $\dot t_i=-q\frac{\partial t_i}{\partial  q}$ and the initial values $t_1=-\frac{1}{12}(1-24q+\dots)$. It turns out that all $t_{i,n}$ can be computed recursively, however, by recursion we can at most claim $t_{i,n}\in\Q$, for more details
see \cite[Section 4.3]{ho14}.
\begin{theo}
\label{23052024jea}
\label{23m2024hamid}
 Let $p\not=2,3$ be a prime.  The followings are equivalent: 
 \begin{enumerate}
     \item 
     There is a curve $C$ in $\T/\Ff_p$ passing through $a:=(-\frac{1}{12}, \frac{1}{12}, -\frac{1}{216})$, smooth at $a$,  tangent to $b=(2,20, \frac{7}{3})$
at $a$ and tangent to the Ramanujan vector field $\vf$ and $\vf^p=\vf$ restricted to $C$.
\item 
 The solution $t_1,t_2,t_3$ of $\vf$ described above is $p$-integral, that is, $p$ does not appear in the denominator of $t_{i,n}$'s.
 \end{enumerate}
\end{theo}
The author strongly believes that the existence of algebraic solutions of modular vector fields as in \cite{ho2020} and generalizations of \cref{23052024jea} can be proved by available methods in algebraic geometry. This will give a purely geometric method for proving the integrality of Calabi-Yau modular forms without computing periods explicitly. Despite the fact that in this article we heavily use the Eisenstein series $E_{p-1},\ E_{p+1}$ for $p$ a prime number, and these objects in the framework of Calabi-Yau varieties do not exists or not yet discovered, we have formulated \cref{13112023tifi}, \cref{06062024ihes} which describes the curve in \cref{21may2024raisimord} using only $\vf$ and primary decomposition of ideals.

We would also like  to announce the following theorem which suggests that there might be a heuristic dictionary for properties of foliations over $\C$ and properties of foliations in characteristic $p$. This is mainly inspired by a  similar work in Hodge theory introduced by  P. Deligne in \cite{de71-0}.
\begin{theo}
\label{08062024drywall}
For any prime $p\not=2,3$, the Ramanujan vector field in $\A^3_{\Ff_p}$ has a first integral
$f\in\Ff_p[t]$, that is $\vf(f)=0$. It is a homogeneous polynomial of degree $p+1$ with $\deg(t_i)=2i,\ i=1,2,3$. Moreover, $\vf$ restricted to $f=0$ has a regular first integral $A$. The curve $A=1, f=0$ is the curve in \cref{23052024jea}.
\end{theo}
A similar theorem has been proved for the Ramanujan vector field in $\C^3$ in \cite[Theorem 1]{ho06-3}.
It has a real analytic first integral $f$ in $\C^3\backslash\{\Delta=0\}$ and $\vf$ restricted to
$f=0$ has also a real analytic first integral.

During the preparation of the present text we have consulted J.V. Pereira, F. Bianchini, F. Voloch, and N. Katz whose names appear throughout the text. My heartfelt thanks go to all of them.

\begin{figure}[t]
\begin{center}
\includegraphics[width=0.5\textwidth]{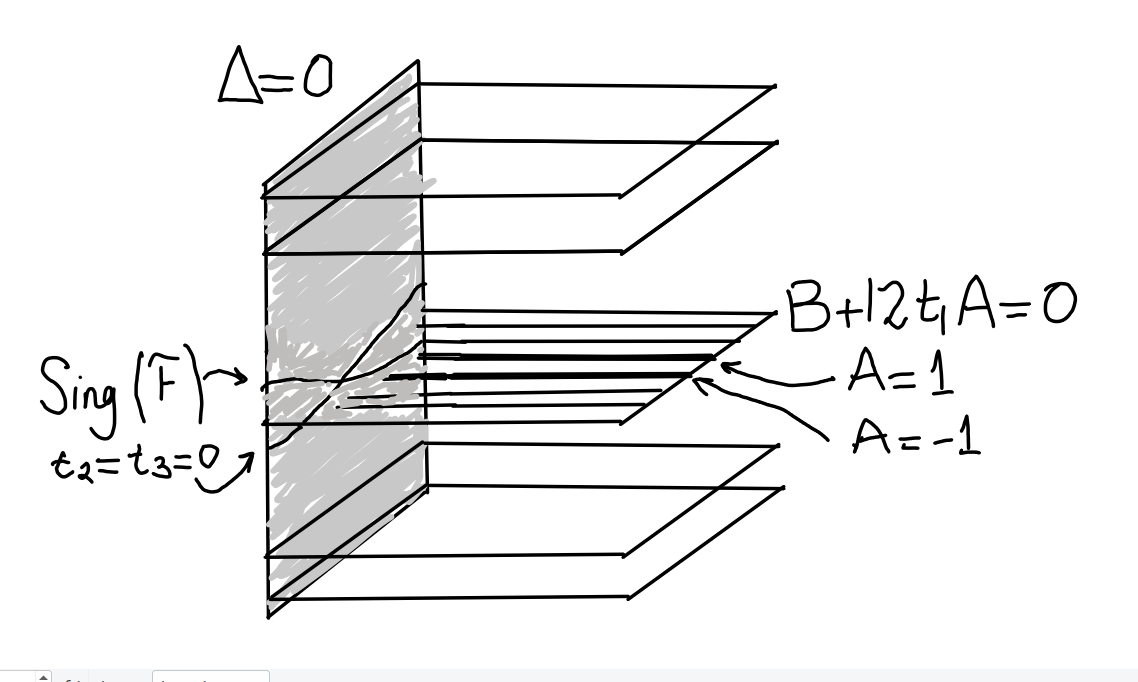}
\caption{Leaves and first integral}
\label{Lattice}
\end{center}
\end{figure}

\section{Bernoulli numbers}
Bernoulli numbers $B_k$ are defined through the equality\index{Bernoulli numbers}\index{$B_k$,  Bernoulli numbers}
$$
\frac{x}{e^x -1}=\sum_{k=0}^{\infty}B_k\cdot\frac{x^k}{k!}.
$$
For instance, $B_0=1, \ B_1=\frac{-1}{2}, \ B_2=\frac{1}{6}, \ B_4=\frac{-1}{30}, \ B_6=\frac{1}{42}$.
It is easy to see that for any odd $k\geqslant 3$ we have $B_k=0$.
\begin{theo}
 We have the following congruence properties for Bernoulli numbers:
 \begin{enumerate}
  \item Von Staudt–Clausen theorem:
$$
B_{k}+\sum_{(p-1)|k, \ p \hbox{ prime }}{\frac {1}{p}}\in \Z
$$
In particular, for $(p-1)|k$  ${\rm ord}_p B_{k}=-1$. 
\item 
Kummer theorem:  If $(p-1)\not| k$ then ${\rm ord}_p \frac{B_k}{k}\geq 0$ and 
$$
\frac{B_k}{k}\equiv_p \frac{B_{k'}}{k},\ \ \forall k \equiv_{p-1}k'\not\equiv_{p-1}0.
$$
 \end{enumerate}
 \end{theo}
See \cite[384-386]{BorevichShafarevich} or \cite[Chapter 15]{IrelandRosen} for a proof. The articles \cite{Serre1973, Swinnerton1973, ka73} use the above theorem to build up the theory of p-adic modular forms. Recall the Eisenstein series: 
\begin{equation}
 E_{2k}(\tau)=1-\frac{4k}{B_{2k}}\sum_{n=1}^{\infty}\sigma_{2k-1}(n)q^n,\ \ k=1,2,3,\ldots. 
\end{equation}
 where $\sigma_i(n):=\sum_{d\mid n}d^i$. 
\begin{prop}
\label{04042024bitaamad}
 The numerator of $\frac{B_{2k}}{2k}$ is the smallest number $a\in\N$ such that
 $aE_{2k}$ can be written as a polynomial with coefficients in $\Z$ of weight $2k$ in $E_4,E_6$ with $\deg(E_i)=i,\ i=4,6$. 
\end{prop}
For the sequence of these numbers see
 \href{https://oeis.org/A001067}{A001067}.  
 Examples of $a=a_{2k}, 2k=2,4,6,8,\cdots, 12$ are 
 $$
1, 1, 1, 1, 1, 691, 1, 3617, 43867, 174611, 77683, 236364091, 
$$
\begin{proof}
This proposition follows from a statement in \cite[page 19]{Swinnerton1973}:
``if $f$ is a modular form and $A$ the additive group generated by the
coefficients of the $q$-series expansion of $f$, then $f$ has a unique expression as an isobaric element of
$A[Q,\Delta]\oplus  R A[Q,\Delta]$". Here, $Q=E_4, R=E_6$ and $\Delta=\frac{1}{1728} (Q^3-R^2)$. The proof is easy,
even though at first it did not appear to me and Frederico Bianchini reminded me the argument. First,  $f-f_0E_4^aE_6^b$,  for some $a,b\in\N$ such that $k:=4a+6b$ is the weight of $f$, is a cusp form. We know that the ideal of cusp form over $\Q$, is generated by $\Delta$
and so $\frac{f-f_0E_4^aE_6^b}{\Delta}$ has coefficients in $A$ and is a modular form of weight $k-12$. The proof is finished by using induction on $k$.
\end{proof}

Let $E_2,E_4,E_6$ be the Eisenstein series.
For a prime $p\not=2,3$ let us define
$$
\IS_p:=\left \{P\in\Ff_p[t_1,t_2,t_3]\ \  \Big| \ \  P(a_1E_2,a_2E_4,a_3E_6)=0\right \}.
$$
This is an ideal in $\Ff_p[t_1,t_2,t_3]$ and we know that it has many elements.
We have $E_{p-1}=1+\frac{2(p-1)}{B_{p-1}}(q+\cdots)$ and Von Staut Clausen theorem says that
${\rm ord}_p \frac{2(p-1)}{B_{p-1}}=+1$ and so $E_{p-1}\equiv_p1$. We write 
\begin{equation}
 \label{07102025bimsa-1}
E_{p-1}=A\left(\frac{1}{12}E_4,-\frac{1}{216}E_6\right)
\end{equation}
and by  \cref{04042024bitaamad} the prime $p$ does not appear in the denominator of $A$ and so it makes sense to consider $A\in\Ff_p[t_2,t_2]$. We get $A-1\in\IS_p$.
In a similar way we use Kummer's theorem for $(p-1)\not|(p+1)$ and so ${\rm ord}_p\frac{B_{p+1}}{2(p+1)}\geq 0$ and this is $\equiv_{p}\frac{B_2}{4}=24$ which implies that 
\begin{equation}
\label{25042024zagier}
 {\rm ord}_p \frac{B_{p+1}}{2(p+1)}=0
\end{equation}
and hence it it invertible modulo $p$.
 This together with Fermat's little theorem imply that  $E_{p+1}\equiv E_2$.
 By \cref{04042024bitaamad} and \eqref{25042024zagier} we can write \begin{equation}
 \label{07102025bimsa-2}
 E_{p+1}\equiv_p B\left(\frac{1}{12}E_4,-\frac{1}{216}E_6  \right)
 \end{equation}
 with $B\in \Ff_p[t_2,t_3]$, and we have another element $B+12t_1\in\IS_p$.
 \begin{prop}
 \label{05062024lopes}
 For $p\geq 5$ we have
 \begin{enumerate}
 \item
  The ideal $\IS_p$ is  generated by $A(t_2,t_3)-1$ and $B(t_2,t_3)+12t_1$.
  \item
  It is invariant under the Ramanujan vector field $\vf$ and
  $(\vf^{p}-\vf)\Ff_p[t_1,t_2,t_3]\subset \IS_p$. 
  \item
  The scheme ${\rm Zero}(\IS_p)$ is an irreducible curve in $\A_{\bar \Ff_p}^3$  with the smooth  point $a$ and tangent to the vector $b$ at $a$, both defined in  \cref{23m2024hamid}.
  \end{enumerate}
 \end{prop}
\begin{proof}
1. Since $\frac{1}{12}B(t_2,t_3)+t_1\in\IS_p$, we need to show that any $P(t_2,t_3)\in\IS_p$ is a multiple of $A(t_2,t_3)-1$.
This has been proved in \cite[Theorem 2 (iv), page 22]{Swinnerton1973}, see also
\cite[page 196]{Serre1973}.

2. Let us consider the following map which is a ring homomorphism:
 $$
 \sring[t_1,t_2,t_3]\to \sring[[q]], P(t_1,t_2,t_3)\mapsto P(a_1E_2,a_2E_4,a_3E_6).
 $$
 In $\sring[[q]]$ we consider the derivation $-q\frac{\partial}{\partial q}$ and it turns out that the following is commutative:
 \begin{equation}
\begin{array}{ccc}
 \sring[t_1,t_2,t_3]&\stackrel{\vf}{\to}&  \sring[t_1,t_2,t_3]\\
\downarrow &&\downarrow\\
 \sring[[q]]&\stackrel{-q\frac{\partial}{\partial q}}{\to} & \sring[[q]].
\end{array}
\end{equation}
  This implies that $\IS_p$ is invariant under the Ramanujan vector field  $\vf$. One can compute  $\vf(A-1)$ and $\vf(B+12t_1)$ in terms of $A-1$ and $B+12t_1$ using the equalities in \cite[Theorem 2 page 22]{Swinnerton1973}.
  By Fermat little theorem we have $a^p\equiv_p a$ for all $a\in\N$, and so,   $(-q\frac{\partial}{\partial q})^p=-q\frac{\partial}{\partial q}$.
 This implies that $\vf^pf-\vf f\in \IS_p$ for all $f\in\sring[t]$.

3. Note that
$$
\left \{P\in\bar \Ff_p[t_1,t_2,t_3]\ \  \Big| \ \  P(E_2,E_4,E_6)=0\right \}=\IS_p\otimes_{\Ff_p} \bar\Ff_p.
$$
This follows by considering the same equality for the vector space of polynomials $P$ of degree $\leq d$ and the fact that the left hand side of the above equality is defined over $\Ff_p$. By definition $\IS_p\otimes_{\Ff_p} \bar\Ff_p$ is a prime ideal.
The point  $a$ is a singular point of the Ramanujan vector field. Of course we can prove this also by explicit generators $A-1$ and $B+12t_1$ of $\IS_p$. The smoothness at $a$ follows from $(p-1)A=4t_2\frac{\partial A}{\partial t_2}+6t_3\frac{\partial A}{\partial t_3}$.
\end{proof}

\begin{rem}\rm
We remark that $\IS_p$ is not generated by $(\vf^p-\vf)t_i,\ \ i=1,2,3$. For the computation below we use the Ramanujan vector field corresponding with the solution $(E_2,E_4,E_6)$ (without constant $a_i$' s).
We first compute the linear part of $\vf$ at $\t$. In the cordinates $x_i=t_i-1,\ \ i=1,2,3$, $\vf$ can be written as
$$
\vf:=
\begin{bmatrix} \frac{\partial}{\partial x_1} & \frac{\partial}{\partial x_2} & \frac{\partial}{\partial x_3}
 \end{bmatrix}
\begin{bmatrix}
\frac{1}{6} & \frac{-1}{12} & 0\\
\frac{1}{3} & \frac{1}{3} & \frac{-1}{3}\\
\frac{1}{2} & -1 & \frac{1}{2}
 \end{bmatrix}
 \begin{bmatrix}
 x_1 \\
 x_2\\
 x_3
\end{bmatrix}+\cdots
$$
where $\cdots$ means higher order terms.
The Jordan decomposition of the linear part  is $A=SJ S^{-1}$, where
$$
J=\begin{bmatrix}
  0&1&0\\
  0&0&0\\
  0&0&1
  \end{bmatrix}
\ \ \
S=\begin{bmatrix}
   1/3 & 2 & 1/21 \\
   2/3 & 0 & -10/21 \\
   1 & 0 & 1
  \end{bmatrix}.
$$
This implies that the linear part of $\vf^p-\vf$ is of rank one. If $(\vf^p-\vf)t_i,\ \ i=1,2,3$ generates $\IS_p$, this must be $2$ because
${\rm Zero}(\IS_p)$ is smooth at $t=(1,1,1)$.
{\tiny
}
\end{rem}

\section{Hasse-Witt invariant}
For the definition of Hasse-Witt invariants, we follow \cite{AchterHowe}. We actually use a notion of Hasse-Witt invariant for differential forms of the second kind.  For a curve $X$ defined over a perfect field of characteristic $p$, there exists a unique map $C:\Omega^1_X\to \Omega^1_X$, called the Cartier operator,
such that
\begin{enumerate}
\item
$C$ is $\frac{1}{p}$-linear, that is  $C$ is additive and $C(f^p\omega)=f C(\omega)$.
\item
$C(df)=0$
\item
$C(f^{p-1} df ) = df$,
\item
a differential $\omega$ is logarithmic, that is $\omega=\frac{df}{f}$ if and only if $\omega$ is closed and
$C(\omega)=\omega$.
\end{enumerate}

This operator induces a $1/p$-linear map on merormorphic differential 1-forms on $X$. We can compute the Cartier operator in the following way.
Let $a$ be a closed point of a smooth curve $X$ and let $t$ be a coordinate system at $a$. The $\O_{X,a}^p$-module $\O_{X,a}$ is freely generated by
functions $1,t,\cdots, t^{p-1}$.  Any meromorphic 1-form  which is holomorphic at $a$ admits an expression
\begin{equation}
\label{16052024gagoaluguel}
\omega=(\sum_{j=0}^{p-1} f_j^pt^j)dt,\ \  f_j\in\O_{X,a}.
\end{equation}
We have
$$
C(\omega)=f_{p-1}dt.
$$
We follow \cite{AchterHowe} and compute the matrix of $C\frac{x^idx}{y}$.
Let 
$f(x)=4x^3-t_2x-t_3$, where $t_2,t_3\in\sk$ and $\sk$ is a perfect field of characteristic $p$. We denote the inverse of the Frobenius map $\sk\to\sk, \ \ t\mapsto t^p$ by $t\mapsto t^{\frac{1}{p}}$ and write
$$
f(x)^{\frac{p-1}{2}}=\sum_{i=1}^{\frac{3(p-1)}{2}} c_ix^i.
$$
By the  $\Gm$ action $x\to kx,\ t_2\to k^{-2}t_2, \ t_3\to k^{-3}t_3$,   we can see that $c_i\in\Z[t_2,t_3],\ \deg(t_3)=6,\ \deg(t_2)=4$ is homogeneous of degree $3(p-1)-2i$. A simple calculation shows that
$$
C\left(\frac{x^{j-1}dx}{y}\right)=\sum_{i=1}^{}c_{ip-j}^{\frac{1}{p}}\frac{x^{i-1}dx}{y}=c_{p-j}^\frac{1}{p}\frac{dx}{y},
$$
$$
\left [C(\frac{dx}{y}),C(\frac{xdx}{y})\right]=
 \left[\frac{dx}{y},\frac{xdx}{y}\right]
 \begin{bmatrix}
 c_{p-1}^{\frac{1}{p}} & c_{p-2}^{\frac{1}{p}}\\
 0  & 0
\end{bmatrix}.
$$
The quantity $c_{p-1}$ is called the Hasse-Witt invariant of the elliptic curve $E: y^2=f(x)$.
\begin{prop}
\label{05062024bita}
Let $E_{t_2,t_3}: y^2= 4x^3-t_2x-t_3$ be an  elliptic curve in the Weierstrass format over a perfect field of characteristic $p\not=2,3$.
We have
$$
c_{p-1}\equiv_p A(t_2,t_3),\ \  c_{p-2}\equiv_p \frac{1}{12} B(t_2,t_3),
$$
where $A,B\in\Q[t_2,t_3]$ are computed via $E_{p-1}=A(a_2E_4,a_3E_6),\ \ E_{p+1}=B(a_2E_4,a_3E_6)$.
\end{prop}
\begin{proof}
 A hint for the proof of the first congruence can be found in
 \cite[last paragraph of page 23]{Swinnerton1973}. In this article we read ''This may be proved in one of two ways.
On the one hand Deligne has shown that the $q$-series expansion of the Hasse invariant reduces to 1; and Theorem 2 shows that this property characterizes $A$ among polynomials of weight $l-1$.
On the other hand the differential equation derived from (ii) is just that which the Hasse invariant is
known to satisfy, see Igusa."
The first proof is reproduced in \cite[page 90]{ka73}, see also \cite[Theorem 3]{Serre1973}. The second congruency at first seemed to be a novelty which has not deserved the attention of masters. The author consulted F. Voloch regarding this which resulted in the following comments. Serre in a note (Algèbre et géométrie, page 81) mentions the article \cite{Robert1980} in which a multiplication by $E_{p+1}$ map is characterized uniquely with a certain property related to Hecke operators. This is also reproduced in \cite[Proposition 7.2]{Edixhoven1992}.  At no point its relation with the Cartier map and the differential form of the second kind $\frac{xdx}{y}$  is discussed.
From another perspective, Katz in \cite[page 57]{Katz1977} describes the action of Frobenius map (dual to Cartier) on the de Rham cohomology of elliptic curves. He uses the letter $B$ to denote a coefficient which must be essentially $c_{p-2}$, but no relation with $E_{p+1}$ is discussed. After a personal communication with N. Katz, the author came to know that the puzzle has been only solved many decades later in \cite[Theorem 3.1]{Katz2021}. This might partially justify the author's rediscovery of the second congruency.  
The reader who dislikes proofs using some complicated language of Algebraic Geometry and prefers
experimental verification, might use the following  computer code in {\sc Singular}
which verifies the proposition for as much as primes that the computer can handle. Here is the code for primes $\leq 300$.
{\tiny
\begin{verbatim}
LIB "foliation.lib"; int np=300;
intvec prli=primes(1,np); int i; int j; 
 for (i=3;i<=size(prli);i=i+1)
   { "*", prli[i];
   ring r=0,(x,t_2,t_3),dp;
   poly A=Eisenstein(prli[i]-1, t_2, t_3); poly B=Eisenstein(prli[i]+1, t_2, t_3 );
   A=subst(subst(A, t_2, 12*t_2),t_3, -216*t_3); B=subst(subst(B, t_2, 12*t_2),t_3, -216*t_3);
   ring rr=int(prli[i]),(x,t_2,t_3),dp;
   poly A=imap(r,A);  poly B=imap(r,B);
   int p12=(prli[i]-1) div 2;
   poly P=(4*x^3-t_2*x-t_3)^p12;
   matrix M=coeffs(P,x);
   M[prli[i],1]-A;
   M[prli[i]-1,1]-1/12*B;
   poly Q=4*x^3-t_2*x-t_3; poly Vx=(1/2)*diff(Q,x);
   for (j=4; j<=prli[i]-1;j=j+2){ Vx=(1/2)*diff(Q,x)*diff(Vx,x)+Q*diff(diff(Vx,x),x);}
   Vx=diff(Vx,x); Vx-A; 
  }
\end{verbatim}
}

\end{proof}


\begin{proof}{(of \cref{21may2024raisimord-2})}
The curve we are looking for in \cref{21may2024raisimord-2} is the one explicitly described in  \cref{05062024lopes}.
Let $(t_1,t_2,t_3)$ be a $\sk$-rational point of ${\rm Zero}(\IS_p)$, and hence, $A(t_2,t_3)=1,\ -\frac{1}{12}B(t_2,t_3)=t_1$.
We write the family of elliptic curves in \cref{21may2024raisimord} in the affine coordinates and write it in the form $y^2=4x^3-t_2x-t_3,\ \alpha=\frac{dx}{y},\ \omega=\frac{(x+t_1)dx}{y}$. By \cref{05062024bita} we know that
$$
C(\alpha)=A^{\frac{1}{p}}\alpha=\alpha,\ \
$$
$$
C(\omega)=(\frac{1}{12}B)^{\frac{1}{p}}\alpha+t_1^{\frac{1}{p}}A^{\frac{1}{p}}\alpha=(\frac{1}{12}B+t_1)^{\frac{1}{p}}\alpha=0.
$$

\end{proof}

\begin{proof}{(of \cref{23052024jea})} 1) implies 2): 
We have an algebraic curve $C={\rm Zero}(I)$ in $\T/\Ff_p$ and its smooth point $a$ in the discriminant loci $\Delta=0$ and we know that it is tangent to $\vf$, that is, $\vf(I)\subset I$. Moreover, we have the vector $b$ given in \cref{23m2024hamid} tangent to $C$ at $a$.
Since $C$ is smooth at $a$, we can parameterize it, that is, there are formal power series $\check t_i(q):=\sum_{i=0}^\infty \check t_{i,n}q^n\in\Ff_p[[q]],\ \ i=1,2,3$ such that 
$$
(\check t_{1,0},\check t_{2,0},\check t_{3,0})=a, \ (\check t_{1,1},\check t_{2,1},\check t_{3,1})=b
$$ 
and any element in the ideal $I$ of $C$ annihilates $\check t(q):=(\check t_1(q),\check t_2(q),\check t_3(q))$. Both vectors $\frac{\partial \check t(q)}{\partial q}$  and $\vf(\check t(q)):=(\vf_1(\check t(q)),\ \vf_2(\check t(q)),\vf_3(\check t(q)) )$, where $\vf=\sum_{i=1}^3\vf_i\frac{\partial}{\partial t_i}$ are tangent to the curve $C$, that is,  
$$
\sum_{i=1}^3 A_i\frac{\partial f}{\partial t_i}=0,\ \forall f\in I,
$$
for $A$ one of these vectors.
Since the constant term of  $\frac{\partial \check t}{\partial q}$ is the non-zero vector $b$, we have
$$
\vf(\check t(q))=a(q)\frac{\partial \check t}{\partial q}, \hbox{ where  } a(q)=\sum_{i=0}^\infty a_iq^i\in\Ff_p[[q]].
$$
Moreover, from $\frac{\partial \check t}{\partial q}(a)=b\not=0$ and $\vf(a)=0$, we conclude that $a_0=0$. We compute the coefficient of $q$ in $\vf(\check t(q))$ and see that it is $-b$, that is,
$$
\left[\frac{\partial \vf_j}{\partial t_j}\right]_{3\times 3}\Big|_{q=0}b^{\tr}=-b^{\tr},
$$
see \cite[Page 333]{ho13}. This implies that $a_1=-1$. We claim that there is a formal change of variables $q\mapsto \tilde q:=f(q)\in\Ff_p[[q]]$ which maps the derivation $a(q)\frac{\partial}{\partial q}$ to  $-\tilde q\frac{\partial}{\partial \tilde q}$. For this we use $\vf^p=\vf$ restricted to the curve $C$ and we restate and prove it in \cref{30102025guerranorio}.  If this is the case,  we can assume that $a(q)=-q$. Therefore,
$q\frac{\partial t}{\partial q}=\vf(t)$ and this implies the $p$-integrality of the solution over $\Q$.

1) implies 2): We use the fact that $t_i=a_iE_{2i}$ and the curve $C$ is given by $A-1=B+12t_1A=0$. 
\end{proof}

Let $\A^{1,\for}_{\Ff_p}$ be the spectrum of $\Ff_p[[q]]$. 
\begin{prop}\rm 
\label{30102025guerranorio}
\footnote{This proposition is a particular case of the so called Poincar\'e linearization theorem in characteristic $p$, see \cref{30102025shanghai}.} 
Let $\vf=a(q)\frac{\partial }{\partial q}=(\lambda q+\cdots)\frac{\partial }{\partial q}, \lambda\in \Ff_p$ be a vector field in $\A^{1,\for}_{\Ff_p}$. There is a coordinate system $\tilde q$ in $\A^{1,\for}_{\Ff_p}$ such that $\vf= \lambda \tilde q\frac{\partial }{\partial \tilde q}$ if and only if  $\vf^p=\vf$. 
\end{prop}
\begin{proof}
The direction $\Rightarrow$ is easy: Let $\tilde q=f(q):=q+\cdots\in\Ff_p[[q]]$. The map 
$$
\A^{1,\for}_{\Ff_p}\to \A^{1,\for}_{\Ff_p},\ q\mapsto \tilde q, 
$$
maps $a(q)\frac{\partial}{\partial q}$ to $\lambda \tilde q\frac{\partial }{\partial \tilde q}$ and the the second vector field  satisfies $\vf^p=\vf$ and hence the first vector field satisfies too. Note that we use Fermat's little theorem $\lambda^p=\lambda$. 

For $\Leftarrow$ we argue as follows: We have to find $f(q):=q+\cdots\in\Ff_p[[q]]$ such that $\vf(f)=\lambda f$.
We claim that 
$$
f=-\sum_{i=1}^{p-1} (\lambda^{-1}\vf)^{i}q
$$
satisfies this. The equality $\vf(f)=\lambda f$ follows from $(\lambda^{-1}\vf)^p=(\lambda^{-1}\vf)$. Moreover, the coefficient of $q^1$ in $P$ is $1=-(1+1+\cdots+1)$, $(p-1)$-times.  
\end{proof}

\begin{proof}{(of \cref{08062024drywall})}
We claim that the $B+12t_1A$ is the first integral of $\vf$, that is, $\vf(B+12t_1A)=0$.
For this we use the differential equations of $A$ and $B$ in \cite[Theorem 2 (ii), page 22]{Swinnerton1973}, see also \cite{Serre1973}. Under the transformation $(t_1,t_2,t_3)\to (kt_1,k^2t_2,k^3t_3)$ with $k\in \sk$, $\vf$ is mapped to $k^{-2}\vf$ and the curve $C_1: A=1,\ B=t_1$ is mapped to another curve $C_{k}$ passing through $(ka_1,k^2a_2,k^3a_3)$. All these curves lie in the hypersurface $B+12 t_1A=0$ which intersects the discriminant hypersurface $\Delta=0$ at two components $\sing(\vf)$ and $t_2=t_3=0$.
This shows that $\vf$ restricted to $B+12t_1A$ has the first integral $A=-\frac{B}{12 t_1}$.
\end{proof}

\begin{rem}\rm
 The vector field $\vf$ is tangent to both $\Delta=0$ and $B+12 t_1A$ and its leaves inside them are algebraic.
 It is interesting to know that these are the only algebraic leaves of $\vf$ over finite fields. This follows from \cref{13112023tifi}, \cref{06062024ihes} which can be verified by computer for examples of $p$.
\end{rem}

\begin{rem}\rm
 Let $E: y^2=P(x),\deg(P)=3$ be an elliptic curve in the Weierstrass format  over a perfect field of characteristic $p\not=2$. 
 It can be easily shown that 
 \begin{equation}
\label{23062024consumismo-2}
\vf:=y\frac{\partial }{\partial x}+ \frac{1}{2}P'(x)\frac{\partial }{\partial y}
\end{equation}
is a derivation/vector field on  $E$. 
Over complex numbers, the Weierstrass uniformization  $z\mapsto[\wp(z):\wp'(z):1]$ maps $\frac{\partial }{\partial z}$ to $\vf$.
For the vector field \eqref{23062024consumismo-2}, it is known that
$$
\vf^p= {\rm HW}(E,\frac{dx}{y}) \vf.
$$
see  for instance \cite[12.4.1.2]{KatzMazur1985}, \cite[3.2.1]{Katz1973}. This gives the following algorithm to compute the Hasse-Witt invariant. The polynomials $V_{2n}:=\vf^{2n}x$ satisfy the recursion:
$$
V_{2n+2}=PV_{2n}''+\frac{1}{2}P'V_{2n}',\ \ V_0=x. 
$$
It turns out that $V_{p-1}$ is a degree one polynomial in $x$ and  
the Hasse-Witt invariant ${\rm HW}(E,\frac{dx}{y})$ is the coefficient of $x$ in $V_{p-1}$. The experimental verification of this fact for many primes is done at the end of the computer code in the proof of \cref{05062024bita}.  
\end{rem}

For an elliptic curve $E$ and a regular differential $1$-form $\alpha$ both defined over $\Ff_p$  with $p\not= 2$, it is a classical fact that the Hasse-Witt invariant satisfies ${\rm HW}(E,\alpha)\equiv_p p+1-\#E(\Ff_p)$. 
If we write $E$ in the Weierstarss format $y^2=4x^3-t_2x-t_3$ then we conclude that 
$$
A(t_2,t_3)\equiv_p p+1-\#E(\Ff_p),
$$
that is, the $A$ polynomial obtained from the Eisenstein series $E_{p-1}$ through the equality \eqref{07102025bimsa-1},  is responsible for $\Ff_p$-point counting of elliptic curves. One might ask whether $B$ in \eqref{07102025bimsa-2} also plays a similar role. Here is the answer, for which the author was not able to find it in the literature. 
\begin{theo}\rm
\label{22072024decouvert}
 Let $E$ be an elliptic curve in the Weierstrass format $y^2=4x^3-t_2x-t_3$ over $\Ff_p,\ \ p\not=2$.
 We have
 $$
\sum_{P\in E(\Ff_p),\ P\not=O}x(P)^{j-1}\equiv_p -c_{p-j},\ \hbox{ for }  0\leq j-1<\frac{p-1}{2},
 $$
 where $c_{p-j}$ is the coefficient of $x^{p-j}$ in $(4x^3-t_tx-t_3)^{\frac{p-1}{2}}$.
 In particular, for $j=2$ and $p\geq 5$ we get
 $$
 B(t_2,t_3)\equiv_p 12 \sum_{P\in E(\Ff_p),\ P\not=O}x(P).
 $$
\end{theo}
\begin{proof}
 The proof is elementary and generalizes \cite[Section 2.11]{Clemens1980}.
 Let $f(x)^{\frac{p-1}{2}}=\sum_{i=0}^{\frac{3(p-1)}{2}} c_jx^j$.
 We have the following equalities modulo $p$:
 \begin{eqnarray*}
  \sum_{P\in E(\Ff_p),\ P\not=O}x(P)^{j-1} &= & \sum_{x\in \Ff_p} (f(x)^{\frac{p-1}{2}}+1 )x^{j-1}\\
  &= &
  \sum_{i=0}^{\frac{3(p-1)}{2}}c_i\left( \sum_{x\in\Ff_p }x^{i+j-1}\right)+ \sum_{x\in\Ff_p }x^{j-1}\\
  &=& -c_{p-j}.
\end{eqnarray*}
The first equality follows from the following.
If for $x\in \Ff_p$  we have $f(x)=0$ then $(x,0)\in E(\Ff_p)$ and $f(x)^{\frac{p-1}{2}}+1=1$ which amounts to counting once.  If $f(x)\not=0$ and there is $y\in \Ff_p$ such that $y^2=f(x)$ then  $f(x)^{\frac{p-1}{2}}\equiv_p1$. We have two points $(x,y),(x,-y)\in  E(\Ff_p)$ and  $f(x)^{\frac{p-1}{2}}+1\equiv_p 2$.
This amounts to counting two points. If there is no such $y$ then  $f(x)^{\frac{p-1}{2}}\equiv_p-1$ and
$f(x)^{\frac{p-1}{2}}+1\equiv_p 0$ and so we have not counted those $x$'s.

In the third equality we have used the classical formula
\begin{equation}
 \sum_{x\in\Ff_p}x^k=\left\{
                 \begin{array}{ll}
                  0 & \ \ p-1\nmid k,\ \ \hbox{or}\ \ k=0\\
                  -1 & \  \ p-1\mid k.
                 \end{array}
                 \right.
\end{equation}
Note that by this equality only $i+j-1=k(p-1)$ for some $k\in\N$ contributes to the sum and for $k\geq 2$ we have $i=2(p-1)-j+1>\frac{3}{2}(p-1)$.
The last statement follows from \cref{05062024bita}. 
\end{proof}

\section{Conjugate filtration}
\label{22102025shanghai}
\cref{21may2024raisimord-2} can be reformulated in terms of conjugate filtration of the de Rham cohomology of elliptic curves  over fields of positive characteristic, for definitions and some results on this, see \cite{Katz1972, DeligneIllusie1987}. This paves the way for generalizations of this theorem for Calabi-Yau varieties. For some partial results related to conjugate filtrations for Calabi-Yau varieties see \cite{Ogus2001, GeerKatsura2002}.

The curve $K$ in \cref{21may2024raisimord-2} has a more natural moduli space interpretation in the following way. For missing definitions the reader is referred to \cite{Katz1970}.  For an elliptic curve $E$ defined over a perfect field, apart from the Hodge filtration $\{0\}=F^2\subset F^1\subset F^0=H^1_\dR(E)$, we have also the conjugate filtration $\{0\}=G_{-1}\subset G_0\subset G_1=H^1_\dR(E)$ such that 
$$
F^1=H^0(E, \Omega^1_{E}),\ \ G_0=H^1(E, {\cal H}^0(\Omega^\bullet_E)), \ \ H^1_\dR(E)/G_0\cong H^0(E, {\cal H}^1(\Omega^\bullet_E)).
$$
This last isomorphism  composed with Cartier isomorphism  gives us 
$$
C: H^1_\dR(E)\to F^1,
$$
which we denote again by $C$.  For a covering $\{U_0,U_1\}$ of $E$ if we write $\omega\in H^1_\dR(E)$ in this covering as 
$$
(\omega_0,\omega_1,f_{10}),\ \omega_1-\omega_0=df_{10},\ \omega_i\in\Omega^1_E(U_i),\ f_{10}\in \Omega^0_E(U_0\cap U_1),\ \ 
$$ 
then $C(\omega)$ is simply the gluing of $C(\omega_0)$ and $C(\omega_1)$,  as in $U_1\cap U_0$ we have $C(\omega_0)=C(\omega_1)$.  The curve $K$ is the moduli space of triples $(E,\alpha,\omega)$, such that 
\begin{equation}
\label{15nov2025russia}
C(\alpha)=\alpha,\ \ C(\omega)=0.
\end{equation}
A more intrinsic way to see these equalities adapted for generalizations is the following: 
\begin{prop}
The equalities \eqref{15nov2025russia} are equivalent to the following:
the composition of isomorphisms 
\begin{equation}
\label{15112025moscow-1}
G_1/G_{0} \to   H^{0}(E,{\mathcal H}^1(\Omega^\bullet_{E}))  \xrightarrow{C}  H^{0}(E,\Omega^1_{E})  \leftarrow  F^1/F^2,
\end{equation}
maps $\alpha$ to $\alpha$  and the composition of isomorphisms 
\begin{equation}
\label{15112025moscow-2}
G_0/G_{-1} \to   H^{1}(E,{\mathcal H}^0(\Omega^\bullet_{E})) \xrightarrow{C}  H^{1}(E,\Omega^0_{E})   \leftarrow  F^0/F^{1},
\end{equation}
maps $\omega$ to $\omega$. 
\end{prop}
\begin{proof}
We only prove $\Rightarrow$. The equality $C(\alpha)=\alpha$ is easily translated to the statement about \eqref{15112025moscow-1}. We prove that the statement about \eqref{15112025moscow-2} follows from $C(\omega)=0$.  In order to see this, we write $\omega:=(\omega_0,\omega_1,f_{10}),\ \ \omega_1-\omega_0=df_{10}$, relative to a covering $\{U_0,U_1\}$ of $E$. Since $C(\omega)=0$, we have $\omega_0=df_0,\ \omega_1=df_1$, and so,  there is $g_{10}$ 
such that $f_{10}=g_{10}^p+f_1-f_0$. We have $\{g_{10}^p\}\in H^{1}(E,{\mathcal H}^0(\Omega^\bullet_{E})) $ and the Cartier map in \eqref{15112025moscow-2} maps this to $\{g_{10}\}$. From another side, $\omega$ in $H^1(E,\Omega^0_E)$ is represented by $\{f_{10}\}=\{g_{10}^p\}\in  H^{1}(E,\Omega^0_{E})$. Therefore, we must prove that  $\{g_{10}\}=\{g_{10}^p\}$ in  $H^{1}(E,\Omega^0_{E})$. This follows from $C(\alpha)=\alpha$ and the fact that under Serre duality the Frobenius map $F: H^1(E,\Omega^0_E)\to H^1(E,\Omega^0_E)$ is identified with the Cartier map $C: H^0(E,\Omega^1_E)\to H^0(E,\Omega^1_E)$ and $\alpha$ is identified with $\{g_{10}\}$, see for instance the discussions in the proof of \cite[Propostion 12.3.6]{KatzMazur1985}.   
\end{proof}

The equalities \eqref{15nov2025russia} imply that $E$ is an ordinary elliptic curve and $H^1_\dR(E)=F^1\oplus G_0$. Note that if we look at the representation of $\omega=\frac{(x+t_1)dx}{y}$ in the algebraic de Rham cohomology relative to a covering containing with   $U_0=E-\{O\}$ then the differential form $\omega$ must be exact in $U_0$. This implies the following statement. Let us consider the affine elliptic curve 
$U_0: y^2=4x^3-t_2x-t_3$ with $t_2,t_3\in \Ff_p$ with $p\not=2,3$ and $\Delta=27t_3^2-t_2^3\not=0$. Let also  $A,B\in \Ff_p[X,Y]$ be the polynomials of degree $p-1$ and $p+1$, respectively,  such that $E_{p-1}=A(\frac{1}{12}E_4,-\frac{1}{216}E_6)$ and $E_{p+1}=B(\frac{1}{12}E_4,-\frac{1}{216}E_6)$, where $E_i$'s are the Eisenstein series. If $A(t_2,t_3)=1$ (and so the elliptic curve is ordinary) then the  differential 
$1$-form $\left (x-\frac{1}{12}B(t_2,t_3)\right)\frac{dx}{y}$ is exact. It turns out that we have  a better statement.
\begin{prop}[Conjecture]
Let us consider the affine elliptic curve $E: y^2=4x^3-t_2x-t_3$ over $\Q[t_2,t_3]$, with $t_2,t_3$ as unknown variables. It is easy to show that on $E$ we have 
 \begin{equation}
 \label{22092025sara}
 (2n-1)\frac{x^{n}dx}{y}=a_0(t_2,t_3)\frac{dx}{y}+a_1(t_2,t_3)\frac{xdx}{y}+d(yQ(x)),
 \end{equation}
 for some homogeneous polynomials $a_0,a_1\in\Q[t_2,t_2],\ \deg(t_2)=4,\ \deg(t_3)=6$ of degree $2n$ and $2n-2$, and  a polynomial
 $Q\in\Q[t_2,t_3][x]$ of degree in $x$ $\leq n-2$. For a prime $p\not=2,3$ and $n=\frac{p+1}{2}$, $p$ does not appear in the denominator of $a_0,a_1$ and $Q$ and
$$
a_0\equiv_p -\frac{1}{12}B,\ \ \ a_1\equiv_p A,
$$
where $A,B\in \Q[X,Y]$ are polynomials of degree $p-1$ and $p+1$, respectively,  such that $E_{p-1}=A(\frac{1}{12}E_4,-\frac{1}{216}E_6)$ and $E_{p+1}=B(\frac{1}{12}E_4,-\frac{1}{216}E_6)$. 
In particular,  
\begin{equation}
\label{22sep2025kirk}
A(t_2,t_3) \frac{xdx}{y}-\frac{1}{12}B(t_2,t_3)\frac{dx}{y},
\end{equation}
is an exact form on the curve $E$ over $\Ff_p[t_2,t_3]$. 
\end{prop}
\begin{proof}
As we mentioned earlier, for ordinary  elliptic curves with $A(t_2,t_3)=1$ we know that the right-hand side of \eqref{22sep2025kirk} is exact and the left-hand side modulo $p$ is zero. Note also that if we choose $t_2,t_3$ from a perfect field $\sk$, then by \cref{05062024bita} the Cartier map evaluated at the differential form \eqref{22sep2025kirk} is zero, and hence, by Cartier isomorphism it must be exact.   The proposition can be verified for as much as primes the computer can handle. Here is the computer code.  
{\tiny 
\begin{verbatim}
proc linearw(poly A, number t(2), number t(3))
      {
      int i=deg(A) div deg(var(1)); poly B=A; poly exact;  poly LM; list Bl;
          while (i>1)
                {LM=lasthomo(B);
                 B=B-LM+leadcoef(LM)*((i-2+1/2)/number(4*(i-2)+6)*t(2)*var(1)^(i-2));
                 if (i>2){B=B+leadcoef(LM)*(i-2)/number(4*(i-2)+6)*t(3)*var(1)^(i-2-1);}
                 exact=exact+leadcoef(LM)*(1/number(4*(i-2)+6)*var(1)^(i-2));
                 i=deg(B) div deg(var(1));
                }
           Bl=subst(B,var(1),0), subst(B,var(1),1)-subst(B,var(1),0);
          return(list(Bl, exact));
      }
LIB "foliation.lib"; int np=40;
intvec prli=primes(1,np); int i; int j; list ll;
for (i=4;i<=size(prli);i=i+1)
   { "*", prli[i]; list ll; 
   ring r=(0,s_2,s_3), (x,t_2,t_3),dp;
   poly A=Eisenstein(prli[i]-1, t_2, t_3); poly B=Eisenstein(prli[i]+1, t_2, t_3 );
   A=subst(subst(A, t_2, 12*t_2),t_3, -216*t_3); B=subst(subst(B, t_2, 12*t_2),t_3, -216*t_3);
   ll=linearw(x^((prli[i]+1) div 2 ),s_2,s_3)[1];
   poly Bt=substpar(substpar(prli[i]*ll[1],s_2,t_2),s_3,t_3); poly At=substpar(substpar(prli[i]*ll[2],s_2,t_2),s_3,t_3);

   ring rr=int(prli[i]),(x,t_2,t_3),dp;
   poly A=imap(r,A);  poly B=imap(r,B);
   poly At=imap(r,At);  poly Bt=imap(r,Bt);
   prli[i], A-At,  Bt+(1/12)*B;
  }
\end{verbatim}
}
\end{proof}



\section{Other aspects of the Ramanujan vector field}
In this section we gather some other arithmetic aspects of the Ramanujan vector field. We consider the Eisenstein series
$E_2,E_4$ and $E_6$  (without $a_i$ constants), and in particular $A, B$ are the original ones in the literature. 
The corresponding differential equation in the vector field format is:
\begin{equation}
\label{16112023yau-2}
\vf:=\frac{1}{12}(t_1^2-t_2)\frac{\partial}{\partial t_1}+\frac{1}{3}(t_1t_2-t_3)\frac{\partial}{\partial t_2}+
\frac{1}{2}(t_1t_3-t_2^2)\frac{\partial}{\partial t_3}.
\end{equation}
All the conjectures in this section must be easy exercises and the main evidence for them is their verification for many prime numbers by computer. The author has not put any effort to prove them theoretically. 
They are motivated by some general discussions for vector fields. 
\begin{conj}
\label{06062024doubles}
 Let $a=(a_1,a_2,a_3)\in\C^3$ with $a_2^3-a_3^2\not=0$ and $\sring=\Z[\frac{1}{6}, a_1,a_2,a_3]$ (polynomial ring in $a_1,a_2,a_3$  and with coefficients in $\Z[\frac{1}{6}]$). Let also $\vf$ be the Ramanujan vector field \eqref{16112023yau-2} in $\A^3_\sring=\spec(\sring[t_1,t_2,t_3])$.
 For an infinite number of primes $p$, $\vf$ is not collinear with $\vf^p$ at $a$, in the scheme $\T_p:=\T\times_{\sring}\spec(\sring/p\sring)$ (that is modulo prime $p$).
 \footnote{This conjecture follows from a generalization of Grothendieck-Katz conjecture for vector fields
in \cref{07092024omidviolento} and the fact that the solutions of $\vf$ passing through $a$  is a transcendental curve, see for instance \cite[Theorem 1]{ho06-3}).
}
\end{conj}
Recall that the Ramanujan vector field
leaves the discriminant locus $\Delta: t_2^3-t_3^2=0$ invariant and its solutions in this locus are algebraic.
\begin{conj}\rm
\label{13112023tifi}
For the Ramanujan vector field $\vf$ and any prime $p\not=2,3$, we have the following statements about ideals in
$\Ff_p[t_1,t_2,t_3]$.
\begin{enumerate}
 \item
 \label{18june2024hugo}
 The radical of the ideal of $\vf^p=0$, that is $\langle \vf^pt_1,\vf^pt_2,\vf^pt_3\rangle$, is generated by $\Delta$.
\item \label{06062024ihes}
The primary decomposition of the ideal of the equality $\vf^p=\vf$, that is $\langle \vf^pt_1-\vf t_1,\vf^pt_2-\vf t_2,\vf^pt_3-\vf t_3\rangle$,  consists of three components: The first two are
$\langle A-1,B-t_1\rangle$ and $\langle A+1, B+t_1\rangle$ mentioned in the proof of \cref{08062024drywall}  and the third component is $\sing(\vf):=\langle t_1^2-t_2,t_1^3-t_3\rangle$ which is inside $\Delta=0$.
\item
\label{13062024sam}
The radical of the collinearness ideal between $\vf^p$ and $\vf$, that is  $\langle \vf^pt_i\vf t_j-  \vf^pt_j\vf^pt_i \mid i,j=1,2,3\rangle$,
is generated by $\Delta\cdot (B-t_1A)$.
\end{enumerate}
\end{conj}
One can check the above conjecture for any prime $p\not=2,3$ using the following code.
{\tiny
\begin{verbatim}
LIB "foliation.lib"; int np=100;
intvec prli=primes(1,np); int n=10; int pr=prli[n]; pr;
ring r=pr, (t_1,t_2,t_3),dp;
list vecfield=1/12*(t_1^2-t_2), 1/3*(t_1*t_2-t_3), 1/2*(t_1*t_3-t_2^2);
list vf; int i; int k; int j; poly Q; int di=size(vecfield);
for (i=1; i<=di;i=i+1){vf=insert(vf, var(i),size(vf));}
for (k=1; k<=di;k=k+1)
    {for (i=1; i<=pr;i=i+1)
         {Q=0;
          for (j=1; j<=di;j=j+1){Q=Q+diff(vf[k], var(j))*vecfield[j];}
          vf[k]=Q;
         }
    }
poly Delta=t_2^3-t_3^2; ideal I=vf[1..size(vf)]; I=radical(I); "vp=0"; I;
ideal K=vf[1]-vecfield[1], vf[2]-vecfield[2],vf[3]-vecfield[3];  "vp=v"; primdecGTZ(K);

matrix CL[2][3]=vecfield[1..di],vf[1..di]; ideal J=minor(CL,2); J=radical(J); "vp cllinear v";   J;
ring rr=0,(x,t_2,t_3),dp;
poly A=Eisenstein(pr-1, t_2, t_3); poly B=Eisenstein(pr+1, t_2, t_3 );
setring r;
poly A=imap(rr,A);  poly B=imap(rr,B);
(Delta*(B-t_1*A)/J[1])*J[1]-Delta*(B-t_1*A);
//---Experminetal verification of the fact that B-t_1A is a first integral
list lv=t_1,t_2,t_3;
Diffvf(B-t_1*A, lv, vecfield);
//----ideal generated by A, B----
ideal I=A,B; radical(I);
I=B-t_1*A,Delta; primdecGTZ(I);
//-----Investigating F_g--------
poly P=(10*B^3-6*B*t_2-4*t_3)/103680;
ideal I=A-1,P; primdecGTZ(I);
pr;
\end{verbatim}
}
\cref{13112023tifi}, \cref{13062024sam} implies the following.
\begin{conj}
\label{26102023omaiorladrao}
For all primes $p\not=2,3$ the Ramanujan vector field in $\A^3_{\Ff_p}$ is not $p$-closed, that is,
$\vf^p$ is not collinear to $\vf$ at a generic point.
\end{conj}
The above  statement can be also verified by the following computer code.
{\tiny
\begin{verbatim}
     LIB "foliation.lib";
     ring r=0, (t_1,t_2,t_3),dp;
     list vf=1/12*(t_1^2-t_2), 1/3*(t_1*t_2-t_3), 1/2*(t_1*t_3-t_2^2);
     int ub=200;
     BadPrV(vf, ub);
\end{verbatim}
}
The fact that the variety given by $A=B=0$ is $t_2=t_3=0$ has been noticed in the literature, see \cite[Theorem 3.1]{Katz2021}. This implies that
the radical of the ideal $\langle A,B\rangle$ is $\langle t_2,t_3\rangle$. Moreover, we can also easily prove that the primary decomposition of the ideal
$\langle \Delta, B-t_1A\rangle$ consists of two components $\langle t_2,t_3\rangle$ and $\sing(\vf)$.  Both facts can be verified experimentally.

After the first draft of the present section was written, Frederico Bianchini in his Ph.D. thesis, was able to prove the following theorem. This can be used in order to prove most of the experimental conjectures in this section.
\begin{theo}
Let $\vf$ be the Ramanujan vector field in $\A^3_{\Ff_p}$, where $p$ is  a prime number different form $2$ and $3$. Then 
\[
\vf^p =A^2\cdot \vf -\left(\frac{1}{12} B + t_1  A\right)^2 \cdot f
+ 
A\cdot \left(\frac{1}{12}  B+ t_1 A\right) 
\cdot h
\]
where $A$ and $B$ are given by \eqref{07102025bimsa-1}  and \eqref{07102025bimsa-2}
and 
$$
h=6t_3\frac{\partial }{\partial t_3}+4t_2\frac{\partial }{\partial t_2}+2t_1\frac{\partial}{\partial t_1},\ f=\frac{\partial}{\partial t_1}.
$$
\end{theo}	

{\tiny
}

\section{CM elliptic curves}
It is interesting to know whether it is possible to use the Ramanujan vector field to characterize CM elliptic curves in characteristic zero or supersingular elliptic curves in characteristic $p$.
\begin{theo}
\label{22102025mauricio}
If  an elliptic curve $y^2=4x^3-g_2x-g_3,\ g_2,g_3\in\Z[\frac{1}{N}]$ for some $N\in\N$ and $27g_3^2-g_2^3$ invertible in $\Z[\frac{1}{N}]$ is CM,  then there is $g_1\in \Z[\frac{1}{N}]$ such that for an infinite number of primes $p$, the vector field  $\vf^p$ is collinear with $\vf$ at $(g_1,g_2,g_3)$ modulo $p$ in $\A_{\Ff_p}^3$ for the Ramanujan vector field $\vf$. 
\end{theo}
\begin{proof}
Let us take a CM elliptic curve over $\Z[\frac{1}{N}]$ by an order $\O=\Z+n\O_{\Q(\sqrt{-D})}$ for some square-free $D\in\N$. Let also $1,u$ be a basis  of $\O$ as $\Z$-module. The element $u$ acts as $f: E/\bar\Q\to E/\bar \Q$, and hence in the de Rham cohomology  $f^*: H^1_\dR(E/\bar\Q)\to H^1_\dR(E/\bar\Q)$  by pull-back of differential forms. According to \cite[Lemma 4.1]{Katz2021}, we can choose  $g_1\in \Z[\frac{1}{N}]$ such that the action of $u$ in the basis  $\alpha:=\frac{dx}{y},\ \ \omega= \frac{xdx}{y}+g_1\frac{dx}{y}$ is diagonalized, that is,  
\footnote{It is a well-known fact $f$ is defined  over the Hilbert class field $K$ of $\Q(\sqrt{-D})$, and hence,  
$f^*: H^1_\dR(E/K)\to H^1_\dR(E/K)$ is defined over this larger field. However, in the proof of  \cite[Lemma 4.1]{Katz2021}, the author assumes that $f^*$ and in particular the $a$-number in his proof is defined over $\Q(\sqrt{-D})$.   
}
\begin{equation}
\label{06102025scheideger}
f^*\alpha=u\alpha,\ f^*\omega=\bar u \omega. 
\end{equation}
Note that $u\bar u$ is the degree of $f$. 
Let us now consider modulo $p$-reduction of $E,f$. Since $f$ is defined over the Hilbert class field $K$ of $\Q(\sqrt{-D})$, we have to take prime ideal $\mathfrak p$ in $\O_K$ whose residue field is of characteristic $p$ and consider reduction modulo $\mathfrak p$. We discard finite number of primes for which $u=\bar u$ modulo $\mathfrak p$. Since $u$ is not a rational number, this follows from Chebotarev density theorem. In characteristic $p$ we have also a one dimensional subspace $G_0=H^1(E, {\cal H}^0(\Omega^\bullet_E)) \subset H^1_\dR(E/\Ff_p)$ of the conjugate filtration which is described in \cref{22102025shanghai}. It is a well-known fact the density of primes $p$ for which $E/\Ff_p$ is ordinary, that is, $E/\Ff_p$ is not a supersingular elliptic curve,  is $1/2$ for $CM$ elliptic curves and it is $1$ otherwise, see for instance the references in \cite{Elkies1987} or \href{https://swc-math.github.io/aws/2024/PAWSLi/2023PAWSLiNotes6.pdf}{this link}. In particular, we have infinite number of primes for which $E/\Ff_p$ is ordinary. In this case, $G_0$ is also characterized as the kernel of the Cartier map $C: H^1_\dR(E/\Ff_p)\to F^1$. It follows from \cref{05062024bita}  that this kernel is given by $\omega:=\frac{xdx}{y}-\frac{B}{12 A}\frac{dx}{y}$. This implies that        
\begin{equation}
\label{22102025Atourhotel}
g_1\equiv_p -\frac{B(g_2,g_2)}{12\cdot A(g_2,g_3)} 
\end{equation}
for infinite number of primes. This observation actually appears in the proof of \cite[Lemma 4.1]{Katz2021}. Katz has also computed $g_1$ (in his notation $A$) for many examples of $CM$ elliptic curves using the above congruency. It turns out that the point $(g_1,g_2,g_3)$ modulo $p$ lies in the hypersurface $B+12t_1A=0$. In the proof of \cref{08062024drywall} we have seen that $B+12t_1A=0$ is invariant under the Ramanujan vector field $\vf$ and all its solutions in this hypersurface are algebraic curves. This implies that $\vf^p$ is collinar with $\vf$ at all points of this hypersurface, and in particular the point $g$. 
\end{proof}
It seems to the author that the converse of \cref{22102025mauricio} is true. Since in its proof we have actually seen that the mentioned property is valid for primes of density $\geq 1/2$, it might be wiser to formulate its converse in the following way: 
\begin{conj}
If $g=(g_1,g_2,g_3)\in\Q^3$ with $27g_3^2-g_2^3\not=0$ and the density of primes $p$ such that $\vf^p$ is collinear with $\vf$ at $g\in\Ff_p^3$ is bigger than or equal to  $1/2$ then $y^2=4x^3-g_2x-g_3$ is CM.  
\end{conj}
This conjecture might be proved in the following way. Assume that $\vf^p$ is collinear with $\vf$ at $g\in\Q^3$. It might not be hard to prove that this is equivalent to the existence of a curve in $\A^3_{\Ff_p}$ passing through $g$ and tangent to the Ramanujan vector field. But we expect that the only algebraic solutions to Ramanujan vector field in $\A^3_{\Ff_p}$ and outside the discriminant $\Delta=0$ are $B+12t_1A=0,\ A=\text{ constant }$. This implies the that we have \eqref{22102025Atourhotel} for a set of primes of density $\geq 1/2$. Our problem is reduced to prove that this implies that the underlying elliptic curve is CM.    
\begin{rem}\rm
It is natural to define for each $u\in \O_{\Q(\sqrt{-D})},\ a\not\in \Z$ the locus in $\T$ parameterizing $(E,\alpha,\omega)$ for which $E$ is $CM$ with $f\in\End(E)$ which satisfies \eqref{06102025scheideger}. It turns out that this is a union of  orbits of the multiplicative subgroup $\Gm$ of $\BG$ acting on $\T$. They are also   tangent to  the vector field $2t_1\frac{\partial }{\partial t_1}+4t_2\frac{\partial }{\partial t_2}+6t_3\frac{\partial }{\partial t_3}$.  
\end{rem}

\section{Final remarks}
As the main goal of the present article has been to prepare the ground for similar investigations in the case of Calabi-Yau varieties, one might try to describe modulo prime properties of topological string partition functions $F_g$ in the case of elliptic curves, see the articles of Dijkgraaf, Douglas, Kaneko and Zagier in \cite[Appendix B]{ho14}. These are homogeneous polynomials in $E_2,E_4$ and $E_6$ of weight $6g-6$, for instance 
$F_2=\frac{1}{103680}(10E_2^3-6E_2E_4-4E_6)$.  For Calabi-Yau threefolds, and in particular for mirror quintic, there is an ambiguity problem for $F_g$'s which has been only established for lower genus, see the references in \cite{GMCD-MQCY3}. Modulo primes investigation of $F_g$'s might give some insight to this problem. For instance, in the last lines of the computer code of \cref{13112023tifi} we have investigated the zero locus of $F_2$ restricted to the curve ${\rm Zero}(\IS_p)$. We found many such points with coordinates in $\Ff_p$, but no pattern for different primes were found. 

The articles of Swinnerton-Dyer and Serre used in this article have originated the Serre conjecture on the modularity of two dimensional Galois representations. It would be of interest to see similar conjectural statements  in the case of Calabi-Yau modular forms.

Let $E$ be an elliptic curve over $\Z$ and assume that its reduction $E/\Ff_p$ modulo $p$ is smooth. We consider $E$ over $\Z_p$, and we know that $H^1_\dR(E/\Z_p)$ is a free $\Z_p$-module of rank $2$. Moreover, by a comparison theorem of Berthelot, $H^1_\dR(E/\Z_p)$ is isomorphic in a canonical way to
$H^1_{\rm cris}(E/\Ff_p)$. The later, has the Frobenius map which lifts to a map $F:H^1_\dR(E/\Z_p)\to H^1_\dR(E/\Z_p)$, for details and references see \cite{ked08}. When we started to write the present paper, we wanted to  formulate \cref{21may2024raisimord-2} using $F$, but later we realized that only the Cartier operator is sufficient.

\chapter{On a Hodge locus}
\label{29082024tanha}
\def\codnum{{\sf C}}
\def\Has{{\mathcal C}}

{\it Abstract:}
There are many instances such that deformation space of the homology class of an algebraic cycle as a Hodge cycle is larger than its deformation space as algebraic cycle. This phenomena can occur for algebraic cycles inside hypersurfaces, however, we are only able to gather evidences for it by computer experiments.  In this article we describe one example of this for cubic hypersurfaces. The verification of the mentioned phenomena in this case is proposed as the first GADEPs problem. 
The main goal is either to verify the (variational) Hodge conjecture in such a case  or gather evidences that it might produce a counterexample to the Hodge conjecture.

\section{Introduction}
\label{sec1}

Let $\T$ be the space of homogeneous polynomials $f(x)$ of degree $d$ in $n+2$ variables 
$x=(x_0,x_1,\ldots,x_{n+1})$ and with coefficients in $\C$ such that the induced hypersurface $X:=\P\{f=0\}$ in $\P^{n+1}$ is smooth. We assume that $n\geq 2$ is even and $d\geq 3$. Consider the   subvariety  of $\T$ parametrizing hypersurfaces containing two projective subspaces $\P^{\frac{n}{2}}, \check\P^{\frac{n}{2}}$ (we call them linear cycles) with  
$\P^{\frac{n}{2}}\cap \check\P^{\frac{n}{2}}=\P^{m}$ for a fixed $-1\leq m\leq \frac{n}{2}-1$ ($\P^{-1}$ is the empty set). We are actually interested in a local analytic branch $V_Z$ of this space which parametrizes deformations of a fixed $X$ together with such two linear cycles. We consider the algebraic cycle 
\begin{equation}
\label{06102022bolsolula}
Z=\cf \P^{\frac{n}{2}}+\check \cf \check \P^{\frac{n}{2}},\ \ \cf\in\N,0\not=\check\cf\in\Z
\end{equation}
and its cohomology class
$$
\delta_0=[Z]\in H^{\frac{n}{2}, \frac{n}{2}}(X)\cap H^n(X,\Z).
$$
Note that $V_Z$ does not depend on $\cf$ and $\check\cf$ and it is $V_Z=V_{\P^\frac{n}{2}}\cap V_{\check \P^\frac{n}{2}}$, where $V_{\P^\frac{n}{2}}$ and $V_{\check \P^\frac{n}{2}}$ are two branches of the subvariety of $\T$ parameterizing hypersurfaces containing  a linear cycle, see \cref{TwoPn2}.
From now on we use the notation $t\in\T$ and denote the corresponding polynomial and hypersurface  by $f_t$ and $X_t$ respectively, being clear that $f_0=f$ and  $X_0=X$. 
 The monodromy/parallel transport $\delta_t\in H^n(X_t,\Z)$ is well-defined for all $t\in(\T,0)$, a small neighborhood of $t$ in $\T$ with the usual/analytic  topology, and it is not necessarily supported in algebraic cycles like the original $\delta_0$. We arrive at the set theoretical definition of the Hodge locus 
 \begin{equation}
  V_{[Z]}:=\left \{t\in(\T,0)\ \  | \ \ \delta_t \hbox{ is a Hodge cycle, that is } \ \ \delta_t\in H^{\frac{n}{2}, \frac{n}{2}}(X_t)\cap H^n(X_t,\Z) \right \}.
 \end{equation}
 We have $V_Z\subset V_{[Z]}$ and claim that 
\begin{conj}
\label{zemarceneiro2022}
For $d=3, \ n=6,8, \ m=\frac{n}{2}-3$ and all $\cf\in\N,0\not=\check\cf\in\Z$, the Hodge locus $V_{[Z]}$ is of dimension  $\dim(V_Z)+1$, and so, $V_{Z}$ is a codimension one subvariety of $V_{[Z]}$.   Moreover, the Hodge conjecture for the Hodge cycle $\delta_t,\ t\in V_{[Z]}$ is true. 
\end{conj}
For $n=4$ the Hodge conjecture is a theorem and the first part of the above  conjecture for $n=4$ is true for trivial reasons.  
If the first part of the above conjecture is true then one might try to verify the Hodge conjecture for the Hodge cycle $\delta_t,\ t\in V_{[Z]}$ which is absolute, see Deligne's lecture in \cite{dmos}. 
It is only verified for $t\in V_Z$ using the algebraic cycle $Z$. By Cattani-Deligne-Kaplan theorem $V_{[Z]}$ for fixed $\cf$ and $\check\cf$ is a union of branches of an algebraic set in $\T$ and we will have the challenge of verifying a particular case of Grothendieck's variational Hodge conjecture. It can be verified easily that the tangent spaces of $V_{[Z]}$ intersect each other in the tangent space of $V_Z$, and hence, we get a pencil of Hodge loci depending on the rational number $\frac{\cf}{\check \cf}$, see \cref{TwoPn2}.  
 Similar computations as for \cref{zemarceneiro2022} in the case of surfaces result in a conjectural counterexample to a conjecture of J. Harris for degree $8$ surfaces, see \cite{ho2019}. 

The seminar \href{https://w3.impa.br/~hossein/GADEPs/GADEPs.html}{"Geometry, Arithmetic and Differential Equations of Periods" (GADEPs)}, started in the pandemic year 2020 and its aim is to gather people in different areas of mathematics around the notion of periods which are certain multiple integrals.  \cref{zemarceneiro2022} is the announcement of the first GADEPs' problems. 
\begin{figure}
\begin{center}
\includegraphics[width=0.6\textwidth]{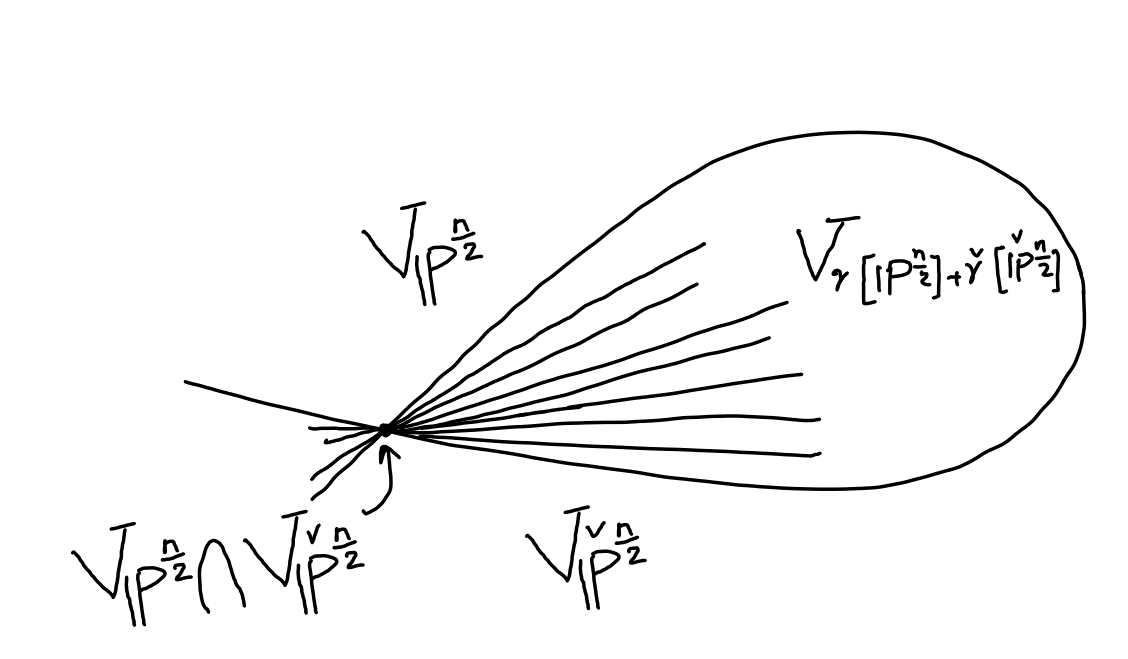}
\caption{A pencil of Hodge loci}
\label{TwoPn2}
\end{center}
\end{figure}

\section{The path to \cref{zemarceneiro2022}}
\label{sec2}

The computational methods introduced in \cite{ho13} can be applied to an arbitrary combination of linear cycles, for some examples see \cite[Chapter 1]{Headaches},  however, for simplicity the author focused mainly in the sum of two linear cycles as announced earlier. We note that $V_{[Z]}$ carries a natural analytic scheme/space structure, that is,  there is an ideal $I=\langle f_1,f_2,\ldots,f_k\rangle \subset \O_{\T,0}$ of holomorphic functions $f_i$ in a small neighborhood $(\T,0)$ of $0$ in $\T$, and the ring structure of $V_{[Z]}$ is $\O_{\T,0}/I$. The holomorphic functions $f_i$ are periods $\int_{\delta_t}\omega_i$, where $\omega_i$'s are global sections of the $n$-th cohomology bundle 
$\cup_{t\in (\T,0)}H^n_\dR(X_t)$ such that for fixed $t$ they form a basis of the piece $F^{\frac{n}{2}+1}H^n_\dR(X_t)$ of Hodge filtration (form now on all Hodge cycles will be considered in homology and not cohomology).  For hypersurfaces, using Griffiths work \cite{gr69}, the holomorphic functions $f_i$'s are 
\begin{equation}
\label{02082023gago}
\mathlarger{\mathlarger{\int}}_{\delta_t}\Resi\left(\frac{x^\beta\Omega}{f^{k}_t}\right),\ \  k=1,2,\cdots,\frac{n}{2},\ \ \ x^\beta\in (\C[x]/\jacob(f_t))_{kd-n-2}
\end{equation}
and $x^\beta$ is a basis of monomials for the degree $kd-n-2$ piece of the Jacobian ring $\C[x]/\jacob(f_t)$ and $\Omega:=\sum_{i=0}^{n+1}x_idx_0\wedge dx_1\wedge\cdots \wedge dx_{i-1}\wedge dx_{i+1}\wedge\cdots \wedge dx_{n+1}$.  The Taylor series of such integrals can be computed and implemented in a computer, however, for simplicity we have done this around the Fermat variety. 

Let us consider the hypersurface $X_t$ in the projective space 
$\Pn {n+1}$ given by the homogeneous polynomial:
\begin{equation}
\label{15dec2016}
f_t:=x_0^{d}+x_1^{d}+\cdots+x_{n+1}^d-\sum_{\alpha}t_\alpha x^\alpha=0,\ \ 
\end{equation}
$$
t=(t_\alpha)_{\alpha\in I}\in(\T,0), 
$$
where $\alpha$ runs through a finite subset $I$ of $\N_0^{n+2}$ with $\sum_{i=0}^{n+1} \alpha_i=d$. 
From now on for all statements and conjectures $X_0$ is the Fermat variety.  The Taylor series for the Fermat variety $X_0$ can be computed explicitly,  see \cite[18.5]{ho13}.
\footnote{For particular cases of such series see \eqref{300620224pilar} and \eqref{09082024david}.}
It is also implemented in computer, see \cite[Section 20.11]{ho13}. Its announcement takes almost a full page and we only content ourselves to the following statement:
\begin{prop}
Let $\delta_0\in H_n(X_0,\Q)$ be a Hodge cycle and $x^\beta$ be a monomial of degree $kd-n-2$. The integral  $\frac{1}{(2\pi i)^{\frac{n}{2}}}\mathlarger{\mathlarger{\int}}_{\delta_t}\Resi\left(\frac{x^\beta\Omega}{f^{k}_t}\right)$ can be written as a power series in $(t_\alpha)_{\alpha\in I}$ with coefficients in an abelian extension of $\Q(\zeta_{d})$. If $\delta_0$ is a sum of linear cycles $\P^{\frac{n}{2}}$ then such an abelian extension is $\Q(\zeta_{2d})$.   
\end{prop}

In \cref{zemarceneiro2022} we have considered $V_{[Z]}$ as an analytic variety. As an analytic scheme and for $X_0$ the Fermat variety, we even claim that $V_{[Z]}$ is smooth which implies  that it is also reduced.  
The first goal is to compare the dimension of Zariski tangent spaces
$\dim(T_tV_{[Z]})$ and $\dim(T_t V_Z)$. 
Computation of $TV_{[Z]}$ is done using the notion of infinitesimal variation of Hodge structures developed by P. Griffiths and his coauthors in \cite{CGGH1983}. In a down-to-earth terms, this is just the data of the linear parts of $f_i$'s. 
It turns out that 
\begin{theo}
\label{02082023guapemirim}
For $m<\frac{n}{2}-\frac{d}{d-2}$ we have $T_0V_{[Z]}=T_0V_Z$, and hence, $V_{[Z]}=V_Z$.
\end{theo}
This is proved in \cite[Theorem 18.1]{ho13} for 
\begin{equation}
\label{02.08.2023chiam}
0<\cf\leq |\check\cf|\leq 10
\end{equation}
and $(n,d)$ in the list
\begin{eqnarray*}
(2,d),\ \ \ d\leq 14,\ \ & &  (4, 3), (4,4),(4,5),(4,6),\ (6,3), (6,4), (8,3),\\
& & (8,3), (10,3), (10,3), (10,3),
\end{eqnarray*}
using computer. 
For the proof of \cref{02082023guapemirim}  we have computed both $\dim T_0V_{[Z]}$ and $\dim(V_Z)$ and we have verified that these dimensions are equal.  The full proof of \cref{02082023guapemirim}  is done in \cite[Theorem 1.3]{RobertoThesis}.
Throughout the paper, the condition \eqref{02.08.2023chiam} is needed for all statements whose proof uses computer, however, note that the number $10$ is just the limit of the computer and the author's patience for waiting the computer produces results. All the conjectures that will appear in this section  are not considered to be so difficult and their proofs or disproofs are in the range of available methods in the literature. 
\begin{conj}
For  $m=\frac{n}{2}-1, (\cf,\check\cf)\not=(1,1)$, the Hodge locus $V_{[Z]}$ as a scheme  is not smooth, and hence the underlying variety of $V_{[Z]}$ might be $V_Z$ itself.
\end{conj}
In \cite[Theorem 18.3, part 1]{ho13} we have proved the above conjecture by computer  for  $(n,d)$ in the list
\begin{eqnarray*}\label{5.12.2017}
& &(2,d),\ 5\leq d\leq 9, (4, 4), (4,5), (6,3),(8,3),
\end{eqnarray*}
see also \cite{Dan2017} for many examples of this situation in the case of surfaces, that is, $n=2$. 
\begin{theo}
\label{02082023barato}
For $m=\frac{n}{2}-1, (\cf,\check\cf)=(1,1)$, $V_{[Z]}$ parameterizes hypersurfaces containing a complete intersection of type $(1,1,\cdots,1,2)$, where $\cdots$ means $\frac{n}{2}$ times. 
\end{theo}
Note that in the situation of \cref{02082023barato}, $\P^{\frac{n}{2}}+\check\P^{\frac{n}{2}}$ is a complete intersection of the mentioned type. In this way \cref{02082023barato} follows from \cite{Dan-2014}, see also \cite[Chapter 11]{ho13-Roberto}. In our search for a Hodge locus $V_{[Z]}$ bigger than $V_Z$ we arrive at the cases 
$$
(d,m)=(3,\frac{n}{2}-3),(3, \frac{n}{2}-2), (4, \frac{n}{2}-2). 
$$
\begin{conj}
In  the case $(d,m)=(3, \frac{n}{2}-2)$ and $(\cf,\check\cf)\not=(1,-1)$, the Hodge locus $V_{[Z]}$ is not smooth.
\end{conj}
This conjecture for $n=6,8$ is proved in  \cite[Theorem 18.3 part 2]{ho13}. 
The same conjecture for $(n,d,m)=(4,4,0)$ is also proved there. 
\begin{conj}
For $(d,m)=(3,\frac{n}{2}-2)$ with $(\cf,\check\cf)=(1,-1)$,  the Hodge locus $V_{[Z]}$ is smooth and it   parameterizes hypersurfaces containing generalized cubic scroll (for the definition see \cref{18082023-jinmiad} and \cite[Section 19.6]{ho13}) .
\end{conj}
This conjecture is obtained after a series of email discussions with P. Deligne in 2018, see \cite[Chapter 1]{Headaches}  and  \cite[Section 19.6]{ho13}. The proof of this must not be difficult (comparing two tangent spaces). The case $(n,d,m)=(4,4,0)$ with $(\cf,\check\cf)=(1,-1)$ is still mysterious, however, it might be solved by similar methods as in the mentioned references.  
The only remaining cases are the case of  \cref{zemarceneiro2022}  for arbitrary even number $n\geq 4$  and $(d,m)=(4,\frac{n}{2}-2),\ n\geq 6$. It turns out \cref{zemarceneiro2022} is false for $n\geq 10$, see \cref{30112023camera}.


If the verification of the (variational) Hodge conjecture is out of reach for $\delta_t$, a direct verification of  the first part of \cref{zemarceneiro2022} might be possible by developing Grobner basis theory  for ideals of formal power series $f_i$ which are not polynomially generated. Such formal power series satisfy polynomial differential equations (due to Gauss-Manin connection), and so, this approach seems to be quite accessible.

\section{Evidence 1}
The first evidence to  \cref{zemarceneiro2022} comes from computing the Zariski tangent spaces of both $V_Z$ and $V_{[Z]}$, for the Fermat variety $X_0$,  and observing that $\dim(\T_tV_{[Z]})=\dim(\T_t V_Z)+1$.  This has been verified by computer for many examples of $n$ in \cite[Chapter 19]{ho13} and the full proof can be found in \cref{sec4}.   However, this is not sufficient as $V_{[Z]}$ carries a natural analytic scheme structure. 
Moreover, $V_{[Z]}$ as a variety  might be singular, even though, the author is not aware of an example.
The Zariski tangent space is only the first approximation of a variety, and one can introduce the $N$-th order approximations $V^N_{[Z]},\ \ N\geq 1$ which we call it the $N$-th infinitesimal Hodge locus,  such that $V^1_{[Z]}$ is the Zariski tangent space. The algebraic variety $V^N_{[Z]}$ is obtained by truncating the defining holomorphic functions of $V_Z$ up to degree $N$.  The non-smoothness results as above follows from the non-smoothness of $V^N_{[Z]}$ for small values of $N$ like $2,3$ (the case $N=2$ has been partially  treated in cohomological terms in \cite{mclean2005}). The strongest evidence to  \cref{zemarceneiro2022} is the following theorem  in \cite[Theorem 19.1, part 2]{ho13} which is proved by heavy computer calculations.
\begin{theo}
\label{22112022protests}
In the context of  \cref{zemarceneiro2022}, for $\cf\in\N,\check\cf\in\Z,\ 1\leq \cf, |\check\cf|\leq 10$,
the infinitesimal Hodge locus $V_{[Z]}^N, N\leq M$
is  smooth  for all $(n,M)=(6,14), (8,6),(10,4),(12,3)$.
\end{theo}
For $n=4$, the Hodge locus $V_{[Z]}$ itself is smooth  for trivial reasons.
There is abundant examples of Hodge cycles for which we know neither to verify the Hodge conjecture (construct algebraic cycles) nor give evidences that they might be counterexamples to the Hodge conjecture, see \cite{Deligne-HodgeConjecture} and \cite[Chapter 19]{ho13}. Finding Hodge cycles for hypersurfaces is extremely difficult, and the main examples in this case are due to T. Shioda for Fermat varieties \cite{Shioda1979}.

We  have proved \cref{22112022protests}  by computer with processor 
{\tt  Intel  Core i7-7700}, $16$ GB Memory plus $16$ GB swap memory and the operating system {\tt Ubuntu 16.04}. 
It turned out that for many cases  such as $(n, N)=(12,3)$, we get the  `Memory Full' error. 
Therefore,  we had to increase the swap memory up to $170$ GB. 
Despite the low speed of the swap which slowed down the computation, 
the computer was  able to use the data and give us the desired output. The computation for this example  took more than $21$ days. 
We only know that at least $18$ GB of the swap were used. 

\section{Evidence 2}
\label{15082023guapiacu}
The main project behind  \cref{zemarceneiro2022} is to discover new Hodge cycles for hypersurfaces by deformation. Once such Hodge cycles are discovered, there is an Artinian Gorestein ring attached to such Hodge cycles which contains some partial data of the defining ideal of the underlying algebraic cycle (if the Hodge conjecture is true), see \cite{voisin89, Otwinowska2003, ho13-Roberto}. In the case of lowest codimension for a Hodge locus, this is actually enough to construct the algebraic cycle (in this case a linear cycle) from the topological data of a Hodge cycle, see \cite{voisin89} for $n=2$ and \cite{Villaflor2022} for arbitrary $n$ but near the Fermat variety, and \cite{EmreHossein2018}. It turns out that in the case of surfaces ($n=2$) the next minimal codimension for Hodge loci (also called Noether-Lefschetz loci) is achieved by surfaces containing a conic, see \cite{voisin89, voisin90} .
Therefore, it is expected that components of Hodge loci of low codimension parametrize hypersurfaces with rather simple algebraic cycles. In our case,  it turns out that $\dim(V_Z)$ grows like the minimal codimension for Hodge loci.  This is as follows.    A formula for the dimension of $V_Z$ for arbitrary $m$ in terms of binomials can be found  in  \cite[Proposicion 17.9]{ho13}:  
  \begin{equation}
 \label{6dec2016}
  \codim(V_{Z})=
  2\codnum_{1^{\frac{n}{2}+1},(d-1)^{\frac{n}{2}+1}}-\codnum_{1^{n-m+1}, (d-1)^{m+1}}.
 \end{equation}
where for a sequence of natural numbers $\underline a=(a_1,\ldots,a_{2s})$ we define
\begin{equation}
 \label{1julio2016-bonn}
 \codnum_{\underline{a}}=
 \bn{n+1+d}{n+1}-
 \sum_{k=1}^{2s}(-1)^{k-1} \sum_{a_{i_1}+a_{i_2}+\cdots+a_{i_k}\leq d }\bn{n+1+d-a_{i_1}-a_{i_2}-\cdots-a_{i_k}}{n+1}
\end{equation}
and the second sum runs through all $k$ elements (without order) of $a_i,\ \ i=1,2,\ldots,2s$.
{\tiny
}
For $d=3$ and $k=\frac{n}{2}$ we have 
\begin{eqnarray*}
\codnum_{1^{k+1+x}, 2^{k+1-x}} &=& \frac{1}{6}(2k+4)(2k+3)(2k+2)-(k+1+x)\frac{1}{2}(2k+3)(2k+2)\\
                                         && -(k+1-x)(2k+2)+\frac{1}{2}(k+1+x)(k+x)(2k+2)\\
                                         & & +(k+1-x)(k+1+x)
                                         -\frac{1}{6}(k+1+x)(k+x)(k+x-1)\\
                                         &=& \frac{1}{6}k^3-\frac{1}{2}k^2x+(\frac{1}{2}x^2-\frac{1}{6})k -\frac{1}{6}x(x-1)(x+1)
\end{eqnarray*}
and so in our case $x=3$ we have 
$$
\codim(V_{Z})=\frac{1}{6}k^3+\frac{3}{2}k^2-\frac{14}{3}k+4
$$
which grows like the minimum codimension $\frac{1}{6}(k+1)k(k-1)$ for Hodge loci.  This minimum codimension is achieved by the space of cubic hypersurfaces containing a linear cycle.
The conclusion is that if the Hodge conjecture is true for $\delta_t,\ t\in V_{[Z]}$ then \cref{zemarceneiro2022} must be an easy exercise. 
Therefore, the author's hope is that  \cref{zemarceneiro2022} and its generalizations will flourish new methods to construct algebraic cycles.

\section{Evidence 3}
\label{18082023-jinmiad}
There is a very tiny evidence that the Hodge cycle in  \cref{zemarceneiro2022} might be a counterexample to the Hodge conjectures. All the author's attempts to produce new components of Hodge loci with the same codimension as of $V_{[Z]}$ has failed. This is summarized in  \cite[Table 19.5]{ho13} which we explain it in this section. 
\begin{defi}\rm
\label{03082023spe}
Let us consider a linear subspace $\P^{\tilde n}\subset \P^{n+1}$, a linear rational surjective map $\pi: \P^{\tilde n}\dashrightarrow \P^{r}$ with indeterminacy set $\P^ {\tilde n-r-1}$, an algebraic cycle $\tilde Z\subset \P^{r}$ of dimension $\frac{n}{2}+r-\tilde n$. The algebraic cycle $Z:=\pi^{-1}(\tilde Z)\subset \P^{\tilde n}\subset \P^{n+1}$ is of dimension $\frac{n}{2}$. If the algebraic cycle $\tilde Z$ is called X then we call $Z$ a generalized X. 
\end{defi}
By construction, it is evident that if $\tilde Z$ is inside a cubic hypersurface $\tilde X$, or equivalently if the ideal of $\tilde Z$ contains a degree $3$ polynomial then $Z$ is also inside a cubic hypersurface $X$. It does not seem to the author that $r=1,2,3,4$ produces a component of Hodge loci of the same codimension as in \cref{zemarceneiro2022}, however, it might be interesting to write down a rigorous statement.  The first case such that the algebraic cycles $\tilde Z\subset \tilde X$  produce infinite number of components of Hodge loci, is the case of two dimensional cycles inside cubic fourfolds, that is, $\dim(\tilde Z)=2, \dim(\tilde X)=4$.   Therefore, we have used algebraic cycles in the above definition for  $r=5$ and $\tilde n=\frac{n}{2}+3$.

For cubic fourfolds, Hodge loci is a union of  
codimension one irreducible subvarieties $\Has_D,\ D\equiv_6 0,2, D\geq 8$ of $\T$, see 
\cite{Hassett2000}. Here, $D$ is the discriminant of the saturated lattice generated 
by $[Z]$ and the polarization $[Z_\infty]=[\P^3 \cap X] $ in $H_4(X,\Z)$ (in \cite{Hassett2000} notation $[Z_\infty]=h^2$), where   
$Z$ is an  algebraic cycle $Z\subset X, \ X\in \Has_D$ whose homology class together $[Z_\infty]$ form a rank two lattice. The loci of cubic fourfolds containing a plane $\P^2$ is 
$\Has_8$.  It turns out that the generalized $\P^2$ is just the linear cycle $\P^{\frac{n}{2}}$ and the space of cubic $n$-folds containing a linear cycle has the smallest possible codimension. These codimesnions are listed under $L$ in \cref{12mar2017-2-corona}.
The loci of cubic fourfolds containing  a cubic ruled surface/cubic scroll is $\Has_{12}$.  The codimension of the space of cubic $n$-folds containing a generalized cubic scroll is listed   in $CS$ in \cref{12mar2017-2-corona}. Under $M$ we have listed the codimension of our Hodge loci in \cref{zemarceneiro2022}. Next comes, $\Has_{14}$ and $\Has_{20}$ for cubic $n$-folds. 
The loci $\Has_{14}$ parametrizes cubic fourfolds with a quartic scroll. For generalized quartic scroll we get codimensions under $QS$.  
The loci of cubic fourfolds with a Veronese surface is $\Has_{20}$ and for generalized Veronese we get the codiemsnions under $V$.  
One gets the impression that as $D$ increases the codimension of any possible generalization of ${\mathcal C}_D$ for  cubic hypersurfaces of dimensions $n$ 
gets near to the maximal codimension, and so, far away from the codimension in \cref{zemarceneiro2022}.

\begin{table}[!htbp]
\centering
\begin{tabular}{|c|c|c|c|c|c|c|c|c|}
\hline
{\tiny $\dim(X_0)$}    &  {\tiny $\dim(\T)$} & {\tiny range of codimensions} &L & CS &M &QS& V & Hodge numbers    \\ \hline
$n$     & $\bn{n+2}{3}$  & {\tiny $\binom{\frac{n}{2}+1}{3}$, $\binom{n+2}{{\rm min }\{3, \frac{n}{2}-2\}}$}  
        &&&&&   &    \tiny $h^{n,0}, h^{n-1,1},\cdots, h^{1,n-1}, h^{0,n}$ \\  \hline \hline 
 $4$    & $20$ &  $1,1$ &1&1&\fbox{1}& 1& 1  & \tiny $0,1,21,1,0$\\  \hline  
 $6$    & $56$ &  $4,8$  &4&6&\fbox{7}& 8 & \fbox{10} & \tiny $0,0,8,71,8,0,0$ \\ \hline  
  $8$   & $120$ & $10,45$  &10&16&\fbox{19}& 23 & 25 &  \tiny $0, 0, 0, 45, 253, 45, 0, 0, 0$ \\  \hline 
  $10$  & $220$ & $20, 220$ &20&32&\fbox{38}& 45 & 47  & \tiny $0, 0, 0, 1, 220, 925, 220, 1, 0, 0, 0$ \\  \hline
  $12$  & $364$ & $35,364$  &35&55&\fbox{65}& 75& 77  &   \tiny $0,0,0,0, 14,1001, 3432, 1001, 14,0,0,0,0$   \\ \hline 
\end{tabular}
\caption{Codimensions of the components of the Hodge/special loci for cubic hypersurfaces.} 
\label{12mar2017-2-corona}
\end{table}

\section{Artinian Gorenstein ideals attached to Hodge cycles}
\label{sec8}
In order to constuct an algebraic cycle $Z$ from its topological class we must compute its ideal $I_Z$ which might be a complicated task. However, we may aim to compute at least one element $g$ of $I_Z$ which is not in the ideal $I_X$ of the ambient space $X$. In the case of surfaces $X\subset \P^3$ this is actually almost the whole task, as we do the intersection $X\cap \P\{g=0\}$, and the only possiblity for $Z$ comes from the irreducible components of this intersection. In general this is as difficult as the original job, and a precise formulation of this has been done in \cite{Thomas2005}.  
The linear part of the Artinian-Gorenstein ideal  of a Hodge cycle  of a hypersurface  seems to be part of the defining ideal of the underlying algebraic cycle, and in this section we aim to explain this.

Let $X=\{f=0\}\subset \Pn {n+1}$ be a smooth hypersurface of degree $d\geq 3 $ and even dimension $n\geq 2$ defined over $\C$, and 
$$
\sigma:=(\frac{n}{2}+1)(d-2).
$$
\begin{defi}\rm
For every Hodge cycle $\delta\in H_n(X,\Z)$ we define its associated Artinian Gorenstein ideal  as the homogeneous ideal
$$
I(\delta)_a:=\left \{Q\in \C[x]_a \Bigg| \mathlarger{\int}_{\delta} \res \left(\frac{QP\Omega}{F^{\frac{n}{2}+1}}\right) =0,\ \ 
\forall P\in \C[x]_{\sigma-a} \right\}. 
$$
By definition $I(\delta)_{m}=\C[x]_m$ for all $m\geq \sigma+1$.
\end{defi}
Let $\plc$ be the intersection of a linear $\Pn {\frac{n}{2}+1}$ with $X$ and
$[\plc]\in H_n(X,\Z)$ be the induced element in homology (the polarization).
We have $I([Z])=\C[x]$ and for an arbitrary Hodge cycle $\delta$, $I(\delta)$ depends only on the equivalence class of $\delta \in H_n(X,\Z)/\Z[\plc]$.  
The main purpose of the present section is to investigate the following: 
\begin{conj}
 \label{5nov2022}
 Let $\delta\in H_n(X,\Z)/ \Z[\plc]$ be a non-torsion Hodge cycle such that $V_\delta$ is smooth. Assume that there is a non-zero linear polynomial $g\in (I_\delta)_1$.  Then $\delta$ is supported in  the hyperplane section  $Y:=\P\{g=0\}\cap X$. 
\end{conj}
If the Hodge conjecture is true then \cref{5nov2022}  says that the linear polynomial $g$  is in the defining ideal of an algebraic cycle $Z$ such that $\delta=[Z]$.  We have the following statement which is stronger than the converse to \cref{5nov2022}. 
Let $\delta=[Z]\in H_n(X,\Z)$ be  an algebraic cycle. Then the defining ideal of $Z$ is inside $I_\delta$.  The proof is the same as \cite[Proposition 11.3]{ho13-Roberto}.

If we take a basis $g_1,g_2,\cdots,g_k$ of $(I_\delta)_1$ and apply the above conjecture for $g=\sum_{i=1}^k t_ig_i$ with arbitrary $t_i\in\C$ then we may conclude that $\delta$ is supported in $\P\{(I_\delta)_1=0\}\cap X$. A rigorous argument for this is needed, but it does not seem to be difficult.   In particular,  $\dim_\C (I_\delta)_1\leq \frac{n}{2}+1$. For $X$ the Fermat variety this consequence is easy and it can be reduced to an elementary problem as \cite[Problem 21.3]{ho13}. 
\cref{5nov2022} is mainly inspired by the following conjecture for which we have more evidences.

  \begin{conj}
  \label{21f2023bimsa}
 If $V_\delta$ is smooth and $\dim_\C (I_\delta)_1=\frac{n}{2}+1$ then $\P\{(I_\delta)_1=0\}=\P^{\frac{n}{2}}$ is inside $X$ and  modulo $\Z[\plc]$  we have $\delta=[\P^\frac{n}{2}]$. 
 \end{conj}
 For $d\not=3,4,6$, $X_0$ the Fermat variety and without the smoothness condition this theorem is proved in \cite[Theorem 1.2]{Villaflor2022}. For $d=3,4,6$ smoothness is necessary as in \cite{DuqueVillaflor} the authors have described many non-smooth components for which the theorem is not true. 

  \begin{prop}
  If \cref{5nov2022} is true then  the hyperplane  $\P\{g=0\}$ is not transversal to $X$ and hence $Y:=\P\{g=0\}\cap  X$ is not smooth.
  \end{prop}
  \begin{proof} 
  If $Y\subset \P^n:=\P\{g=0\}$ is smooth then by Lefschetz' hyperplane section theorem $H_{n}(Y,\Z)\cong H_n(\P^{n},\Z)$  and the latter  is generated by any $\P^{\frac{n}{2}}\subset \P^n$. From another side if we take any $\P^{\frac{n}{2}+1}\subset \P^n\subset \P^{n+1}$  we have $\plc\subset Y\subset \P^n$, and $[\plc]=d[\P^{\frac{n}{2}}]$ in $H_n(\P^n,\Z)$. This implies that a $d$ multiple of the generator of $H_n(Y,\Z)$ is $[\plc]$, and so $\delta$ must be a torsion in  $H_n(X,\Z)/ \Z[\plc]$.
   \end{proof}

\section{Singular cubic hypersurfaces}
If \cref{5nov2022} is true then the Hodge cycle $\delta$ is supported in a singular cubic hypersurface of dimension $n$, and our analysis of $\delta$ reduces to the study of singularities of cubic hypersurfaces.  
Cubic hypersurfaces have many linear subspaces and  it is worth to mention the following result:
\begin{theo}[\cite{Borcea1990}]
	\label{dim k planes}
Let $X=\{f_1=f_2=\cdots=f_r=0\}\subset \P^{n+r}$ be a complete intersection of dimension $n$, where $f_1,f_2,\ldots,f_r,\ \ \deg(f_i)=d_i$ are
homogeneous polynomials in the projective coordinates of $\P^{n+r}$.
For a generic $X$,  the variety $\Omega_X(k)$ of $k$-planes inside $X$ is non-empty and smooth of pure dimension
$\delta=(k+1)(n+r-k)-\sum_{i=1}^r\binom{d_j+k}{k}$, provided $\delta\geq 0$ and $X$ is not a quadric. In the case $X$ a quadric, we require
$n\geq 2k$. Furthermore, if $\delta>0$ or if in the case $X$ a quadric, $n>2k$, then  $\Omega_X(k)$ is
connected (hence irreducible).
\end{theo}
For the case of our interest $r=1,d=3$, and one dimension below linear cycles that is $k=\frac{n}{2}-1$, we have $$
\delta=\frac{k+1}{6}\left(6(n+1-k)-(k+3)(k+2)\right)=\frac{n}{12}(\frac{n}{2}+2)\left(5-\frac{n}{2}\right).
$$
It follows that the number of $\P^4$'s in a generic cubic tenfold is finite. It turns out that such a number is  $1812646836$, see  \cite{HashimotoKadets}.    
For $n=10$ and $k=\frac{n}{2}-2=3$ we have $\delta=8$, that is, the variety of $\P^3$'s inside a generic cubic tenfold is of dimension $8$.
Next, we focus on singular cubic hypersurfaces. 
  \begin{prop}
  Any line passing through two distinct points of $\sing(X)$ is inside $X$.
  \end{prop}
  \begin{proof}
   If $p$ and $q$ are two distinct singular points of $X$ then the line passing through $p$ and $q$ intersects $X$ in more than four points (counting with multiplicity) and hence it must be inside $X$.
  \end{proof}
 \begin{prop}
 \label{23nov2022}
A singular cubic hypersurface $X\subset \P^{n+1}$ is either a cone over another cubic hypersurface of dimension $n-1$ or it is birational to $\P^n$.
 \end{prop}
 \begin{proof}
  Let $p\in X$ be any singularity of $X$. We define $\P^n_p$ to be the space of lines in $\P^{n+1}$ passing through $p$ and
  $$
  X_1:=\{l\in\P^n_p\mid l\subset X\}.
  $$
  We have the  map
  $$
  \alpha: \P^n_p\backslash X_1\to X,  \ \ \
  l\mapsto \hbox{The third intersection point of $l$ with $X$.}
  $$
  If for all point $q\in X$ the line passing through $p$ and $q$ lies in $X$ then the image of $\alpha$ is the point $p$. In this case $X$ is a cone over another cubic hypersurface of dimension $n-1$ and $p$ is the vertex of the cone.
  Let us assume that this is not the case. Then $\alpha$ is a birational map between $\P^n_p$ and $X$.
\end{proof}
  It is useful to rewrite the above proof in a coordinate system $[x_0:x_1:\cdots:x_{n+1}]$. We take the affine chart $x=(x_1,x_2,\ldots,x_{n+1})\in\C^n$ given by $x_0=1$ and assume that the singularity $p$ is at the origin $0\in\C^{n+1}$. The hypersurface $X$ is given by $f=x_0f_2-f_3$, where $f_i$' are homogenuous polynomials of degree $i$ in $x$. If $f_2=0$ then $X$ is a cone over the cubic hypersurface $\P\{f_3=0\}\subset \P^n$. Otherwise, we have the birational map
  $$
  \alpha: \P^n\dashrightarrow X,\ \ [x]\mapsto [f_3(x):xf_2(x)].
  $$
 We would like to describe $\sing(X)$ and do the desingularization of $X$. In the following we consider $\{f_i=0\},\ \ i=2,3$ as affine subvarieties of $\C^{n+1}$ and $\P\{f_i=0\},\ \ i=2,3$ as projective varieties in $\P^{n}_\infty$.
 \begin{prop}
  We  have
  \begin{eqnarray}
& &  \sing\{f_2=0\}\cap\sing\{f_3=0\}\subset \sing(X)\cap \C^{n+1}\subset \{f_2=0\}\cap\{f_3=0\} \label{24112022-1}\\
& & \sing(X)\cap \P^n_\infty=\sing\P\{f_3=0\}\cap \P\{f_2=0\}  \label{24112022-2}.
\end{eqnarray}
Moreover, any line between $0\in\C^{n+1}$ and $p\in \sing(X)\cap \C^{n+1}$ either lies in $\sing(X)$ for which $p\in\sing(f_2=0)\cap\sing(f_3=0)$ or it intersects $\sing(X)$ only at $0$ and $p$.
 \end{prop}
 \begin{proof}
 The variety $X$ is given by $x_0f_2(x)-f_3(x)=0$ and hence $\sing(X)$ is given by $x_0f_2(x)-f_3(x)=f_2=x_0\frac{\partial f_2}{\partial x_i}-\frac{\partial f_3}{\partial x_i}=0, \ i=1,2,\ldots,n+1$.
  The inclusions \eqref{24112022-1} and \eqref{24112022-2} are immediate.
 \end{proof}

\section{Computing Artinian Gorenstein ring over formal power series}

The hypersurface $X_t,\ \ t\in V_{[Z]}\backslash V_Z$ is not given explicitly, as its existence is conjectural. Therefore, it might be difficult to study its Artinian Gorenstein ring. However, as we can write the Taylor series of the periods of $X_t, t\in (\T,0)$ explicitely, see \cite[Sections 13.9, 13.10, 18.5]{ho13} we might try to study such rings over, not only over $\C$, but also over formal power series. 
In this section we explain this idea. 

In \cite[Section 19.3]{ho13}, we have taken a parameter space which is transversal to $V_Z$ at $0$ and it has the complimentary dimension. Therefore, it intersects $V_Z$ only at $0$. From now on we use $V_Z$ and $V_{[Z]}$ for this new parameter space, and hence by our construction $V_Z=\{0\}$. \cref{zemarceneiro2022} is equivalent to the follwing: The Hodge locus $V_{[Z]}$ is a smooth curve ($\dim(V_{[Z]})=1$). We note that \cref{22112022protests} is proved first for this new parameter space.  In particular, this implies that the new parameter space is also transversal to $T_0V_{[Z]}$.

For a smooth hypersurface defined over the ring $\O_{\T,0}$ of holomorphic functions in a neighborhood of $0$,  and a continuous family of cycles $\delta=\delta_t\in H_n(X_t,\Z)/\Z [\plc],\ t\in (\T,0)$, the Hodge locus $V_{\delta}$ is given by the zero locus of an ideal ${\mathcal I}(\delta)\subset \O_{\T,0}$.
\begin{defi}
Let $\sigma:=(\frac{n}{2}+1)(d-2)$. We define  the Artinian Gorenstein ideal of the Hodge locus $V_{\delta_t}$ as the homogeneous ideal
\begin{equation}
\label{22nov2022AB}
I(\delta)_a:=\left \{Q\in \O_{\T,0}[x]_a \Bigg| \mathlarger{\int}_{\delta_t} \res \left(\frac{QP\Omega}{F_t^{\frac{n}{2}+1}}\right)\in {\mathcal I}(\delta),\ \ 
\forall P\in \C[x]_{\sigma-a} \right\}. 
\end{equation}
We define the Artinian Gorenstein algebra of the Hodge locus  as $R(\delta):=\O_{\T,0}[x]/ I(\delta)$. By definition $I(\delta)_{m}=\O_{\T,0}[x]_m$ for all $m\geq \sigma+1$ and so $R(\delta)_a=0$.
\end{defi}
Note that we actually need that the integral in \eqref{22nov2022AB} vanishes identically over ${\rm Zero}({\mathcal I}(\delta))$. Since $\mathcal I(\delta)$ might not be reduced, these two definitions might not be equivalent. Since in \cref{zemarceneiro2022} we expect that $V_{[Z]}$ is smooth, these two definitions are the same. 
In a similar way we can replace $\O_{\T,0}$ with the ring $\check \O_{\T,0}$ of formal power series, and in particular, with the truncated rings $\O_{\T,0}^N:=\O_{\T,0}/m_{\T,0}^{N+1}\cong \check\O_{\T,0}/\check m_{\T,0}^{N+1}$. 
\begin{conj}
\label{26mar2023huairou}
 For all even number $n\geq 6$ the linear part $I(\delta)_1$ of $I(\delta)$ is not zero.
\end{conj}
It seems quite possible to prove this conjecture using \cite[Section 3]{voisin1988} and \cite[Theorem 3, Proposition 6]{Otwinowska2002}. In these reference the authors prove that if a Hodge locus $V_\delta$ has minimal codimension then $\dim I(\delta)=\frac{n}{2}+1$. Note that the codimension of our Hodge locus as a function in $n$ grows as the
minimal codimension for a Hodge loci, see \cref{15082023guapiacu}.   
Despite this, we want to get some evidence for \cref{26mar2023huairou}. The main goal of this section is to explain the computer code which verifies the following statement: The linear part $I^N(\delta)_1$ of $I^N(\delta)$ is not zero for $(n,N)=(6,7),(8,3)$.
\href
{https://w3.impa.br/~hossein/WikiHossein/files/Singular%20Codes/2024_09_2024_On_A_Hodge_Locus_Linear_Part_Of_Artinian_Gorenstein.txt}
{For the computer code for experimental verification of  \cref{26mar2023huairou} see the tex file of the present text in arxiv or the author's webpage.} 
{\tiny
}

We fix the canonical basis $x^I$ of the Jacobian ring $S_0:=\C[x]/\jacob(F_0)$, where $F_0:=x_0^{d}+x_1^{d}+\cdots+x_{n+1}^d$ is the Fermat polynomial. This is also the basis for  $\C[x]/\jacob(F_t)$ in a Zariski neighborhood of $0\in\T$. From this basis we take out the basis for $(S_0)_1$ and $(S_0)_{\sigma-1}$, where $\sigma=(d-2)(\frac{n}{2}+1)$. These are:
\begin{eqnarray*}
 (S_0)_1 &:&  x_0,x_1,\cdots,x_{n+1}\\
 (S_0)_{\sigma-1} &:&   x_0^{i_0}x_1^{i_1}\cdots x_{n+1}^{i_{n+1}},\ \ \sum i_j=\sigma-1,\ 0\leq i_j\leq d-2. 
\end{eqnarray*}
Let $a_1:=n+1=\#(S_0)_1$ and $b_1:= \#(S_0)_{\sigma-1}$. For a Hodge cycle  $\delta_0\in H_n(X_0,\Z)$, we define the $a_1\times b_1$ matrix in the following way:
$$
A_t:=\left[ \int_{\delta_t} \omega_{PQ}\right],\ \ P\in  (S_0)_1,\ \ Q\in (S_0)_{\sigma-1}. 
$$
For the Hodge cycle in \cref{zemarceneiro2022} we want to compute $I(\delta)_1$ which is equivalent to compute the kernel of $A_t$ modulo ${\mathcal I}(\delta)$ from the left, that is $1\times a_1$ vectors $v$ with $vA_t=0$ modulo ${\mathcal I}(\delta)$. At first step we aim to compute the rank of $A_t$.  Let $\mu$ be the rank of $A_t$ over $\O_{\T,0}/{\mathcal I}(\delta)$. This means that the determinant of all $(\mu+1)\times(\mu+1)$ minors of 
$A_t$ are in the ideal modulo ${\mathcal I}(\delta)$, but there is a $\mu\times\mu$ minor whose determinant is not in ${\mathcal I}(\delta)$. Recall that  ${\mathcal I}(\delta)$ is conjecturally reduced! These statements can be experimented by computer after truncating the entries of $A_t$.

\section{Kloosterman's work}
\label{30112023camera}
 In 2023 R. Kloostermann sent the author the preprint \cite{Kloosterman2023} in which among many other things proves that \cref{zemarceneiro2022} is not true for $n\geq 10$. It turns out that for a generic cubic hypersurface $X_0$ of dimension $n\geq 10$, we have $T_0V_{Z}=T_0V_{[Z]}$, even though for Fermat variety this is not true. Actually according to \cref{22112022protests} at Fermat variety for $n=10$ and $12$ we have the equalities of fourth order and third order neighborhoods, respectively. The equality $T_0V_{Z}=T_0V_{[Z]}$ can be easily checked by computer  using Villaflor's elegant formula in \cite{RobertoThesis}. I could have done this in 2019 when Villaflor defended his thesis. However, I was too exhausted by my computer calculations in \cite{ho13}. 
 \href{https://w3.impa.br/~hossein/WikiHossein/files/Singular%20Codes/2024_09_07_On_A_Hodge_Locus_Kloosterman.txt}
 {The computer code for this verification which uses an example of hypersurface of dimension $10$ in \cite{Kloosterman2023} can be obtained in the tex file of the present text in arxiv  or the author's webpage.}
{\tiny 
}
Recently in \cite[Corollary 1.6]{BKU} the authors have proved that the Hodge loci  corresponding to all tensor product of $H^n(X_t,\Z)$ and its dual and of positive dimension in the Griffiths period domain is algebraic for  $d=3$ and $n\geq 10$. This Hodge loci is formulated in terms of Mumford-Tate groups and it is much larger than the Hodge loci considered in the present text. Even though \cref{zemarceneiro2022} fails for $n\geq 10$, there might be infinite number of special components $V_i,\ i\in\N$ of Hodge loci in our context which might lie inside a proper algebraic subset of $V$ of $\T$ (in their terminology maximal element for inclusion). 
It might be helpful to construct such a $V$ explicitly. If we consider the variation of Hodge structures over $V$ then \cite[Corollary 1.6]{BKU} imply that the generic period domain for $V$ becomes a product of monodromy invariant factors $D_j$ and either the level of Hodge decomposition attached to one of these factors is $\leq 2$ 
or for all except a finite $i\in \N$, the Hodge locus $V_i$ has zero dimension projected to one of $D_j$. A similar discussion must be also valid for \cref{zemarceneiro2022}, that is $n=6,8$,  as the Hodge locus $V_{[Z]}$ is atypical in the sense of \cite{BKU}. For the summary of results on Hodge loci in this general framework see \cite{Klingler}.


\section{\cref{zemarceneiro2022} for tangent spaces ( By R. Villaflor)}
\label{sec4}

Let $X=\{x_0^3+x_1^3+\cdots+x_{n+1}^3=0\}\subseteq\P^{n+1}$ be the cubic Fermat variety of even dimension $n$. Let 
$$
\P^{\frac{n}{2}+3}:=\{x_{6}-\zeta_{2d}x_{7}=x_{8}-\zeta_{2d}x_{9}=\cdots=x_n-\zeta_{2d}x_{n+1}=0\},
$$
$$
\P^\frac{n}{2}:=\{x_0-\zeta_{2d}x_1=x_2-\zeta_{2d}x_3=x_4-\zeta_{2d}x_5=0\}\cap \P^{\frac{n}{2}+3},
$$
$$
\check\P^\frac{n}{2}:=\{x_0-\zeta_{2d}^{\alpha}x_1=x_2-\zeta_{2d}^{\alpha}x_3=x_4-\zeta_{2d}^{\alpha}x_5=0\}\cap \P^{\frac{n}{2}+3},
$$ 
where $\alpha\in \{3,5,7,\ldots,2d-1\}$. Then 
$$
\P^{\frac{n}{2}-3}:=\P^\frac{n}{2}\cap\check\P^\frac{n}{2}=\{x_0=x_1=x_2=x_3=x_4=x_{5}=0\}\cap\P^{\frac{n}{2}+3}.
$$
For $Z$ as in \eqref{06102022bolsolula}, let $V_{[Z]}$, $V_{[\P^\frac{n}{2}]}$ and $V_{[\check\P^\frac{n}{2}]}$ be their corresponding Hodge loci. 
\begin{prop}
	\label{pro:2301}
    We have $\dim T_0V_{[Z]}=\dim \T_0V_Z+1$.
\end{prop}
\begin{proof}
In fact, by \cite[Proposition 17.9]{ho13} we have $\dim V_Z=\dim T_0V_{[\P^\frac{n}{2}]}\cap T_0V_{[\check\P^\frac{n}{2}]}$ and so we are reduced to show that
$$
\dim \frac{T_0V_{[Z]}}{T_0V_{[\P^\frac{n}{2}]}\cap T_0V_{[\check\P^\frac{n}{2}]}}=1.
$$
By \cite[Corollaries 8.2 and 8.3]{RobertoThesis} this is equivalent to show that
$$
\dim \frac{(J^F:P_1+P_2)_3}{(J^F:P_1)_3\cap(J^F:P_2)_3}=1,
$$
where $J^F=\langle x_0^2,x_1^2,\ldots,x_{n+1}^2\rangle$ is the Jacobian ideal of $X$, $P_1:=R_1Q$, $P_2:=R_2Q$,
$$
Q:=\prod_{k\ge 6\text{ even}}{(x_k+\zeta_{6}x_{k+1})},
$$
$$
R_1:=c_1\cdot{(x_0+\zeta_{6}x_{1})}{(x_2+\zeta_{6}x_{3})}{(x_4+\zeta_{6}x_{5})},
$$
and 
$$
R_2:=c_2\cdot{(x_0+\zeta_{6}^{\alpha}x_{1})}{(x_2+\zeta_{6}^{\alpha}x_{3})}{(x_4+\zeta_{6}^{\alpha}x_{5})},
$$
for some $c_1,c_2\in\C^\times$. Let $I:=\langle x_0^2,x_1^2,x_2^2,x_3^2,x_4^2,x_5^2\rangle\subseteq\C[x_0,x_1,x_2,x_3,x_4,x_5]$. We claim that the natural inclusion
$$
(I:R_1+R_2)_3\hookrightarrow (J^F:P_1+P_2)_3
$$
induces an isomorphism of $\C$-vector spaces
\begin{equation}
\label{keyiso}
\frac{(I:R_1+R_2)_3}{(I:R_1)_3\cap(I:R_2)_3}\simeq\frac{(J^F:P_1+P_2)_3}{(J^F:P_1)_3\cap(J^F:P_2)_3}.
\end{equation}
Note first that 
$$
(J^F:Q)=\langle x_0^2,x_1^2,x_2^2,x_3^2,x_4^2,x_5^2,x_6-\zeta_{6}x_7,x_7^2,x_8-\zeta_{6}x_9,x_9^2,\ldots,x_n-\zeta_{6}x_{n+1},x_{n+1}^2\rangle
$$ 
since both are Artin Gorenstein ideals of socle in degree $\frac{n}{2}+4$ (here we use Macaulay theorem \cite[Theorem 2.1]{RobertoThesis}) and the right hand side is clearly contained in $(J^F:Q)$. In order to prove \eqref{keyiso}, let $r\in (I:R_1+R_2)_3$ such that $r\in (J^F:P_i)_3=((J^F:Q):R_i)_3$ for both $i=1,2$, then $r\cdot R_i\in (J^F:Q)\cap \C[x_0,x_1,x_2,x_3,x_4,x_5]=I$ and so $r\in (I:R_i)_3$ for each $i=1,2$. Conversely, given $q\in (J^F:P_1+P_2)_3$ write it as $q=s+t+u$, 
where $s\in \C[x_0,x_1,x_2,x_3,x_4,x_5]$, $t\in \langle x_6-\zeta_6x_7,x_8-\zeta_6x_9,\ldots,x_n-\zeta_6x_{n+1}\rangle\subseteq\C[x_0,x_1,\ldots,x_{n+1}]$ and $u\in \langle x_7,x_9,\ldots,x_{n+1}\rangle\subseteq\C[x_0,x_1,x_2,x_3,x_4,x_5]\otimes\C[x_7,x_9,x_{11},\ldots,x_{n+1}]$. Since $q\cdot(R_1+R_2)\in (J^F:Q)$, letting $x_6=x_7=\cdots=x_{n+1}=0$ it follows that $s\cdot (R_1+R_2)\in I$, i.e. $
s\in (I:R_1+R_2)$.
On the other hand is clear that $t\in (J^F:P_1)\cap (J^F:P_2)$, then in order to finish the claim it is enough to show that $u\in (J^F:P_1)\cap (J^F:P_2)$. Note that this is clearly true for all monomials appearing in the expansion of $u$ divisible by some $x_i^2$ for $i>6$ odd. Hence we may assume that
$$
u=\sum_{i>6\text{ odd}} p_i(x_0,x_1,\ldots,x_5)\cdot x_i+\sum_{j>i>6\text{ both odd}}p_{ij}(x_0,\ldots,x_5)\cdot x_ix_j
$$
$$
+\sum_{k>j>i>6\text{ all odd}}p_{ijk}(x_0,\ldots,x_5)\cdot x_ix_jx_k.
$$
Note also that
$$
(J^F:Q)\cap \C[x_0,x_1,x_2,x_3,x_4,x_5]\otimes\C[x_7,x_9,x_{11},\ldots,x_{n+1}]=
$$
$$
\langle x_0^2,x_1^2,\ldots,x_5^2,x_7^2,x_9^2,\ldots,x_{n+1}^2\rangle
$$
is a monomial ideal. From here it is clear that $u\cdot(R_1+R_2)\in (J^F:Q)$ if and only if $p_i\cdot(R_1+R_2)\in I$, $p_{ij}\cdot(R_1+R_2)\in I$ and $p_{ijk}\cdot(R_1+R_2)\in I$  for all $k>j>i>6$ odd numbers. Then $p_i\in (I:R_1+R_2)_2$, $p_{ij}\in (I:R_1+R_2)_1$ and $p_{ijk}\in (I:R_1+R_2)_0=0$. By \cite[Proposition 2.1]{RobertoThesis} we know $(I:R_1+R_2)_e=(I:R_1)_e\cap (I:R_2)_e$ for all $e\neq 3$, then $p_i\in (I:R_1)_2\cap(I:R_2)_2$ and $p_{ij}\in (I:R_1)_1\cap (I:R_2)_1$ for all $j>i>6$ both odd and so $u\in (J^F:P_1)\cap (J^F:P_2)$ as claimed. This proves \eqref{keyiso}. Finally, since $(I:R_1+R_2)$, $(I:R_1)$ and $(I:R_2)$ are all Artin Gorenstein ideals of socle in degree 3 but they are not equal, we get that $(I:R_1+R_2)_3$ is a hyperplane of $\C[x_0,\ldots,x_5]_3$ while $(I:R_1)_3\cap(I:R_2)_3$ is a codimension 2 linear subspace of $\C[x_0,\ldots,x_5]_3$, hence
$$
\dim \frac{(I:R_1+R_2)_3}{(I:R_1)_3\cap (I:R_2)_3}=1.
$$
\end{proof}  

\begin{rem}\rm
\label{01112022}
The proof of the above proposition works in general for any degree $d$ such that the intersection of both linear cycles is $m$-dimensional with $(d-2)(\frac{n}{2}-m)=d$. It is easy to see that this is only possible for $(d,m)=(3,\frac{n}{2}-3)$ and $(d,m)=(4,\frac{n}{2}-2)$. We expect a similar property as in \cref{zemarceneiro2022} for the later case, see \cite[Section 19.8]{ho13}.
\end{rem}

\chapter{Vector fields modulo primes}
\label{10082025sharr}
{\it I recall a story (...) about a man whose canoe flipped in turbulent waters. He fought and struggled to get free from the currents but he failed, drowned and died. Shortly after, his body surfaced near by. Apparently, had he surrendered to the currents, instead of fighting them, he had a good chance of making out alive, (from the   \href{https://iamronen.com/blog/2010/02/28/life-currents/}{blog  iamronen.com }). 

}
\section{Introduction}
In this chapter  we  aim to study a local-global principle for vector fields which generalizes the Grothendieck-Katz conjecture on linear differential equations. 
Vector fields from an algebraic point of view are simple, however, the foliations induced by them might show complicated dynamical behavior. The Lorenz attractor 
shows this dynamical behavior in an algebraically simple example. 
We investigate modulo primes behavior of vector fields. This study is mainly for detecting algebraic solutions of vector fields and we do not know yet whether it is related to the dynamics of solutions of vector fields.
\footnote{We use the notations in \cref{26sept2024baddahan} and for simplicity we do not reproduce them here.}

\section{Derivations or vector fields}
For simplicity we can take $\T:=\A^n_\sring=\spec(\sring[z_1,z_2,\ldots,z_n]/I)$, where $I$ is an ideal in $\sring[z_1,z_2,\ldots,z_n]$. A vector field in $\T$ is written in the form 
$$
\vf=\vf_i(z)\frac{\partial }{\partial z_1}+\vf_2(z)\frac{\partial }{\partial z_2}+\cdots+ \vf_n(z)\frac{\partial }{\partial z_n},\ \ \vf_i\in\sring[z], 
$$
where $\frac{\partial }{\partial z_i}$ is the unique vector field in $\T$ with  $\frac{\partial }{\partial z_i}(dz_j)=1$ if $i=j$ and $=0$ otherwise. Moreover, $\vf(I)\subset I$.   The $\O_\T$-module of vector fields $\Vec_{\T}$ is isomorphic to the $\O_\T$-module 
of derivations.
\rm\index{Derivation}\index{$\vf$, a vector field}
A map $\vf: \O_{\T}\to \O_{\T}$ is called a derivation
if it is $\sring$-linear and it satisfies the Leibniz rule
$$
\vf(fg)=f\vf(g)+\vf(f)g,\ \ f,g\in\O_{\T}.
$$
We denote by $\Der(\O_\T)$ the $\O_\T$-module of derivations. \index{$\Der(\O_\T)$, the sheaf of derivations in $\T$ }
\begin{prop}
 We have an isomorphism of $\O_\T$-modules
 \begin{eqnarray*}
& &  \Vec_{\T}\cong \Der(\O_\T),\\
 & & \vf\mapsto (f\mapsto \vf(df)).
 \end{eqnarray*}
 \end{prop}
\begin{proof}
This isomorphism maps the vector field $\vf$  to the corresponding derivation $\check \vf$  obtained by
$\check\vf(f)=\vf(df),\ \ f\in\O_\T$.
This equality  also defines its inverse $\check \vf\mapsto \vf$, that is, if $\check \vf\in \Der(\O_\T)$ then $\vf\in\Theta_\T$ is defined through
$\vf(\sum g_idf_i):=\sum g_i \check \vf(f_i)$. We have to show that this map is well-defined, that is, if $\sum g_idf_i=0$ then $\sum g_i \check \vf(f_i)=0$. For this, recall that $\T:=\spec(\sring[z]/I)$ is affine, where $I\subset \sring[z]$ is an ideal. In this case a derivation in $\T$ is given by $\check\vf:=\sum_{i=1}^n \vf_i(z)\frac{\partial }{\partial z_i}$ with $\check\vf(I)\subset I$. Moreover, $\Omega^1_\T$ is the quotient of $\sring[z]dz_1+\sring[z]dz_2+\cdots+\sring[z]dz_n$ with the $\sring[z]$-module generated by $dI, Idz_i, d(fg)-fdg-gdf,\ f,g\in\sring[z]$.
\end{proof}

\section{Fundamental theorems of ODE's}
In a course in ordinary differential equations one first learns the existence and uniqueness of solutions. Let $\A^1_\sring=\spec(\sring[z])$ be the affine line over $\sring$ and $\frac{\partial}{\partial z}$ be the vector field on it such that $\frac{\partial}{\partial z}(z)=1$. Let also $\sring_{\Q}:=\sring\otimes_\Z\Q$.
The next two theorems  are classical in the theory of ordinary differential equations, however, they are not usually written in an algebro-geometric context, that is why we reproduce them here.  We would like to point out the fact that in order to define the underlying holomorphic objects, we only need to be able to invert any natural number in the ring $\sring$, and hence the usage of $\sring_{\Q}$ instead of $\sring$ is justified.
\begin{theo}
\label{5jan2015}
Let $\t$ be  a smooth $\sring$-valued  point of $\T$, $\vf$ be a vector field in $\T$  with $\vf(\t)\not =0$ and $A=\A^1_{\sring}$ be the one dimensional affine scheme.
There is a unique holomorphic map
$$
\varphi: (A^\hol_{\sring_{\Q}},0) \to (\T_{\sring_{\Q}},\t)\ \ \  \varphi(0)=\t
$$
such that $\varphi$ maps the vector field $\frac{\partial }{\partial z}$ to $\vf$.
\end{theo}
\begin{proof}
Let us take an algebraic  coordinate system $z=(z_1,z_2,\ldots,z_n)$ in $(\T,t)$ and set $\vf_i:=\vf(z_i)$. We write 
$\vf_i$ as formal power series in $z$ with coefficients in $\sring$.
We denote by $\varphi_i$ the pull-back of $z_i$ by $\varphi$.    The fact that $\varphi$ maps $\frac{\partial}{\partial z}$  to $\vf$ is translated into the following ordinary differential equation:
\begin{equation}
\label{09112023anit}
\left\{
\begin{array}{lll}
\frac{\partial\varphi_1}{\partial z} &=&\vf_1(\varphi_1(z), \varphi_1(z),\cdots, \varphi_1(z))   \\
\frac{\partial\varphi_2}{\partial z} &=& \vf_2 (\varphi_1(z), \varphi_1(z),\cdots, \varphi_1(z)) \\
&\vdots& \\
\frac{\partial\varphi_n}{\partial z} &=& \vf_n(\varphi_1(z), \varphi_1(z),\cdots, \varphi_1(z))
\end{array}
\right. .
\end{equation}
From now on we use $\varphi$ for $(\varphi_1,\varphi_2,\ldots,\varphi_n)$, and so, the above differential equation can be written as
$\frac{\partial \varphi}{\partial z}=\vf(\varphi)$, where $\vf$ is identified with $(\vf_1,\vf_2,\ldots,\vf_n)$. Let $\vf=p_0+p_1+\cdots$, where $p_i$ is the homogeneous piece of degree $i$ of $\vf$. Moreover, let us write:
$$
\varphi=\sum_{i=0}^\infty \varphi_iz^i,\ \varphi_i\in\sring_{\Q}^n,\ \varphi_0:=0
$$
and substitute all these in the above differential equation.  It turns out that
$i\cdot \varphi_i$ can be written in a unique way in
terms of $\varphi_j,\ j<i$ with coefficients in $\sring$. We need to invert $i$, that is why in the statement of theorem we have to use $\sring_{\Q}$.  This guaranties the existence of a unique
formal power series $\varphi$.  Note that if $\vf(\t)=0$ then $\varphi_i=0$ for all $i\geq 1$
and so $\varphi$ is the constant map. It is not at all clear why $\varphi$ must be convergent. For this we use Picard operator associated with the differential equation \eqref{09112023anit} and the contracting map principle. For more details see \cite[\S 1.4, page 4]{IlyashenkoYakovenko}.
\end{proof}
\begin{theo}
\label{21jan2015}
Let $\t$ be  a smooth $\sring$-valued point of $\T$  and $\vf$ be a vector field in $\T$  with $\vf(\t)\not =0$.
There is a holomorphic coordinate system  $(z_1,z_2,\ldots,z_n)$ in $(\T^\hol_{\sring_\Q},\t)$
such that  $\vf=\frac{\partial }{\partial z_1}$.
\end{theo}
\begin{proof}
Let $z$ be an algebraic  coordinate system in $(\T,t)$. We are looking for  a coordinate system $F$ such that
the push forward of the vector field $\frac{\partial}{\partial z_1}$ by $F$ is $\vf$. This is equivalent to
$$
\begin{pmatrix}
 \frac{\partial F_1}{\partial z_1} & \frac{\partial F_1}{\partial z_2}&\cdots& \frac{\partial F_1}{\partial z_n}\\
 \frac{\partial F_2}{\partial z_1} & \frac{\partial F_2}{\partial z_2}&\cdots& \frac{\partial F_2}{\partial z_n}\\
 \vdots & \vdots & \cdots &\vdots\\
\frac{\partial F_n}{\partial z_1} & \frac{\partial F_n}{\partial z_2}&\cdots& \frac{\partial F_n}{\partial z_n}\\
\end{pmatrix}
\begin{pmatrix}
 1 \\ 0 \\ \vdots \\ 0
\end{pmatrix}=
\begin{pmatrix}
 \vf_1\circ F \\ \vf_2\circ F \\ \vdots \\ \vf_n\circ F
\end{pmatrix},
$$
where $F=(F_1,F_2,\ldots,F_n)$.
In a similar way as in \cref{5jan2015} we have a unique solution $F$
to the above differential equation with
$$
F(0,\tilde z)=(0,\tilde z),
$$
where $\tilde z=(z_2,\ldots,z_n)$.
We have
$$
\begin{bmatrix}
 \frac{\partial F_i}{\partial z_j}(0)
\end{bmatrix}=
\begin{bmatrix}
\vf_1(0) & 0\\
*& I_{(n-1)\times (n-1)}
\end{bmatrix}
$$
By a linear change of coordinates in $z$, we may assume that $\vf_1(0)\not =0$,
and so $F$ is the desired coordinate system.
\end{proof}
\begin{rem}\rm
\label{10112023rik}
 In the proof of \cref{21jan2015} observe that if $\vf_i$'s does not depend on $z_{i_j},\ j=1,2,\ldots$ then we can assume that $F_{i_j}=z_{i_j}$, that is, we do not need to change the coordinate $z_{i_j}$.
\end{rem}

\section{Algebraic solutions}

\begin{defi}
\label{27062024aahoo}
Let $\varphi$ be the solution in \cref{5jan2015}.
The Zariski closure of $\varphi$ is a subscheme of $\T$ given by the ideal which consists of all regular functions $f\in \O_{\T}$ such that  the pull-back of $f$ by $\varphi$ is zero:
$$
\T_\varphi:=\Zero\left(\IS_\varphi\right),\ \ \IS_\varphi:=\left\{   f\in \O_{\T} \ | \ f\circ\varphi=0 \right\}.
$$
This is an $\sring$-scheme by definition.
We say that $\varphi$ is algebraic if for all $f\in\O_{\T}$, $f\circ\varphi$ are algebraic functions in $z$.
We say that the image of $\varphi$ is algebraic if $\dim(\T_\varphi)=1$. Equivalently,
for all $f,g\in\O_{\T}$, there is a polynomial in $P(x,y)\in\sring[x,y]$ such that $P(f,g)=0$.
\end{defi}
Note that in the above definition it does not make any difference to use $P(x,y)\in\sring[x,y]$ or $P\in\C[x,y]$, see \cref{27jan2024toop}.
The $\sring$-scheme $\T_\varphi$ is tangent to the vector field $\vf$: 
For $f$ in the ideal sheaf of $\T_\varphi$, since $\frac{\partial}{\partial z}$  is mapped to $\vf$ under $\varphi$,  we have
$$
\vf(f)\circ\varphi =\frac{\partial}{\partial z} (f\circ\varphi)=0.
$$
We say that a solution $\varphi$ of a vector field $\vf$ is 
defined over $\sring$ if $
\varphi=\sum_{i=0}^\infty \varphi_iz^i,\ \varphi_i\in\sring^n$, 
that is,  we do not need to invert all integers in $\sring$.
The fact that a solution of a vector field is still  defined over $\sring$, that is we do not need to invert integers in $\sring$,  has strong consequence.
\begin{conj}
\label{11july2024autoescola}
Let $\T$ be an $\sring$-scheme, $\vf$ be a vector field on $\T$ and $\t$ be a smooth $\sring$-valued point of $\T$.
If the solution $\varphi$ of $\vf$ through $\t$ is defined over a finitely generated subring $\sring$ of $\C$ and $\vf(\t)\not=0$ then $\varphi$ is algebraic. 
\footnote{If a solution $\varphi$ of $\vf$ is defined over $\sring$ then the ideal $\IS_\Q:=\{ f\in \O_{\T^\hol_{\sring_\Q},\t} \ \Big| \ f(\varphi)=0\}$
is also defined over $\sring$, that is, if we define $\IS$ in a similar way, replacing $\sring_\Q$ with $\sring$, then  $\IS_\Q=\IS\otimes_\Z\Q$. In this way, \cref{11july2024autoescola} is a particular case of \cref{19062024kontsevich}.
}

\end{conj}
 The condition $\vf(\t)\not=0$ cannot be dropped from the above conjecture. We have an abundant literature  on holomorphic functions  of the form $y(z):=\sum_{n=0}^\infty y_nz^n\in\Z[[z]]$ which satisfies a linear differential equation $\sum_{i=0}^kP_i(z)y^{(i)}=0$  with $P_i\in\Z[z]$ and $P_k\not=0,\ P_k(0)=0$, see for instance \cite{zu02}. Writing this as a vector field in 
 $(z,y,y',\cdots, y^{k-1})$, we observe that it is meromorphic in $z=0$ which contains the point 
 $\t:=(0,y(0),y'(0),\cdots, y^{k-1}(0))$. Even if we multiply this vector field with $P_k(z)$ in order to get a polynomial expression, we observe that $\t$ is a singular point of the new vector field. 

\begin{rem}\rm
In \cref{04112025dodokhtarbimsa} we have  written a particular case of \cref{11july2024autoescola} in order to make it accessible to a general audience. We will also introduce \cref{11072024graphofscience} which in this case is equivalent to say that if there is $N\in\N$ such that $N^ny_n\in\Z$ for all $n\in\N_0$ then $y(x)$ is algebraic. 
\end{rem}

Before, introducing local-global principles for vector fields, let us first introduce the analog of \cref{19102023chinagain-2} and \cref{19102023chinagain} in this case. Recall that $\sring\subset \C$ is a finitely generated $\Z$-algebra and $\sk$ is its quotient field.
\begin{theo}
\label{12112023rik-2}
Let $\T$ be an $\sring$-scheme, $\t$ be an $\sring$-valued smooth point of $\T$ and $\vf$ be a vector field in $\T$ with $\vf(\t)\not=0$. Moreover, assume that there is a function $z\in\O_{\T}$ such that
$z(\t)=0,\ \vf(z)=1$. 
If the  solution $\varphi$ of $\vf$ through $\t$ is algebraic then there is $N\in \sring$ such that the composition
$$
\frac{N^m\vf^m}{m!}(\t): \O_{\T_\sk,\t}\to \O_{\T_{\sk},\t}\stackrel{\t}{\to}\sk, \ \forall m\in\N
$$
is defined over $\sring$, that is, it sends  $\O_{\T}$ to $\sring$. If all the solutions of $\vf$ are algebraic then there is $N\in\N$ such that
$$
\frac{N^m\vf^m}{m!}: \O_{\T}\to \O_\T,\ \forall m\in\N
$$
is well-defined, that is, its image is in $\O_\T$.
\end{theo}
\begin{proof}
The proof is almost the same as the proof of \cref{19102023chinagain}.
For the first part, the theorem is of local nature, and so, we can assume that $\T:=\spec(\ring)$ is affine and $\ring$ is generated over $\sring$ by $y_0,y_1,\ldots,y_n,\ y_0=z$. We know that  $\varphi(A^\hol,0)$ is parameterized over $\sring_\Q$, see \cref{5jan2015}, and it is  algebraic. Therefore, there are  polynomials $P_i\in\C[x,y]$ in two variables and coefficients in $\C$ such that $P_i(z\circ \varphi,y_i\circ\varphi)=0$. Using a similar argument as in \cref{27jan2024toop} we can assume that $P_i$ has coefficients in $\sring_\Q$, and after a multiplication with an integer, it has coefficients in $\sring$.
We write the Taylor series of $ \check y_i:=y_i\circ \varphi$ in the variable $\check z:=z\circ \varphi$ and by \cref{26012024margetifi?-2} we have $N\in\N$ such that 
$$
\frac{N^m\vf^m}{m!}(\t)=\frac{N^m\frac{\partial^m \check y_i}{\partial\check z^m}}{m!}(0) \in\sring.
$$
By \eqref{26nov2023zendegi-1}, we get the statement for an arbitrary element of $\O_{\T}$. For the second part instead of  \cref{5jan2015}   we use \cref{21jan2015}. We get polynomials $P_i$ in two variables with coefficients in $\O_{\T}$ and the rest of the proof is similar.
\end{proof}
\begin{rem}\rm
 The condition on the existence of $z$ can be achieved by taking any $z\in\O_{\T}$ with $\vf(z)(\t)\not=0$ and replacing $\vf$ with $\frac{\vf}{z}$. It would be interesting to formulate a similar statement without this condition.
 For this we might try to divide the ideal of the  collinear  scheme $\vf||\vf^m$ 
 with $m!$. It seems to be necessary  to find a characteristic zero version of the following identity in characteristic $p$:
$$
(f\vf)^p=f^p\vf^p+f\vf^{p-1}(f^{p-1})\vf \hbox{ modulo } p,\ f\in\O_{\T},
$$
see \cite[5.4.0]{Katz1970}.
\end{rem}
The  analog of  \cref{05112023tina} can be written immediately.
\begin{conj}
\label{11072024graphofscience}
 Let $\T$ be an $\sring$-scheme, $\t$ be a $\sring$-valued smooth point of $\T$ and $\vf$ be a vector field in $\T$ with $\vf(\t)\not=0$. 
If for some $N\in \sring$ the map
$$
\frac{N^m\vf^m}{m!}(\t): \O_{\T}\to \sring, \ \forall m\in\N
$$
is well-defined, then  the solution of  $\vf$ through $\t$ is algebraic. Moreover, if 
$$
\frac{N^m\vf^m}{m!}: \O_{\T}\to \O_\T, \ \forall m\in\N
$$
is well-defined then  all the solutions of $\vf$ are algebraic.
\end{conj}
Note that in \cref{11072024graphofscience} we do not assume the existence of $z\in\O_{\T}$ such that
$z(\t)=0,\ \vf(z)=1$. We think that the hypothesis of this conjecture is strong enough to imply the existence of
a function $z\in\O_{\T^\hol_{\sring_\Q},\t}$ algebraic over $\O_{\T}$. \cref{11072024graphofscience} implies
\cref{11july2024autoescola}, however, note that the hypothesis of 
\cref{11july2024autoescola} is weaker that the hypothesis of \cref{11072024graphofscience}.

\begin{prop}
The first and second part of \cref{12112023rik-2} imply respectively \cref{19102023chinagain-2} and \cref{19102023chinagain}.
\end{prop}
\begin{proof}
We consider $\T:=\A^1_\sring\backslash\{\Delta=0\}\ \times\  \A^{\Nn}_\sring$ with the coordinate system $(z,x)$  and the following vector field in $\T$:
$$
\vf(z)=1,\ \ \vf(x)=\Am x.
$$
We have $\vf^m(z)=0,\ \ \vf^m(x)=\Am_mx$.
\end{proof}

\begin{prop}
\cref{11072024graphofscience} implies \cref{15july2024impatech}.
\end{prop}
\begin{proof}
 Let $x_1(z)$ be an integral function and $\Z[x]$ be the associated integral $\Z$-algebra such that $x_1$ is the first entry of $x$. We have $\frac{\partial x}{\partial z}=\vf(x)$, where $\vf\in\Z[x]^N$. We consider $\vf$ as a vector field in $\A^N_\Z$, and hence, $\varphi(z):=x(z)$ is a solution of $\vf$. The entries of $\vf(x(z))$ are not identically zero, otherwise, $x$ and in particular $x_1(z)$, are constant and hence algebraic. Therefore, we can choose $z_0\in\C$ in the domain of definition of $\varphi(z)$ such that $\t:=\varphi(z_0)\not=0$. We define $\sring:=\Z[t]$, and consider the vector field $\vf$ and the $\sring$-valued point $t$ in $\A_\sring^N$. The fact that $\Z[x]$ is a $\Z$-integral algebra implies the hypothesis of \cref{11072024graphofscience}. 
\end{proof}

\section{Local-Global principle for vector fields}

For a prime number $p\in\N$, we consider the ring $\sring_p:=\sring/p\sring$ which might not be a field ($p$ might not be prime in the larger ring $\sring$) and the reduction modulo $p$, $\T_p:=\T\times_{\sring} \spec(\sring_p)$.
For a vector field $\vf$ in $\T$, we consider it as a derivation $\vf: \O_\T\to \O_\T$ and its iteration
$$
\vf^p: \O_\T\to \O_\T,\ \ p\in\N.
$$
It satisfies $\vf^p(fg)=\sum_{i=0}^p\binom{n}{i}\vf^if\vf^{p-i}g$. 
For $p$ a prime number this equality modulo $p$ is  $\vf^p(fg)=(\vf^pf)g+f(\vf^pg)$ which means that $\vf^p$ is a vector field in $\T_p$.
\begin{prop}
\label{12112023rik}
Let $\T$ be an $\sring$-scheme, $\t$ be an $\sring$-valued point of $\T$ and $\vf$ be a vector field in $\T$ with $\vf(\t)\not=0$.  If the  solution $\varphi$ of $\vf$ through $\t$ is defined over $\sring$ then for any good prime $\vf^p$ in $\T_p$ vanishes at $\t$. In other words,  the subscheme $\check\T$ of $\T_p$ defined by the ideal generated by $\vf^p\O_{\T_p}$  contains the point $\t$.
\footnote{
Inspired by \cref{12112023rik}, we have analyzed the locus $\vf^p=0$ for the Ramanujan vector field and we have got \cref{13112023tifi}, \cref{18june2024hugo}.
}
\end{prop}
\begin{proof}
We know that $\varphi$ maps $\frac{\partial}{\partial z}$ to $\vf$. This is the same as to say that
\begin{equation}
 \begin{array}{ccc}
\O_{\T}&\stackrel{\vf}{\to}& \O_{\T} \\
\downarrow && \downarrow\\
\O_{A^\hol,0} &\stackrel{\frac{\partial}{\partial z}}{\to} & \O_{A^\hol,0}
\end{array},
\end{equation}
commutes. Making modulo $p$ we have
\begin{equation}
 \begin{array}{ccc}
\O_{\T_p}&\stackrel{\vf^p}{\to}& \O_{\T_p} \\
\downarrow && \downarrow\\
\O_{A^\hol_p,0} &\stackrel{(\frac{\partial}{\partial z})^p}{\to} & \O_{A^\hol_p,0}
\end{array}.
\end{equation}
Note that we use the fact that $\varphi$ is defined over $\sring$, and so, the above commutative diagram makes sense.
It says that  $\varphi$ maps $(\frac{\partial}{\partial z})^p$ to $\vf^p$ in $\T_p$.
But $ (\frac{\partial}{\partial z})^p=\frac{\partial^p}{\partial z^p}$ is zero in $\A^1_{\sring_p}$.
This implies that $\vf^p(\O_{\T_p})$ is mapped to $0$ under $\O_{\T_p}\to \O_{A^\hol_p,0}$. This is a morphism of $\sring_p$-algebras, and hence, $\vf^p(\O_{\T_p})\subset \mk_{\T_p}$.  This means that $\vf^p\O_{\T_p}$ vanishes at $\t$.
\end{proof}

\begin{defi}
\label{22june2024piri}
Let $\vf$ and $\wf$ be two vector fields on $\T$, all defined over $\sring$. The collinear or parallel scheme  $\vf||\wf$ of $\vf$ and $\wf$ is a subscheme of $\T$ given by the ideal generated by
$$
\left|\begin{matrix}
\vf(P) & \vf(Q)\\
\wf(P) & \wf(Q)
\end{matrix}\right|,\ \ P,Q\in\O_\T.
$$
\end{defi}
\begin{prop}
 \label{12112023rik-3}
 Let $\T$ be an $\sring$-scheme, $\t$ be an $\sring$-valued point of $\T$ and $\vf$ be a vector field in $\T$ with $\vf(\t)\not=0$.  If the  image of the solution $\varphi$ of $\vf$ through $\t$ is algebraic then for all but a finite number of primes $p$, $\vf^p$ and $\vf$ are collinear at $\t$. In other words, the collinear scheme $\vf||\vf^p$ contains the point $\t$.
\end{prop}
\begin{proof}
Let $\IS$ be the ideal of $\T_\varphi$.
 If $\vf$ is tangent to  $\T_\varphi$ in $\T$ then by definition $\vf(\IS)\subset \IS$, and hence, $\vf^p(\IS)\subset \IS$. This implies that  $\vf^p$ is tangent to $(\T_\varphi)_p$ in $\T_p$. Since $\T_\varphi$ is one dimensional we get the result. 
\end{proof}
It is natural to ask whether the converse of  \cref{12112023rik-3} is true. Namely,
let $\T$ be an $\sring$-scheme, $\t$ be an $\sring$-valued point of $\T$, and $\vf$ be a vector field  in $\T$ with $\vf(\t)\not=0$.
If for all but a finite number of primes $p$, $\vf$ is collinear with $\vf^p$ at the point $\t$, then the image of the  solution $\varphi$  of $\vf$ through $\t$ is algebraic.
\footnote{Unfortunately this statement is not true, see \cref{07092024omidviolento}.}
The following conjecture has been already treated in the literature. 
\begin{conj}
\label{02112023mainconj-vv}\footnote{This is a particular case of \cref{14112023tifi}}.
Let $\T$ be an $\sring$-scheme  and $\vf\not=0$ be a vector field  in $\T$.
If for all but a finite number of primes $p$, $\vf$ is collinear with $\vf^p$, then all the solutions of
$\vf$ are algebraic.
\end{conj}

\section{$p$-curvature and vector fields}
In this section we discuss the $p$-curvature of connections in terms of vector fields and foliations.

Let $\V$ be a  smooth $\sring$-scheme, $H$ be a vector bundle on $\V$ and $\nabla: H\to \Omega_V^1\otimes_{\O_\V} H$ be a connection, all defined over $\sring$. We can regard this as
$$
\Theta_\V\to \End_{\sring}(H),\ \ \vf\mapsto \nabla_\vf. 
$$
\begin{defi}
 The $p$-curvature of $\nabla$ is the map 
 $$
 \psi: \Theta_\V\to \End_{\sring}(H),\ \ \psi(\vf):=\nabla_\vf^p-\nabla_{\vf^p},
 $$
 where we have taken reduction modulo $p$ (that is over $\sring_p:=\sring/p\sring$) of the underlying objects. 
\end{defi}
It is easy to see  that  $\psi: \Theta_\V\to \End_{\O_\V}(H)$, see \cite[Section 5]{Katz1970}, that is, the image of $\psi$ lies in $\O_\V$-linear endomorphisms of $H$. It is well-known that Gauss-Manin connections have nilpotent p-curvatures. 
 \begin{theo}
 [P. Deligne, N. Katz, \cite{Katz1970} Corollary 7.5, page 383] 
  Let $X\to \V $ be a family of smooth projective varieties over a field of characteristic $p$ and 
  $$
  m=\#\left \{(p,q) |  p+q=n,\ h^{p,q}(X_t)\not=0\right \}.
  $$
  Then the $m$-th power of the $p$-curvature 
  $$
  \psi(\vf)^m: H^n(X/V)\to H^n(X/V),\ \ 
  $$
  is identically zero for all $\vf\in \Theta$. 
\end{theo}
Let $s_1,s_2,\ldots,s_n :U\to H$ be trivializer sections of $H$ in an open set $U\subset\V$, 
that is, we have an isomorphism  
$$
U\times \A_\sring^n\to \pi^{-1}(U),\ \ (\v, x)\mapsto \sum_{i=1}^n  x_is_i(\v).
$$ 
As we are using the language of $\sring$-schemes we must say what is this in the level of functions. 
Let us consider the section $x\cdot s$, where  $x=[x_1,x_2,\ldots,x_n]$ and $s=[s_1,s_2,\ldots,s_n]^\tr$ and write 
$$
\nabla s=dx\otimes s+x\nabla s=(dx+x\gma)\otimes s, 
$$
where $\gma$ is computed via $\nabla s=\gma\otimes s$ and it is called the connection matrix. It is an $n\times n$ matrix with entries in $\Omega^1_U$.
Let us define  the matrix $\gma(\vf)_n$ through the equality: 
$$
\nabla_{\vf}^ns=\gma(\vf)_n\otimes s. 
$$
We can check easily that it is given by the recursion 
$$
\gma(\vf)_1=\gma(\vf),\ \ \gma(\vf)_{n+1}=\vf (\gma(\vf)_n)+\gma(\vf)_n\gma(\vf).
$$
The $p$-curvature can be written as
$$
\psi(\vf)=\gma(\vf)_p-\gma(\vf^p). 
$$
We would like to reinterpret $p$-curvature in terms of vector fields. 
We have the projection $\pi: H\to\V$ which is a morphism of $\sring$-schemes and gives the inclusion 
$\O_\V\subset \O_\T$ by pull-back of  functions via $\pi$.
A section $s: U\to H$  of $H$ in an open subset $U\subset\V$ is a map  $s^{*}: \O_H\to \O_U$ which is identity restricted to  $\O_U\subset \O_H$. Let $\T$ be the total space of the vector bundle $H$. This is an $\sring$-scheme in a natural way and it is just $H$ without its vector bundle structure. 
 We have a submodule $\Omega\subset\Omega^1_\T$ which can be constructed in the following way.
In the local chart $\pi^{-1}(U)$, $\Omega$ is the $\O_{\pi^{-1}(U)}$-module generated by the entries of 
$dx+x\gma$. We can easily check that $\Omega$ does not depend on the trivilizing sections $s$. If we  
choose another one $\check s$ then we  write $\check s=S s$, where $S$ is a $ n \times n $ matrix with entries in $\O_{U}$. We have $\check\gma=dSS^{-1}+S\gma S^{-1}$, $x=yS$ and 
$$
dx+x\gma=d(yS)+yS\gma=(dy)S+ydS+yS\gma=(dy+ y\check \gma)S.
$$
By local description of $\Omega$, it follows that it is integrable in the strongest format, that is $d\Omega\subset \Omega\wedge \Omega^1_\T$. Therefore, it defines a foliation $\F(\Omega)$ in $\T$. 
Let $\Theta$ be the dual of $\Omega$

\begin{theo}(\cite[Theorem 6.23]{ho2020})
\label{29ag2024vicio}
We have a canonical isomorphism $f:\Theta_\V\to\Theta$ which is given by the following identities in local charts: 

 \begin{eqnarray}\label{02102025yanqi-1}
& &f(\vf)(g)= \vf(g),\ g\in\O_\V, \\ \label{02102025yanqi-2}
& & f(\vf)(x) =x\gma(\vf),
\end{eqnarray}
\end{theo}
We have used $f$ to construct modular vector fields in \cite[Theorem 6.23]{ho2020}. 
In geometric terms, this is the lifting of vector fields from $\V$ to $\T$ such that they are tangent to solutions of $dx=-x\gma$. 

\begin{prop}
\label{29082024naotenhointernete}
The vector field  $\wf:=f(\vf)^p-f(\vf^p)$ is tangent to the fibers of $\T\to\V$, that is, $\wf(g)=0,\ g\in\O_\V$. 
Moreover, in a local chart  
$$
\wf(x)=x\psi(\vf):=x(\gma(\vf)_p-\gma(\vf^p)).
$$
\end{prop}
\begin{proof}
 This follows from \cref{29ag2024vicio}. For $g\in\O_V$ we have 
 \begin{eqnarray*}
 \wf(g) & = & f(\vf)^p(g)-f(\vf^p)(g)\\
 &\stackrel{\eqref{02102025yanqi-1}}{=}  & \vf^p(g)-\vf^p(g)=0. 
 \end{eqnarray*}
 In a local chart with $x$-coordinates as above we have:
 \begin{eqnarray*}
 \wf(x) & = & f(\vf)^p(x)-f(\vf^p)(x)\\
 &\stackrel{\eqref{02102025yanqi-2}, \eqref{02102025yanqi-1}}{=}  & f(\vf)^{p-1}(x\gma(\vf))-x\gma(\vf^p) \\ 
 &=& \cdots\\
 &=& x(\gma(\vf)_p-\gma(\vf^p)).
 \end{eqnarray*}

\end{proof}

\section{Examples of vector fields}
In this section we would like to highlight that vector fields can be rich from both dynamical and arithmetic point of view, even though the intersection of these two aspects must yet to be discovered, and for an expert in dynamical system the whole theory developed in the present text might be a useless tool (specially for experts in \cref{22032024lorenz} and \cref{11072024fred}).  We have already seen arithmetic aspects of the Ramanujan vector field in \cref{Ramanujan27062024}. 

\begin{exam}\rm
\label{22032024lorenz}
Let us consider the ordinary differential equation in $(x,y,z)$ and depending on three parameters $\sigma,\rho$ and $\beta$:
\begin{equation}
\label{27062024lorenz}
{\rm R}:
 \left \{ \begin{array}{l}
\dot x= \sigma(y-x)\\
\dot y=x(\rho-z)-y \\
\dot z=xy-\beta z
\end{array} \right.,
\end{equation}
which is called the Lorenz system and and it has a set of chaotic solutions known under the name Lorenz attractor. A simple web search will lead us to its fascinating acpects and we avoid them here. A simple plot of its trajectories in $\R^3$, for instance for $\sigma=10,\ \rho=28,\ \beta=8/3$,  shows its chaotic behaviour, see for instance the \href{https://en.wikipedia.org/wiki/Lorenz_system}{wikipedia webpage for 'Lorenze system'},  and tells us that its solutions are far from being algebraic. Let us take $\sigma,\rho,\beta\in\sring$ and consider \eqref{27062024lorenz} as a vector field $\vf$ in $\A^3_\sring$ which we call it Lorenz vector field. 
We have verified that the Lorenz  vector field, with parameters as above, 
is not $p$-closed  for all $p\not=3$ and $p\leq 200$, that is, $\vf$ and $\vf^p$ are not collinear at a generic point. 
{\tiny
\begin{verbatim}
     LIB "foliation.lib";
     ring r=0, (x,y,z),dp;
     number sigma=10; number rho=28; number beta=8/3;
     list vf=sigma*(y-x), x*(rho-z)-y, x*y-beta*z;
     int ub=200;
     BadPrV(vf, ub);
\end{verbatim}
}

\end{exam}

\begin{exam}\rm
The following vector field in $\A^4_\sring:=\spec(\sring[x_2,x_3,y_2,y_3])$ with $\sring=\Z[\frac{1}{6}]$
\begin{align}
\label{14March2018}
\vf:=&\left(2x_2-6x_3+\frac{1}{6}(x_2-y_2)x_2\right)\frac{\partial}{\partial x_2}+
 \left(3x_3-\frac{1}{3}x_2^2+\frac{1}{4}(x_2-y_2)x_3\right)\frac{\partial}{\partial x_3}\\
 -&\left(2y_2-6y_3+\frac{1}{6}(y_2-x_2)y_2\right)\frac{\partial}{\partial y_2}-
 \left(3y_3-\frac{1}{3}y_2^2+\frac{1}{4}(y_2-x_2)y_3\right)\frac{\partial}{\partial y_3},\nonumber
\end{align}
has been introduced in \cite{ho2018}. It has two invariant hypersurfaces $\{27x_3^2-x_2^3=0\}$ and $\{27y_3^2-y_2^3=0\}$, and it is proved that outside these hypersurfaces 
it has an enumerable set of algebraic leaves birational to modular curves $X_0(N)$ and these are the only algebraic solutions outside the mentioned hypersurfaces. 
We have verfied this  vector field is not $p$-closed  for all $p\not=2,3$ and $p\leq 50$,
{\tiny
\begin{verbatim}
LIB "foliation.lib";
ring r=0, (x_2,x_3,y_2,y_3),dp;
list vf=2*x_2-6*x_3+(1/6)*(x_2-y_2)*x_2,   3*x_3-(1/3)*x_2^2+(1/4)*(x_2-y_2)*x_3,
             -(2*y_2-6*y_3+(1/6)*(y_2-x_2)*y_2),-(3*y_3-(1/3)*y_2^2+(1/4)*(y_2-x_2)*y_3);
int ub=50; BadPrV(vf, ub);
\end{verbatim}
}

{\tiny
}
\end{exam}

\begin{exam}\rm
\label{11072024fred}
 The differential equation
\begin{equation}
\label{21juin07}
\left \{ \begin{array}{l}
\dot x=2y+ \frac{x^2}{2}\\
\dot y=3x^2-3+\frac{9}{10}y
\end{array} \right.
\end{equation}
has a limit cycle which is depicted in \cref{limitcycle}. This is related to Hilbert's 16th problem, see \cite{il02, Movasati-Uribe}. It turns out that only modulo $p=3$, $\vf$ and $\vf^p$ are collinear and in other cases different from $2,3$, these vector fields are never collinear. 
{\tiny 
\begin{verbatim}
LIB "foliation.lib";
ring r=0, (x, y),dp;
list vf=2*y+1/2*x^2, 3*x^2-3+9/10*y;
int ub=100; BadPrV(vf, ub);
\end{verbatim}
}
\end{exam}
\begin{figure}[h!]
\begin{center}
\includegraphics[width=0.6\textwidth]{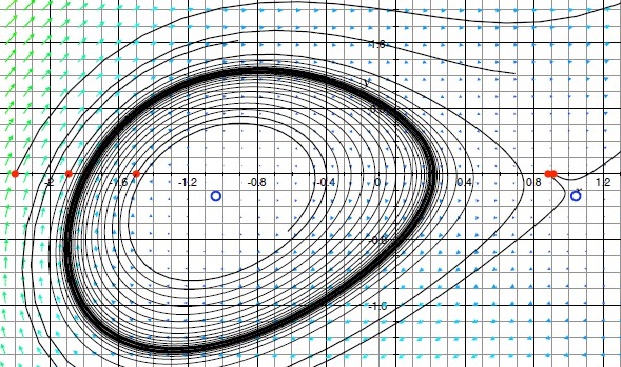}
\caption{A limit cycle crossing $(x,y)\sim (-1.79,0)$}
\label{limitcycle}
\end{center}
\end{figure}

\section{Vector fields on elliptic curves}
In this section we recall \cref{15072024urbanik} and \cref{04092025artan}  and we prove that the first theorem is basically equivalent to \cref{02112023mainconj-vv} for certain class of vector fields. This is mainly inspired by \cite[Corollary 2.5]{bost01}.

Let $E: y^2=p(x),\ \deg(p)=3,\ p\in\sring[x],\ \sring:=\Z[\frac{1}{6}]$ be an elliptic curve in the Weierstrass format. The vector field  
 \begin{equation}
\label{23062024consumismo}
\vf:=y\frac{\partial }{\partial x}+ \frac{1}{2}p'(x)\frac{\partial }{\partial y}
\end{equation}
in $\A^2_\sring$ is tangent to $E$, and hence, it induces a vector field on $E$. Over complex numbers, the Weierstrass uniformization  $z\mapsto[\wp(z):\wp'(z):1]$ maps $\frac{\partial }{\partial z}$ to $\vf$.
The tangent bundle of $E$ is trivial and this is the trivializer section of this bundle. Therefore, we can talk about $\vf$ for elliptic curves, not necessarily written in the Weierstrass format. For any regular differential $1$-form $\alpha$ on $E$ there is a unique vector field in  $E$ such that $\alpha(\vf)=1$. This is called the dual of $\alpha$. The tangent space of $E$ is a trivial line bundle and $\vf$ is the trivializer of such a bundle.   

\begin{theo}
\label{04092025artan}
Let $E$ be an elliptic curve  over a perfect field of characteristic $p$ and $\vf$ be a vector field on $E$ with $\alpha(\vf)=1$. We have 
$$
\vf^p= {\rm HW}(E,\alpha)^\frac{1}{p} \vf,
$$
where  ${\rm HW}(E,\alpha)$ is the Hasse-Witt invariant. Moreover,  if $E,\alpha,\vf$ are defined over $\Ff_p$ then 
$$
 {\rm HW}(E,\alpha)\equiv_p p+1-\#E(\Ff_p).
 $$
\end{theo}
See \cite[12.4.1.2]{KatzMazur1985}, \cite[3.2.1]{Katz1973}. Knowing \cref{04092025artan} we can formulate the first density statement for vector fields.
\begin{theo}
Let $E$ be an elliptic curve  over $\Z$ and $\vf$ be a vector field on $E$ without any zeros. The density of primes such that $\vf^p=0$ on $E/\Ff_p$  is $\frac{1}{2}$ for CM elliptic curves and it is zero otherwise.   
\end{theo}
By  \cref{04092025artan} we know that we are looking for the density of primes such that $E/\Ff_p$ is supersingular. For the proof see \cite{Elkies1987} and the references therein.

\begin{theo}
\label{15072024urbanik}
 Let $E_i,\ i=1,2$ be two elliptic curves over rational numbers. The following are equivalent:
 \begin{enumerate}
  \item 
  $E_1$ is isogeneous to $E_2$ over $\Q$.
  \item
  For all except a finite number of primes $\#E_1(\Ff_p)=\#E_2(\Ff_p)$. 
 \end{enumerate}
\end{theo}
This has been proved in  \cite[IV 2.3]{Serre1968} and \cite[\S 5, Korollar 2]{Faltings1983}. Let $E_1,E_2$ be two elliptic curves as above and let $\vf_i,\ i=1,2$ be the vector fields in $E_i$ as in \eqref{23062024consumismo} all defined over $\sring:=\Z[\frac{1}{N}]$, where $N$ is the product of bad primes of $E_i$'s after choosing models over $\sring$. We consider their parallel extensions (see \cite[Section 2.5]{ho2020}) in $\T:=E_1\times E_2$ and use the same letter $\vf_i$.  

\begin{prop}
Consider vector fields in $\T$  of the form  $\vf:=r_1\vf_1+r_2\vf_2$ with non-zero $r_1,r_2\in\sring$.
\cref{02112023mainconj-vv} for $\vf$ implies \cref{15072024urbanik}. 
\end{prop}
We remark that \cref{02112023mainconj-vv} for such $\vf$'s  follows from the main theorem of \cite{bost01}.
 \begin{proof}
Using Hasse bound $|p+1-\#E(\Ff_p)|\leq 2\sqrt{p}$,  we observe that for big primes 
 $\#E_1(\Ff_p)=\#E_2(\Ff_p)$ and $\#E_1(\Ff_p)\equiv_p\#E_2(\Ff_p)$ are equivalent. Therefore, we can replace the second item of \cref{15072024urbanik} with $\#E_1(\Ff_p)\equiv_p\#E_2(\Ff_p)$.

Let $f: E_1\to E_2$ be an isogeny defined over $\Q$. We  take a model of $f$ over $\sring:=\Z[\frac{1}{N}]$, possibly with $N$ bigger. Let $r\in\sring,\ r\not=0$  be such that $f^*\alpha_2=r\alpha_1$. 
We have a vector field   $\vf:=\vf_1+r\vf_2$ which is tangent to the graph of isogeny and all its translations (for this one may take coordinate system $x$ and $y$ for $E_1$ and $E_2$ such that in these coordinates  $f$ is given $x\mapsto \tilde f(x)$). This implies that  $\vf^p$ is collinear with $\vf$. From another side  we have also
 \begin{eqnarray*}
\vf^p &=& \sum_{i=0}^p\binom{p}{i}(\vf_1)^i(r\vf_2)^{p-i}\\
&=& \vf_1^p+r^p\vf_2^p\\
&=& a_1\vf_1+a_2r\vf_2,\ \ a_i:=p+1-\#E_i(\Ff_p).   
 \end{eqnarray*}
In the first equality we have used $[\vf_1,\vf_2]=0$ and in the third equality we have used \cref{04092025artan} and Fermat's little theorem.
Since $\vf^p$ is collinear with $\vf$ and $r\not\equiv_p 0$ (we might exclude few primes which violates this property), we get $a_1\equiv_p a_2$.  
 
 Let us assume that $a_1\equiv_pa_2$ for all but a finite number of primes. We take arbitrary non-zero $r_1,r_2\in\sring$ and our hypothesis implies that $\vf^p=a_1\vf$. We apply 
 \cref{02112023mainconj-vv} and we get an algebraic curve in $\T$ passing through $(O,O)$ whose projection to $E_1$ and $E_2$ is not constant as it is the unique curve tangent to $\vf$ and passing through $(O,O)$. It is also defined over $\Q$ because of the uniqueness above and the fact that $(O,O)$ is defined over $\Q$. Note that we can compose the image and domain of an isogeny by multiplication  by a natural number map and this produces $\vf$'s with different $r_1,r_2$. 
\end{proof}


\section{Poincar\'e linearization theorem}
\cref{30102025guerranorio} can be considered as a kind of Poincar\'e linearization theorem in positive characteristic, see for instance \cite{casa87, IlyashenkoYakovenko}. This theorem is mainly stated in the complex context  for singularities of holomorphic foliations, and \cref{30102025guerranorio} motivated us to write a generalization of this  theorem in positive characteristic. It turns out that in characteristic $p$ the hypothesis of \cref{30102025shanghai} are quite different from the same theorem in characteristic zero.   

Let $\Ff_p[[x]]$ be the ring of formal power series in $x=(x_1,x_2,\ldots,x_n)$ and with coefficients in $\Ff_p$. We define $(\A^{n, \for}_{\Ff_p},0)$ to be the spectrum of this ring. This is supposed to mimic a small formal neighborhood of $0$ in $\C^n$. We are going to talk about vector fields in  $(\A^{n, \for}_{\Ff_p},0)$. These are given by $\vf:=\sum_{i=1}^n\vf_i(x)\frac{\partial }{\partial x_i}$ with $\vf_i\in\Ff_p[[x]]$. The linear part of $\vf$ is $\wf:=\sum_{i=1}^n\wf_{i}(x)\frac{\partial }{\partial x_i}$, where $\wf_{i}$ is the linear part of $\vf_i$. We write this as 
\begin{equation}
\label{31102025mauricio}
\wf=
\begin{bmatrix}
\frac{\partial}{\partial x_1} &
\frac{\partial}{\partial x_2}&
\ldots &
\frac{\partial}{\partial x_n}
\end{bmatrix}
A
\begin{bmatrix} 
x_1 \\x_2\\ \vdots\\ x_n
\end{bmatrix}, \hbox{ equivalently } \wf(x)=Ax,
\end{equation}
where $A$ is an $n\times n$ matrix with entries in $\Ff_p$ and $x=[x_1,x_2,\ldots,x_n]^{\tr}$. 
\begin{theo}[Poincar\'e linearization theorem in characteristic $p$]
\label{30102025shanghai}
Let $\vf$ be a vector field in $(\A^{n, \for}_{\Ff_p},0)$ with the linear part $\wf$ such that $A^{p-1}=I_{n\times n}$. There is an isomorphism $f: (\A^{n, \for}_{\Ff_p},0)\to (\A^{n, \for}_{\Ff_p},0)$ tangent to the identity and sending $\vf$ to $\wf$ if and only if $\vf^p=\vf$. 
\end{theo} 
\begin{proof}
Let us write $f$ in coordinates $x$, that is,  $f^*x_i=f_i(x)=x_i+\cdots\in\Ff_p[[x]]$. The fact that $f$ maps the vector field $\vf$ to $\wf$ is translated into 
$$
\begin{pmatrix}
 \frac{\partial f_1}{\partial x_1} & \frac{\partial f_1}{\partial x_2}&\cdots& \frac{\partial f_1}{\partial x_n}\\
 \frac{\partial f_2}{\partial x_1} & \frac{\partial f_2}{\partial x_2}&\cdots& \frac{\partial f_2}{\partial x_n}\\
 \vdots & \vdots & \ddots &\vdots\\
\frac{\partial f_n}{\partial x_1} & \frac{\partial f_n}{\partial x_2}&\cdots& \frac{\partial f_n}{\partial x_n}\\
\end{pmatrix}
\begin{pmatrix}
 \vf_1 \\ \vf_2 \\ \vdots \\ \vf_n
\end{pmatrix}=
\begin{pmatrix}
 \wf_1(f_1,f_2,\ldots,f_n) \\  \wf_2(f_1,f_2,\ldots,f_n)\\ \vdots \\ \wf_n(f_1,f_2,\ldots,f_n)
\end{pmatrix}.
$$
For $\wf$ linear we write this as $\vf f=Af$, where $f=[f_1,f_2,\cdots,f_n]^{\tr}$.

The direction $\Leftarrow$ is easy: If there is an $f$ such that $\vf(f)=Af$ then $\vf^{p}f=A^{p}f=A^{p-1}\vf(f)=\vf(f)$. Since $f$ is tangent to the identity, that is, $f_i=x_i+\cdots$, we have $\vf^{p}x=\vf x$, and hence, $\vf^p=\vf$. 

For the the direction $\Leftarrow$ we set:
$$
f:=-(\vf^{p-1}x+A\vf^{p-2}x+\cdots+A^{p-2}\vf x).
$$
It follows from $\vf^px=\vf x$ and $A^{p-1}=I_{n\times n}$ that $\vf(f)=Af$. 
The linear part of $f$ is $-(p-1)A^{p-1}=I_{n\times n}$. 
\end{proof}


\bibliography{biblio.bib}

@article {ho06-3,
    AUTHOR = {Movasati, H. },
     TITLE = {On elliptic modular foliations},
   JOURNAL = {Indag. Math. (N.S.)},
  FJOURNAL = {Koninklijke Nederlandse Akademie van Wetenschappen.
              Indagationes Mathematicae. New Series},
      NOTE ={},
    VOLUME = {19},
      YEAR = {2008},
    NUMBER = {2},
     PAGES = {263--286},
      ISSN = {0019-3577},
   MRCLASS = {11G05 (14Gxx)},
  MRNUMBER = {MR2489330},
       URL={http://arxiv.org/abs/0806.3926},
       URL={},
}

@Book{ho13,
 Author = {Movasati, H.},
 Title = {{A course in Hodge theory. With emphasis on multiple integrals}},
 ISBN = {978-1-57146-400-2/pbk},
 Pages = {382},
 Year = {2021},
 Publisher = {Somerville, MA: International Press},
 Language = {English},
 MSC2010 = {14-01 14C30}
}

@article {ho16,
    AUTHOR = {Movasati, H.  and Reiter, S.},
     TITLE = {Painlev\'e {VI} equations with algebraic solutions and families of curves},
   JOURNAL = {Journal of Experimental Mathematics},
    VOLUME = {19},
      YEAR = {2010},
    NUMBER = {2},
     PAGES = {161-173},
       URL={http://arxiv.org/abs/0806.1213},
       URL={},


}

@article {ho17,
    AUTHOR = {Movasati, H. and  Reiter, S.},
     TITLE = {Heun equations coming from geometry},
   JOURNAL = {Bull. Braz. Math. Soc},
    VOLUME = {43(3)},
      YEAR = {2012},
    NUMBER = {},
     PAGES = {423-442},
       URL={http://arxiv.org/abs/0902.0760},
       URL={},
}

@article{HosseinMurad,
    AUTHOR = {Alim, M. and Movasati, H. and Scheidegger, E. and S.-T. Yau },
     TITLE = {Gauss-{M}anin connection in disguise: {C}alabi-{Y}au threefolds},
   JOURNAL = {Comm. Math. Phys.},
  FJOURNAL = {},
    VOLUME = {334},
      YEAR = {2016},
    NUMBER = {3},
     PAGES = {889-914},
      ISSN = {},
     CODEN = {},
   MRCLASS = {},
  MRNUMBER = {},
MRREVIEWER = {},
       DOI = {},
       URL = {},
}

@article{GMCD-NL,
    AUTHOR = {Movasati, H. },
     TITLE = {Gauss-{M}anin connection in disguise: {N}oether-{L}efschetz and {H}odge loci},
   JOURNAL = {Asian Journal of Mathematics},
  FJOURNAL = {},
    VOLUME = {21},
      YEAR = {2017},
    NUMBER = {3},
     PAGES = {463-482},
      ISSN = {},
     CODEN = {},
   MRCLASS = {},
  MRNUMBER = {},
MRREVIEWER = {},
       DOI = {},
       URL = {},
}

@article{GMCD-MQCY3,
    AUTHOR = {Movasati, H. },
     TITLE = {Gauss-{M}anin connection in disguise: {C}alabi-{Y}au modular forms},
   JOURNAL = {Surveys of Modern Mathematics, Int. Press, Boston.},
  FJOURNAL = {},
    VOLUME = {},
      YEAR = {2017},
    NUMBER = {},
     PAGES = {},
      ISSN = {},
     CODEN = {},
   MRCLASS = {},
  MRNUMBER = {},
MRREVIEWER = {},
       DOI = {},
       URL = {},
}

@article {RobertoThesis,
    AUTHOR = {Villaflor Loyola, Roberto},
     TITLE = {Periods of complete intersection algebraic cycles},
   JOURNAL = {Manuscripta Math.},
  FJOURNAL = {Manuscripta Mathematica},
    VOLUME = {167},
      YEAR = {2022},
    NUMBER = {3-4},
     PAGES = {765--792},
      ISSN = {0025-2611},
   MRCLASS = {32G20 (13H10 14C25 14D07)},
  MRNUMBER = {4385390},
MRREVIEWER = {Fumio Hazama},
       DOI = {10.1007/s00229-021-01290-x},
       URL = {https://doi.org/10.1007/s00229-021-01290-x},
}

@article {ho2018,
 Author = {Hossein {Movasati}},
 Title = {On elliptic modular foliations. {II}},
 FJournal = {Moscow Mathematical Journal},
 Journal = {Mosc. Math. J.},
 ISSN = {1609-3321},
 Volume = {22},
 Number = {1},
 Pages = {103--120},
 Year = {2022},
 Language = {English},
 URL = {www.mathjournals.org/mmj/2022-022-001/2022-022-001-004.html},
 zbMATH = {7551762}
}

@ARTICLE{EmreHossein2018,
 Author = {Movasati, Hossein and Sert{\"o}z, Emre Can},
 Title = {On reconstructing subvarieties from their periods},
 FJournal = {Rendiconti del Circolo Matem{\`a}tico di Palermo. Serie II},
 Journal = {Rend. Circ. Mat. Palermo (2)},
 ISSN = {0009-725X},
 Volume = {70},
 Number = {3},
 Pages = {1441--1457},
 Year = {2021},
 Language = {English},
 DOI = {10.1007/s12215-020-00562-x},
 zbMATH = {7424490},
 Zbl = {1478.32050}
}

@article {ho2019,
 Author = {Movasati, Hossein},
 Title = {Special components of {Noether}-{Lefschetz} loci},
 FJournal = {Rendiconti del Circolo Matem{\`a}tico di Palermo. Serie II},
 Journal = {Rend. Circ. Mat. Palermo (2)},
 ISSN = {0009-725X},
 Volume = {70},
 Number = {2},
 Pages = {861--874},
 Year = {2021},
 Language = {English},
 DOI = {10.1007/s12215-020-00523-4},
 zbMATH = {7383960},
 Zbl = {1480.14005}
}

@Book{Movasati-Uribe,
 Author = {Marco {Uribe} and Hossein {Movasati}},
 Title = {{Limit cycles, Abelian integral and Hilbert's sixteenth problem. Paper from the 31st Brazilian mathematics colloquium -- 31\(^{\text o}\) Col\'oquio Brasileiro de Matem\'atica, IMPA, Rio de Janeiro, Brazil, July 30 -- August 5, 2017}},
 FJournal = {{Publica\c{c}\~oes Matem\'aticas do IMPA}},
 Journal = {{Publ. Mat. IMPA}},
 ISBN = {978-85-244-0437-5/pbk},
 Pages = {106},
 Year = {2017},
 Publisher = {Rio de Janeiro: Instituto Nacional de Matem\'atica Pura e Aplicada (IMPA)},
 Language = {English},
 MSC2010 = {37-02 37G15 14D05 34C07 34C05 34C08},
 Zbl = {1406.37005}
}

@Book{ho13-Roberto,
 Author = {Hossein {Movasati} and Roberto {Villaflor}},
 Title = {{A course in Hodge theory: Periods of algebraic cycles. 33\(^{\text o}\) Col\'oquio Brasileiro de Matem\'atica, IMPA, Rio de Janeiro, Brazil, 2021}},
 FJournal = {{Publica\c{c}\~oes Matem\'aticas do IMPA}},
 Journal = {{Publ. Mat. IMPA}},
 ISBN = {},
 Pages = {210},
 Year = {2021},
 Publisher = {Rio de Janeiro: Instituto Nacional de Matem\'atica Pura e Aplicada (IMPA)},
 Language = {English},
 MSC2010 = {37-02 37G15 14D05 34C07 34C05 34C08},
 Zbl = {1406.37005}
}

@Book{ho2020,
 Author = {Movasati, H.},
 Title = {Modular and automorphic forms \& beyond},
 FSeries = {Monographs in Number Theory},
 Series = {Monogr. Number Theory},
 ISSN = {1793-8341},
 Volume = {9},
 ISBN = {978-981-12-3867-3; 981-12-3869-3},
 Year = {2022},
 Publisher = {Singapore: World Scientific},
 Language = {English},
 DOI = {10.1142/12325},
 zbMATH = {7423960}
}

@article {Headaches,
    AUTHOR = {Movasati, H.},
     TITLE = {{H}eadaches in {H}odge theory},
   JOURNAL = {Available at Author's webpage},
      YEAR = {20xx},
    NUMBER = {},
     PAGES = {},
       URL={https://w3.impa.br/~hossein/myarticles/Headaches.pdf},
}

@article {il02,
    AUTHOR = {Ilyashenko, Yu.},
     TITLE = {Centennial history of {H}ilbert's 16th problem},
   JOURNAL = {Bull. Amer. Math. Soc. (N.S.)},
  FJOURNAL = {American Mathematical Society. Bulletin. New Series},
    VOLUME = {39},
      YEAR = {2002},
    NUMBER = {3},
     PAGES = {301--354 (electronic)},
      ISSN = {0273-0979},
     CODEN = {BAMOAD},
   MRCLASS = {34-02 (34C07 37C10 37F75)},
  MRNUMBER = {2003c:34001},
MRREVIEWER = {Lubomir Gavrilov},
}

@preamble{
   "\def\cprime{$'$} "
}

@article {cadeka,
    AUTHOR = {Cattani, Eduardo H. and Deligne, Pierre and Kaplan, Aroldo G.},
     TITLE = {On the locus of {H}odge classes},
     JOURNAL = {J. Amer. Math. Soc.},
  FJOURNAL = {Journal of the American Mathematical Society},
    VOLUME = {8},
      YEAR = {1995},
    NUMBER = {2},
     PAGES = {483--506},
      ISSN = {},
     CODEN = {},
   MRCLASS = {},
  MRNUMBER = {},
MRREVIEWER = {},
}

@book {casa87,
    AUTHOR = {Camacho, C{\'e}sar and Sad, Paulo},
     TITLE = {Pontos singulares de equa\c c\~oes diferenciais anal\'\i
              ticas},
    SERIES = {16$\sp {\rm o}$ Col\'oquio Brasileiro de Matem\'atica. [16th
              Brazilian Mathematics Colloquium]},
 PUBLISHER = {Instituto de Matem\'atica Pura e Aplicada (IMPA), Rio de
              Janeiro},
      YEAR = {1987},
     PAGES = {iv+132},
      ISBN = {85-244-0029-3},
   MRCLASS = {58F14 (32C40 32L99 34A20 34C05 58F18 58F21)},
  MRNUMBER = {MR953780 (90a:58126)},
MRREVIEWER = {J. S. Joel},
}

@incollection {de71-0,
    AUTHOR = {Deligne, Pierre},
     TITLE = {Th\'eorie de {H}odge. {I}},
 BOOKTITLE = {Actes du {C}ongr\`es {I}nternational des {M}ath\'ematiciens
              ({N}ice, 1970), {T}ome 1},
     PAGES = {425--430},
 PUBLISHER = {Gauthier-Villars, Paris},
      YEAR = {1971},
   MRCLASS = {14C30 (14F15)},
  MRNUMBER = {0441965},
MRREVIEWER = {J. H. M. Steenbrink},
}

@book {dmos,
    AUTHOR = {Deligne, Pierre and Milne, James S. and Ogus, Arthur and Shih,
              Kuang-yen},
     TITLE = {Hodge cycles, motives, and {S}himura varieties},
    SERIES = {Lecture Notes in Mathematics},
    VOLUME = {900},
      NOTE = {Philosophical Studies Series in Philosophy, 20},
 PUBLISHER = {Springer-Verlag},
   ADDRESS = {Berlin},
      YEAR = {1982},
     PAGES = {ii+414},
      ISBN = {3-540-11174-3},
   MRCLASS = {14Kxx (10D25 12A67 14A20 14F30 14K22)},
  MRNUMBER = {84m:14046},
}

@incollection {ked08,
    AUTHOR = {Kedlaya, K.  S.},
     TITLE = {{$p$}-adic cohomology: from theory to practice},
 BOOKTITLE = {{$p$}-adic geometry},
    SERIES = {Univ. Lecture Ser.},
    VOLUME = {45},
     PAGES = {175--203},
 PUBLISHER = {Amer. Math. Soc.},
   ADDRESS = {Providence, RI},
      YEAR = {2008},
   MRCLASS = {14F30 (14F40)},
  MRNUMBER = {MR2482348},
MRREVIEWER = {Michel Gros},
}

@article {gr69,
    AUTHOR = {Griffiths, Phillip A.},
     TITLE = {On the periods of certain rational integrals. {I}, {II}},
   JOURNAL = {Ann. of Math. (2) 90 (1969), 460-495; ibid. (2)},
    VOLUME = {90},
      YEAR = {1969},
     PAGES = {496--541},
   MRCLASS = {14.01},
  MRNUMBER = {41 \#5357},
MRREVIEWER = {F. Gherardelli},
}

@incollection {ka73,
    AUTHOR = {Katz, N. M.},
     TITLE = {{$p$}-adic properties of modular schemes and modular forms},
 BOOKTITLE = {Modular functions of one variable, III (Proc. Internat. Summer
              School, Univ. Antwerp, Antwerp, 1972)},
     PAGES = {69--190. Lecture Notes in Mathematics, Vol. 350},
 PUBLISHER = {Springer},
   ADDRESS = {Berlin},
      YEAR = {1973},
   MRCLASS = {10D15 (14D20)},
  MRNUMBER = {MR0447119 (56 \#5434)},
MRREVIEWER = {V. V. Sokurov},
}

@article {st77,
    AUTHOR = {Steenbrink, Joseph},
     TITLE = {Intersection form for quasi-homogeneous singularities},
   JOURNAL = {Compositio Math.},
    VOLUME = {34},
      YEAR = {1977},
    NUMBER = {2},
     PAGES = {211--223},
   MRCLASS = {14B05 (32C40)},
  MRNUMBER = {MR0453735 (56 \#11995)},
MRREVIEWER = {Helmut Hamm},
}

@article {zu02,
    AUTHOR = {Zudilin, V. V.},
     TITLE = {On the integrality of power expansions related to
              hypergeometric series},
   JOURNAL = {Mat. Zametki},
  FJOURNAL = {Rossi\u\i skaya Akademiya Nauk. Matematicheskie Zametki},
    VOLUME = {71},
      YEAR = {2002},
    NUMBER = {5},
     PAGES = {662--676},
      ISSN = {0025-567X},
   MRCLASS = {11J91},
  MRNUMBER = {1936191 (2003h:11080)},
MRREVIEWER = {Eugene M. Matveev},
       DOI = {10.1023/A:1015827602930},
       URL = {http://dx.doi.org/10.1023/A:1015827602930},
}

@book{ dw94,
	AUTHOR = "Bernard Dwork and Giovanni Gerotto and Francis J. Sullivan",
	TITLE = "An introduction to {$G$}-functions",
	SERIES = "Annals of Mathematics Studies",
	VOLUME = "133",
	PUBLISHER = "Princeton University Press",
	ADDRESS = "Princeton, NJ",
	YEAR = "1994",
	PAGES = "xxii+323",
	ISBN = "0-691-03681-0",
	MRCLASS = "12H25 (11S80 12J25 13F25 14G20 34A99)",
	MRNUMBER = "1274045 (96c:12009)",
	MRREVIEWER = "Adolfo Quir{\'o}s"
}

@article {bost01,
    AUTHOR = {Bost, J.-B.},
     TITLE = {Algebraic leaves of algebraic foliations over number fields},
   JOURNAL = {Publ. Math. Inst. Hautes \'Etudes Sci.},
  FJOURNAL = {Publications Math\'ematiques. Institut de Hautes \'Etudes
              Scientifiques},
    NUMBER = {93},
      YEAR = {2001},
     PAGES = {161--221},
      ISSN = {0073-8301},
   MRCLASS = {14G40 (11G35 37F75)},
  MRNUMBER = {1863738 (2002h:14037)},
MRREVIEWER = {Carlo Gasbarri},
       DOI = {10.1007/s10240-001-8191-3},
       URL = {http://dx.doi.org/10.1007/s10240-001-8191-3},
}

@article {CGGH1983,
    AUTHOR = {Carlson, James and Green, Mark and Griffiths, Phillip and
              Harris, Joe},
     TITLE = {Infinitesimal variations of {H}odge structure. {I, II,III}},
   JOURNAL = {Compositio Math.},
  FJOURNAL = {Compositio Mathematica},
    VOLUME = {50},
      YEAR = {1983},
    NUMBER = {2-3},
     PAGES = {109--205},
      ISSN = {0010-437X},
     CODEN = {CMPMAF},
   MRCLASS = {32G13 (14C30 32J25)},
  MRNUMBER = {720288 (86e:32026a)},
MRREVIEWER = {Sampei Usui},
       URL = {http://www.numdam.org/item?id=CM_1983__50_2-3_109_0},
}

@article {voisin89,
    AUTHOR = {Voisin, Claire},
     TITLE = {Composantes de petite codimension du lieu de
              {N}oether-{L}efschetz},
   JOURNAL = {Comment. Math. Helv.},
  FJOURNAL = {Commentarii Mathematici Helvetici},
    VOLUME = {64},
      YEAR = {1989},
    NUMBER = {4},
     PAGES = {515--526},
      ISSN = {0010-2571},
     CODEN = {COMHAX},
   MRCLASS = {14J10 (14J05)},
  MRNUMBER = {1022994 (91c:14041)},
MRREVIEWER = {Tie Luo},
       DOI = {10.1007/BF02564692},
       URL = {http://dx.doi.org/10.1007/BF02564692},
}

@article {voisin90,
    AUTHOR = {Voisin, Claire},
     TITLE = {Sur le lieu de {N}oether-{L}efschetz en degr\'es {$6$} et
              {$7$}},
   JOURNAL = {Compositio Math.},
  FJOURNAL = {Compositio Mathematica},
    VOLUME = {75},
      YEAR = {1990},
    NUMBER = {1},
     PAGES = {47--68},
      ISSN = {0010-437X},
     CODEN = {CMPMAF},
   MRCLASS = {14J10 (14J25)},
  MRNUMBER = {1059955 (91i:14028)},
MRREVIEWER = {L. B{\u{a}}descu},
       URL = {http://www.numdam.org/item?id=CM_1990__75_1_47_0},
}

@article {mclean2005,
    AUTHOR = {Maclean, Catriona},
     TITLE = {A second-order invariant of the {N}oether-{L}efschetz locus
              and two applications},
   JOURNAL = {Asian J. Math.},
  FJOURNAL = {The Asian Journal of Mathematics},
    VOLUME = {9},
      YEAR = {2005},
    NUMBER = {3},
     PAGES = {373--399},
      ISSN = {1093-6106},
   MRCLASS = {14C30 (14D07 14N15)},
  MRNUMBER = {2214958 (2007b:14015)},
MRREVIEWER = {Ciro Ciliberto},
       DOI = {10.4310/AJM.2005.v9.n3.a5},
       URL = {http://dx.doi.org/10.4310/AJM.2005.v9.n3.a5},
}

@article {voisin1988,
    AUTHOR = {Voisin, Claire},
     TITLE = {Une pr\'ecision concernant le th\'eor\`eme de {N}oether},
   JOURNAL = {Math. Ann.},
  FJOURNAL = {},
    VOLUME = {280},
      YEAR = {1988},
    NUMBER = {4},
     PAGES = {605-611},
      ISSN = {},
     CODEN = {},
   MRCLASS = {},
  MRNUMBER = {},
MRREVIEWER = {}
}

@article {ho14,
    AUTHOR = {Movasati, H. },
     TITLE = {Quasi-modular forms attached to elliptic curves, {I}},
   JOURNAL = {Ann. Math. Blaise Pascal},
  FJOURNAL = {Annales Math. Blaise Pascal},
    VOLUME = {19},
      YEAR = {2012},
    NUMBER = {2},
     PAGES = {307--377},
      ISSN = {1259-1734},
   MRCLASS = {11F99 (11G05 14H52 33C05)},
  MRNUMBER = {3025138},
MRREVIEWER = {{\'A}lvaro Lozano-Robledo},
       URL = {http://ambp.cedram.org/item?id=AMBP_2012__19_2_307_0},
}

@book {IlyashenkoYakovenko,
    AUTHOR = {Ilyashenko, Y. and Yakovenko, S.},
     TITLE = {Lectures on analytic differential equations},
    SERIES = {Graduate Studies in Mathematics},
    VOLUME = {86},
 PUBLISHER = {American Mathematical Society, Providence, RI},
      YEAR = {2008},
     PAGES = {xiv+625},
      ISBN = {978-0-8218-3667-5},
   MRCLASS = {34-02 (32S65 34-01 34C07 34M25 34M50 37C10 37F75)},
  MRNUMBER = {2363178 (2009b:34001)},
MRREVIEWER = {Christiane Rousseau},
}

@Article{Shioda1979,
    Author = {Tetsuji {Shioda}},
    Title = {{The Hodge conjecture and the Tate conjecture for Fermat varieties.}},
    FJournal = {{Proceedings of the Japan Academy. Series A}},
    Journal = {{Proc. Japan Acad., Ser. A}},
    ISSN = {0386-2194},
    Volume = {55},
    Pages = {111--114},
    Year = {1979},
    Publisher = {Japan Academy, Ueno Park, Tokyo},
    Language = {English},
    DOI = {10.3792/pjaa.55.111},
    MSC2010 = {14G99 14K15 14J25 14C30},
    Zbl = {0444.14017}
}

@InCollection{Deligne-HodgeConjecture,
    Author = {P. {Deligne}},
    Title = {{The Hodge conjecture.}},
    BookTitle = {{The millennium prize problems}},
    ISBN = {0-8218-3679-X/hbk},
    Pages = {45--53},
    Year = {2006},
    Publisher = {Providence, RI: American Mathematical Society (AMS); Cambridge, MA: Clay Mathematics Institute},
    Language = {English},
    MSC2010 = {14-01 14C30},
    Zbl = {1194.14001}
}

@book {MTH,
    AUTHOR = {John J O'Connor  and Edmund F Robertson},
     TITLE = {MacTutor History of Mathematics archive},
    SERIES = {},
    VOLUME = {},
 PUBLISHER = {http://www-groups.dcs.st-and.ac.uk/$\sim$history/},
   ADDRESS = {},
      YEAR = {2016},
     PAGES = {},
      ISBN = {},
   MRCLASS = {},
  MRNUMBER = {},
MRREVIEWER = {},
}

@article {Otwinowska2003,
    AUTHOR = {Otwinowska, Ania},
     TITLE = {Composantes de petite codimension du lieu de
              {N}oether-{L}efschetz: un argument asymptotique en faveur de
              la conjecture de {H}odge pour les hypersurfaces},
   JOURNAL = {J. Algebraic Geom.},
  FJOURNAL = {Journal of Algebraic Geometry},
    VOLUME = {12},
      YEAR = {2003},
    NUMBER = {2},
     PAGES = {307--320},
      ISSN = {1056-3911},
   MRCLASS = {14C30 (14J70)},
  MRNUMBER = {1949646},
MRREVIEWER = {Giuseppe Lombardo},
       DOI = {10.1090/S1056-3911-02-00349-1},
       URL = {http://dx.doi.org/10.1090/S1056-3911-02-00349-1},
}

@incollection {Dan-2014,
    AUTHOR = {Dan, Ananyo},
     TITLE = {Noether-{L}efschetz locus and a special case of the
              variational {H}odge conjecture: using elementary techniques},
 BOOKTITLE = {Analytic and algebraic geometry},
     PAGES = {107--115},
 PUBLISHER = {Hindustan Book Agency, New Delhi},
      YEAR = {2017},
   MRCLASS = {14C30 (14D07)},
  MRNUMBER = {3728128},
MRREVIEWER = {Fumio Hazama},
}

@Article{Hassett2000,
    Author = {Brendan {Hassett}},
    Title = {{Special cubic fourfolds.}},
    FJournal = {{Compositio Mathematica}},
    Journal = {{Compos. Math.}},
    ISSN = {0010-437X; 1570-5846/e},
    Volume = {120},
    Number = {1},
    Pages = {1--23},
    Year = {2000},
    Publisher = {Cambridge University Press, Cambridge; London Mathematical Society, London},
    Language = {English},
    DOI = {10.1023/A:1001706324425},
    MSC2010 = {14J35 14J28 14J45 14J10},
    Zbl = {0956.14031}
}

@Article{Weil1977,

    Title = {Abelian varieties and the {H}odge ring },
    FJournal = {},
    Journal = {Andr\'e Weil: Collected papers III},
    ISSN = {},
    Volume = {},
    Pages = {421-429},
    Year = {1977},
    Publisher = {Springer, Berlin/Heidelberg},
    Language = {},
    DOI = { },
    MSC2010 = {},
    Zbl = {}
}

@Article{Dan2017,
 Author = {Ananyo {Dan}},
 Title = {{On generically non-reduced components of Hilbert schemes of smooth curves.}},
 FJournal = {{Mathematische Nachrichten}},
 Journal = {{Math. Nachr.}},
 ISSN = {0025-584X; 1522-2616/e},
 Volume = {290},
 Number = {17-18},
 Pages = {2800--2814},
 Year = {2017},
 Publisher = {Wiley (Wiley-VCH), Weinheim},
 Language = {English},
 MSC2010 = {14C05 32J25 14D07},
 Zbl = {1387.14027}
}

@Article{rem,
    Author = {Reinhold {Remmert}},
    Title = {{Sur les espaces analytiques holomorphiquement s\'eparables et holomorphiquement convexes.}},
    FJournal = {{Comptes Rendus Hebdomadaires des S\'eances de l'Acad\'emie des Sciences, Paris}},
    Journal = {{C. R. Acad. Sci., Paris}},
    ISSN = {0001-4036},
    Volume = {243},
    Pages = {118--121},
    Year = {1956},
    Publisher = {Gauthier-Villars, Paris},
    Language = {French},
    Zbl = {0070.30401}
}

@Book{IrelandRosen,
 Author = {Kenneth {Ireland} and Michael {Rosen}},
 Title = {{A classical introduction to modern number theory.}},
 FJournal = {{Graduate Texts in Mathematics}},
 Journal = {{Grad. Texts Math.}},
 ISSN = {0072-5285},
 Volume = {84},
 ISBN = {0-387-97329-X},
 Pages = {xiv + 389},
 Year = {1990},
 Publisher = {New York etc.: Springer-Verlag},
 Language = {English},
 MSC2010 = {11-01 11Axx 11Gxx 11Nxx 11Rxx 11Txx 11Dxx 11G05 11G40 11D41 14H52},
 Zbl = {0712.11001}
}

@Book{Weil1946,
 Author = {Andr\'e {Weil}},
 Title = {{Foundations of algebraic geometry. Revised and enlarged edition.}},
 FJournal = {{Colloquium Publications. American Mathematical Society}},
 Journal = {{Colloq. Publ., Am. Math. Soc.}},
 ISSN = {0065-9258},
 Volume = {29},
 Year = {1962},
 Publisher = {American Mathematical Society (AMS), Providence, RI},
 Language = {English},
 MSC2010 = {14-01 14-02},
 Zbl = {0168.18701}
}

@ARTICLE{DuqueVillaflor,
       author = {{Duque Franco}, Jorge and {Villaflor Loyola}, Roberto},
        title = "{On fake linear cycles inside Fermat varieties}",
      journal = {arXiv e-prints},
         year = 2021,
        month = dec,
          eid = {arXiv:2112.14818},
        pages = {arXiv:2112.14818},
archivePrefix = {arXiv},
       eprint = {2112.14818},
 primaryClass = {math.AG},
       adsurl = {https://ui.adsabs.harvard.edu/abs/2021arXiv211214818D}
}

@article {Villaflor2022,
    AUTHOR = {Villaflor Loyola, R.},
     TITLE = {Small codimension components of the {H}odge locus containing
              the {F}ermat variety},
   JOURNAL = {Commun. Contemp. Math.},
  FJOURNAL = {Communications in Contemporary Mathematics},
    VOLUME = {24},
      YEAR = {2022},
    NUMBER = {7},
     PAGES = {Paper No. 2150053, 25},
      ISSN = {0219-1997},
   MRCLASS = {14C25 (13H10 14C30)},
  MRNUMBER = {4476311},
       DOI = {10.1142/S021919972150053X},
       URL = {https://doi.org/10.1142/S021919972150053X},
}

@article {Borcea1990,
    AUTHOR = {Borcea, Ciprian},
     TITLE = {Deforming varieties of {$k$}-planes of projective complete
              intersections},
   JOURNAL = {Pacific J. Math.},
  FJOURNAL = {Pacific Journal of Mathematics},
    VOLUME = {143},
      YEAR = {1990},
    NUMBER = {1},
     PAGES = {25--36},
      ISSN = {0030-8730},
   MRCLASS = {14J40 (14D99 14M10)},
  MRNUMBER = {1047398},
MRREVIEWER = {Fabio Bardelli},
       URL = {http://projecteuclid.org/euclid.pjm/1102646199},
}

@article {Thomas2005,
    AUTHOR = {Thomas, R. P.},
     TITLE = {Nodes and the {H}odge conjecture},
   JOURNAL = {J. Algebraic Geom.},
  FJOURNAL = {Journal of Algebraic Geometry},
    VOLUME = {14},
      YEAR = {2005},
    NUMBER = {1},
     PAGES = {177--185},
      ISSN = {1056-3911},
   MRCLASS = {14C30},
  MRNUMBER = {2092131},
MRREVIEWER = {Giuseppe Lombardo},
       DOI = {10.1090/S1056-3911-04-00378-9},
       URL = {https://doi.org/10.1090/S1056-3911-04-00378-9},
}

@article{HashimotoKadets, title={38406501359372282063949 and All That: Monodromy of Fano Problems}, volume={2022}, ISSN={1073-7928, 1687-0247}, url={https://academic.oup.com/imrn/article/2022/5/3349/5962044}, DOI={10.1093/imrn/rnaa275}, number={5}, journal={International Mathematics Research Notices}, author={Hashimoto, Sachi and Kadets, Borys}, year={2022}, month={Feb}, pages={3349–3370}, language={en} }

@article {Otwinowska2002,
    AUTHOR = {Otwinowska, Ania},
     TITLE = {Sur la fonction de {H}ilbert des alg\`ebres gradu\'{e}es de
              dimension 0},
   JOURNAL = {J. Reine Angew. Math.},
  FJOURNAL = {Journal f\"{u}r die Reine und Angewandte Mathematik. [Crelle's
              Journal]},
    VOLUME = {545},
      YEAR = {2002},
     PAGES = {97--119},
      ISSN = {0075-4102},
   MRCLASS = {13D40 (13C40)},
  MRNUMBER = {1896099},
MRREVIEWER = {Ralf Fr\"{o}berg},
       DOI = {10.1515/crll.2002.038},
       URL = {https://doi.org/10.1515/crll.2002.038},
}

@misc{BKU,
      title={On the distribution of the {H}odge locus}, 
      author={Gregorio Baldi and Bruno Klingler and Emmanuel Ullmo},
      year={2022},
      eprint={2107.08838},
      archivePrefix={arXiv},
      primaryClass={math.AG}
}

@misc{klingler,
      title={Hodge theory, between algebraicity and transcendence}, 
      author={Bruno Klingler},
      year={2021},
      eprint={2112.13814},
      archivePrefix={arXiv},
      primaryClass={math.AG}
}

@article {Pereira2002,
    AUTHOR = {Pereira, Jorge Vit\'{o}rio},
     TITLE = {Invariant hypersurfaces for positive characteristic vector
              fields},
   JOURNAL = {J. Pure Appl. Algebra},
  FJOURNAL = {Journal of Pure and Applied Algebra},
    VOLUME = {171},
      YEAR = {2002},
    NUMBER = {2-3},
     PAGES = {295--301},
}

@article {Katz1972,
    AUTHOR = {Katz, Nicholas M.},
     TITLE = {Algebraic solutions of differential equations ({$p$}-curvature
              and the {H}odge filtration)},
   JOURNAL = {Invent. Math.},
  FJOURNAL = {Inventiones Mathematicae},
    VOLUME = {18},
      YEAR = {1972},
     PAGES = {1--118},
      ISSN = {0020-9910,1432-1297},
   MRCLASS = {14D05 (14C30 14F30 34A20)},
  MRNUMBER = {337959},
MRREVIEWER = {T.\ Oda},
       DOI = {10.1007/BF01389714},
       URL = {https://doi.org/10.1007/BF01389714},
}

@book {Andre1989,
    AUTHOR = {Andr\'{e}, Yves},
     TITLE = {{$G$}-functions and geometry},
    SERIES = {Aspects of Mathematics},
    VOLUME = {E13},
 PUBLISHER = {Friedr. Vieweg \& Sohn, Braunschweig},
      YEAR = {1989},
     PAGES = {xii+229},
      ISBN = {3-528-06317-3},
   MRCLASS = {11J82 (11-02 11G10 11J87 12H25)},
  MRNUMBER = {990016},
MRREVIEWER = {C.\ L.\ Stewart},
       DOI = {10.1007/978-3-663-14108-2},
       URL = {https://doi.org/10.1007/978-3-663-14108-2},
}

@book {Esnault2023,
    AUTHOR = {Esnault, H\'{e}l\`ene},
     TITLE = {Local {S}ystems in {A}lgebraic-{A}rithmetic {G}eometry},
    SERIES = {Lecture Notes in Mathematics},
    VOLUME = {2337},
 PUBLISHER = {Springer, Cham},
      YEAR = {2023},
     PAGES = {vii+94},
      ISBN = {978-3-031-40839-7; 9783031408403},
   MRCLASS = {99-06},
  MRNUMBER = {4651437},
       DOI = {10.1007/978-3-031-40840-3},
       URL = {https://doi.org/10.1007/978-3-031-40840-3},
}

@Misc{Seshadri1960,
 Author = {Seshadri, Conjeeveram Srirangachari},
 Title = {The {Cartier} operation. {Applications}},
 Year = {1960},
 Language = {French},
 HowPublished = {Vari{\'e}t{\'e}s de {Picard}. {S{\'e}m}. {C}. {Chevalley} 3 (1958/59), {No}. 6, 26 p. (1960).},
 zbMATH = {3196713},
 Zbl = {0121.38001}
}

@article{Kloosterman2023,
 author = {Kloosterman, Remke},
 title = {On a conjecture on {Hodge} loci of linear combinations of linear subvarieties},
 fjournal = {Rendiconti del Circolo Matem{\`a}tico di Palermo. Serie II},
 journal = {Rend. Circ. Mat. Palermo (2)},
 issn = {0009-725X},
 volume = {74},
 number = {6},
 pages = {31},
 note = {Id/No 187},
 year = {2025},
 language = {English},
 doi = {10.1007/s12215-025-01307-4},
 zbMATH = {8100823}
}

@article {Katz1970,
    AUTHOR = {Katz, Nicholas M.},
     TITLE = {Nilpotent connections and the monodromy theorem:
              {A}pplications of a result of {T}urrittin},
   JOURNAL = {Inst. Hautes \'{E}tudes Sci. Publ. Math.},
  FJOURNAL = {Institut des Hautes \'{E}tudes Scientifiques. Publications
              Math\'{e}matiques},
    NUMBER = {39},
      YEAR = {1970},
     PAGES = {175--232},

}

@Article{Landau1904,
 Author = {Landau, E.},
 Title = {Eine {Anwendung} des {Eisensteinschen} {Satzes} auf die {Theorie} der {Gau{{\ss}}schen} {Differentialgleichung}.},
 FJournal = {Journal f{\"u}r die Reine und Angewandte Mathematik},
 Journal = {J. Reine Angew. Math.},
 ISSN = {0075-4102},
 Volume = {127},
 Pages = {92--102},
 Year = {1904},
 Language = {German},
 DOI = {10.1515/crll.1904.127.92},
 zbMATH = {2653716},
 JFM = {35.0463.01}
}

@book {Edwards1990,
    AUTHOR = {Edwards, Harold M.},
     TITLE = {Divisor theory},
 PUBLISHER = {Birkh\"{a}user Boston, Inc., Boston, MA},
      YEAR = {1990},
     PAGES = {xiv+166},
      ISBN = {0-8176-3448-7},
   MRCLASS = {11Rxx (11-01 13F05 14H05)},
  MRNUMBER = {1200892},
       DOI = {10.1007/978-0-8176-4977-7},
       URL = {https://doi.org/10.1007/978-0-8176-4977-7},
}

@book {Katz1996,
    AUTHOR = {Katz, Nicholas M.},
     TITLE = {Rigid local systems},
    SERIES = {Annals of Mathematics Studies},
    VOLUME = {139},
 PUBLISHER = {Princeton University Press, Princeton, NJ},
      YEAR = {1996},
     PAGES = {viii+223},
      ISBN = {0-691-01118-4},
   MRCLASS = {14F20 (14F05)},
  MRNUMBER = {1366651},
MRREVIEWER = {Abdellah\ Mokrane},
       DOI = {10.1515/9781400882595},
       URL = {https://doi.org/10.1515/9781400882595},
}

@article {Beukers2004,
    AUTHOR = {Beukers, Frits and van der Waall, Alexa},
     TITLE = {Lam\'{e} equations with algebraic solutions},
   JOURNAL = {J. Differential Equations},
  FJOURNAL = {Journal of Differential Equations},
    VOLUME = {197},
      YEAR = {2004},
    NUMBER = {1},
     PAGES = {1--25},
      ISSN = {0022-0396,1090-2732},
   MRCLASS = {34M35 (12H20 34M15)},
  MRNUMBER = {2030146},
       DOI = {10.1016/j.jde.2003.10.017},
       URL = {https://doi.org/10.1016/j.jde.2003.10.017},
}

@incollection {ChudChud1985-1,
    AUTHOR = {Chudnovsky, D. V. and Chudnovsky, G. V.},
     TITLE = {Applications of {P}ad\'{e} approximations to the
              {G}rothendieck conjecture on linear differential equations},
 BOOKTITLE = {Number theory ({N}ew {Y}ork, 1983--84)},
    SERIES = {Lecture Notes in Math.},
    VOLUME = {1135},
     PAGES = {52--100},
 PUBLISHER = {Springer, Berlin},
      YEAR = {1985},
}

@article {MatzatPut2003,
    AUTHOR = {Matzat, B. Heinrich and van der Put, Marius},
     TITLE = {Iterative differential equations and the {A}bhyankar
              conjecture},
   JOURNAL = {J. Reine Angew. Math.},
  FJOURNAL = {Journal f\"{u}r die Reine und Angewandte Mathematik. [Crelle's
              Journal]},
    VOLUME = {557},
      YEAR = {2003},
     PAGES = {1--52},
      ISSN = {0075-4102,1435-5345},
   MRCLASS = {12H05},
  MRNUMBER = {1978401},
MRREVIEWER = {K.\ Kiyek},
       DOI = {10.1515/crll.2003.032},
       URL = {https://doi.org/10.1515/crll.2003.032},
}

@article {KisinEsnault2018,
    AUTHOR = {Esnault, H\'{e}l\`ene and Kisin, Mark},
     TITLE = {{$D$}-modules and finite monodromy},
   JOURNAL = {Selecta Math. (N.S.)},
  FJOURNAL = {Selecta Mathematica. New Series},
    VOLUME = {24},
      YEAR = {2018},
    NUMBER = {1},
     PAGES = {145--155},
      ISSN = {1022-1824,1420-9020},
   MRCLASS = {14F10 (14D07 14G15 14G25)},
  MRNUMBER = {3769728},
MRREVIEWER = {Stefan\ Schr\"{o}er},
       DOI = {10.1007/s00029-016-0294-2},
       URL = {https://doi.org/10.1007/s00029-016-0294-2},
}

@article {BombieriSperber1982,
    AUTHOR = {Bombieri, E. and Sperber, S.},
     TITLE = {On the {$p$}-adic analyticity of solutions of linear
              differential equations},
   JOURNAL = {Illinois J. Math.},
  FJOURNAL = {Illinois Journal of Mathematics},
    VOLUME = {26},
      YEAR = {1982},
    NUMBER = {1},
     PAGES = {10--18},
      ISSN = {0019-2082},
   MRCLASS = {12H25 (14G20 30G05 34A20)},
  MRNUMBER = {638550},
MRREVIEWER = {Michael\ F.\ Singer},
}

@incollection {Andre2004,
    AUTHOR = {Andr\'{e}, Yves},
     TITLE = {Sur la conjecture des {$p$}-courbures de {G}rothendieck-{K}atz
              et un probl\`eme de {D}work},
 BOOKTITLE = {Geometric aspects of {D}work theory. {V}ol. {I}, {II}},
     PAGES = {55--112},
 PUBLISHER = {Walter de Gruyter, Berlin},
      YEAR = {2004},
      ISBN = {3-11-017478-2},
   MRCLASS = {12H25 (11G99 14C99)},
  MRNUMBER = {2023288},
MRREVIEWER = {Alain\ Salinier},
}

@article {Katz1982,
    AUTHOR = {Katz, Nicholas M.},
     TITLE = {A conjecture in the arithmetic theory of differential
              equations},
   JOURNAL = {Bull. Soc. Math. France},
  FJOURNAL = {Bulletin de la Soci\'{e}t\'{e} Math\'{e}matique de France},
    VOLUME = {110},
      YEAR = {1982},
    NUMBER = {2},
     PAGES = {203--239},
      ISSN = {0037-9484},
   MRCLASS = {14D05 (12H99 14D10 14F30)},
  MRNUMBER = {667751},
MRREVIEWER = {V.\ V.\ Shokurov},
       URL = {http://www.numdam.org/item?id=BSMF_1982__110__203_0},
}

@article {Esnault2020,
    AUTHOR = {Esnault, H\'{e}l\`ene and Groechenig, Michael},
     TITLE = {Rigid connections and {$F$}-isocrystals},
   JOURNAL = {Acta Math.},
  FJOURNAL = {Acta Mathematica},
    VOLUME = {225},
      YEAR = {2020},
    NUMBER = {1},
     PAGES = {103--158},
      ISSN = {0001-5962,1871-2509},
   MRCLASS = {14F30 (14H60)},
  MRNUMBER = {4176065},
MRREVIEWER = {Daniel\ Robert\ Gulotta},
       DOI = {10.4310/ACTA.2020.v225.n1.a2},
       URL = {https://doi.org/10.4310/ACTA.2020.v225.n1.a2},
}

@article {Zudilin2001,
    AUTHOR = {Zudilin, V. V.},
     TITLE = {Cancellation of factorials},
   JOURNAL = {Mat. Sb.},
  FJOURNAL = {Matematicheski\u{\i} Sbornik},
    VOLUME = {192},
      YEAR = {2001},
    NUMBER = {8},
     PAGES = {95--122},
      ISSN = {0368-8666,2305-2783},
   MRCLASS = {11J91},
  MRNUMBER = {1862246},
MRREVIEWER = {Eugene\ M.\ Matveev},
       DOI = {10.1070/SM2001v192n08ABEH000588},
       URL = {https://doi.org/10.1070/SM2001v192n08ABEH000588},
}

@article{Kontsevich2023,
    AUTHOR = {Kontsevich, Maxim},
     TITLE = {Two applications of Grothendieck’s Algebraicity Conjecture},
   JOURNAL = {\href{https://indico.math.cnrs.fr/category/603/attachments/3705/5500/talk_algebraic.pdf}{Slides of a talk}},
  FJOURNAL = {},
    VOLUME = {},
      YEAR = {},
    NUMBER = {},
     PAGES = {},
      ISSN = {},
     CODEN = {},
   MRCLASS = {},
  MRNUMBER = {},
MRREVIEWER = {},
       DOI = {},
       URL = {},
}

@incollection {Serre1973,
    AUTHOR = {Serre, Jean-Pierre},
     TITLE = {Formes modulaires et fonctions z\^{e}ta {$p$}-adiques},
 BOOKTITLE = {Modular functions of one variable, {III} ({P}roc. {I}nternat.
              {S}ummer {S}chool, {U}niv. {A}ntwerp, {A}ntwerp, 1972)},
    SERIES = {Lecture Notes in Math.},
    VOLUME = {Vol. 350},
     PAGES = {191--268},
 PUBLISHER = {Springer, Berlin-New York},
      YEAR = {1973},
   MRCLASS = {10D05},
  MRNUMBER = {404145},
MRREVIEWER = {K.-B.\ Gundlach},
}

@ARTICLE{Delaygue-Rivoal,
       author = {{Delaygue}, {\'E}. and {Rivoal}, T.},
        title = "{On Abel's problem and Gauss congruences}",
      journal = {arXiv e-prints},
         year = 2022,
        month = sep,
          eid = {arXiv:2209.03301},
        pages = {arXiv:2209.03301},
          doi = {10.48550/arXiv.2209.03301},
archivePrefix = {arXiv},
       eprint = {2209.03301},
 primaryClass = {math.NT},
       adsurl = {https://ui.adsabs.harvard.edu/abs/2022arXiv220903301D}
}

@Misc{Swinnerton1973,
 Author = {Swinnerton-Dyer, H. P. F.},
 Title = {On l-adic representations and congruences for coefficients of modular forms},
 Year = {1973},
 Language = {English},
 HowPublished = {Modular {Functions} of one {Variable} {III}, {Proc}. internat. {Summer} {School}, {Univ}. {Antwerp} 1972, {Lect}. {Notes} {Math}. 350, 1-55 (1973).},
 zbMATH = {3420855},
 Zbl = {0267.10032}
}

@Book{BorevichShafarevich,
 Author = {Borevich, Z. I. and Shafarevich, I. R.},
 Title = {Number theory. {Translated} by {Newcomb} {Greenleaf}},
 FSeries = {Pure and Applied Mathematics (Academic Press)},
 Series = {Pure Appl. Math., Academic Press},
 ISSN = {0079-8169},
 Volume = {20},
 Year = {1966},
 Publisher = {New York {and} London: Academic Press},
 Language = {English},
 zbMATH = {3233758},
 Zbl = {0145.04902}
}

@InCollection{AchterHowe,
 Author = {Achter, Jeffrey D. and Howe, Everett W.},
 Title = {Hasse-{Witt} and {Cartier}-{Manin} matrices: a warning and a request},
 BookTitle = {Arithmetic geometry: computation and applications. 16th international conference on arithmetic, geometry, cryptography, and coding theory, AGC2T, CIRM, Marseille, France, June 19--23, 2017. Proceedings},
 ISBN = {978-1-4704-4212-5; 978-1-4704-5104-2},
 Pages = {1--18},
 Year = {2019},
 Publisher = {Providence, RI: American Mathematical Society (AMS)},
 Language = {English},
 DOI = {10.1090/conm/722/14534},
 zbMATH = {7117888},
 Zbl = {1439.11145}
}

@Article{Robert1980,
 Author = {Robert, Gilles},
 Title = {Congruences entre s{\'e}ries d'{Eisenstein}, dans le cas supersingulier},
 FJournal = {Inventiones Mathematicae},
 Journal = {Invent. Math.},
 ISSN = {0020-9910},
 Volume = {61},
 Pages = {103--158},
 Year = {1980},
 Language = {French},
 DOI = {10.1007/BF01390118},
 zbMATH = {3689483},
 Zbl = {0442.10020}
}

@Misc{Katz1977,
 Author = {Katz, Nicholas M.},
 Title = {A result on modular forms in characteristic p},
 Year = {1977},
 Language = {English},
 HowPublished = {Modular {Funct}. one {Var}. {V}, {Proc}. int. {Conf}., {Bonn} 1976, {Lect}. {Notes} {Math}. 601, 53-61 (1977).},
 zbMATH = {3608125},
 Zbl = {0392.10026}
}

@Article{Edixhoven1992,
 Author = {Edixhoven, Bas},
 Title = {The weight in {Serre}'s conjectures on modular forms},
 FJournal = {Inventiones Mathematicae},
 Journal = {Invent. Math.},
 ISSN = {0020-9910},
 Volume = {109},
 Number = {3},
 Pages = {563--594},
 Year = {1992},
 Language = {English},
 DOI = {10.1007/BF01232041},
 zbMATH = {220751},
 Zbl = {0777.11013}
}

@Article{Katz2021,
 Author = {Katz, Nicholas M.},
 Title = {On a question of {Zannier}},
 FJournal = {Experimental Mathematics},
 Journal = {Exp. Math.},
 ISSN = {1058-6458},
 Volume = {30},
 Number = {3},
 Pages = {422--428},
 Year = {2021},
 Language = {English},
 DOI = {10.1080/10586458.2018.1551818},
 zbMATH = {7419661},
 Zbl = {1485.11098}
}

@Misc{Clemens1980,
 Author = {Clemens, C. Herbert},
 Title = {A scrapbook of complex curve theory},
 Year = {1980},
 Language = {English},
 HowPublished = {The {University} {Series} in {Mathematics}. {New} {York}, {London}: {Plenum} {Press}. {IX}, 186 p. \$ 22.50 (1980).},
 zbMATH = {3713857},
 Zbl = {0456.14016}
}

@Article{Chiarellotto1995,
 Author = {Chiarellotto, Bruno},
 Title = {On {Lam{\'e}} operators which are pull-backs of hypergeometric ones},
 FJournal = {Transactions of the American Mathematical Society},
 Journal = {Trans. Am. Math. Soc.},
 ISSN = {0002-9947},
 Volume = {347},
 Number = {8},
 Pages = {2753--2780},
 Year = {1995},
 Language = {English},
 DOI = {10.2307/2154754},
 zbMATH = {870203},
 Zbl = {0851.34024}
}

@Article{Dahmen2007,
 Author = {Dahmen, Sander R.},
 Title = {Counting integral {Lam{\'e}} equations by means of dessins d'enfants},
 FJournal = {Transactions of the American Mathematical Society},
 Journal = {Trans. Am. Math. Soc.},
 ISSN = {0002-9947},
 Volume = {359},
 Number = {2},
 Pages = {909--922},
 Year = {2007},
 Language = {English},
 DOI = {10.1090/S0002-9947-06-03924-9},
 zbMATH = {5120534},
 Zbl = {1131.34060}
}

@Book{WaallThesis,
	AUTHOR = "Alexa van der Waall",
	TITLE = "Lamé Equations with Finite Monodromy",
	NOTE = "\href{https://dspace.library.uu.nl/bitstream/handle/1874/865/full.pdf?isAllowed=y&sequence=1}{Ph.D. thesis}",
	PUBLISHER = "",
	YEAR = "2002",
	PAGES = "",
	ISBN = "",
	MRCLASS = "Thesis",
	MRNUMBER = ""
}

@misc{Hofmann2012,
      title={Uniformizing differential equations of arithmetic (1;e)-groups}, 
      author={Jörg Hofmann},
      year={2012},
      eprint={1205.5680},
      archivePrefix={arXiv},
      primaryClass={math.CA},
      url={https://arxiv.org/abs/1205.5680}, 
}

@Book{KatzMazur1985,
 Author = {Katz, Nicholas M. and Mazur, Barry},
 Title = {Arithmetic moduli of elliptic curves},
 FSeries = {Annals of Mathematics Studies},
 Series = {Ann. Math. Stud.},
 Volume = {108},
 Year = {1985},
 Publisher = {Princeton University Press, Princeton, NJ},
 Language = {English},
 DOI = {10.1515/9781400881710},
 zbMATH = {3920653},
 Zbl = {0576.14026}
}

@Misc{Katz1973,
 Author = {Katz, N.},
 Title = {Une formule de congruence pour la fonction {{\(\zeta\)}}},
 Year = {1973},
 Language = {French},
 HowPublished = {Sem. {Geom}. algebrique, {Bois}-{Marie}, 1967-1969, {SGA} 7 {II}, {Lect}. {Notes} {Math}. 340, {Expose} {XXII}, 401-438 (1973).},
 zbMATH = {3432446},
 Zbl = {0275.14015}
}

@Book{Serre1968,
 Author = {Serre, Jean-Pierre},
 Title = {Abelian {{\(\ell\)}}-adic representations and elliptic curves},
 FSeries = {Research Notes in Mathematics},
 Series = {Res. Notes Math.},
 Volume = {7},
 ISBN = {1-56881-077-6},
 Year = {1998},
 Publisher = {Wellesley, MA: A K Peters},
 Language = {English},
 zbMATH = {1101830},
 Zbl = {0902.14016}
}

@Article{Faltings1983,
 Author = {Faltings, G.},
 Title = {Finiteness theorems for abelian varieties over number fields.},
 FJournal = {Inventiones Mathematicae},
 Journal = {Invent. Math.},
 ISSN = {0020-9910},
 Volume = {73},
 Pages = {349--366},
 Year = {1983},
 Language = {German},
 DOI = {10.1007/BF01388432},
 zbMATH = {3944027},
 Zbl = {0588.14026}
}

@InCollection{GeerKatsura2002,
 Author = {van der Geer, G. and Katsura, T.},
 Title = {An invariant for varieties in positive characteristic.},
 BookTitle = {Algebraic number theory and algebraic geometry. Papers dedicated to A. N. Parshin on the occasion of his sixtieth birthday},
 ISBN = {0-8218-3267-0},
 Pages = {131--141},
 Year = {2002},
 Publisher = {Providence, RI: American Mathematical Society (AMS)},
 Language = {English},
 zbMATH = {1857794},
 Zbl = {1048.14026}
}

@article {DeligneIllusie1987,
    AUTHOR = {Deligne, Pierre and Illusie, Luc},
     TITLE = {Rel\`evements modulo {$p^2$} et d\'ecomposition du complexe de
              de {R}ham},
   JOURNAL = {Invent. Math.},
  FJOURNAL = {Inventiones Mathematicae},
    VOLUME = {89},
      YEAR = {1987},
    NUMBER = {2},
     PAGES = {247--270},
      ISSN = {0020-9910,1432-1297},
   MRCLASS = {14F40 (14C30)},
  MRNUMBER = {894379},
MRREVIEWER = {Thomas\ Zink},
       DOI = {10.1007/BF01389078},
       URL = {https://doi.org/10.1007/BF01389078},
}

@misc{BostanCarusoRoques2024,
      title={Algebraic solutions of linear differential equations: an arithmetic approach}, 
      author={Alin Bostan and Xavier Caruso and Julien Roques},
      year={2024},
      eprint={2304.05061},
      archivePrefix={arXiv},
      primaryClass={math.NT},
      url={https://arxiv.org/abs/2304.05061}, 
}

@article {Ogus2001,
    AUTHOR = {Ogus, Arthur},
     TITLE = {On the {H}asse locus of a {C}alabi-{Y}au family},
   JOURNAL = {Math. Res. Lett.},
  FJOURNAL = {Mathematical Research Letters},
    VOLUME = {8},
      YEAR = {2001},
    NUMBER = {1-2},
     PAGES = {35--41},
      ISSN = {1073-2780},
   MRCLASS = {14F40 (14J10 14J32)},
  MRNUMBER = {1825258},
MRREVIEWER = {Elmar\ Grosse-Kl\"onne},
       DOI = {10.4310/MRL.2001.v8.n1.a5},
       URL = {https://doi.org/10.4310/MRL.2001.v8.n1.a5},
}

@article {LenstraStevenhagen,
    AUTHOR = {Stevenhagen, P. and Lenstra, Jr., H. W.},
     TITLE = {Chebotar\"ev and his density theorem},
   JOURNAL = {Math. Intelligencer},
  FJOURNAL = {The Mathematical Intelligencer},
    VOLUME = {18},
      YEAR = {1996},
    NUMBER = {2},
     PAGES = {26--37},
      ISSN = {0343-6993,1866-7414},
   MRCLASS = {11R44 (01A70 11-03 11R27 11R45)},
  MRNUMBER = {1395088},
MRREVIEWER = {M.\ Ram\ Murty},
       DOI = {10.1007/BF03027290},
       URL = {https://doi.org/10.1007/BF03027290},
}

@article {Elkies1987,
    AUTHOR = {Elkies, Noam D.},
     TITLE = {The existence of infinitely many supersingular primes for
              every elliptic curve over {${\bf Q}$}},
   JOURNAL = {Invent. Math.},
  FJOURNAL = {Inventiones Mathematicae},
    VOLUME = {89},
      YEAR = {1987},
    NUMBER = {3},
     PAGES = {561--567},
      ISSN = {0020-9910,1432-1297},
   MRCLASS = {11G05 (14G25)},
  MRNUMBER = {903384},
MRREVIEWER = {David\ Grant},
       DOI = {10.1007/BF01388985},
       URL = {https://doi.org/10.1007/BF01388985},
}

@misc{LamLitt2025,
      title={Algebraicity and integrality of solutions to differential equations}, 
      author={Yeuk Hay Joshua Lam and Daniel Litt},
      year={2025},
      eprint={2501.13175},
      archivePrefix={arXiv},
      primaryClass={math.AG},
      url={https://arxiv.org/abs/2501.13175}, 
}
\bibliographystyle{alpha}



\end{document}